\documentclass{amsart}

\usepackage{amsfonts, amsmath, amssymb}
\usepackage{graphicx, epic, enumerate}
\usepackage{float}
\usepackage{xcolor}
\usepackage{gensymb}
\usepackage{cellspace}
\newtheorem{theorem}{Theorem}
\newtheorem{lemma}{Lemma}
\newtheorem{corollary}{Corollary}

\theoremstyle{definition}
\newtheorem{definition}{Definition}
\newtheorem{remark}{Remark}
\newtheorem{example}{Example}

\usepackage{ifxetex}
\ifxetex
  \usepackage{fontspec}
\else
  \usepackage[T1]{fontenc}
  \usepackage[utf8]{inputenc}
  \usepackage{lmodern}
\fi

\begin{document}

\title{The $L$-move and Markov theorems for trivalent braids}
\author{Carmen Caprau}
\address{Department of Mathematics, California State University, Fresno, CA 93740}
\email{ccaprau@csufresno.edu}

\author{Gabriel Coloma}
\address{Departamento de Matem\'{a}ticas, Universidad de Puerto Rico, Mayag\"{u}ez, Puerto Rico 00681}
\email{gabriel.coloma@upr.edu}

\author{Marguerite Davis}
\address{Department of Mathematics, Ithaca College,  Ithaca, NY 14850}
\email{mdavis7@ithaca.edu}

\subjclass[2010]{57M25, 57M15; 20F36}
\keywords{L-moves, Markov-type moves, spatial trivalent graphs, trivalent braids}
\thanks{This work was supported by NSF Grant DMS-1460151 and Simons Foundation grant $\#$355640}

\begin{abstract}
The $L$-move for classical braids extends naturally to trivalent braids. We follow the $L$-move approach to the Markov Theorem, to prove a one-move Markov-type theorem for trivalent braids. We also reformulate this $L$-Move Markov theorem and prove a more algebraic Markov-type theorem for trivalent braids. Along the way, we provide a proof of the Alexander's theorem analogue for spatial trivalent graphs and trivalent braids.
\end{abstract}

\maketitle

\section{Introduction}

  The Alexander~\cite{A} and Markov~\cite{M} Theorems are fundamental results in classical knot theory. Alexander Theorem states that any oriented link is isotopic to the closure of some braid (which is not unique). Markov Theorem characterizes braids that yield isotopic links via the closure operation. Specifically, the closures of two classical braids represent isotopic links if and only if the braids are related by a finite sequence consisting of braid isotopy and two additional moves, usually referred to as \textit{Markov moves}: conjugation by a crossing and bottom right stabilization.
  
  There is another type of braid move, the so-called $L$-move, which was introduced by S. Lambropoulou in~\cite{L} (see also~\cite{LR}). This move replaces the two moves of the Markov equivalence and yields a `one-move Markov-type theorem' for oriented links.

  As an extension of classical knot theory, spatial graph theory seeks to classify, up to isotopy, embeddings of graphs in three-space. In this paper we focus on oriented spatial trivalent graphs whose vertices are neither sources nor sinks. Just as oriented links can be represented by diagrams that are closures of braids, so can oriented spatial trivalent graphs be represented by closures of trivalent braids (with the same number of top and bottom endpoints). We borrow the $L$-move approach and show that the $L$-move can be extended to the setting of trivalent braids. Then we prove that this type of move and trivalent braid isotopy are sufficient to prove an $L$-move Markov-type theorem for trivalent braids. With this theorem at hand, we are able to state and prove an algebraic Markov-type theorem for trivalent braids.

  We remark that the $L$-move was extended to other diagrammatic situations, including virtual braids~\cite{KL2}, virtual singular braids~\cite{CPM}, and virtual trivalent braids~\cite{CDS}.

  The paper is organized as follows: We start with a brief discussion about spatial trivalent graphs, trivalent braids, and trivalent braid isotopy; this is done in Section~\ref{sec:braidsintro}. In Section~\ref{sec:AlexanderThm} we describe our preparation for braiding and braiding algorithm for spatial trivalent graphs, which implicitly proves the Alexander-type theorem for spatial trivalent graphs. We introduce in Section~\ref{sec:TL-equiv} the $TL$-equivalence among trivalent braids, and use it in Section~\ref{sec:ProofLMarkov} to prove our one-move Markov-type theorem based on the $L$-move for trivalent braids. Finally, we close with Section~\ref{sec:ProofAMarkov}, where we state and prove a more algebraic Markov-type theorem for trivalent braids.

\section{Trivalent braids and spatial trivalent graphs} \label{sec:braidsintro}

  In this section, we briefly review spatial trivalent graphs and trivalent braids.

  A trivalent graph is a finite graph whose vertices have valency three and a \textbf{spatial trivalent graph} (or shortly STG) is a trivalent graph embedded in $\mathbb{R}^3$. Two spatial trivalent graphs are called \textbf{ambient isotopic} if there exists an orientation-preserving self-homeomorphism on $\mathbb{R}^3$ taking one graph onto the other.

  When studying STGs we work with their diagrams. A \textbf{diagram} of a spatial trivalent graph is a projection of a spatial trivalent graph into a plane. It is well-known that two spatial trivalent graphs are ambient isotopic if and only if their diagrams are related by planar isotopy and a finite sequence of the moves $R1-R5$
  depicted in Figure~\ref{fig:graph-moves} (see for example~\cite{Ka1}). Note that for the $R3$ move, the sliding strand may be either an understrand or overstrand. We refer to these moves as the \textbf{extended Reidemeister moves} for STG diagrams. In addition, if two STG diagrams differ by a finite sequence of the extended Reidemeister moves we refer to them as being \textbf{isotopic} (or \textbf{equivalent}).

  \begin{figure}[ht]
    \[\raisebox{-13pt}{\includegraphics[height=0.4in]{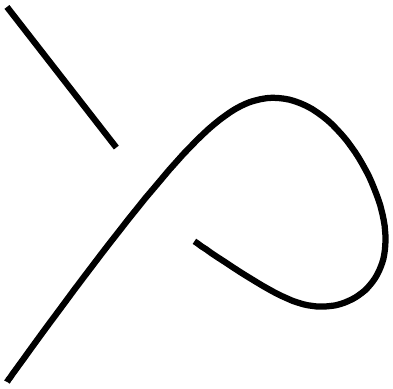}}\,
    \stackrel{R1}{\longleftrightarrow} \,
    \raisebox{-13pt}{\includegraphics[height=0.4in]{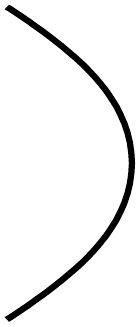}} \,
    \stackrel{R1}{\longleftrightarrow} \,
    \reflectbox{\raisebox{17pt}{\includegraphics[height=0.4in, angle =
        180]{poskink}}} \,\,\qquad\,\,
    \raisebox{-13pt}{\includegraphics[height=0.4in]{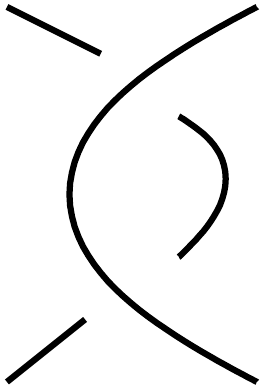}} \,
    \stackrel{R2}{\longleftrightarrow} \,
    \raisebox{-13pt}{\includegraphics[height=0.4in]{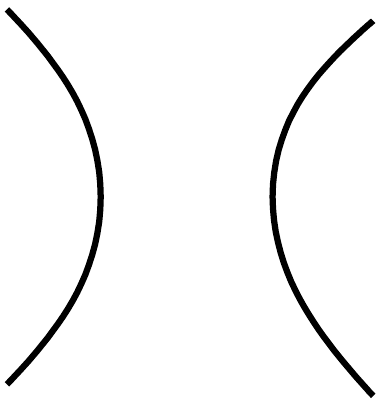}} \,\,
    \qquad \,\,
    \raisebox{-13pt}{\includegraphics[height=0.4in]{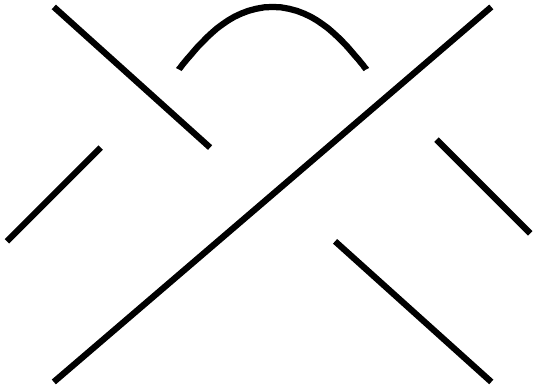}} \,
    \stackrel{R3}{\longleftrightarrow} \,
    \raisebox{-13pt}{\includegraphics[height=0.4in]{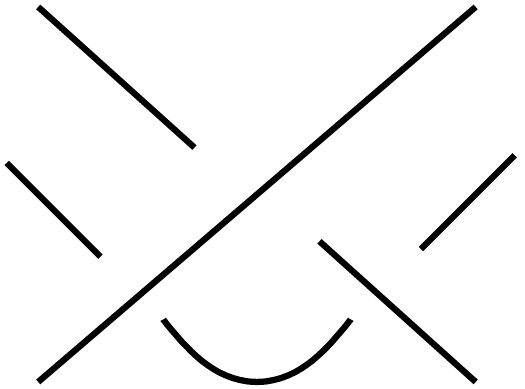}}\]
    \[ \raisebox{-13pt}{\includegraphics[height=0.45in]{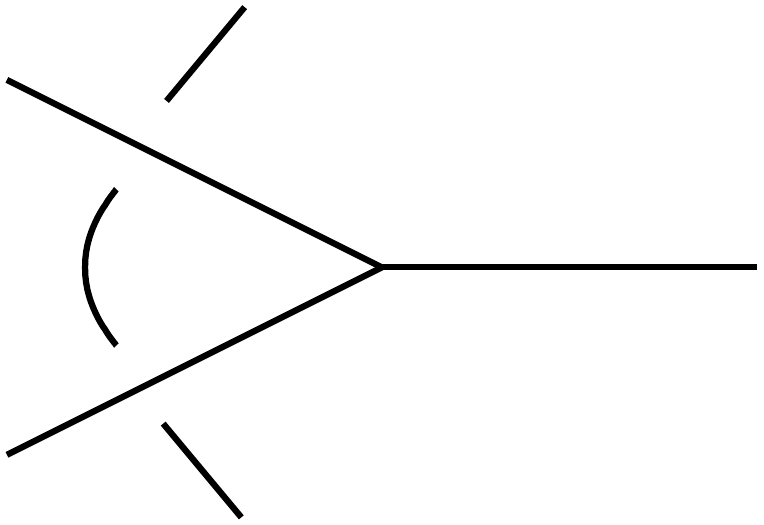}}\,\,
    \stackrel{R4}{\longleftrightarrow} \,\,
    \raisebox{-13pt}{\includegraphics[height=0.45in]{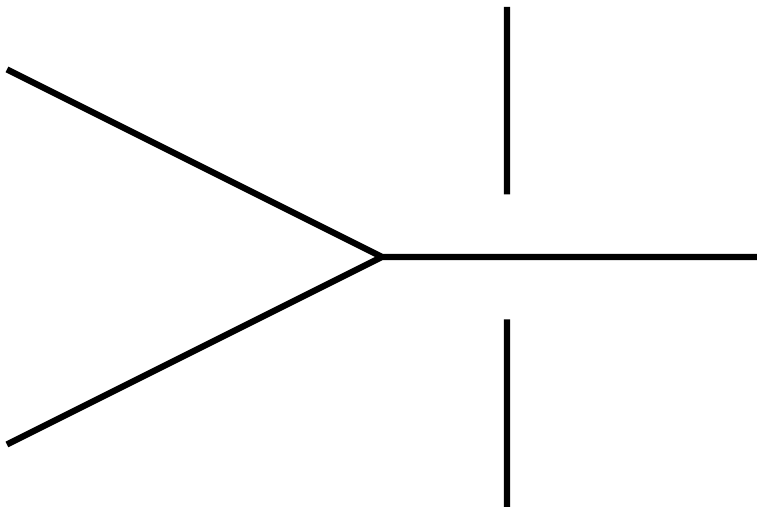}} \,\,
    \qquad \,\,
    \raisebox{-13pt}{\includegraphics[height=0.45in]{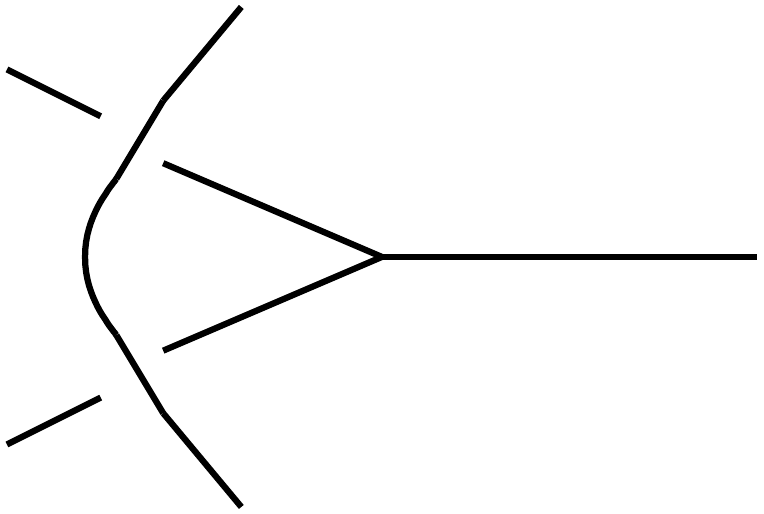}} \,\,
    \stackrel{R4}{\longleftrightarrow}\,\,
    \raisebox{-13pt}{\includegraphics[height=0.45in]{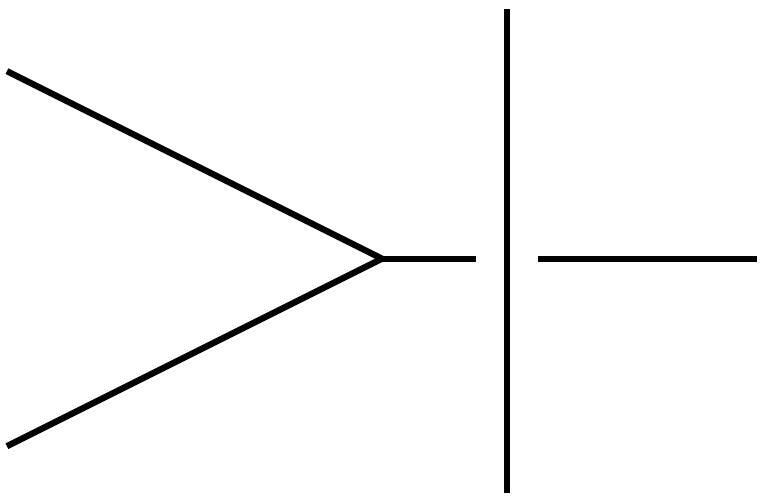}} \]
    \[
    \raisebox{-13pt}{\includegraphics[height=0.4in]{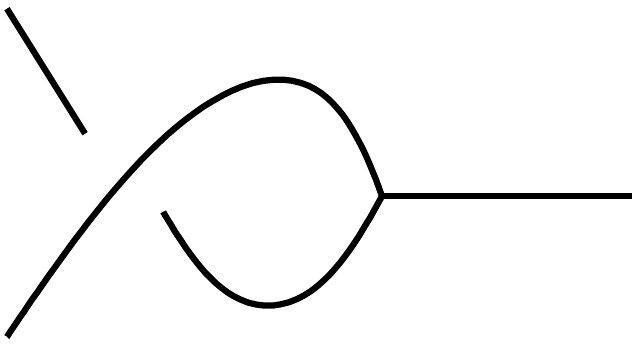}}\,\,\stackrel{R5}{
      \longleftrightarrow} \,\,
    \raisebox{-13pt}{\includegraphics[height=0.4in]{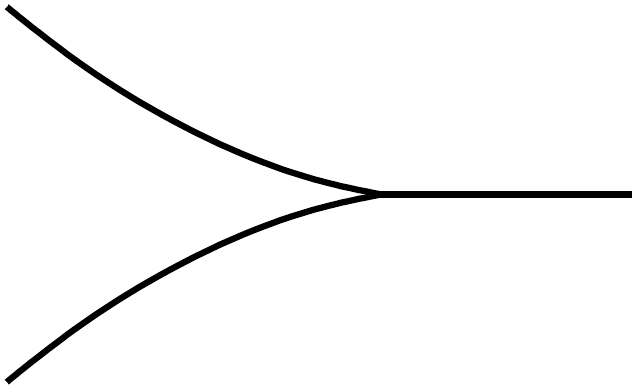}} \,\,
    \stackrel{R5}{\longleftrightarrow} \,\,
    \raisebox{-13pt}{\includegraphics[height=0.4in]{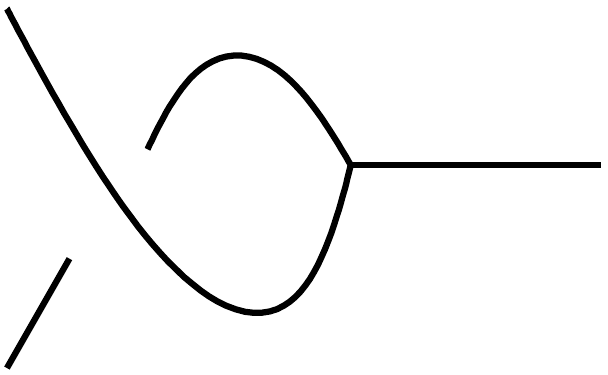}}\]
    \caption{Extended Reidemeister moves for STG
      diagrams} \label{fig:graph-moves}
  \end{figure}

  A trivalent tangle can be regarded as a local region in a spatial trivalent graph diagram. A \textbf{trivalent braid} is a trivalent tangle in braid form. We denote by $TB^m_n$ the set of trivalent braids with $m$ top endpoints and $n$ bottom endpoints, and we refer to an element in $TB^m_n$ as an $(m, n)$ trivalent braid. See, for instance, the $(3,2)$ braid in Figure \ref{32braid}.

\begin{figure}[ht]
  \centering
  \includegraphics[height=0.7in]{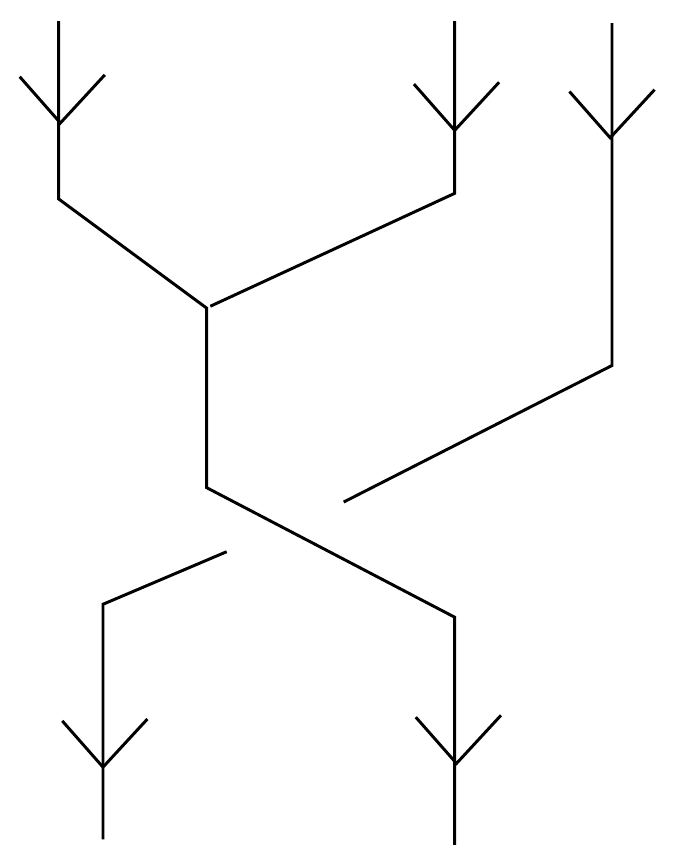}
  \caption{Example of a $(3,2)$ braid.}
  \label{32braid}
\end{figure}

The \textbf{closure} $\overline{b}$ of an $(n, n)$ trivalent braid $b$ is the STG diagram obtained by connecting the $n$ top endpoints of $b$ with the corresponding bottom endpoints, drawing $n$ non-intersecting arcs (see Figure \ref{clos}).

\begin{figure}
  \[
  \raisebox{20pt}{\includegraphics[height = 0.8in]{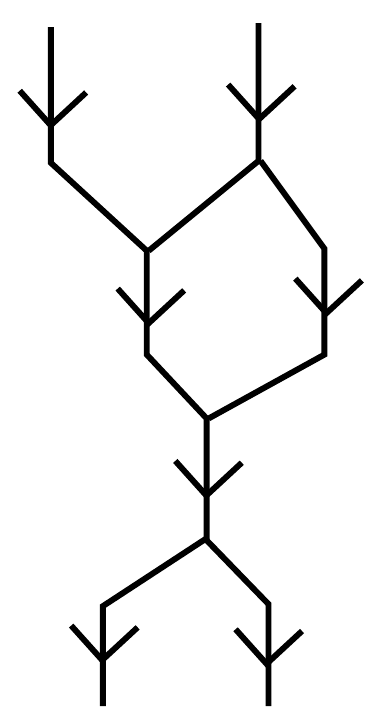}}
  \hspace{1cm}
  \raisebox{40pt}{$\underrightarrow{\text{Closure}}$} 
  \hspace{1cm}
  \includegraphics[height = 1.2in]{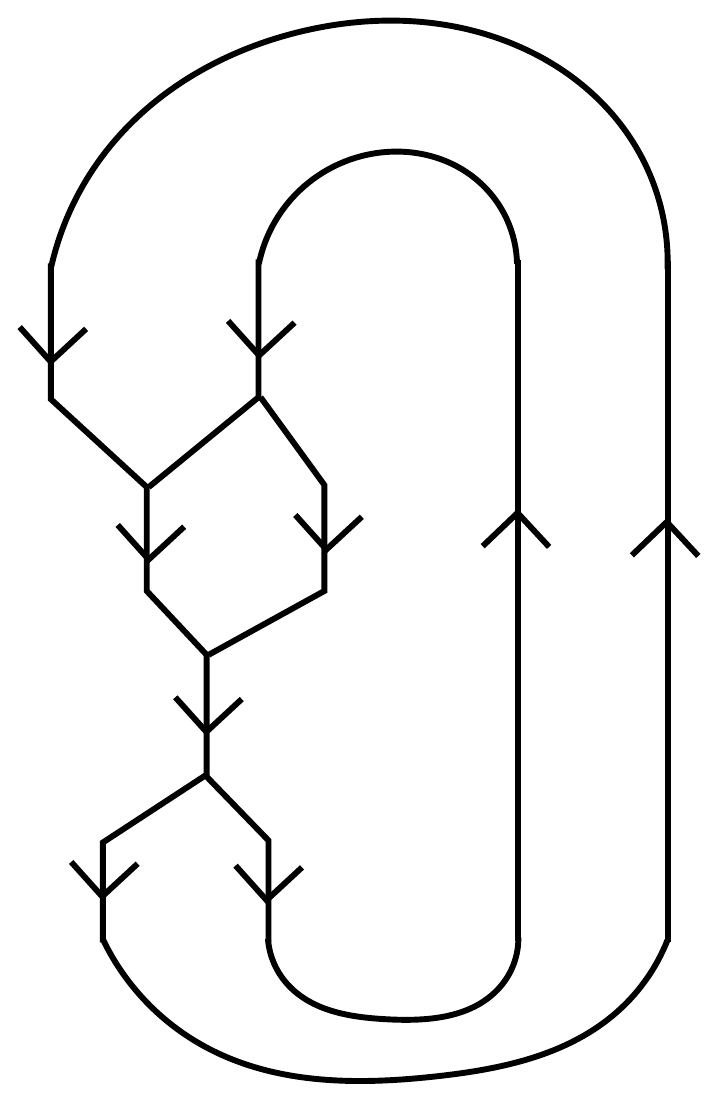}
  \]
  \caption{Example of closure operation}
  \label{clos}
\end{figure}

If $b_1 \in TB^m_n$ and $b_2 \in TB^n_s$, then we can compose $b_1$ with $b_2$. The \textbf{composition} $b_1b_2$ is the trivalent braid obtained by placing $b_1$ on top of $b_2$ and connecting the bottom endpoints of $b_1$ with the top endpoints of $b_2$. Note that $b_1b_2 \in TB^m_s$ (see Figure \ref{comp}).

\begin{figure}[ht]
  \[
  \raisebox{0pt}{\includegraphics[height=1in]{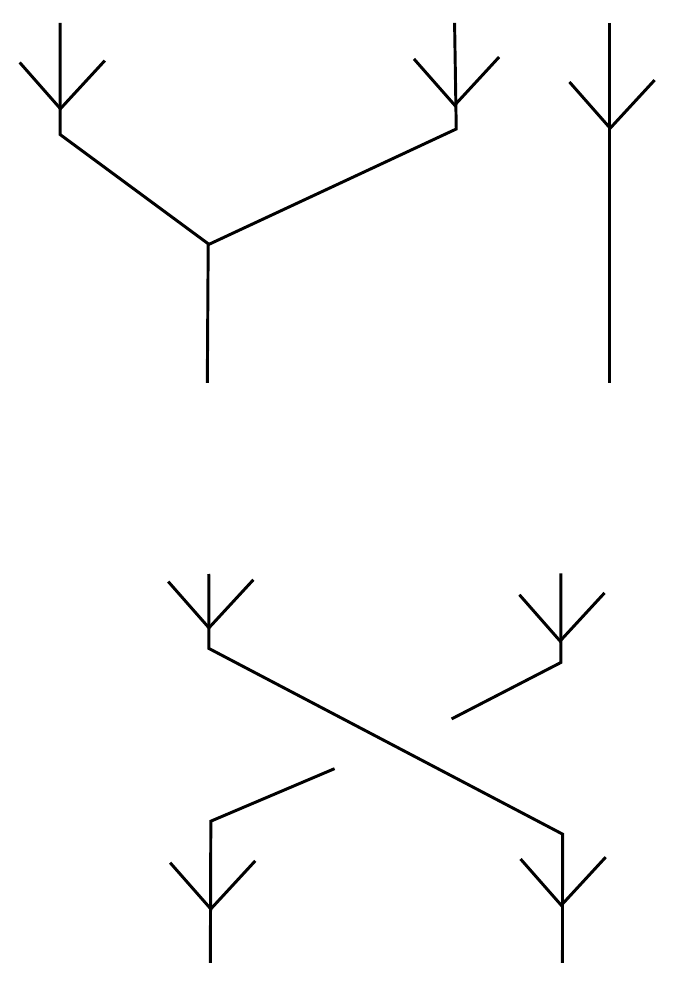}} \hspace{1cm}
  \raisebox{40pt}{$\longrightarrow$} \hspace{1cm}
  \raisebox{10pt}{\includegraphics[height=0.8in]{BraidAlg/Notabraid.pdf}}
  \put(-180, 60){\fontsize{11}{11}\text{$b_1$}}
   \put(-180, 15){\fontsize{11}{11}\text{$b_2$}}
    \put(5, 40){\fontsize{11}{11}\text{$b_1b_2$}}
  \]
  \caption{Composition of trivalent braids}
  \label{comp}
\end{figure}

The \textbf{identity $(n, n)$ braid}, denoted by $1_n$, is the braid with $n$ parallel strands free of crossings or vertices. In addition, the braids $\sigma_i, \sigma^{-1}_i$ and  $y_i, \lambda_{i}$ depicted in Figure~\ref{fig:ElementaryBraids} are called \textbf{elementary trivalent braids}; the only restriction for the index $i$ is that it must be less than the number of braid strands at the horizontal level where such elementary trivalent braid is located (note that a trivalent braid in $TB^m_n$ may contain an arbitrary number of strands at some horizontal level). With these at hand, any trivalent braid can be regarded as a composition of elementary trivalent braids, and the corresponding representation of the braid is called a `word'.

\begin{figure}[ht]
  \[ \sigma_i \,\,\, =\,\,\,
  \raisebox{-17pt}{\includegraphics[height=.5in]{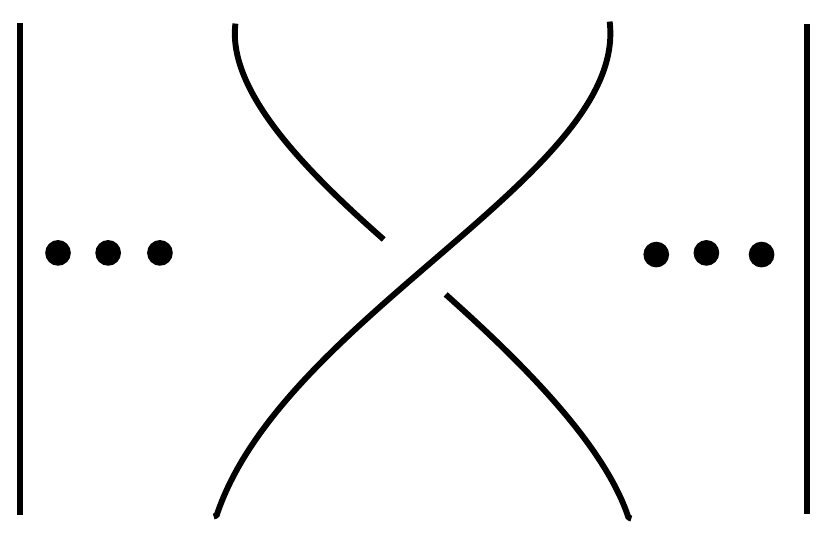}} \hspace{1cm}
  \sigma_i^{-1} \,\,\,=\,\,\,
  \raisebox{-17pt}{\includegraphics[height=.5in]{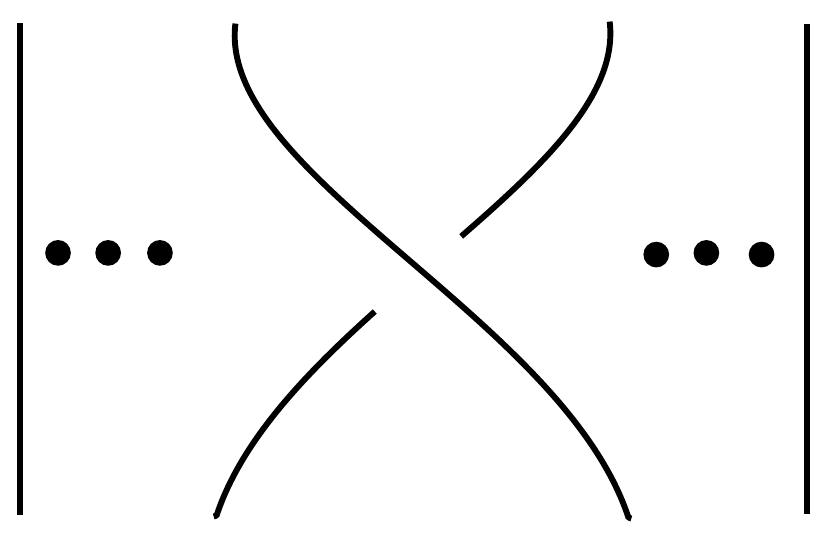}} \put(-180,
  21){\fontsize{7}{7}$1$} \put(-165,
  21){\fontsize{7}{7}$i$} \put(-150,
  21){\fontsize{7}{7}$i+1$}
  \put(-127,21){\fontsize{7}{7}$n$} \put(-58,
  21){\fontsize{7}{7}$1$} \put(-40,
  21){\fontsize{7}{7}$i$} \put(-25,
  21){\fontsize{7}{7}$i+1$} \put(-3,21){\fontsize{7}{7}$n$}
  \]

  \[y_i \,\, =\,\, \raisebox{-17pt}{\includegraphics[height=.5in,
    width=0.8in]{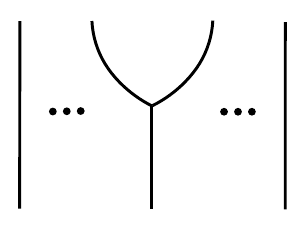}} \hspace{1cm} \lambda_i \,\,\,=\,\,\,
  \raisebox{-17pt}{\includegraphics[height=.5in,
    width=0.83in]{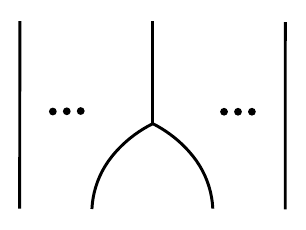}} \put(-177,
  21){\fontsize{7}{7}$1$} \put(-162,
  21){\fontsize{7}{7}$i$} \put(-146,
  21){\fontsize{7}{7}$i+1$}
  \put(-125,21){\fontsize{7}{7}$n$} \put(-58,
  21){\fontsize{7}{7}$1$} \put(-31,
  21){\fontsize{7}{7}$i$}
  \put(-10,21){\fontsize{7}{7}$n-1$} \put(-177,
  -21){\fontsize{7}{7}$1$} \put(-152,
  -21){\fontsize{7}{7}$i$}
  \put(-130,-21){\fontsize{7}{7}$n-1$} \put(-58,
  -21){\fontsize{7}{7}$1$} \put(-45,
  -21){\fontsize{7}{7}$i$} \put(-27,
  -21){\fontsize{7}{7}$i+1$} \put(-5,-21){\fontsize{7}{7}$n$}
  \]
  \caption{Elementary trivalent braids} \label{fig:ElementaryBraids}
\end{figure}
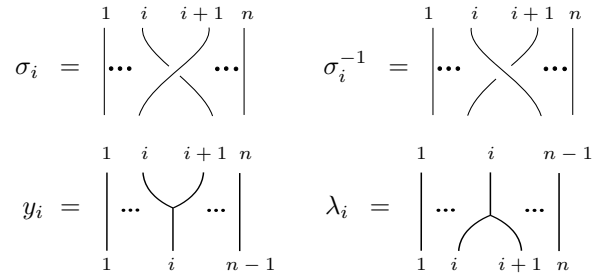

Similar to the case of classical braids, trivalent braids are considered up to isotopy. Two $(m, n)$ trivalent braids are called \textbf{isotopic} if they are related by a finite sequence of the following replacements for subwords (note that for each such move below there is also the variant which uses $\sigma_i^{-1}$):

\begin{itemize}
\item $\sigma_i\sigma_i^{-1}=\sigma_i^{-1}\sigma_i=1_n$
  \[ \raisebox{-13pt}{\includegraphics[height=0.5in]{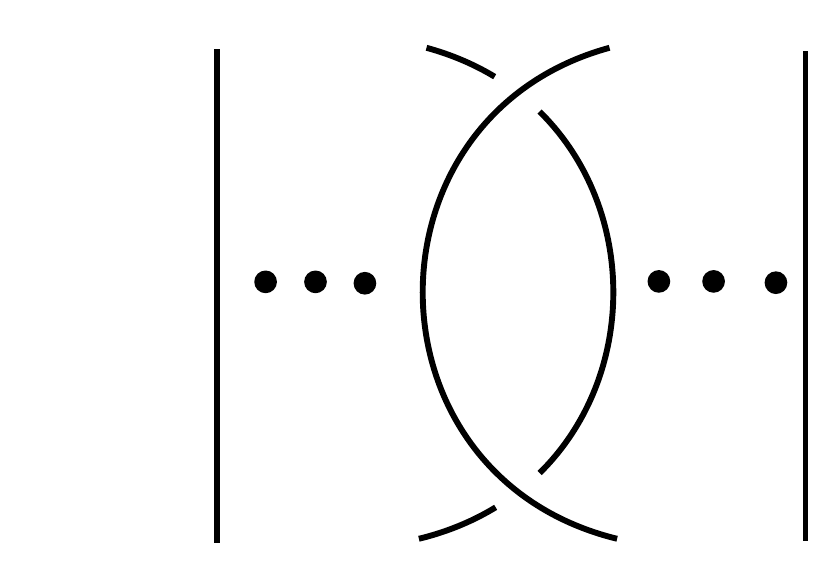}}
  \hspace{0.2cm} \stackrel{R2}{\sim} \hspace{0.2cm} \raisebox{-14
    pt}{\includegraphics[height=0.5in]{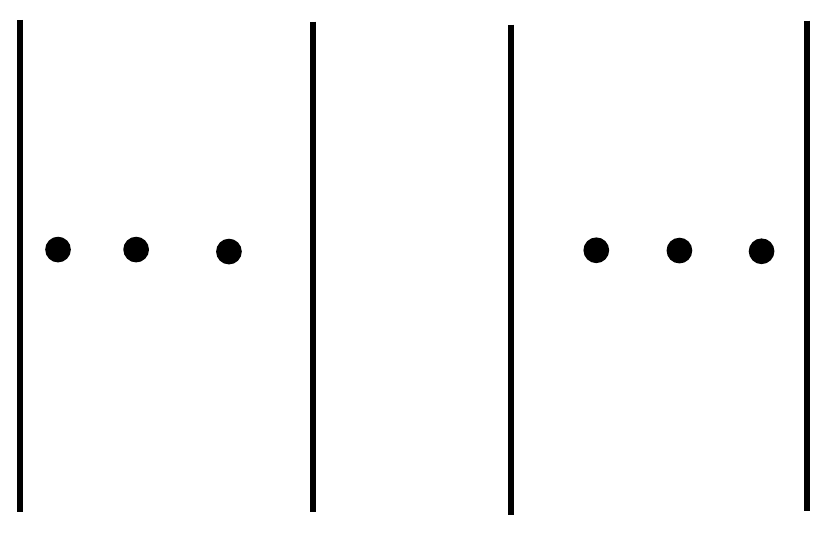}} \]

\item $\sigma_i\sigma_{i+1}\sigma_i=\sigma_{i+1}\sigma_i\sigma_{i+1}$
  \[ \raisebox{-.5cm}{\includegraphics[height=0.5in]{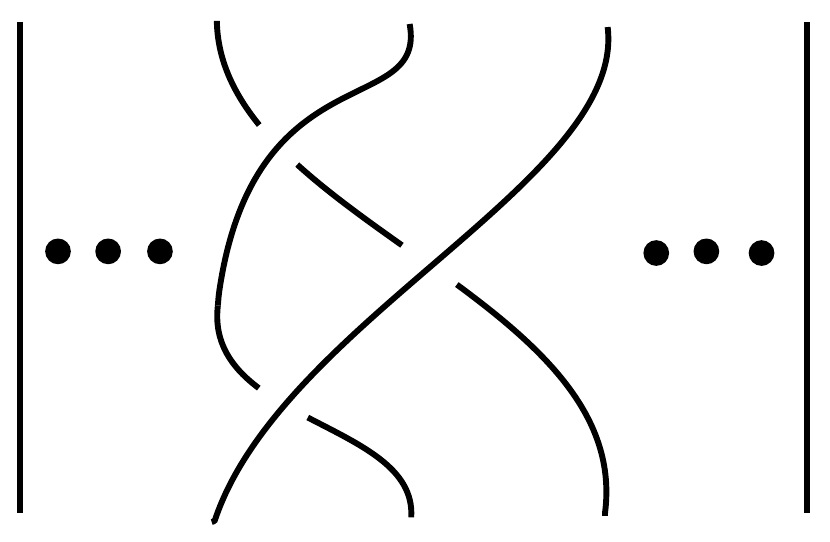}}
  \hspace{0.2cm} \stackrel{R3}{\sim} \hspace{0.2cm}
  \raisebox{-.5cm}{\includegraphics[height=0.5in]{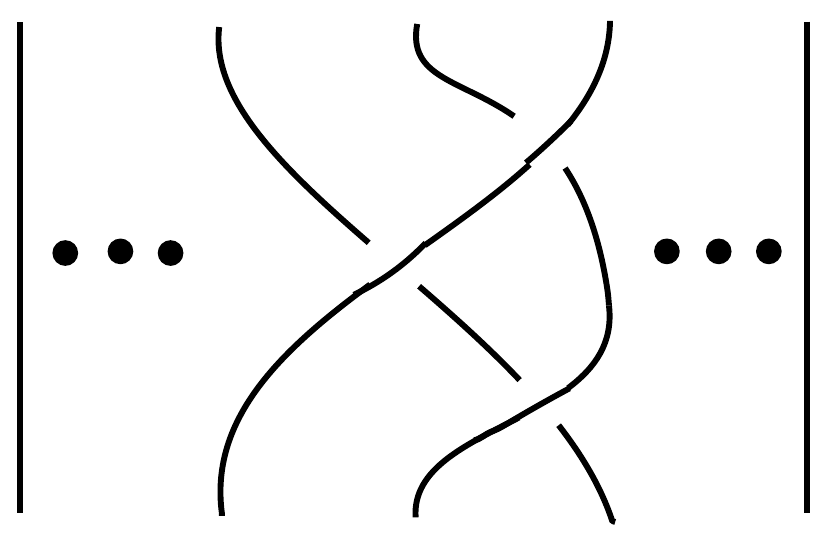}} \]

\item $y_i = \sigma_i y_i$, \hspace{5cm}
  $\lambda_i = \lambda_i \sigma_i$
  \[ \raisebox{-15pt}{\includegraphics[height=0.5in]{y1}}\,\,
  \stackrel{R5}{\sim} \,\,
  \raisebox{-15pt}{\includegraphics[height=0.5in]{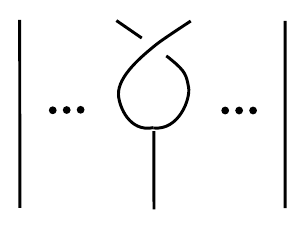}}
  \hspace{0.45in}
  \raisebox{-15pt}{\includegraphics[height=0.5in]{lambda1}}\,\,
  \stackrel{R5}{\sim} \,\,
  \raisebox{-15pt}{\includegraphics[height=0.5in, width=0.7in]{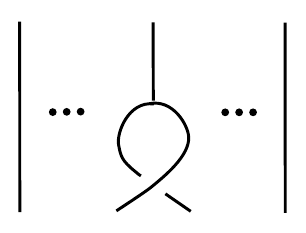}}
  \]

\item $\sigma_i \sigma_{i+1} y_i = y_{i+1} \sigma_i$, \hspace{3.5cm}
  $\sigma_{i+1} \sigma_i y_{i+1} = y_i \sigma_i$
  \[
  \raisebox{20pt}{\includegraphics[height=0.5in, angle =
    180]{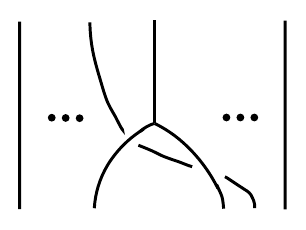}}\,\, \stackrel{R4}{\sim} \,\,
  \raisebox{20pt}{\includegraphics[height=0.5in, angle = 180]{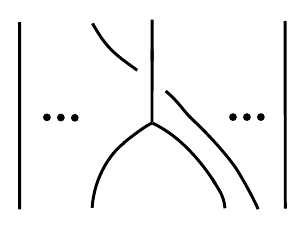}}
  \hspace{0.45in}
  \raisebox{-15pt}{\includegraphics[height=0.5in]{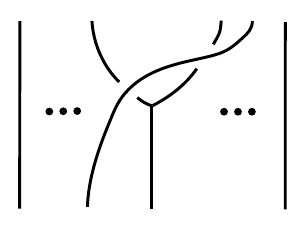}}\,\,
  \stackrel{R4}{\sim} \,\,
  \raisebox{-15pt}{\includegraphics[height=0.5in]{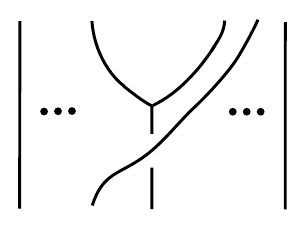}}
  \]

\item $\lambda_i \sigma_{i+1} \sigma_i = \sigma_i \lambda_{i+1}$,
  \hspace{3.5cm}
  $\lambda_{i+1} \sigma_i \sigma_{i+1} = \sigma_i \lambda_i$
  \[\raisebox{20pt}{\includegraphics[height=0.5in, angle =
    180]{R41}}\,\, \stackrel{R4}{\sim} \,\,
  \raisebox{20pt}{\includegraphics[height=0.5in, angle = 180]{R42}}
  \hspace{0.45in}
  \raisebox{-15pt}{\includegraphics[height=0.5in]{LR42a}}\,\,
  \stackrel{R4}{\sim} \,\,
  \raisebox{-15pt}{\includegraphics[height=0.5in]{LR41a}}
  \]
\item Commuting relations:
  \[\sigma_i b_{j} = b_j \sigma_i, \,\, \, y_i b_{j-1} = b_j y_i, \,\,
  \, \lambda_i b_j = b_{j-1} \lambda_i, \,\, \text{where} \,\,i+1 <
  j \]
  and
  $b_i \in \{ \sigma_i^{\pm 1}, y_i, \lambda_i\}$.
\end{itemize}

In this paper, we work with oriented STG diagrams whose vertices are either zip or unzip vertices (see Figure~\ref{fig:orientations}).  A \textbf{zip vertex} is a trivalent vertex with two of its edges oriented toward it and one edge oriented away from it. On the other hand, an \textbf{unzip vertex} has one edge oriented toward it and two edges oriented away from it. We do not allow \textbf{sink} or \textbf{source} vertices, where all edges are oriented toward or, respectively, away from it.

\begin{figure}[ht]
  \[ {\includegraphics[height=0.7in]{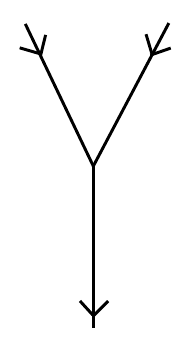}} \hspace{2cm}
  {\includegraphics[height=0.7in]{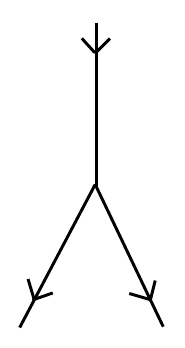}} \put(-105,
  -10){\fontsize{7}{7}\text{zip}} \put(-25,
  -10){\fontsize{7}{7}\text{unzip}}
  \]
  \caption{Allowed orientations near a
    vertex} \label{fig:orientations}
\end{figure}

We say that a spatial trivalent graph is \textbf{well-oriented} if it contains only zip and unzip vertices. We remark that any STG can be well oriented (a proof of this can be found, for example, in~\cite{V}).

We use the convention that trivalent braids have downward orientation. By the handshaking lemma, every (unoriented) trivalent graph has an even number of vertices (all of degree three). Similarly, an $(n,n)$ trivalent braid has an even number of vertices. It is easy to show that for a well-oriented STG and an $(n,n)$ trivalent braid, each has half of its vertices zipped and half unzipped.


\section{Alexander-type theorem for trivalent braids} \label{sec:AlexanderThm}

  In this section we shall generalize the Alexander Theorem~\cite{A} for classical knots and links to spatial trivalent graphs. We remark that the Alexander theorem for oriented spatial graphs was first proved by K. Kanno and K. Taniyama~\cite{KaTa}; we shall give our own proof here, so that the discussion in the next section on the Markov Theorem is simple.
  
  We will first show how to braid an STG diagram. We present a braiding algorithm analogous to the one in~\cite{L, LR} and adapt it to the setting of spatial trivalent graphs. For the purposes of this paper, we care about how the isotopy moves on diagrams affect the final braids. The conventions introduced in the following braiding process (that is, preparation for braiding and braiding algorithm) were carefully chosen so as to simplify the examination of the resulting braids.

\subsection{Preparation for braiding}

  We work strictly with well-oriented spatial trivalent graphs, so from now on we assume all STG diagrams are well-oriented. In addition, STG diagrams are assumed to be piecewise linear. This allows us to subdivide an arc into two smaller arcs, by marking it with a point. From now on, when we refer to vertices, we strictly mean trivalent vertices, thus, distinguishing vertices from subdivision points. Also, we consider local maxima and minima in arcs to be subdivision points.

  Before we begin braiding an arbitrary STG diagram, we need to establish certain requirements for the diagrams. These requirements describe a sort of `general picture' of how an STG diagram must look like in order to proceed with the braiding process. A large portion of the preparation for braiding is devoted to explaining such requirements and how to isotope an arbitrary STG diagram such that it satisfies these requirements. As we shall see below, the isotopy needed in order to meet such requirements is local, and, in particular, can be reduced to small changes involving planar isotopy and the $R5$ move.

  Now STG diagrams lie in the plane, which is equipped with the top-to-bottom direction. This allows us to impose particular restrictions on the diagrams. For example, we require STG diagrams to contain strictly up-arcs and down-arcs (no horizontal arcs). Furthermore, there should not be pairs of horizontally aligned crossings or vertices, so as to have the vertices and crossings in the corresponding braid lying on different horizontal levels. In addition, vertically aligned vertices,
 crossings, or subdivision points are not allowed, so as to avoid triple points when creating new pairs of braid strands with the same endpoint. Later, when we show how the braiding process is performed, the justification for rejecting such vertical alignment shall be made clear.

  The goal of the braiding process is to preserve the down-arcs in a diagram and replace the up-arcs with pairs of braid strands oriented downwards. An arbitrary up-arc may cross with several other arcs. We subdivide each up-arc into smaller pieces, such that each up-arc between two subdivision points contains at most one crossing. We label all up-arcs with an ``o'' or ``u'' indicating whether it is the over- or under-strand of a crossing in the diagram. Note that for \textbf{free up-arcs} (up-arcs that do not contain crossings), we have a choice whether to label them with an ``o'' or ``u''.

  In order to simplify the braiding for vertices, we impose the condition that a subdivision point cannot coincide with a vertex. We want all arcs incident with a vertex to be oriented downwards. By doing this, we isolate the vertices from the braiding of the up-arcs. We say that a vertex is in \textbf{regular position} if, in a small neighborhood, it is incident only with down-arcs.

  Now we will introduce conventions for bringing vertices into regular position. Given an STG diagram, every vertex is, roughly speaking, in either one of the four positions illustrated in Figure~\ref{vert} (there cannot be horizontal arcs). We remark that a vertex in any one of these positions must be oriented so as to satisfy the conventions for either zip or unzip vertices. Note that a vertex incident only with down-arcs must necessarily be either $Y$- or $\lambda$-type, since we do not allow sink or source vertices. We will start by treating the cases for either $Y$- or $\lambda$-type vertices incident with at least one up-arc.

  \begin{figure}[ht]
    \[
    \includegraphics[height = .5in]{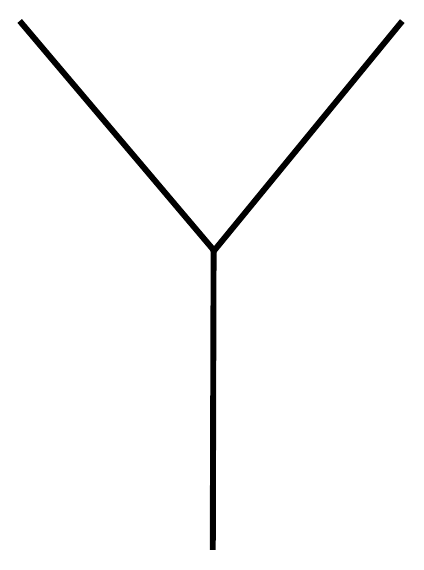}
    \hspace{1.5cm}
    \includegraphics[height = .5in]{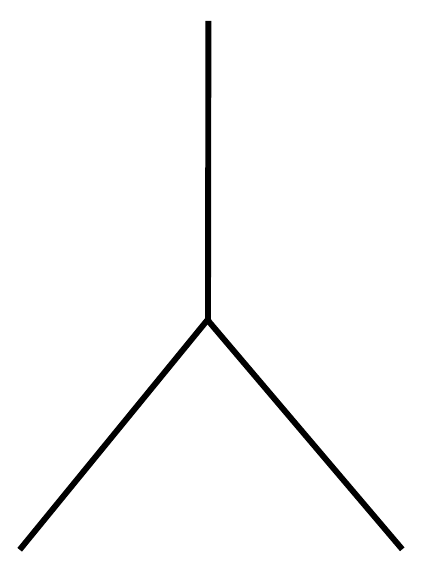}
    \hspace{1.5cm}
    \includegraphics[height = .5in]{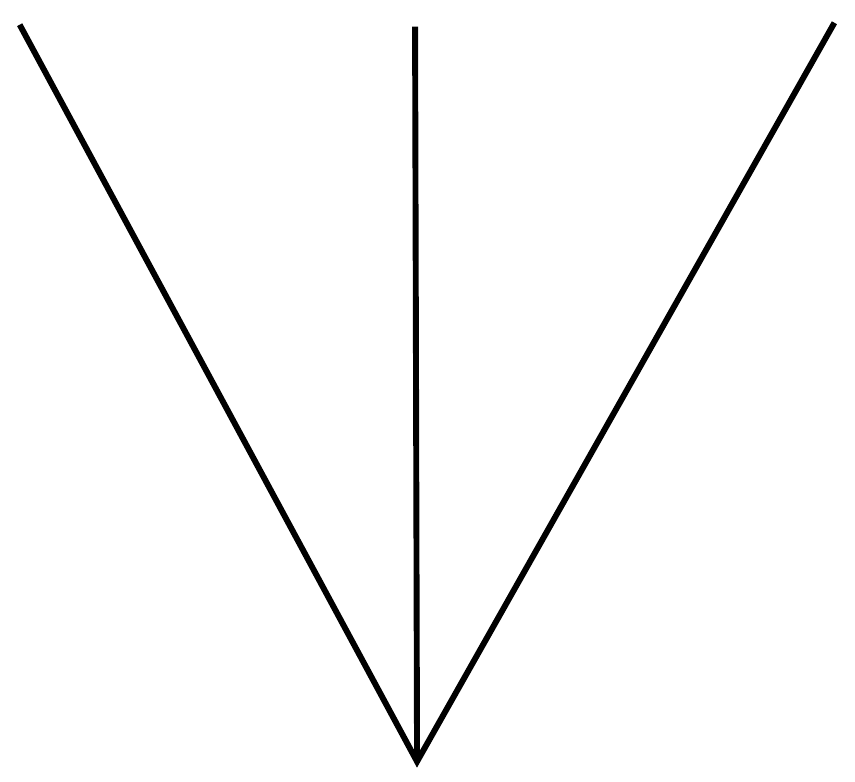}
    \hspace{1.5cm}
    \includegraphics[height = .5in]{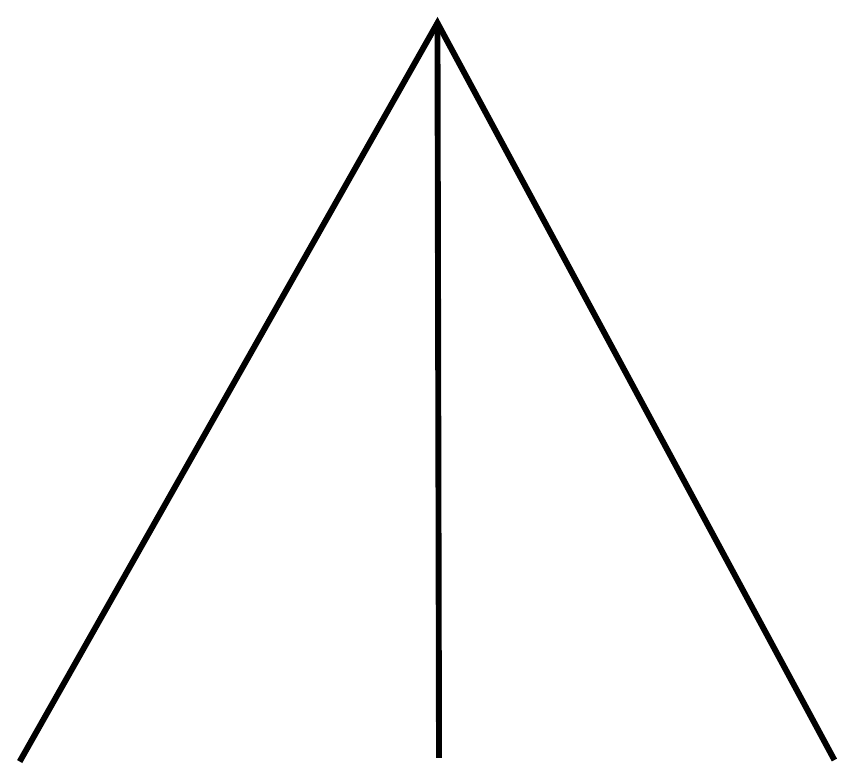}
    \put(-260,
    -15){\small{$Y$-type}} \put(-190,
    -15){\small{$\lambda$-type}} \put(-113,
    -15){\small{$W$-type}} \put(-32, -15){\small{$M$-type}}
    \]
    \caption{Types of trivalent vertices} \label{vert}
  \end{figure}

  In Figure~\ref{vgp}, we consider the various possible orientations for a $Y$-type vertex and show how to put it in regular position. For a $Y$-type vertex incident with only one up-arc, we simply perform planar isotopy on the problematic arc (see the first row of Figure \ref{vgp}). Note that the corrected diagram is a $\lambda$-type vertex. When correcting a $Y$-type vertex incident with exactly two up-arcs, we perform an $R5$ move introducing a crossing between the
  two up-arcs (see the second row of Figure~\ref{vgp}). Note that we have a choice for the type of crossing introduced by the $R5$ move, and that the corrected diagram is a $Y$-type vertex.

  The most interesting case is a $Y$-type vertex incident with three up-arcs; this reduces to performing planar isotopy on an arc so that, near the vertex, it is
  oriented downwards, and performing an $R5$ move that introduces a crossing between the remaining two up-arcs (see the bottom row of Figure~\ref{vgp}). It is important that the planar isotopy be performed on one of the two arcs oriented away from the vertex. Note that we have a choice concerning which arc the planar isotopy is performed on, and the type of crossing (positive or negative) introduced by the $R5$ move. Note that the resulting diagram is a $\lambda$-type vertex.

  \begin{figure}
    \begin{center}
      \begin{tabular}{| c | c| }
        \hline
        &  \\[-.2cm]
        $\raisebox{-10pt}{\includegraphics[height=0.5in]{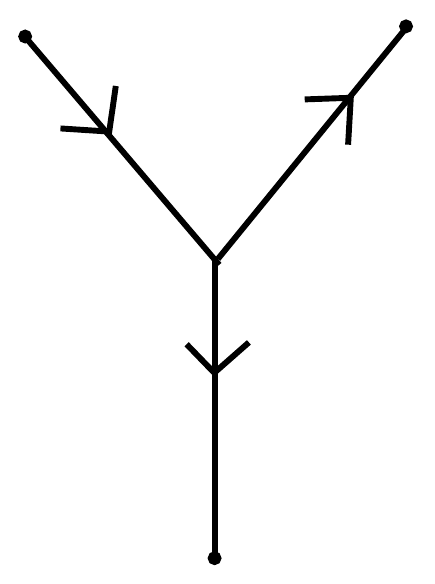}} \rightarrow \reflectbox{\raisebox{-10pt}{\includegraphics[height=.5in]{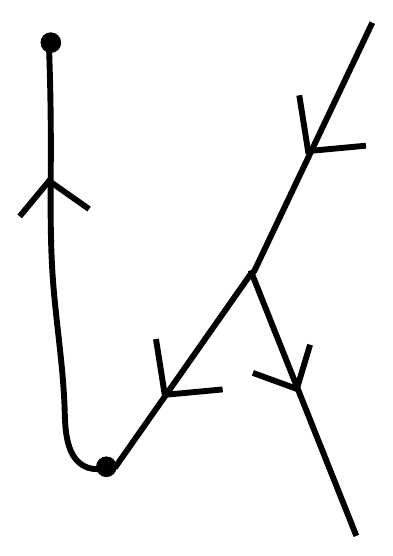}}}$ & $\raisebox{-10pt}{\includegraphics[height=0.5in]{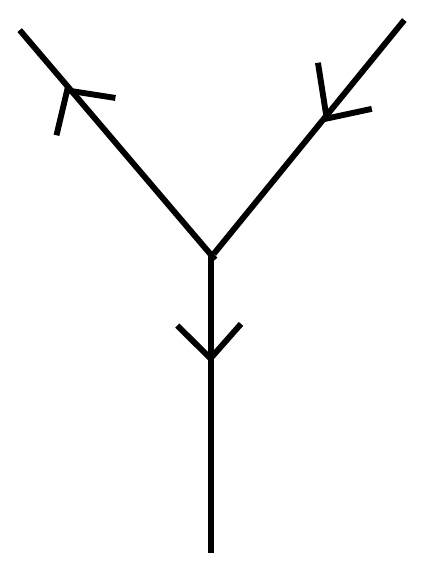}} \rightarrow \raisebox{-10pt}{\includegraphics[height=.5in]{V/UDD2}}$ \\
        [-.2cm] &

        \\ \hline
        &  \\[-.2cm]
        $\raisebox{-10pt}{\includegraphics[height=0.5in]{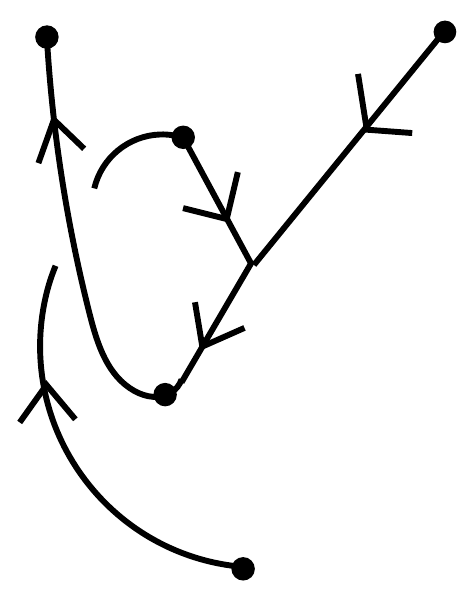}} \leftarrow \raisebox{-10pt}{\includegraphics[height=0.5in]{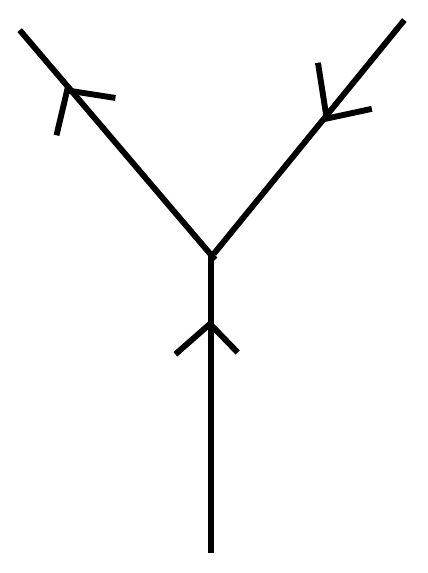}} \rightarrow \raisebox{-10pt}{\includegraphics[height=0.5in]{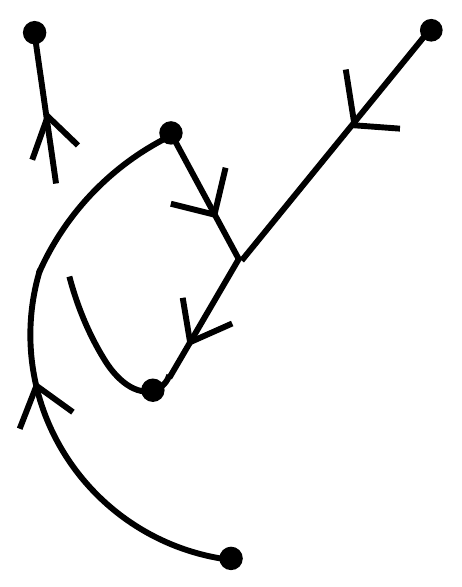}} $ &  $ \reflectbox{\raisebox{-10pt}{\includegraphics[height=0.5in]{V/UDU2}}} \leftarrow \raisebox{-10pt}{\includegraphics[height=0.5in]{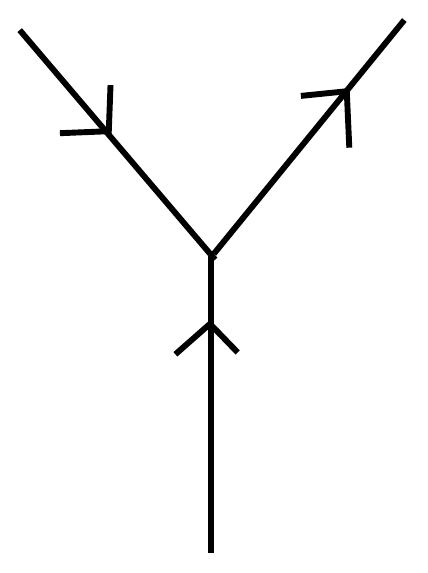}} \rightarrow  \reflectbox{\raisebox{-10pt}{\includegraphics[height=0.5in]{V/UDU3}}}$ \\
        [-.2cm] &
        \\ \hline

        \multicolumn{2}{| Sc |}{ $\raisebox{-10pt}{\includegraphics[height=0.5in]{V/UUU5}} \hspace{.1in} \text{or} \hspace{.1in} \raisebox{-10pt}{\includegraphics[height=0.5in]{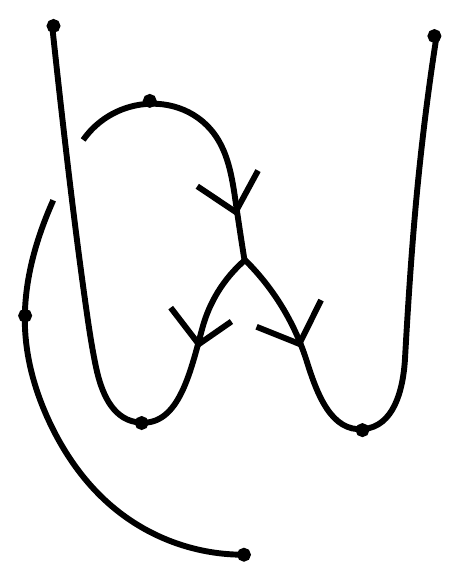}} \leftarrow \raisebox{-10pt}{\includegraphics[height=0.5in]{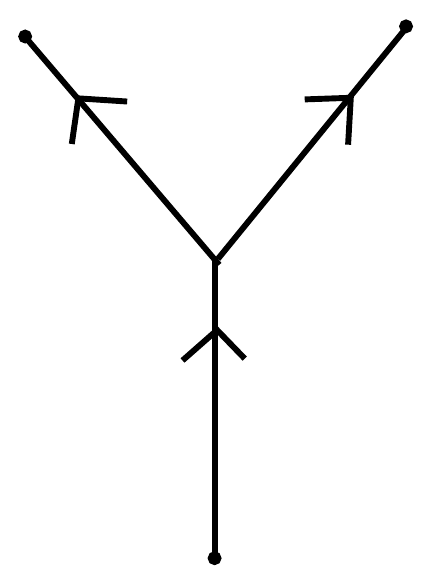}} \rightarrow \raisebox{-10pt}{\includegraphics[height=0.5in]{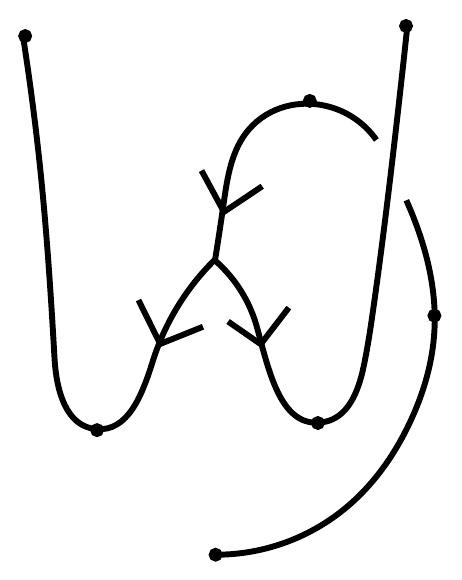}} \hspace{.1in} \text{or} \hspace{.1in} \raisebox{-10pt}{\includegraphics[height=0.5in]{V/UUU4}}$} \\

        \hline

      \end{tabular}
    \end{center}
    \caption{Adjusting $Y$-type vertices into regular position} \label{vgp}
  \end{figure}

  The process for bringing a
  $\lambda$-type vertex into regular position is similar and is
  depicted in Figure \ref{vgp2}.

  \begin{figure}
    \begin{center}
      \begin{tabular}{| c | c| }
        \hline
        &  \\[-.2cm]
        $ \raisebox{-.2in}{\includegraphics[height=0.5in]{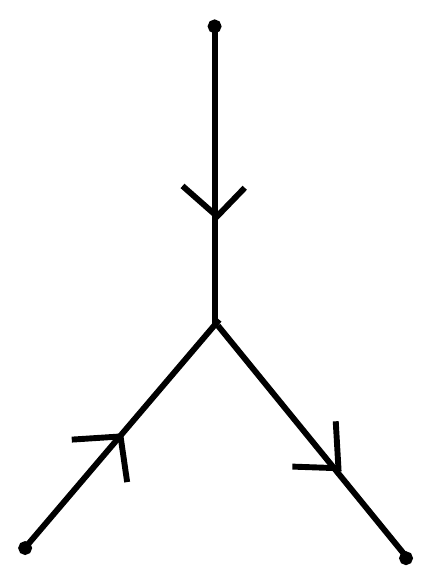}}\,\, \rightarrow \,\, \raisebox{-.2in}{\includegraphics[height=.5in]{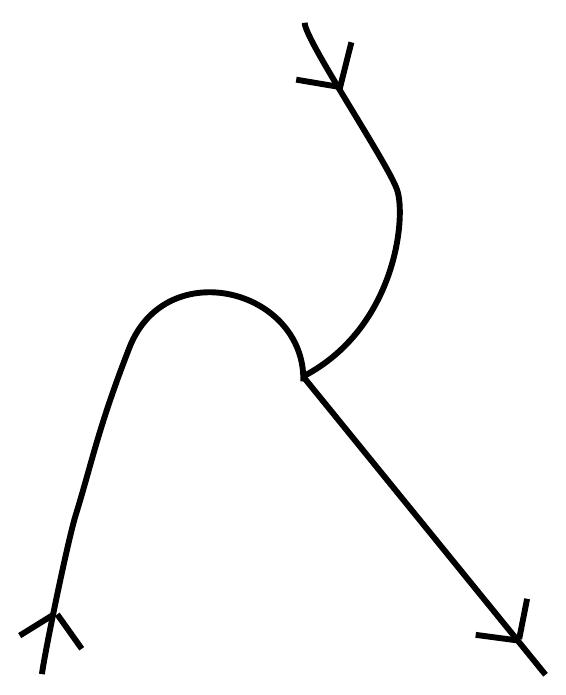}} $ & $ \raisebox{-.2in}{\includegraphics[height=0.5in]{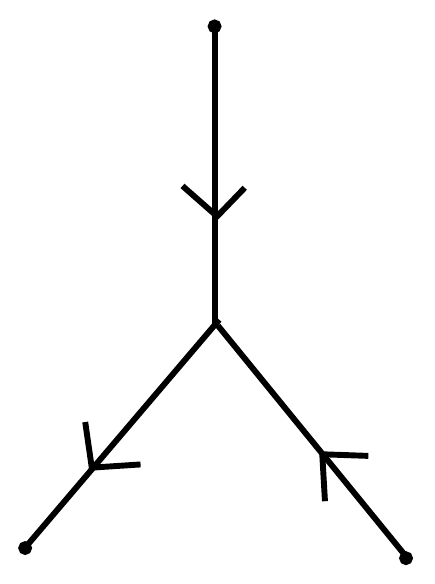}}\,\, \rightarrow \,\, \raisebox{-.2in}{\includegraphics[height=.5in]{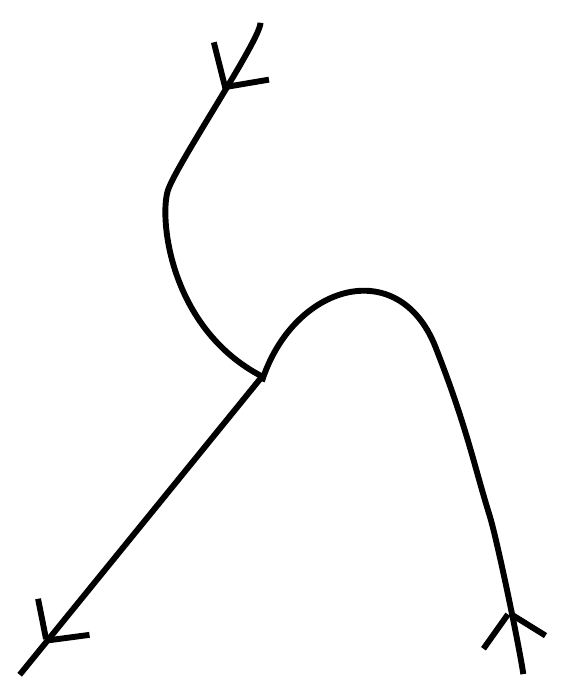}} $ \\
        [-.2cm] &

        \\ \hline
        &  \\[-.2cm]
        $\raisebox{-10pt}{\includegraphics[height=0.5in]{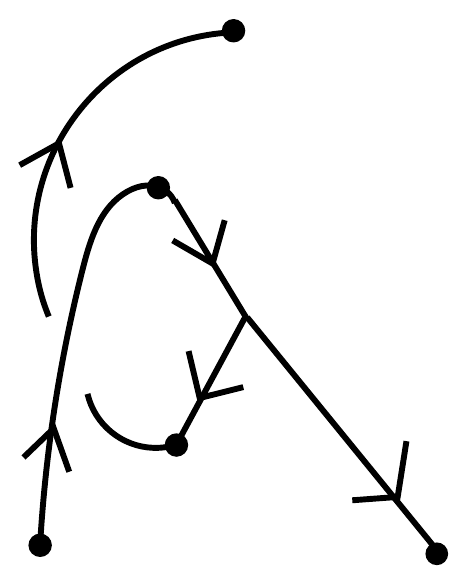}} \leftarrow \raisebox{-10pt}{\includegraphics[height=0.5in]{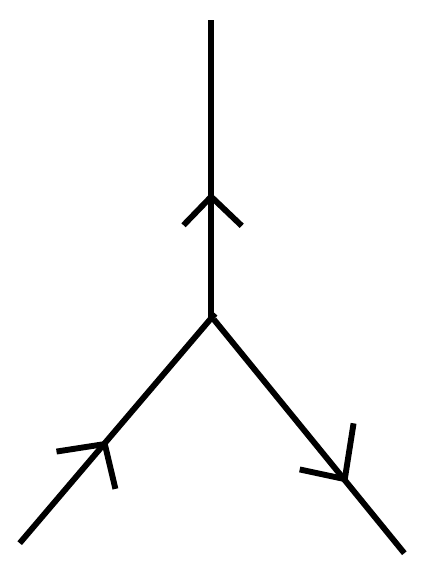}} \rightarrow \raisebox{-10pt}{\includegraphics[height=0.5in]{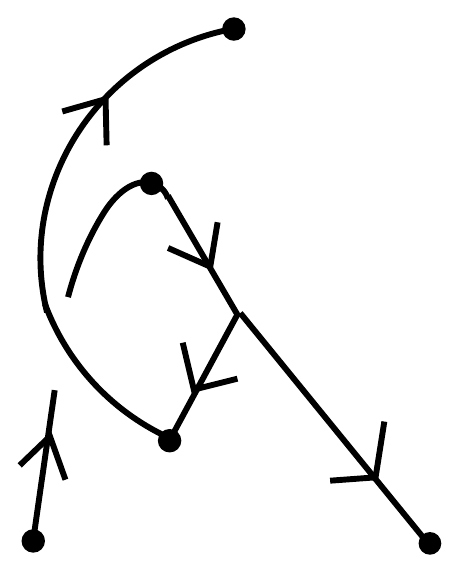}}$ &  $ \reflectbox{\raisebox{-10pt}{\includegraphics[height=0.5in]{V/UDU2L}}} \leftarrow \reflectbox{\raisebox{-10pt}{\includegraphics[height=0.5in]{V/UDUL}}} \rightarrow  \reflectbox{\raisebox{-10pt}{\includegraphics[height=0.5in]{V/UDU3L}}}$ \\
        [-.2cm] &
        \\ \hline

        \multicolumn{2}{| Sc |}{$\raisebox{-10pt}{\includegraphics[height=0.5in]{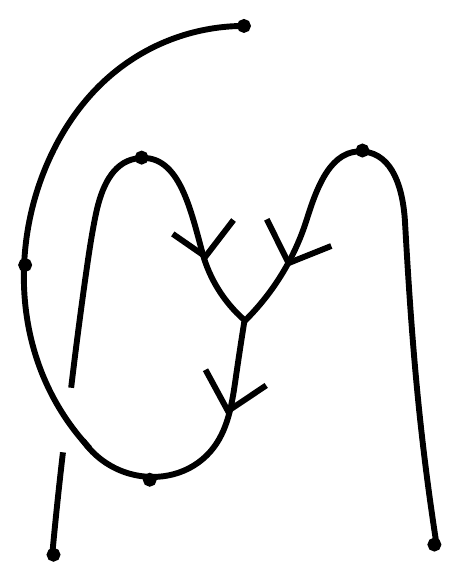}} \hspace{.1in} \text{or} \hspace{.1in} \raisebox{-10pt}{\includegraphics[height=0.5in]{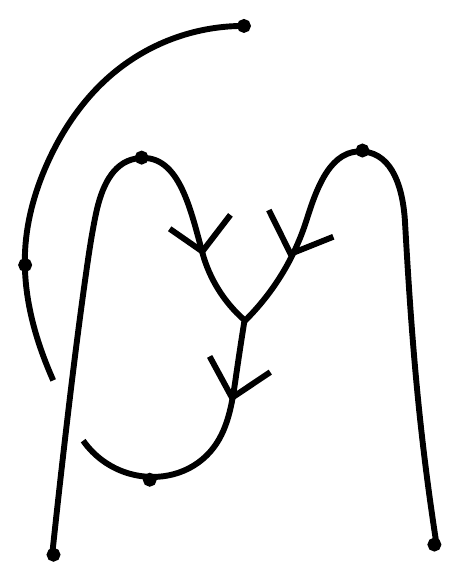}} \leftarrow \raisebox{-10pt}{\includegraphics[height=0.5in]{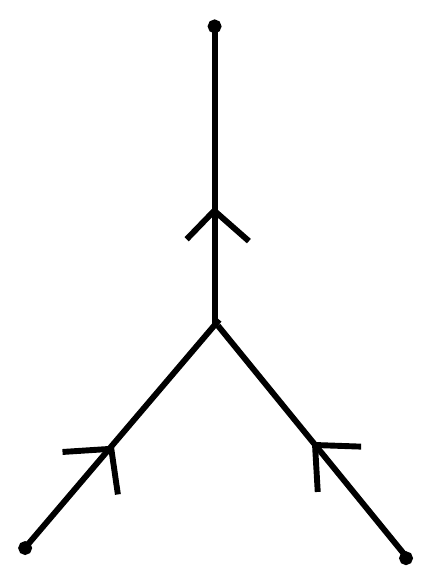}} \rightarrow \raisebox{-10pt}{\includegraphics[height=0.5in]{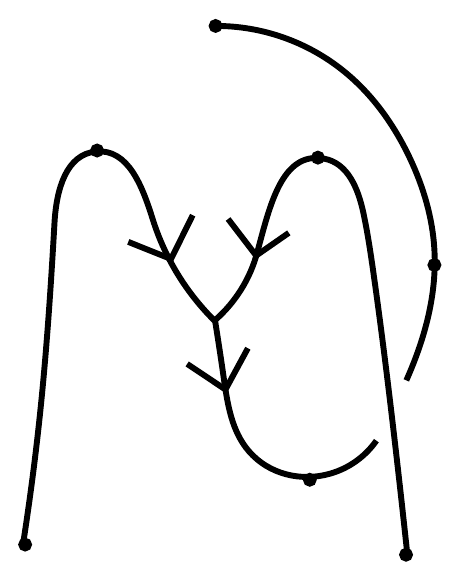}} \hspace{.1in} \text{or} \hspace{.1in} \raisebox{-10pt}{\includegraphics[height=0.5in]{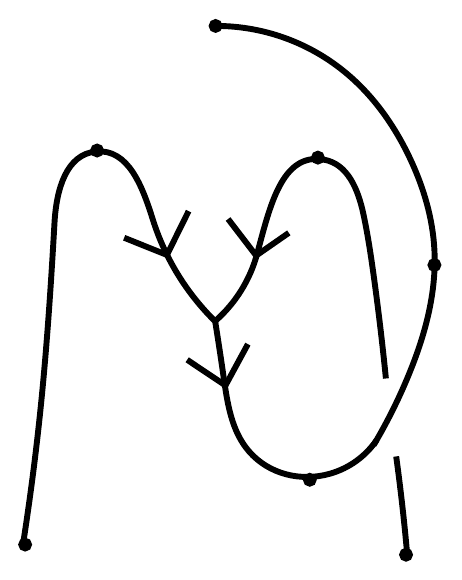}}$} \\

        \hline

      \end{tabular}
    \end{center}
    \caption{Adjusting $\lambda$-type vertices into regular
      position} \label{vgp2}
  \end{figure}

  Now that we have taken care of $Y$- and $\lambda$-type vertices, only two cases remain to be considered, namely $M$- and $W$-type vertices (see rightmost two diagrams in Figure~\ref{vert}). In fact, these can be reduced to the case of either a $Y$- or $\lambda$-type vertex. For example, given an $M$-type vertex, we
  perform planar isotopy on the rightmost arc to obtain a $\lambda$-type vertex (see left hand side of Figure~\ref{vyl}). Now, using the previous conventions for $\lambda$-type vertices, we proceed to bring it into regular position. Note that we have not assumed a particular orientation for the arcs incident with the vertex,
  besides that the vertex is either a zip or unzip vertex. We treat $W$-type vertices similarly, except that in these cases, the isotopy is performed on the leftmost strand, so as to obtain a $Y$-type vertex (as shown on the right hand side of Figure~\ref{vyl}). Applying these conventions to all vertices in a given STG diagram, we obtain a new diagram which is isotopic to the original one and whose vertices are all in regular position.

  \begin{figure}[ht]
    \[
    \raisebox{-20pt}{\includegraphics[height =
      .5in]{BraidAlg/Vertices/3Down.pdf}} \hspace{.4cm}
    \longrightarrow \hspace{.4cm}
    \raisebox{-20pt}{\includegraphics[height =
      .5in]{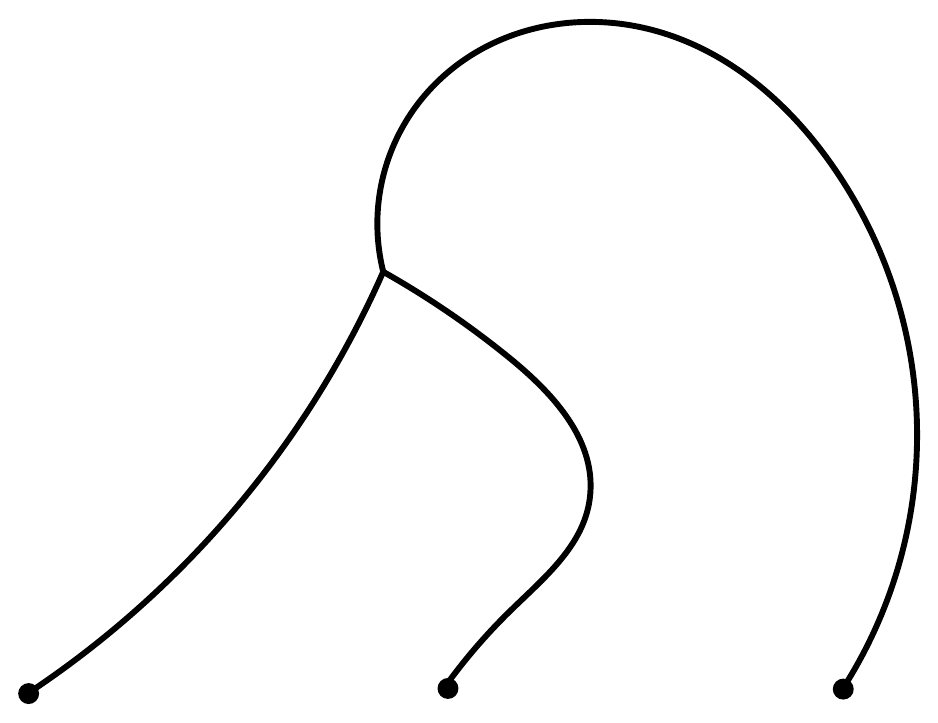}} \hspace{2cm}
    \raisebox{-20pt}{\includegraphics[height =
      .5in]{BraidAlg/Vertices/Trident.pdf}} \hspace{.4cm}
    \longrightarrow \hspace{.4cm}
    \raisebox{-20pt}{\includegraphics[height =
      .5in]{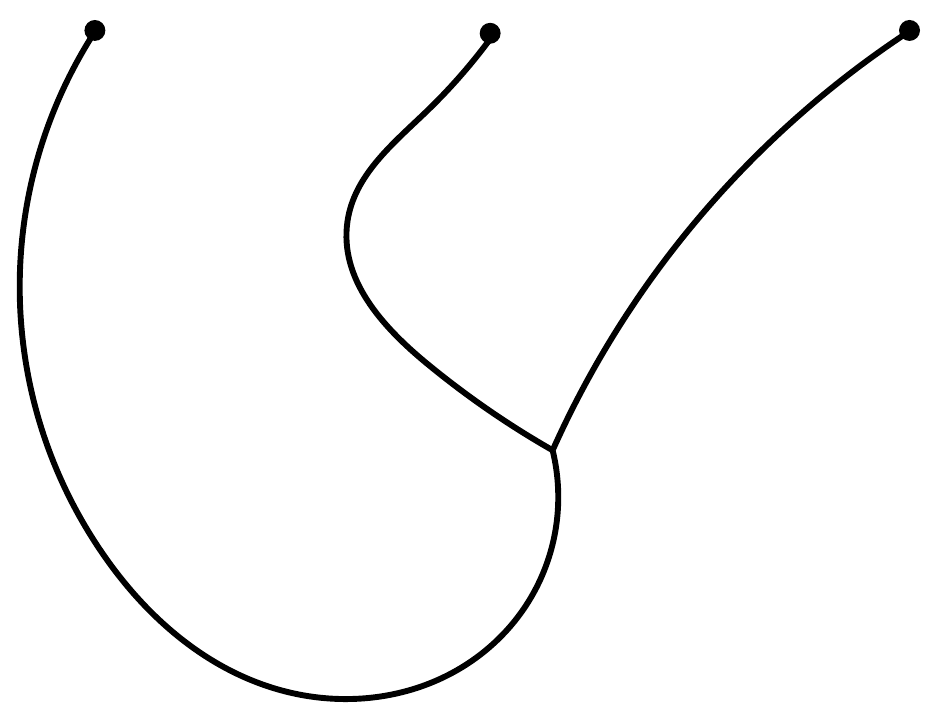}}
    \]
    \caption{Spreading arcs around a vertex} \label{vyl}
  \end{figure}

  We now shift our focus to up-arcs which are not incident with a vertex. We shall use the concept of \textbf{sliding triangle} introduced in~\cite{LR}. The
sliding triangle associated with an up-arc is the right triangle with hypotenuse the up-arc and with right angle lying below the hypotenuse (see Figure~\ref{st}). As we shall see when we explain the actual braiding, the sliding triangle serves as a guideline for how to arrange the braided outputs of an up-arc. We say a sliding triangle is of type \textbf{over} or \textbf{under} according to the label of the up-arc it is associated with. Also, we consider sliding triangles to be adjacent whenever the corresponding up-arcs have a common subdivision point.
  
  \begin{figure}[ht]
    \[
    \includegraphics[height=1in]{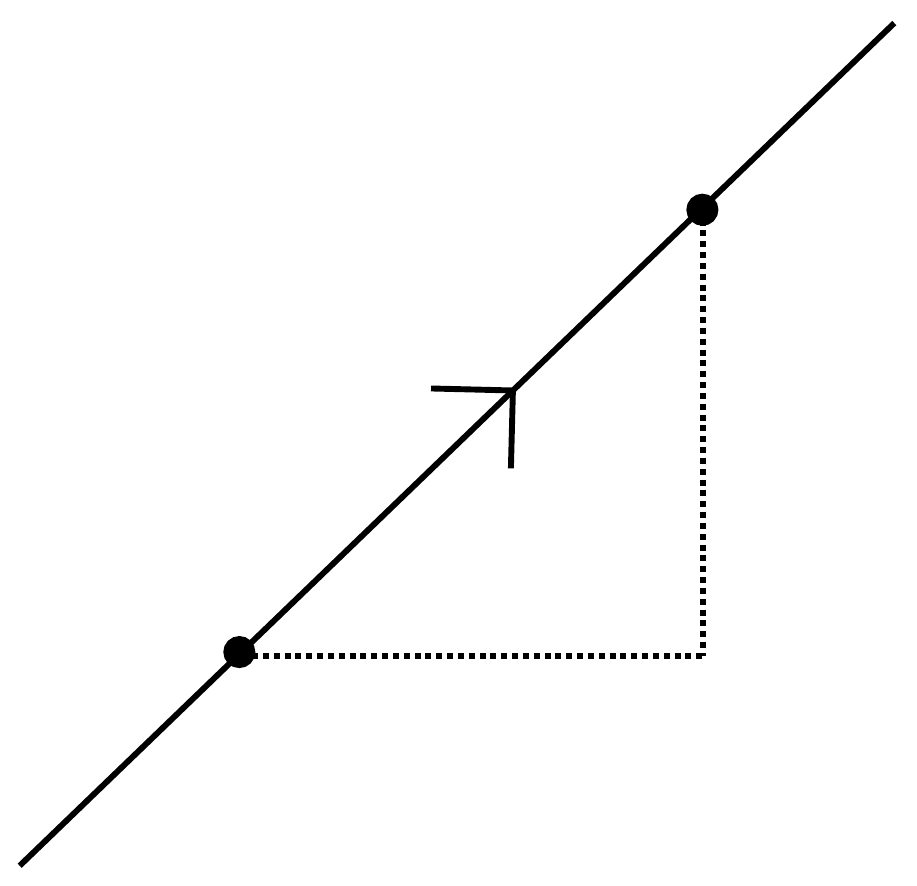}
    \]
    \caption{A sliding triangle associated with an up-arc}
    \label{st}
  \end{figure}

  We introduce the \textbf{triangle condition} (see~\cite{LR}) which states that non-adjacent sliding triangles can overlap only if they have opposite labels. Later, once we introduce the braiding moves for up-arcs, we will go into detail justifying why the triangle condition is needed. In short, the triangle condition ensures
  that the braiding moves do not interfere with each other, so that the order in which we eliminate the up-arcs is irrelevant.

\begin{lemma}
  Let $G$ be an STG diagram with vertices in regular position, and let $G'$ be a subdivision of $G$. Then there is a refinement of $G'$ such that (for appropriate choices of under/over for free up-arcs) the triangle condition is satisfied. \label{lemA1}
\end{lemma}
\begin{proof}
  Consider a crossing which contains one up-arc. We want to have the sliding triangle of such up-arc so that it overlaps only with the other strand in the crossing. If the sliding triangle of the up-arc overlaps with any other arc or sliding triangle outside the crossing, then we further subdivide the up-arc such that the resulting up-arc containing the crossing becomes small enough that its sliding triangle covers only a small neighborhood around the crossing.

  For the case of a crossing containing two up-arcs, we argue as above so as to have the sliding triangles corresponding to the crossing up-arcs isolated from the rest of the diagram. Note that the corresponding sliding triangles have opposite labels (since they correspond to arcs with opposite labels), so  their intersection respects the triangle condition.

  Similarly, for the case of free up-arcs, we can add subdivision points, if necessary, so that the corresponding sliding triangles are disjoint from the rest of the diagram, as shown in Figure~\ref{TC}. Consequently, we have the liberty to assign any label to the triangles. Note that we do not encounter any problems in regions containing vertices, since the latter are in regular position.
\end{proof}

\begin{figure}[ht]
  \[
  \raisebox{-30pt}{\includegraphics[height =
    .8in]{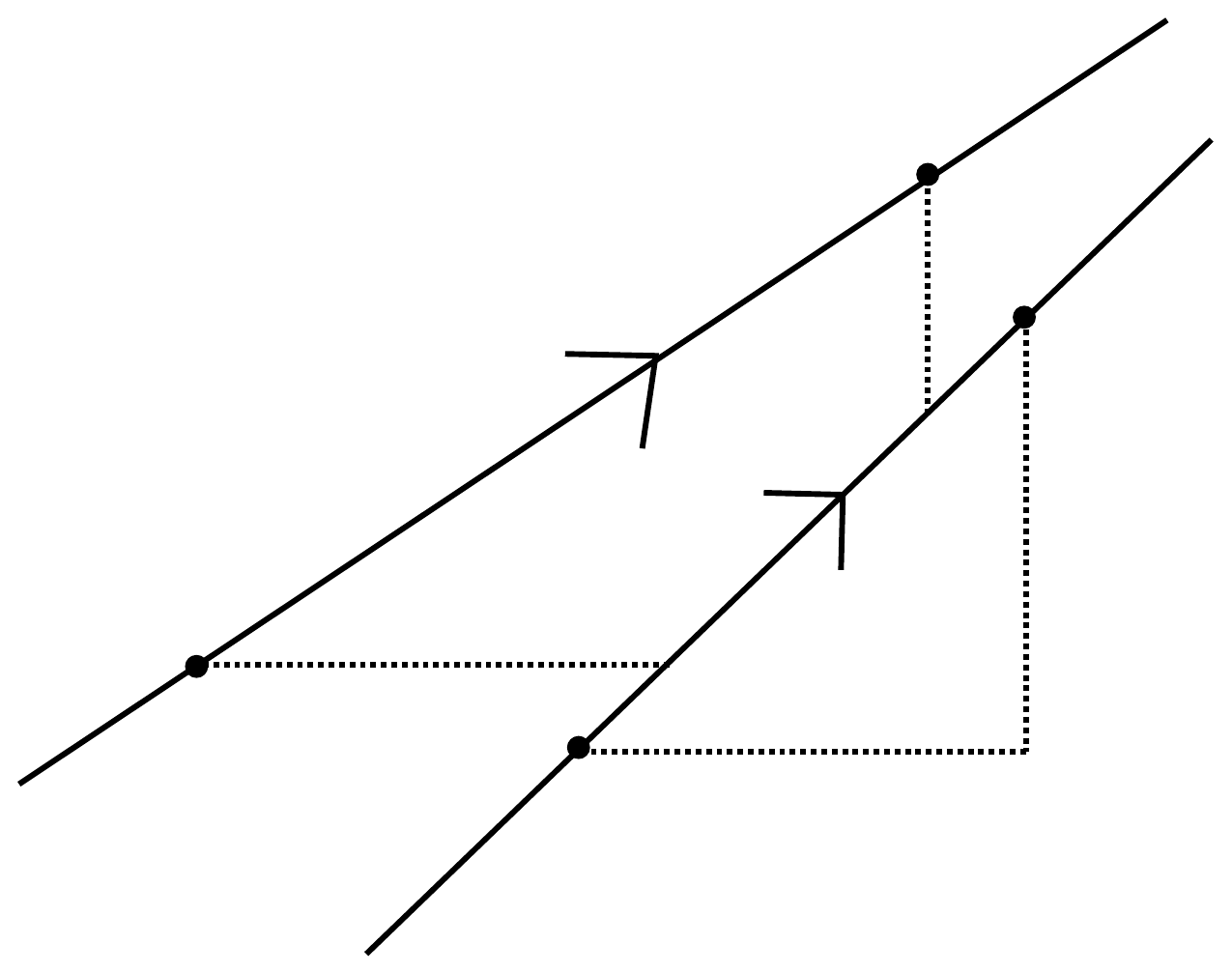}} \hspace{.6cm}
  \longrightarrow \hspace{.6cm}
  \raisebox{-30pt}{\includegraphics[height =
    .7in]{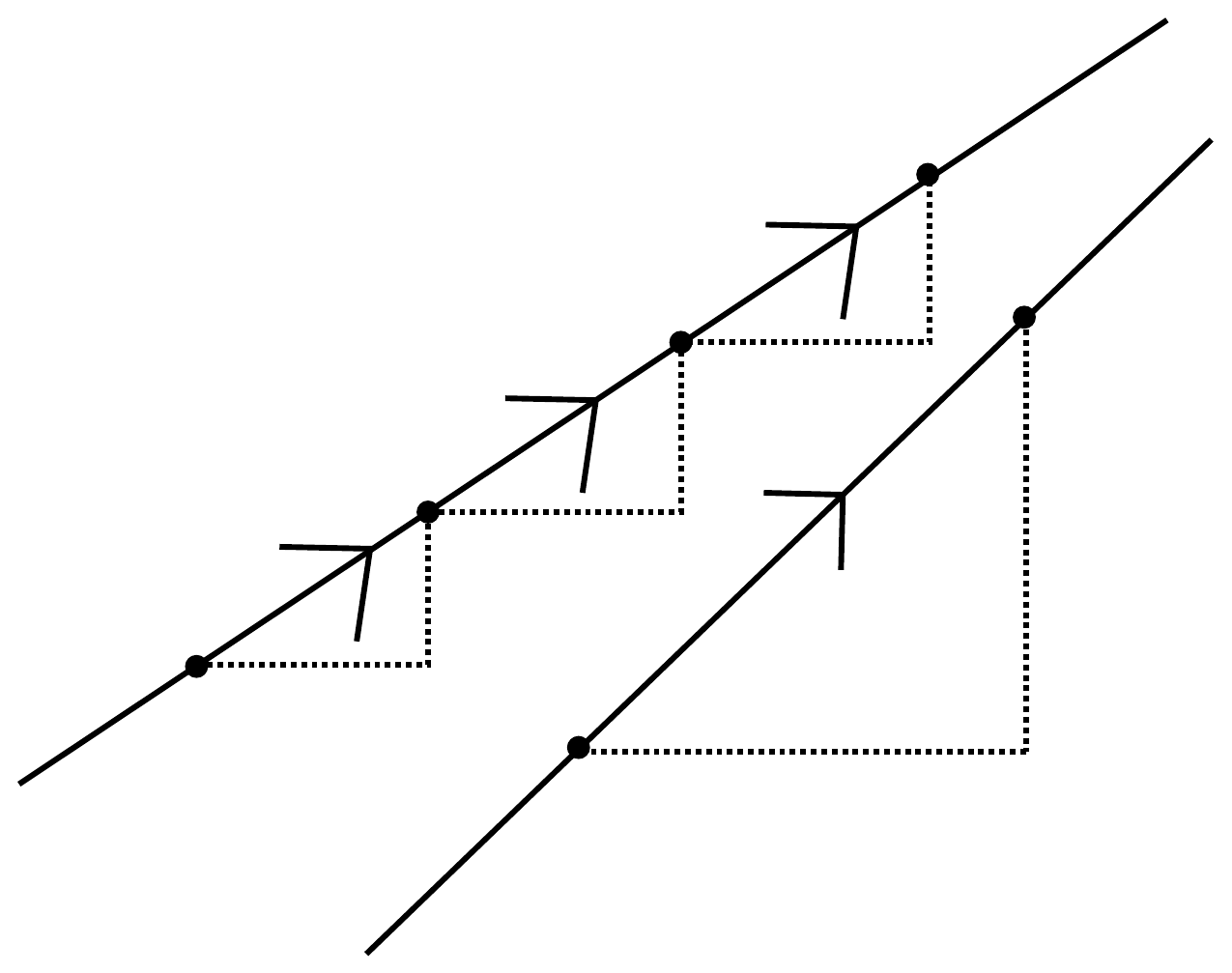}}
  \]
  \caption{Adjusting sliding triangles} \label{TC}
\end{figure}

From the proof of Lemma~\ref{lemA1}, we see that given a diagram with a subdivision which satisfies the triangle condition, then any refinement of the subdivision (with appropriate labels) satisfies the triangle condition as well.

With this, we conclude the requirements an STG diagram must satisfy so that it is ready to be braided. We summarize the previous discussion in the following definition.

\begin{definition}
  An STG diagram with regular vertices is said to be in \textbf{general position} if the following conditions hold:
  \begin{enumerate}
  \item There are no horizontal arcs;
  \item There are no crossings, subdivision points, or vertices that are either horizontally or vertically aligned;
  \item All nonadjacent sliding triangles must satisfy the triangle condition, and if they intersect, this must be along a common interior (and not a single point).
  \end{enumerate} \label{gen pos}
\end{definition}

We refer to \textbf{direction sensitive moves} as the local shifts on an STG diagram in order to put it in general position. These shifts can be in the horizontal or vertical directions. For example, whenever two subdivision points are either vertically or horizontally aligned, we can correct these singularities by performing planar isotopy locally on one of the subdivision points, so that they are no longer horizontally/vertically aligned. A similar argument can be used when correcting alignment of any combination of either subdivision points, crossings, or vertices. In addition, whenever two non-adjacent sliding triangles intersect at a point, we can choose a subdivision point corresponding to one of the participating sliding triangles, and replace it by another subdivision point arbitrarily close to the original one, so that condition (3) of general position is not violated. On the other hand, if two non-adjacent sliding triangles with common labels intersect, we further subdivide one of the corresponding up-arcs and change labels appropriately (We refer the reader to the proof of Lemma 3.5 in~\cite{LR} for details). Note that we allow different choices to be made when shifting a diagram into general position.

Next we shall introduce the last two kinds of the direction sensitive moves. As we shall see later, these particular moves are crucial for our one-move Markov type theorem to work; it turns out that these moves, together with planar isotopy and $R1-R5$, will allow us to only consider instances of isotopy moves between diagrams in general position (see Lemma~\ref{lemma:isotopyRegPosition}). The \textbf{swing moves} form a special case of the direction sensitive moves. Figure~\ref{swing} shows a version of the swing move in which an arc slides across a local minimum; in general a swing move allows an arc to swing over/under an extremum point. It is well-known that the swing moves involving arcs with a local maximum can be obtained from those with a minimum together with $R2$ moves. Note that since we regard local extrema points as subdivision points, we cannot have an arc intersecting an extremum point in an STG diagram; a swing move avoids the coincidence of a minimum or maximum and a crossing.

\begin{figure}[ht]
  \[ \raisebox{-20pt}{\includegraphics[height=0.5in]{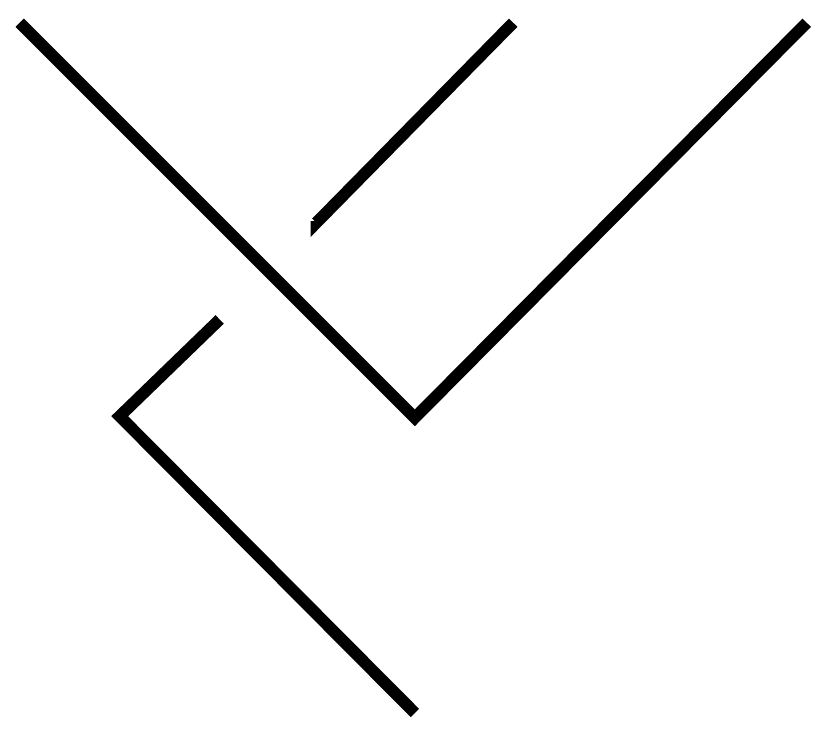}}\,\,
  \longleftrightarrow \,\,
  \raisebox{-20pt}{\includegraphics[height=0.5in]{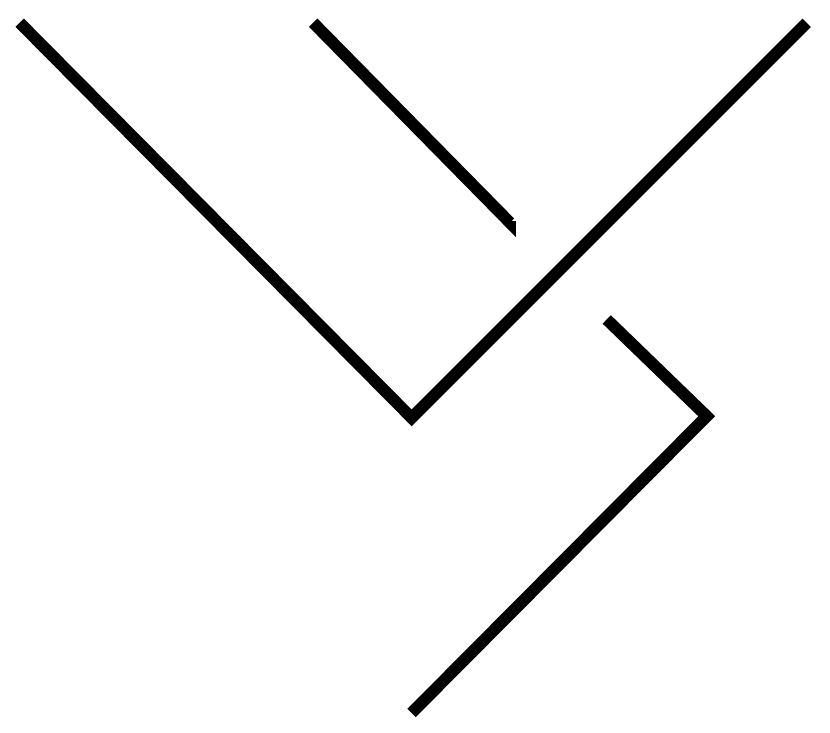}}\]
  \caption{A swing move}\label{swing}
\end{figure}

The last of the direction sensitive moves is the \textbf{switch move} for $Y$- and $\lambda$-type vertices (see Figure~\ref{twist}). These moves are intimately related to our requirement that vertices be in regular position. Recall that when bringing a $Y$- or $\lambda$-type vertex into regular position, if it is incident with at least two up-arcs, we have a choice concerning the type of crossing introduced by the $R5$ move (see the last two rows in Figures~\ref{vgp} and~\ref{vgp2}). Ultimately, all possible choices are related via the switch move, the swing move, and isotopy moves.

\begin{figure}[ht]
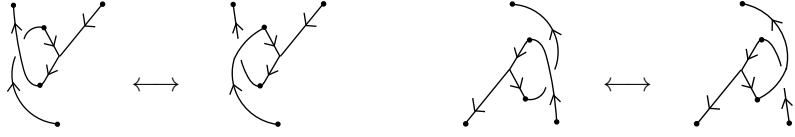

  \[ \raisebox{-15pt}{\includegraphics[height=0.7in]{V/UDU2}}
  \hspace{.2cm} \longleftrightarrow \hspace{.5cm}
  \raisebox{-15pt}{\includegraphics[height=0.7in]{V/UDU3}}
  \hspace{1.7cm}
  \reflectbox{\raisebox{-15pt}{\includegraphics[height=0.7in]{V/UDU2L}}}
  \hspace{.4cm} \longleftrightarrow \hspace{.4cm}
  \reflectbox{\raisebox{-15pt}{\includegraphics[height=0.7in]{V/UDU3L}}}\]
  \caption{The switch moves for Y- and
    $\lambda$-type vertices} \label{twist}
\end{figure}

\begin{lemma} \label{lemma:isotopyRegPosition} 
Let $G$ and $G'$ be isotopic STG diagrams in general position. Then there is a sequence of isotopy moves relating $G$ and $G'$, all of which are between diagrams in general position.
\end{lemma}

\begin{proof}
  The idea is the following: given a finite sequence of isotopy moves relating $G$ and $G'$, we can correct the middle stages in such sequence which are not in general position, and, thus obtain an alternative sequence in which all diagrams are in general position. It is well known that instances of isotopy in regions free of vertices can be easily achieved while still respecting the general position conditions inside such regions (for a proof of this, see Lemma 3.5 in \cite{LR}). Thus it remains to be seen that the same is true for regions that contain a vertex. Note first that the extended Reidemeister moves $R1-R4$ applied to an STG diagram in general position yield a diagram still in general position. On the other hand, the $R5$ move and planar isotopy applied near a vertex $v$ in regular position may convert that vertex into non-regular position. Note that since $G'$ is in general position, any such problematic move is countered, later in the sequence, by another move that brings back $v$ into regular position. Therefore, we can readjust the vertex $v$ to regular position according to the rules in Figures~\ref{vgp} and~\ref{vgp2}, and, thus, complete the sequence with the aid of the direction sensitive moves, in particular the swing and switch moves. Hence, we can replace the `troublesome' moves between STG diagrams, where the second diagram is not in regular position, by a sequence of isotopy moves between STG diagrams in general position.
\end{proof}

\begin{remark}
  From the discussion above, it follows easily that given two isotopic STG diagrams in general position, they differ by a finite sequence of direction sensitive moves (planar isotopy, swing moves, and switch moves) and the extended Reidemeister moves, where each diagram in such sequence is in general position as well.
\end{remark}

From here on, we shall assume that all STG diagrams are in general position.
\subsection{The braiding algorithm}\label{sec:braiding}

We will now illustrate our braiding algorithm. The idea is to keep the down-arcs, and eliminate the up-arcs by replacing them with pairs of vertically aligned braid strands. The braiding algorithm outlined here is inspired by the one presented in~\cite{LR}. These algorithms share many similarities, namely the way in which up-arcs are braided and the conditions that determine the general position. The main difference is that our set-up requires a careful treatment of arcs
incident with a vertex. It is essential that every vertex of a given STG diagram is in regular position prior to shifting the diagram into general position. By doing this, we isolate the regions containing a vertex from the braiding of the up-arcs.

We will first show how to braid crossings. For that, consider an up-arc which is the over-strand of a crossing, thus labeled ``o''. The braiding consists of first sliding
the up-arc across the sliding triangle, making sure the horizontal arc has a negative slope, so that it does not conflict with the general position requirement. We then cut the vertical segment, and pull the upper cut-point upward and the lower downward (see Figure~\ref{crossing}).  Note that the new vertical strands are both oriented downwards, and are vertically aligned. In addition, both vertical strands cross \textit{over} all other arcs in the diagram. This is indicated abstractly in diagrams by adding the label ``o'' to both vertical strands (the braid box in Figure~\ref{crossing} indicates a magnified region in the diagram).

If the up-arc is the under-strand in the crossing, appropriately labeled ``u'', then the braiding of it is identical, except that the new pair of vertical braid strands both cross \textit{under} all other arcs in the diagram. By the same reasoning, both vertical braid strands are labeled ``u''.
\begin{figure}[ht]
  \[
  \raisebox{-25pt}{\includegraphics[height =
    .8in]{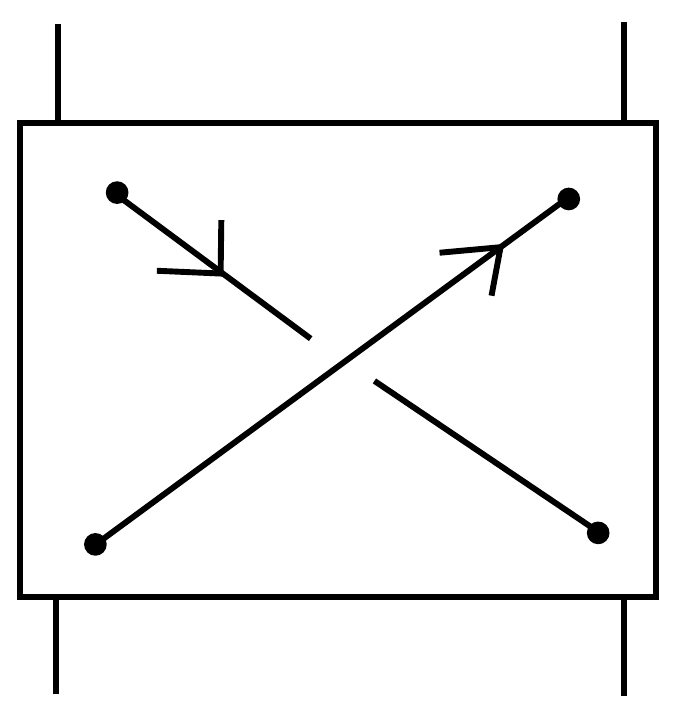}} \longrightarrow
  \raisebox{-25pt}{\includegraphics[height =
    .8in]{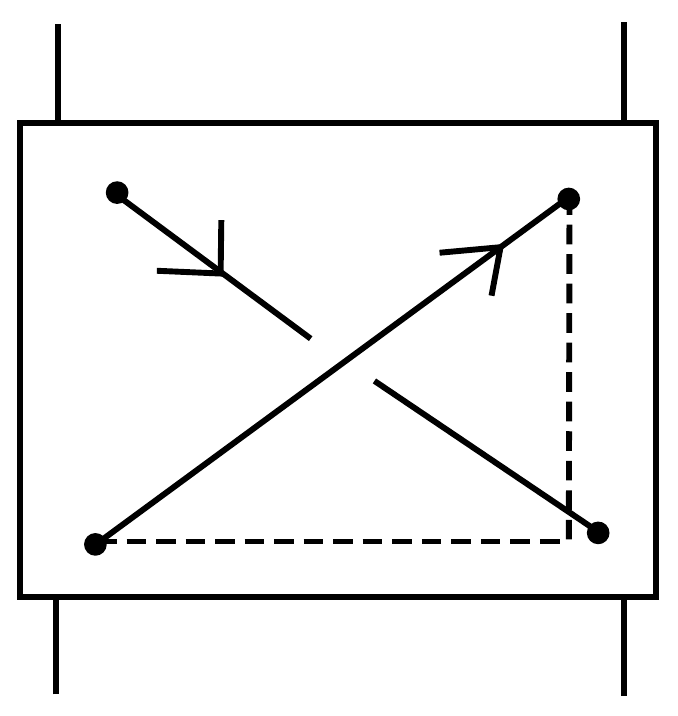}} \put(-100,
  15){\fontsize{9}{9}$o$} \longrightarrow
  \raisebox{-25pt}{\includegraphics[height =
    .8in]{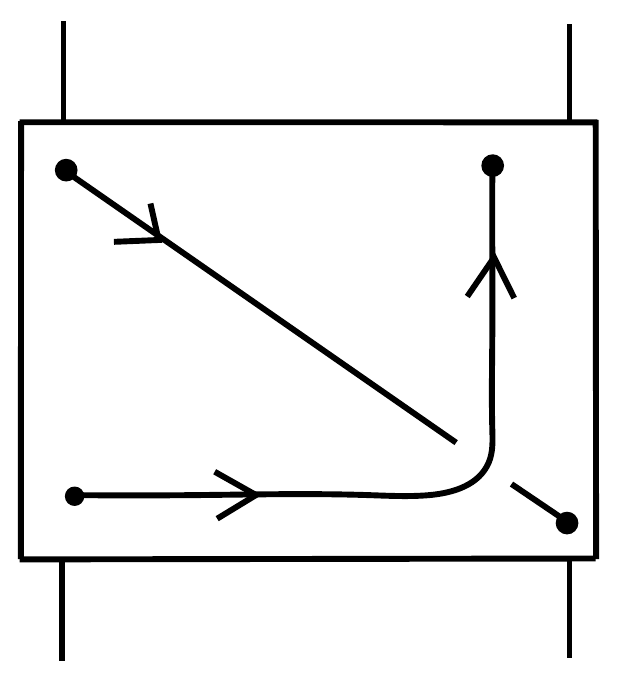}} \longrightarrow
  \raisebox{-36pt}{\includegraphics[height =
    1.06in]{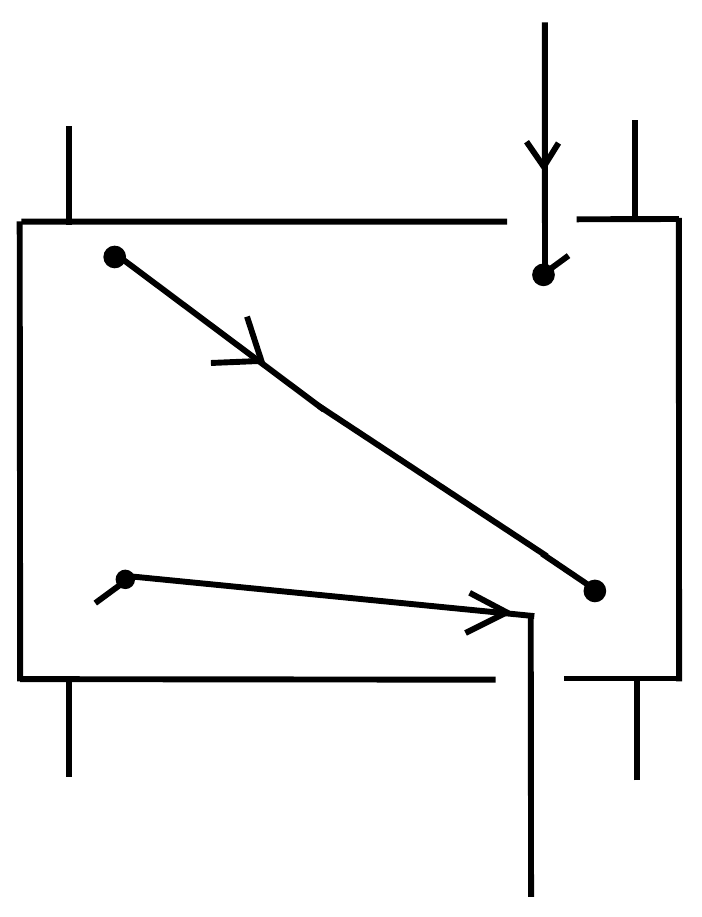}} \put(-23,
  30){\fontsize{9}{9}$o$} \put(-23, -30){\fontsize{9}{9}$o$}
  \]
  \caption{Braiding a crossing} \label{crossing}
\end{figure}

The braiding of a free up-arc is done similarly, making sure the new pair of vertical braid strands cross either under or over all other arcs in the diagram, in accordance with the label of the original up-arc (see Figure~\ref{free}). In essence, by braiding a free up-arc, we simply replace the arc with a pair of vertical braid strands oriented downwards and vertically aligned with the endpoint of the free up-arc. We shall refer to the braiding of an up-arc as a \textbf{basic braiding move.}
\begin{figure}[ht]
  \[
  \raisebox{-38pt}{\includegraphics[height =
    1.05in]{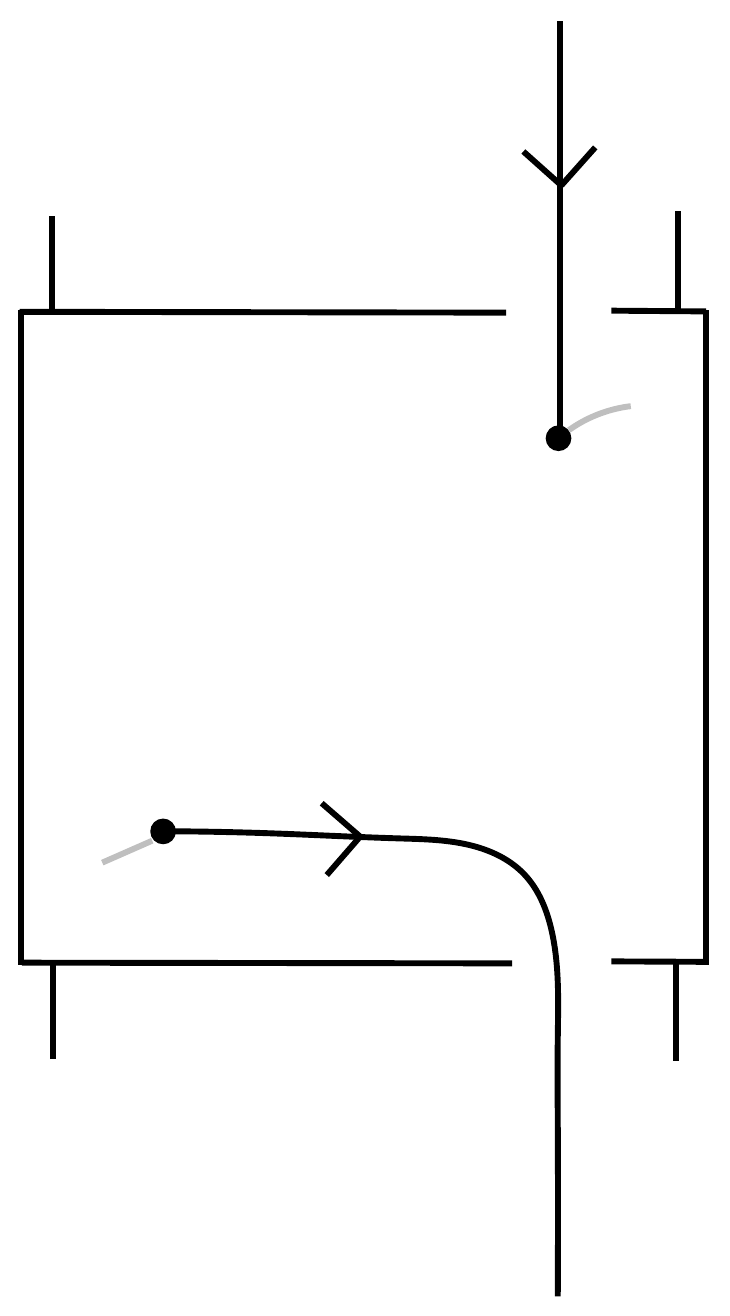}} \,\, \longleftarrow \,\,
  \raisebox{-28pt}{\includegraphics[height = .8in]{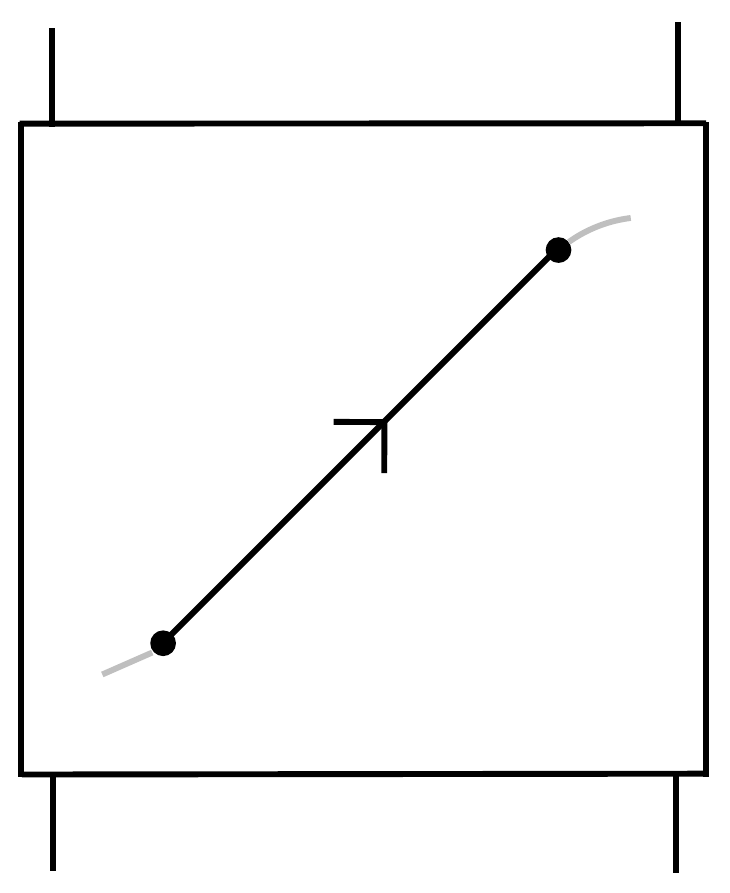}}
  \put(-93, 30){\fontsize{9}{9}$o$} \put(-93,
  -30){\fontsize{9}{9}$o$} \,\, \longrightarrow \,\,
  \raisebox{-38pt}{\includegraphics[height = 1.05in]{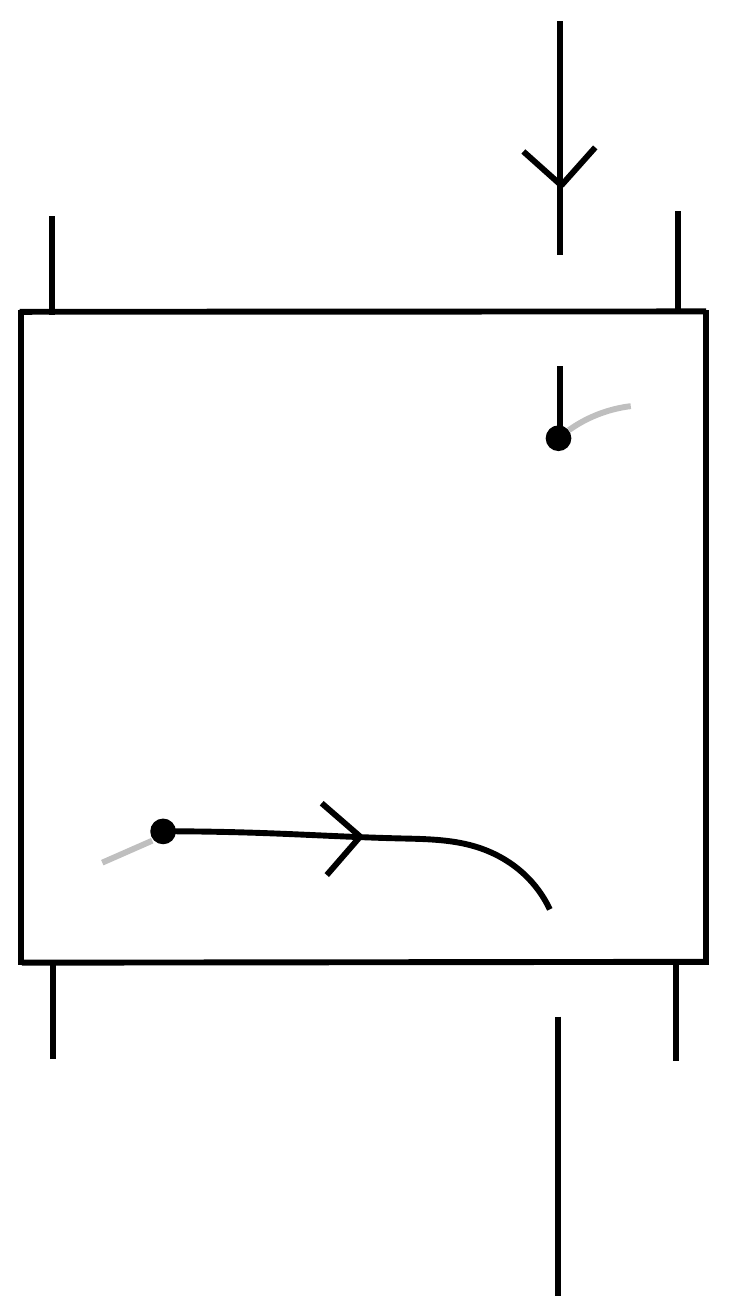}}
  \put(-20, 30){\fontsize{9}{9}$u$} \put(-20,-30){\fontsize{9}{9}$u$}
  \]
  \caption{Braiding a free up-arc} \label{free}
\end{figure}

It is important to remark that by connecting the two newly created pair of vertical braid strands (outside of the diagram, around the braid axis), we obtain a trivalent tangle isotopic to the original one. This holds when braiding either a free up-arc or a crossing.

\begin{remark} We can finally make clear the following:
  \begin{enumerate}
  \item[(i)] It is now clear the reason for rejecting vertical alignment of any combination of vertices, subdivision points or crossings. The problematic alignments are usually caused by subdivision points which determine the top of an up-arc, because the braiding produces pairs of vertical braid strands aligned with such points. For example, when braiding a diagram that displays vertical alignment between subdivision points, the resulting pairs of vertical braid strands correspond to the same endpoints. In addition, when braiding a diagram that displays vertical alignment between a subdivision point and a crossing or a vertex, we obtain a multiple point (which in a flat projection is a vertex of degree six or five) between the crossing or vertex and one of the new vertical strands.
  \item [(ii)] We can now give a proper justification as to why we introduced the triangle condition. Let $G$ be an STG diagram equipped with subdivision points and labels, and suppose that $G$ contains two overlapping sliding triangles with the same label. The middle diagram of Figure~\ref{ambiguous} depicts the  magnified region with such overlapping triangles. Now the order in which we braid the up-arcs of $G$ will affect the final braid. The diagram in the left hand side of Figure~\ref{ambiguous} is obtained by first braiding the up-arc with left-to-right orientation; conversely, the diagram on the right hand side of the same figure is obtained by first braiding the up-arc with right-to-left orientation. Note that the two braid diagrams differ by the type of crossing introduced. We avoid this kind of behavior by introducing the triangle condition, which ensures that the braiding moves do not interfere with each other. That is, the order in which we braid the up-arcs does not affect the final braid. This will be important for the proof of the Markov-type theorem.
  \end{enumerate}
\end{remark}

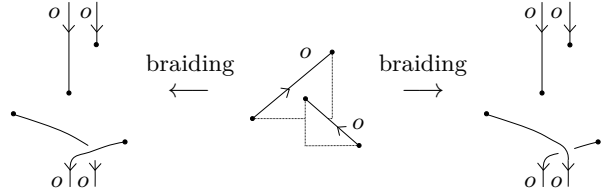
\begin{figure}[ht]
  \[
  \raisebox{-35pt}{\includegraphics[height=1in]{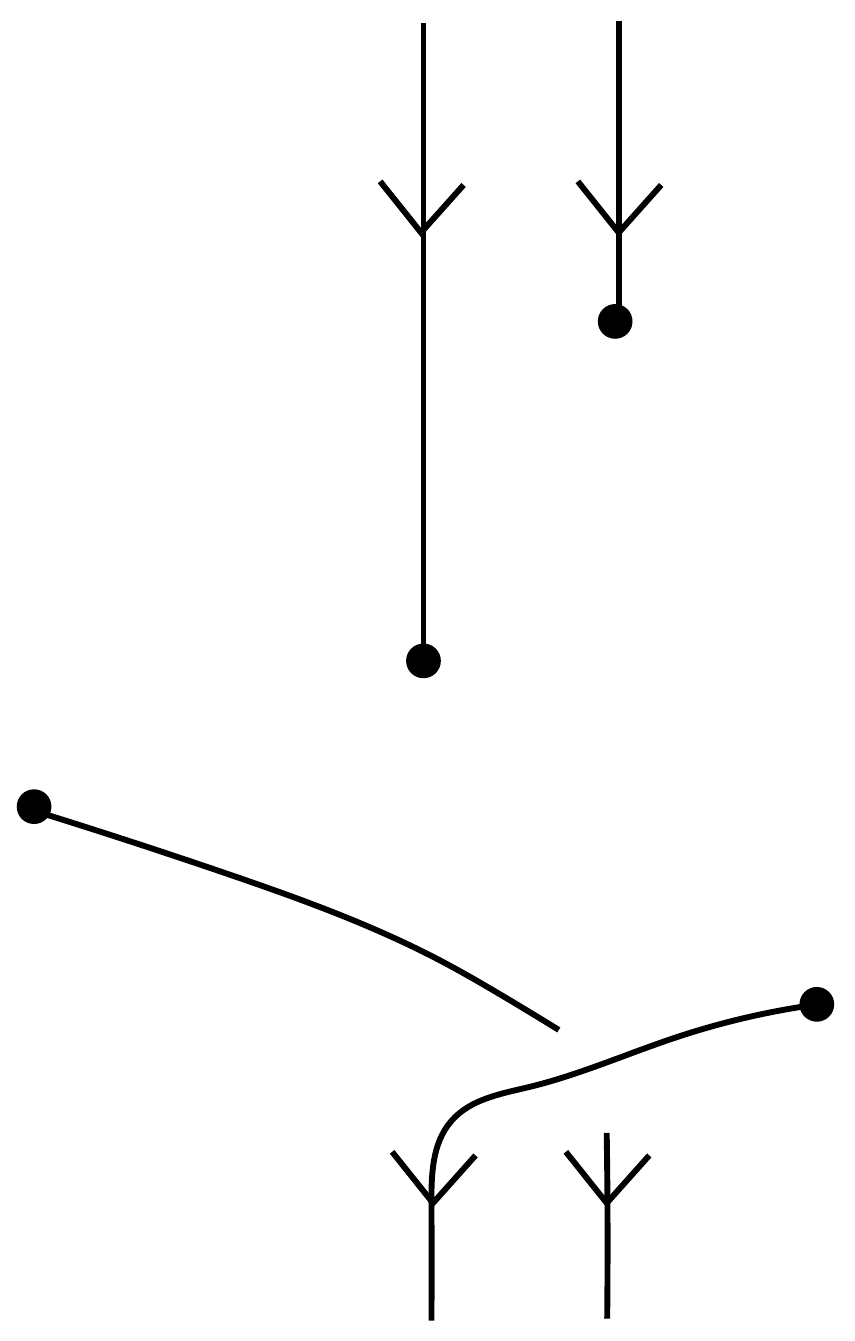}}
  \hspace{.4cm} \longleftarrow \hspace{.4cm}
  \raisebox{-20pt}{\includegraphics[height=.54in]{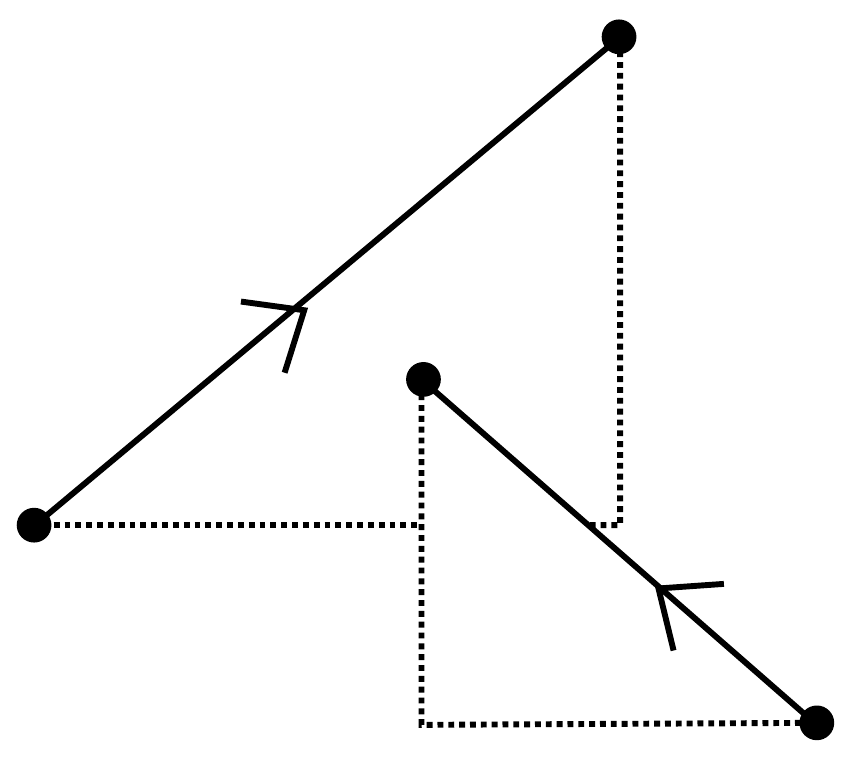}}
  \hspace{.4cm} \longrightarrow \hspace{.4cm}
  \raisebox{-35pt}{\includegraphics[height=1in]{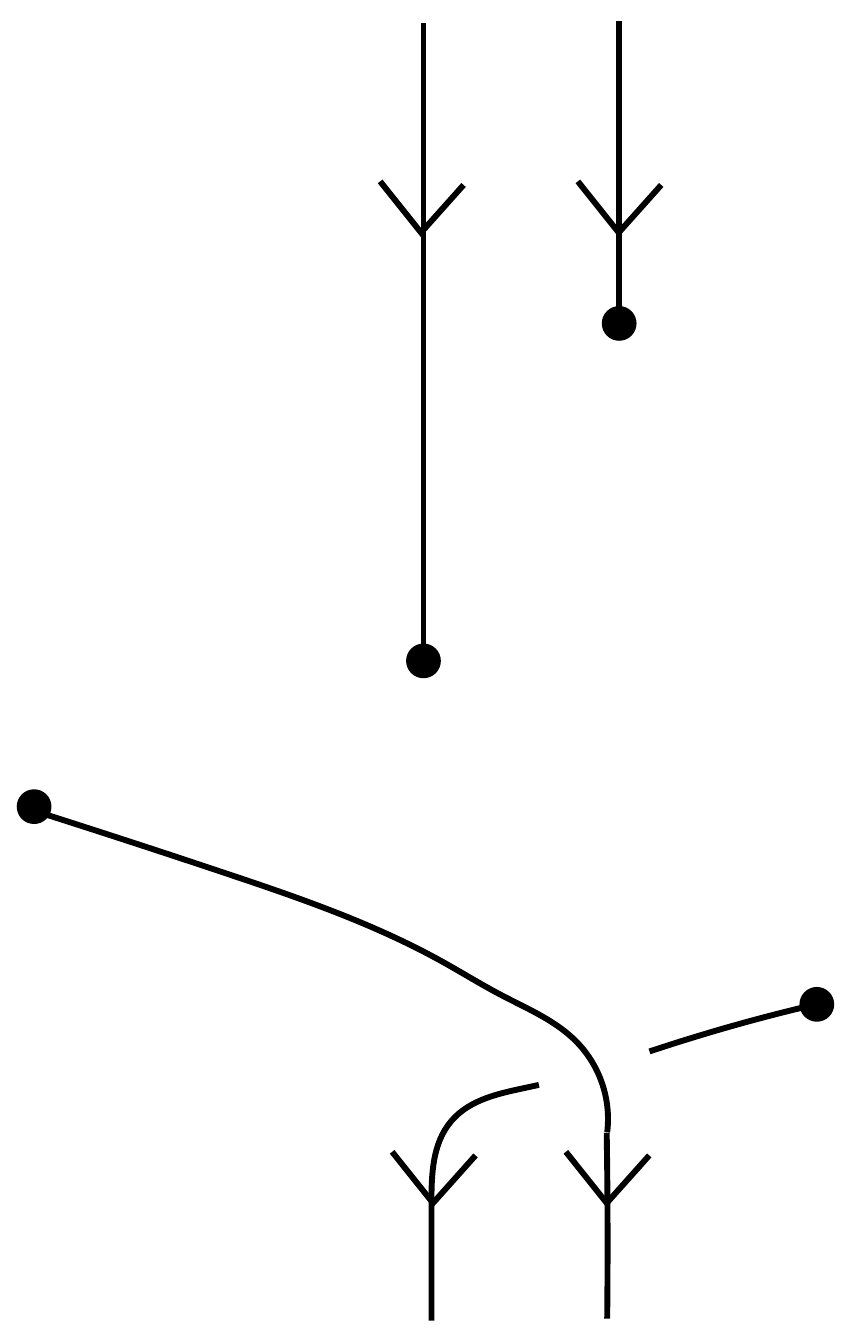}}
  \put(-210,
  30){\fontsize{9}{9}$o$} \put(-209,
  -34){\fontsize{9}{9}$o$} \put(-198,
  30){\fontsize{9}{9}$o$} \put(-199,
  -34){\fontsize{9}{9}$o$} \put(-115,
  15){\fontsize{9}{9}$o$} \put(-95,
  -12){\fontsize{9}{9}$o$} \put(-30,
  30){\fontsize{9}{9}$o$} \put(-30,
  -34){\fontsize{9}{9}$o$} \put(-20,
  30){\fontsize{9}{9}$o$} \put(-20,
  -34){\fontsize{9}{9}$o$} \put(-85, 10){\small{braiding}} \put(-173,
  10){\small{braiding}}
  \]
  \caption{Forbidden sliding triangles}
  \label{ambiguous}
\end{figure}

After braiding every up-arc in a given STG diagram in general position, we obtain a trivalent braid diagram. Recall that well-oriented STG diagrams contain an even number of vertices, half of which are zip and the other half are unzip vertices.  This in fact means that by braiding an STG diagram, we indeed obtain an
$(n,n)$ trivalent braid, where $n$ is a positive integer. Hence the closure operation is well defined. Ultimately, by braiding an STG diagram, we obtain a trivalent braid whose closure is isotopic to the original STG diagram. We state formally the previous discussion in the following theorem.

\begin{theorem}[Alexander-type Theorem for STGs]
  Every well-oriented spatial trivalent graph can be represented as the closure of a trivalent braid.
\end{theorem}

\begin{remark}
  In our braiding algorithm, we require that each up-arc is subdivided such that it contains at most one crossing. In fact, we can relax this condition a bit by allowing for up-arcs to cross over any number of strands. The only requirement is that an up-arc contains crossings of only one type. That is, an up-arc must cross either over or under\---- but not both\---- any other strand.
\end{remark}


\begin{example} We provide in Figure~\ref{example} an example of the braiding algorithm. We start with a well-oriented spatial trivalent graph diagram representing a spatial version of the Petersen graph. The first step is to put the vertices in regular position according to the conventions in Figures~\ref{vgp},~\ref{vgp2}, and~\ref{vyl}. Then we convert the resulting diagram into a diagram in general position (recall Definition~\ref{gen pos}). Finally, we braid the free up-arcs and all crossings containing up-arcs to arrive at a trivalent braid whose closure is isotopic to the original STG diagram.
 
\begin{figure}[ht]
 \[
 \raisebox{10pt}{\includegraphics[height=1.4in]{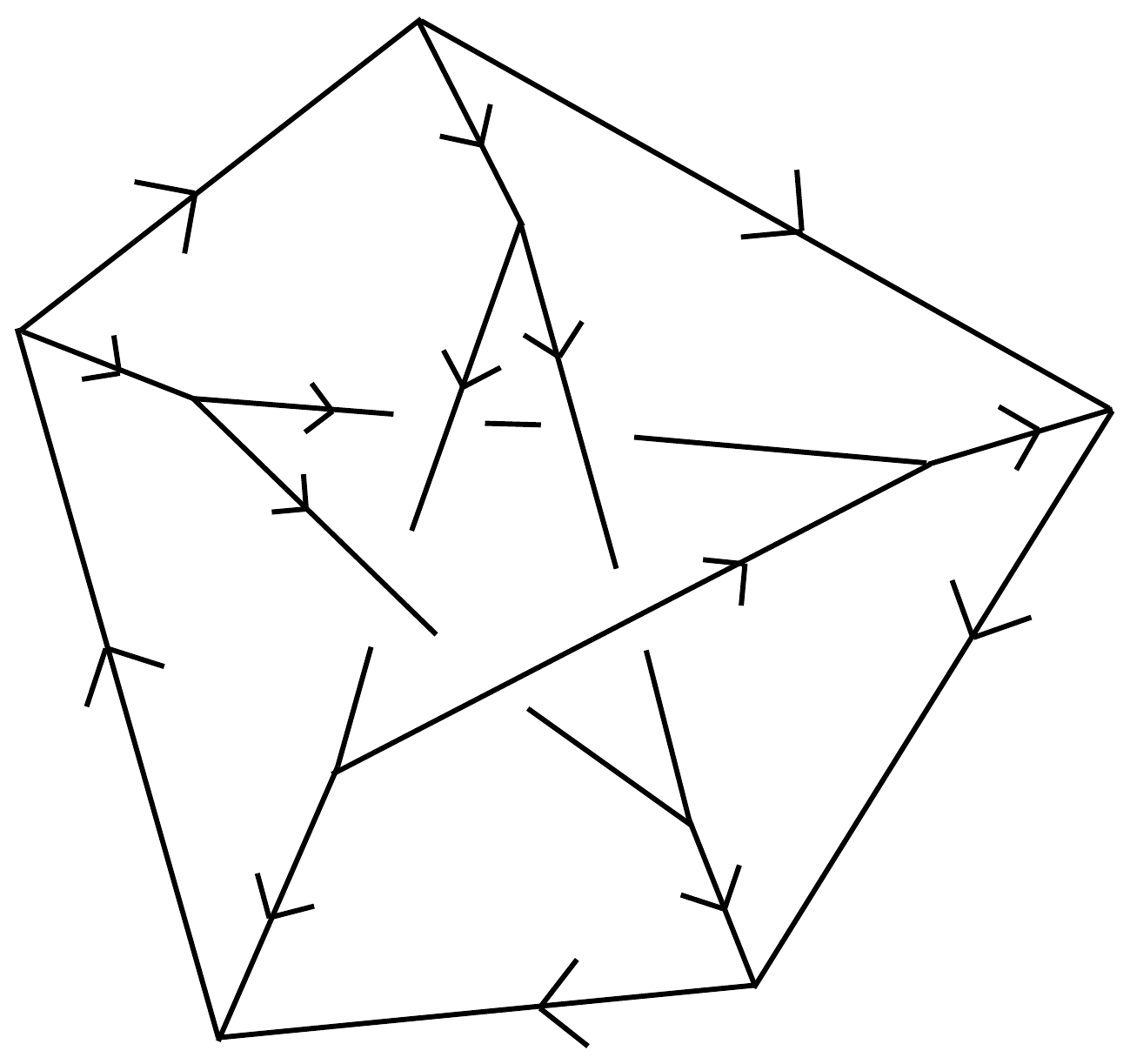}} \hspace{.2cm} \raisebox{60pt}{$\longrightarrow$} \hspace{.2cm} \includegraphics[height=1.7in]{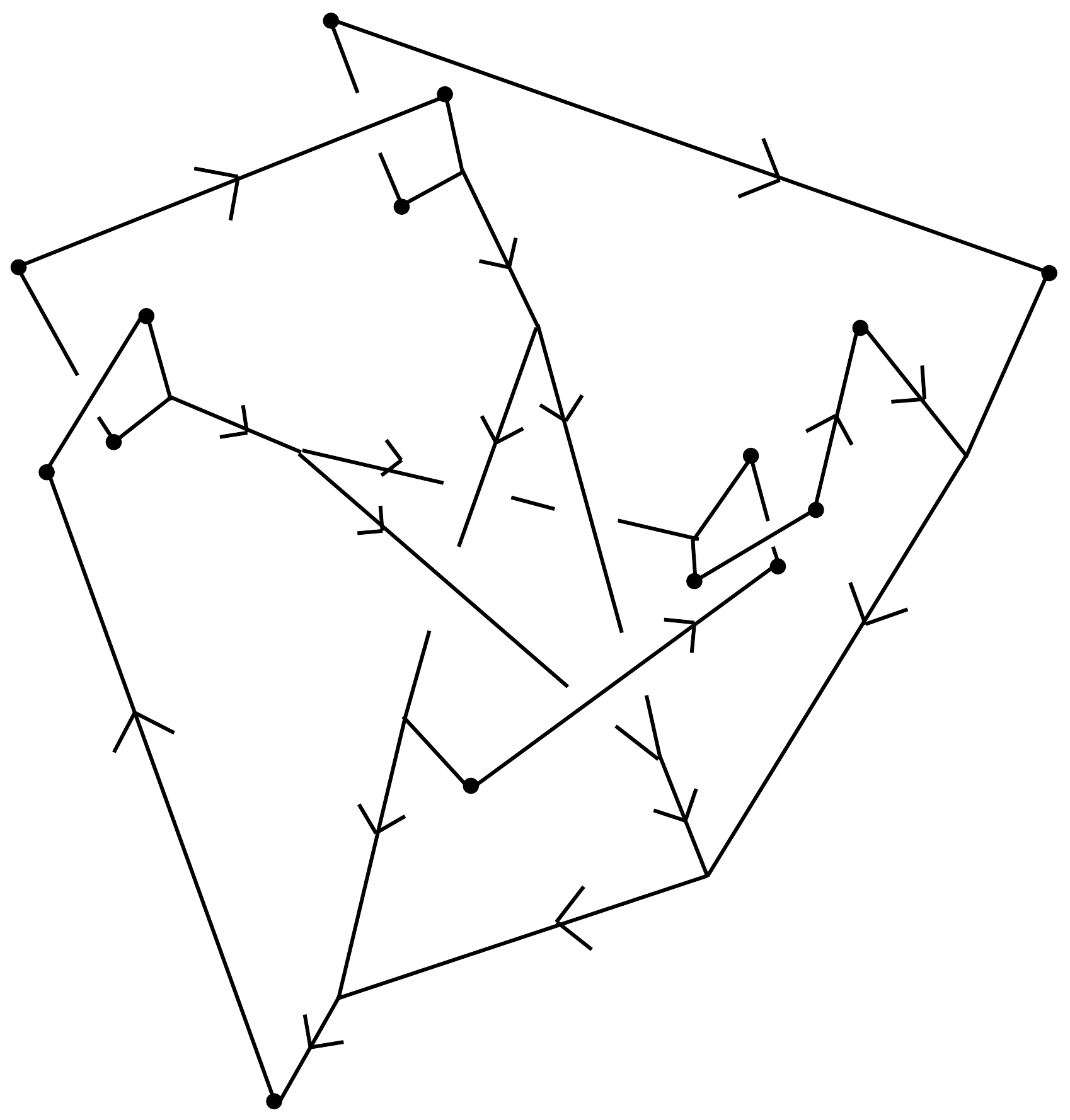} \raisebox{60pt}{$\longrightarrow$}
\]
\put(-110,0){\fontsize{9}{9}STG diagram}
\put(13,0){\fontsize{9}{9}Vertices in regular position}
\vspace{.5cm}
\[
\raisebox{10pt}{\includegraphics[height=1.7in]{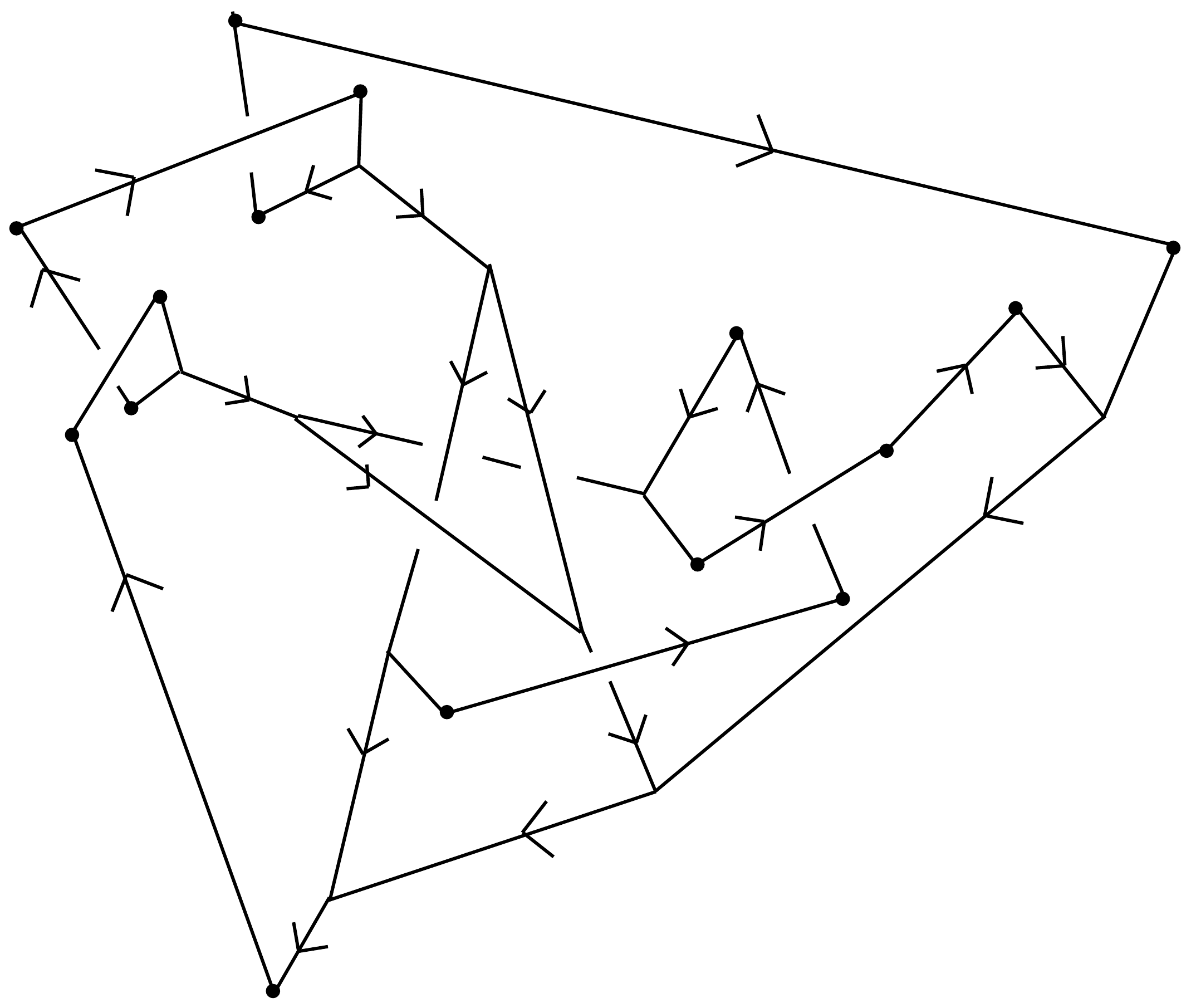}} \hspace{.2cm} \raisebox{60pt}{$\longrightarrow$} \hspace{.2cm} \includegraphics[height=2.1in]{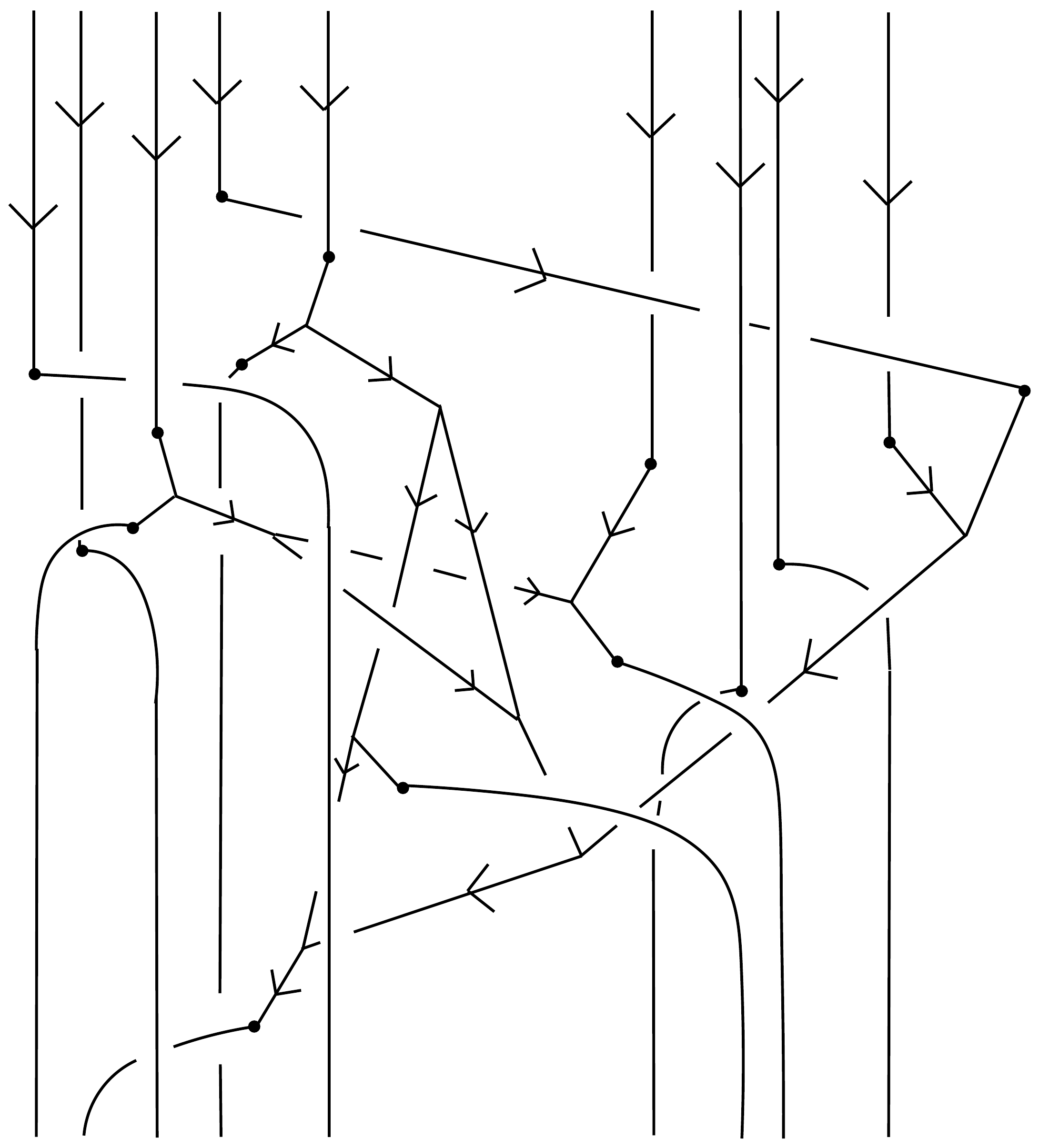}
 \]
\put(-150,0){\fontsize{9}{9}Diagram in general position}
\put(40,0){\fontsize{9}{9}Corresponding braid}
\caption{An example of the braiding algorithm} \label{example}
\end{figure}
\end{example}


\section{Markov-type theorems for trivalent braids}

Analogous to classical knot theory and in particular to the work by Markov~\cite{M}, we want to classify trivalent braids that, upon the closure operation, yield STG diagrams representing isotopic spatial trivalent graphs.

\subsection{Trivalent $L$-equivalence}\label{sec:TL-equiv}

The goal of this section is to define an equivalence relation on the set of trivalent braids, which we refer to as \textbf{trivalent $L$-equivalence} (or \textbf{$TL$-equivalence}). This equivalence relation can be seen as an extension of the $L$-equivalence between classical braids. The $L$-equivalence is described solely by braid isotopy and the $L$-moves introduced by Lambropoulou in her Ph.D. thesis~\cite{L}, and used to prove the `one-move Markov theorem' for oriented links (see also~\cite[Theorem 2.3]{LR}). The $L$-moves for classical braids extend naturally to trivalent braids, as we shall explain in this section.

\begin{definition}
  A \textbf{basic $L$-move} on a trivalent braid consists of cutting an arc of the braid at a point (such point cannot be a vertex), and pulling the upper cut-point downward and the lower cut-point upward, therefore, creating a new pair of vertical braid strands, as explained in Figure~\ref{basicL}. The new pair of braid strands are vertically aligned with the cut-point, and they either cross over or under\---- but not both\---- any other arc of the braid. Consequently, there are two types of the basic $L$-moves, namely an \textbf{under $L$-move} ($L_u$-move) and an \textbf{over $L$-move} ($L_o$-move).
\end{definition}

\begin{figure}[ht]
  \[
  \raisebox{-30pt}{\includegraphics[height=1in]{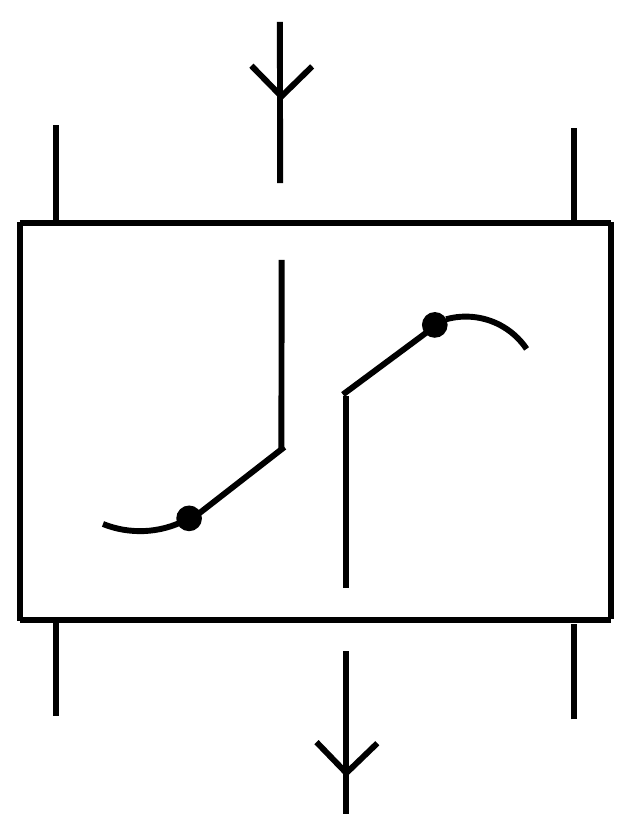}}
  \hspace{.29cm} \longleftrightarrow \hspace{.29cm}
  \raisebox{-23pt}{\includegraphics[height=.8in]{BasicLMove1}}
  \hspace{.29cm} \longleftrightarrow \hspace{.29cm}
  \raisebox{-30pt}{\includegraphics[height=1in]{BasicLMove2}}
  \put(-27,28){\fontsize{9}{9}$o$}
  \put(-34,-21){\fontsize{9}{9}$o$}
  \put(-223,28){\fontsize{9}{9}$u$}
  \put(-227,-21){\fontsize{9}{9}$u$}
  \put(-192,10){\fontsize{9}{9}$L_u$-move} \put(-185,20){\text{Basic}}
  \put(-93,10){\fontsize{9}{9}$L_o$-move} \put(-87,20){\text{Basic}}
  \]
  \caption{Basic $L_u$- and $L_o$-moves} \label{basicL}
\end{figure}

Using braid isotopy, an $L$-move may be formulated with a crossing (positive or negative) which can be either to the right or to the left of the 
cut-point; we refer to these as a \textbf{right $L$-move} or a \textbf{left $L$-move}, respectively. More precisely, we refer to this version of the $L$-move as either \textbf{right/left $+L$-move} or \textbf{right/left $-L$-move}, in accordance with the type of crossing being created (see Figure~\ref{crossL}).

\begin{figure}[ht]
  \[
  \raisebox{-30pt}{\includegraphics[height=1in]{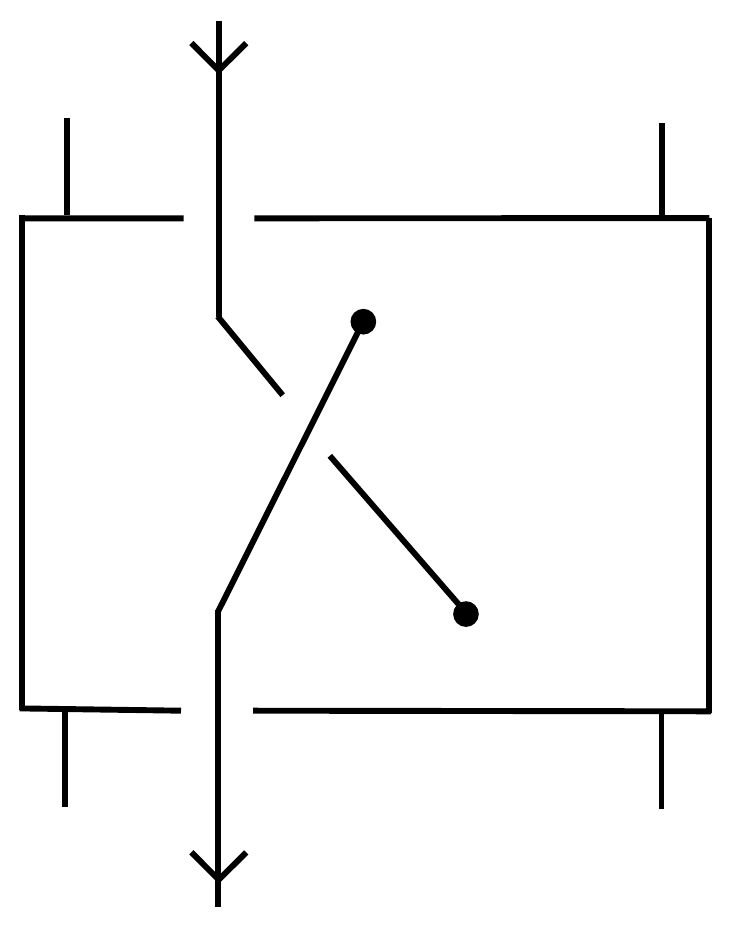}}
  \hspace{.4cm} \longleftrightarrow \hspace{.4cm}
  \raisebox{-23pt}{\includegraphics[height=.8in]{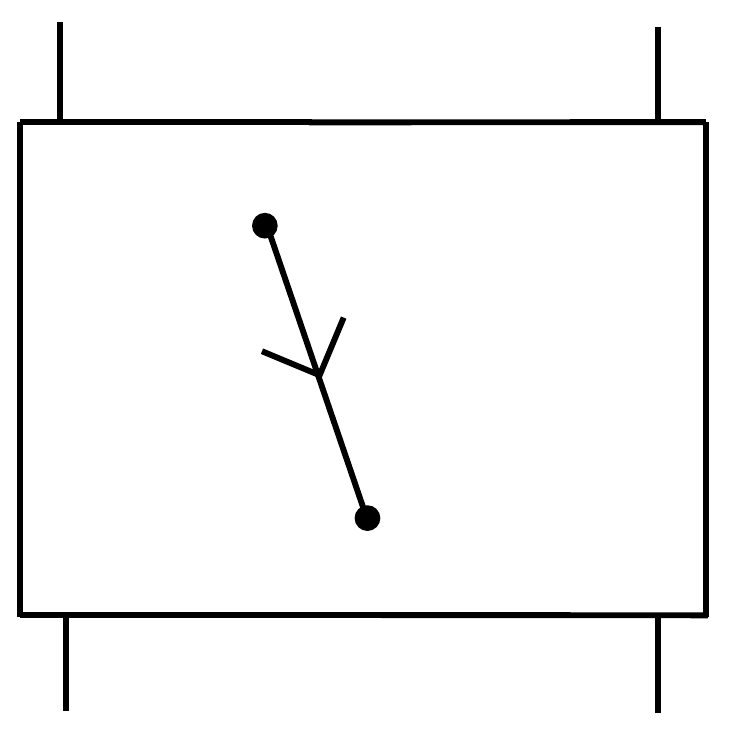}}
  \hspace{.4cm} \longleftrightarrow \hspace{.4cm}
  \raisebox{-30pt}{\includegraphics[height=1in]{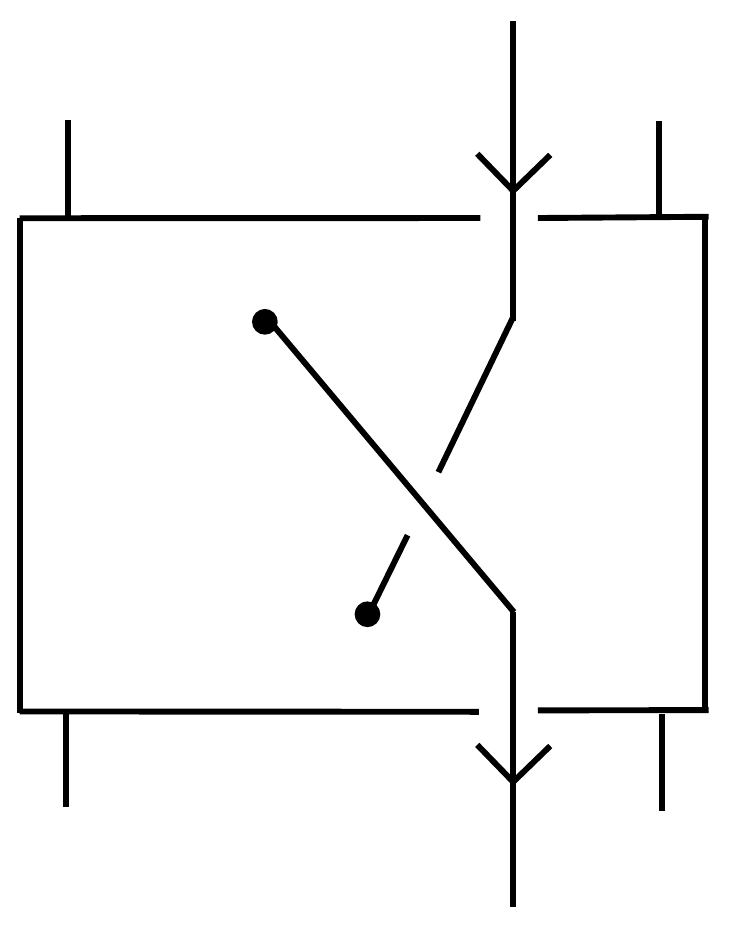}}
  \put(-27,30){\fontsize{9}{9}$o$}
  \put(-27,-23){\fontsize{9}{9}$o$}
  \put(-244,30){\fontsize{9}{9}$o$}
  \put(-244,-23){\fontsize{9}{9}$o$} \put(-188,20){\small{Left}}
  \put(-206,10){\small{\fontsize{9}{9}$+L_o$-move}}
  \put(-88,20){\small{Right}}
  \put(-100,10){\small{\fontsize{9}{9}$-L_o$-move}}
  \]
  \[
  \raisebox{-30pt}{\includegraphics[height=1in]{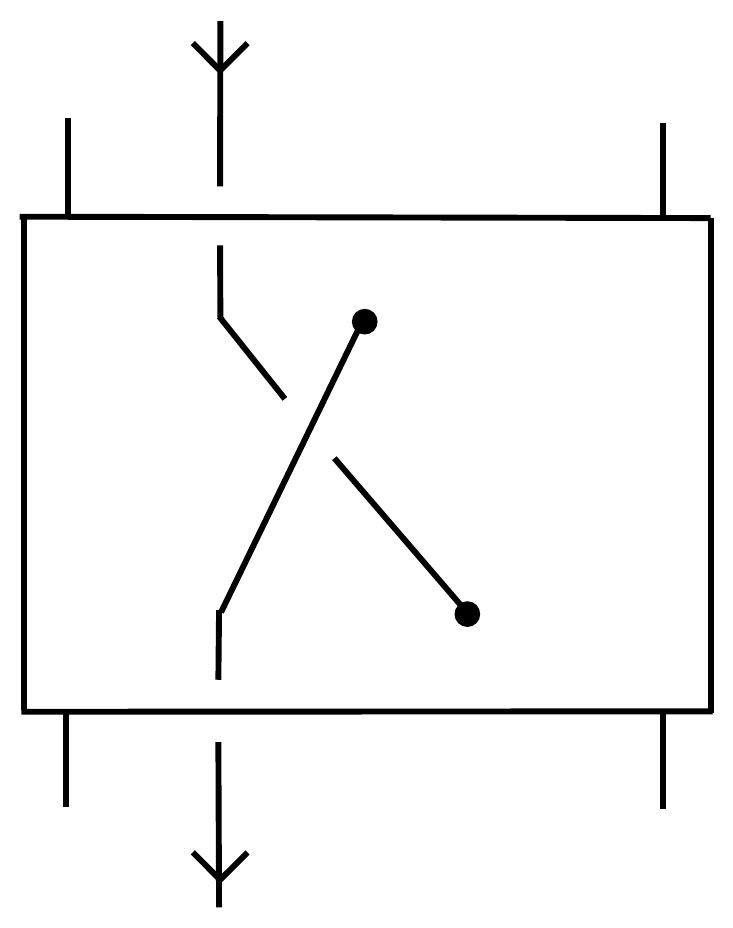}}
  \hspace{.4cm} \longleftrightarrow \hspace{.45cm}
  \raisebox{-23pt}{\includegraphics[height=.8in]{CrossLMove}}
  \hspace{.4cm} \longleftrightarrow \hspace{.4cm}
  \raisebox{-30pt}{\includegraphics[height=1in]{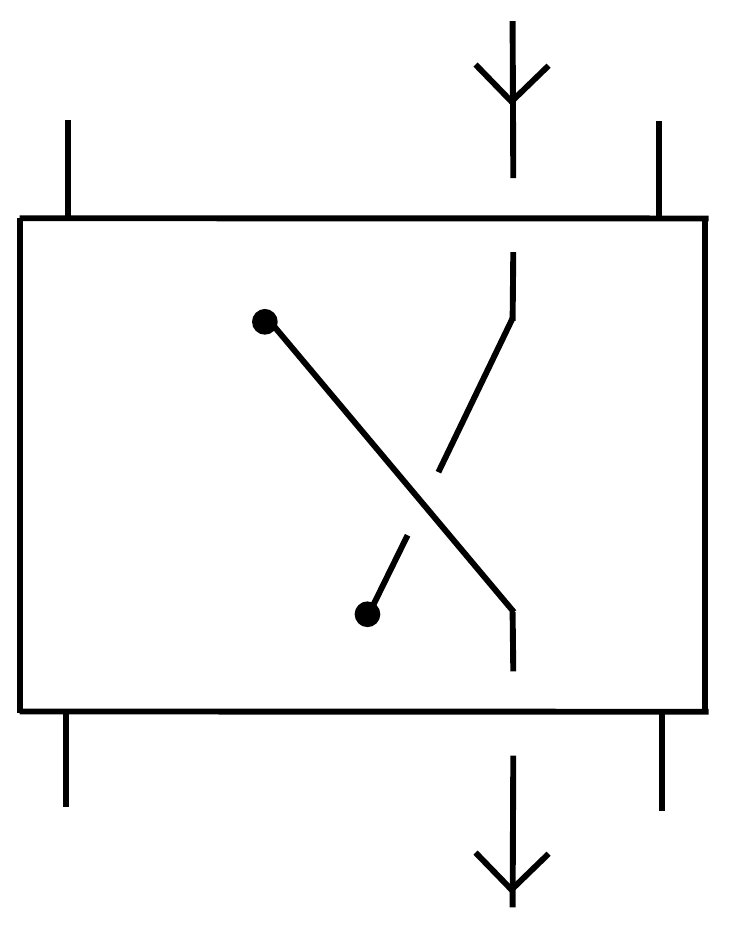}}
  \put(-27,30){\fontsize{9}{9}$u$}
  \put(-27,-23){\fontsize{9}{9}$u$}
  \put(-244,30){\fontsize{9}{9}$u$}
  \put(-244,-23){\fontsize{9}{9}$u$} \put(-188,20){\small{Left}}
  \put(-206,10){\small{\fontsize{9}{9}$+L_u$-move}}
  \put(-88,20){\small{Right}}
  \put(-100,10){\small{\fontsize{9}{9}$-L_u$-move}}
  \]
  \caption{Left and right $L$-moves} \label{crossL}
\end{figure}

\newpage

\begin{remark} Some comments about the $L$-moves on trivalent braids are needed.
  \begin{enumerate}
  \item[(i)] The effect of the $L$-move is to stretch an arc around the braid axis, where the arc is being stretched either over or under the braid diagram. Therefore, such a move between trivalent braids yields isotopic closures.
  \item[(ii)] Although we defined the $L$-moves in one direction, we allow for an $L$-move to be undone in a braid. That is, we allow for the contraction of a pair of vertical strands that correspond to an $L$-move, so as to obtain a down-arc.
  \end{enumerate}
\end{remark}

\begin{definition}
  The \textbf{$TL$-equivalence} is the equivalence relation on the set of trivalent braids determined by (1) braid isotopy and (2) right $L$-moves.
\end{definition}

Note that we did not include neither the basic $L$-moves nor the left $L$-moves in the definition of $TL$-equivalence. In the next lemma, we show that these $L$-moves follow from the right $L$-moves and braid isotopy (see also~\cite{LR}). This will give us the freedom to use all versions of the $L$-moves when comparing $TL$-equivalent braids.

\begin{lemma}
  The basic $L$-moves and left $L$-moves follow from the right $L$-moves together with braid isotopy.
  \label{L equiv}
\end{lemma}
\begin{proof}
  Figure~\ref{prop:basicL} shows how a basic $L_o$-move follows from the right $L_o$-moves together with Reidemeister moves $R2$ in braid form. Then Figure~\ref{prop:crosstype} shows that the left $L_o$-moves follow from basic $L_o$-moves. Therefore, left $L_o$-moves follow from right $L_o$-moves. A similar argument can be used to show that the basic and left $L_u$-moves follow from the right $L_u$-moves and braid isotopy.
\end{proof}

\begin{figure}[ht]
  \[
  \raisebox{-35pt}{\includegraphics[height=1in]{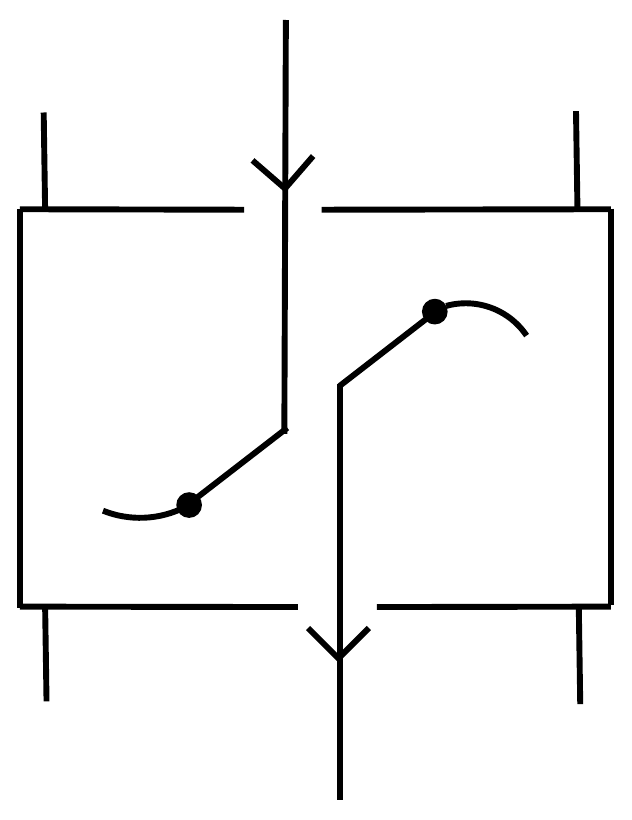}}
  \hspace{.4cm} \longleftrightarrow \hspace{.4cm}
  \raisebox{-35pt}{\includegraphics[height=1in]{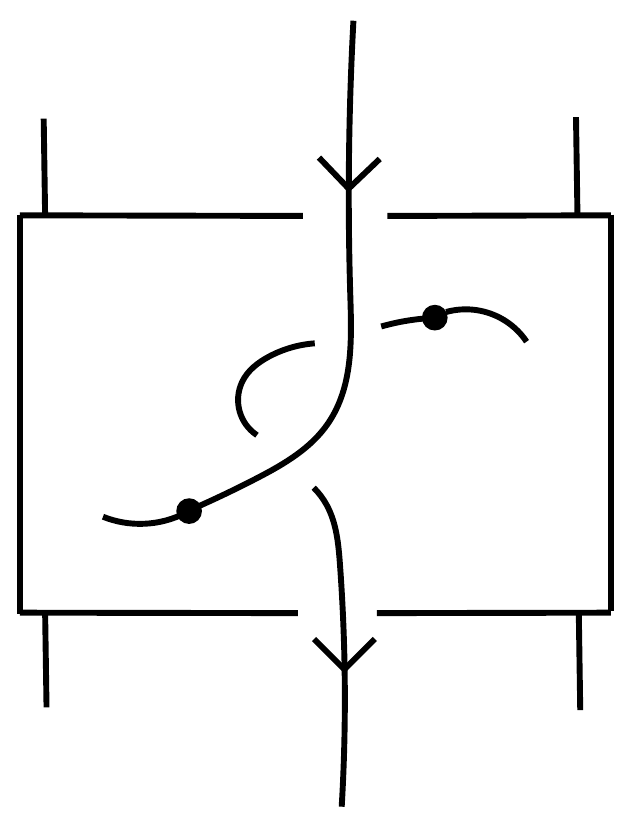}}
  \hspace{.4cm} \longleftrightarrow \hspace{.4cm}
  \raisebox{-26pt}{\includegraphics[height=.75in]{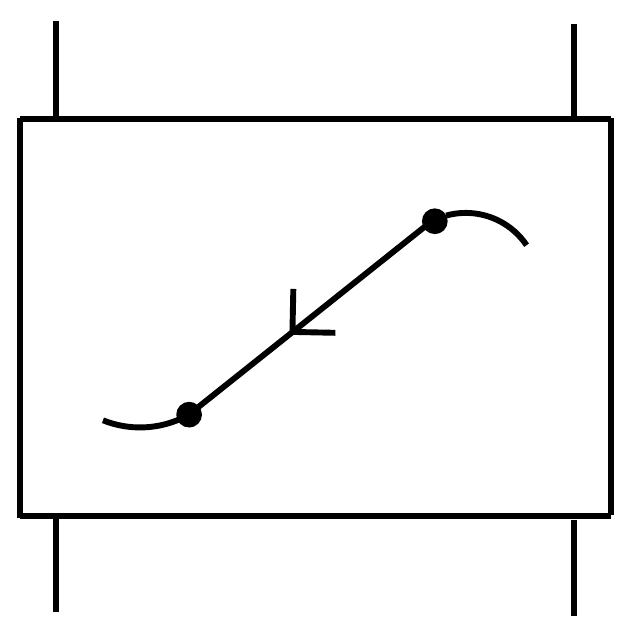}}
  \put(-230,30){\fontsize{9}{9}$o$}
  \put(-226,-30){\fontsize{9}{9}$o$}
  \put(-123,30){\fontsize{9}{9}$o$}
  \put(-124,-30){\fontsize{9}{9}$o$} \put(-193,10){\small{br.
      $R2$}} \put(-90,20){\small{Right}}
  \put(-98,10){\small{\fontsize{9}{9}$+L_o$-move}}
  \put(-260,-45){\small{Basic \fontsize{9}{9}$L_o$-move}}
  \]
  \[
  \raisebox{-35pt}{\includegraphics[height=1in]{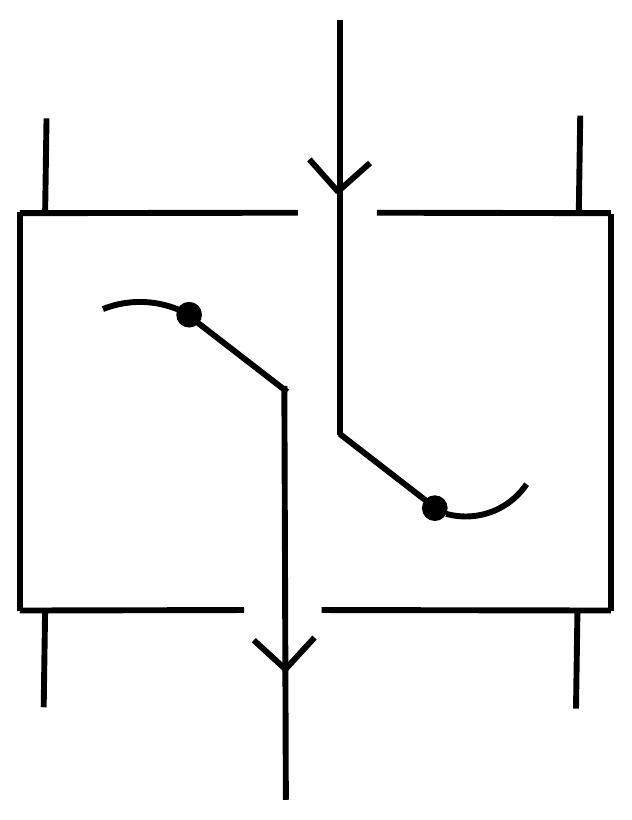}}
  \hspace{.4cm} \longleftrightarrow \hspace{.4cm}
  \raisebox{-35pt}{\includegraphics[height=1in]{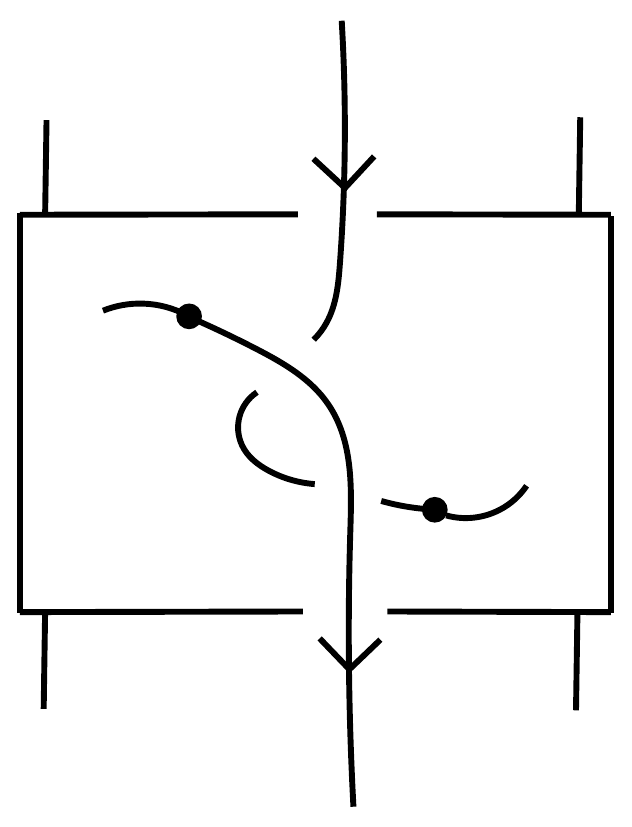}}
  \hspace{.4cm} \longleftrightarrow \hspace{.4cm}
  \raisebox{-26pt}{\includegraphics[height=.75in]{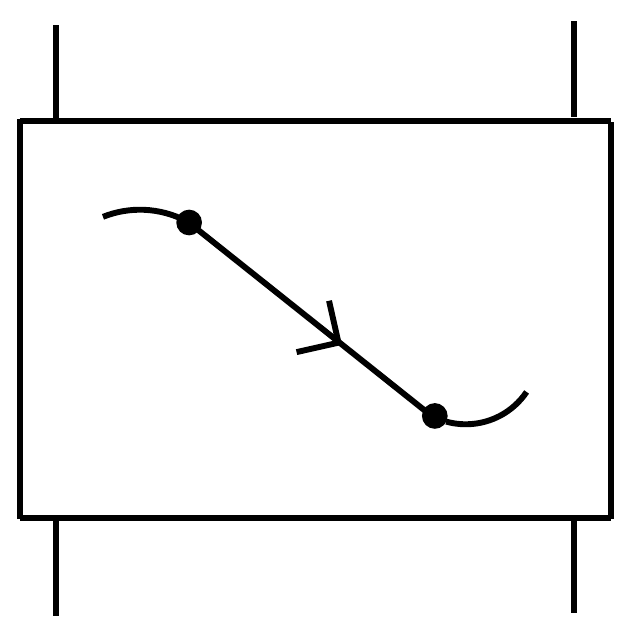}}
  \put(-227,30){\fontsize{9}{9}$o$}
  \put(-231,-30){\fontsize{9}{9}$o$}
  \put(-123,30){\fontsize{9}{9}$o$}
  \put(-124,-30){\fontsize{9}{9}$o$} \put(-193,10){\small{br.
      $R2$}} \put(-90,20){\small{Right}}
  \put(-98,10){\small{\fontsize{9}{9}$-L_o$-move}}
  \put(-260,-45){\small{Basic \fontsize{9}{9}$L_o$-move}}
  \]
  \caption{Basic $L_o$-moves via right $\pm
    L_o$-moves} \label{prop:basicL}
\end{figure}

\begin{figure}[ht]
  \[
  \raisebox{-35pt}{\includegraphics[height=1in]{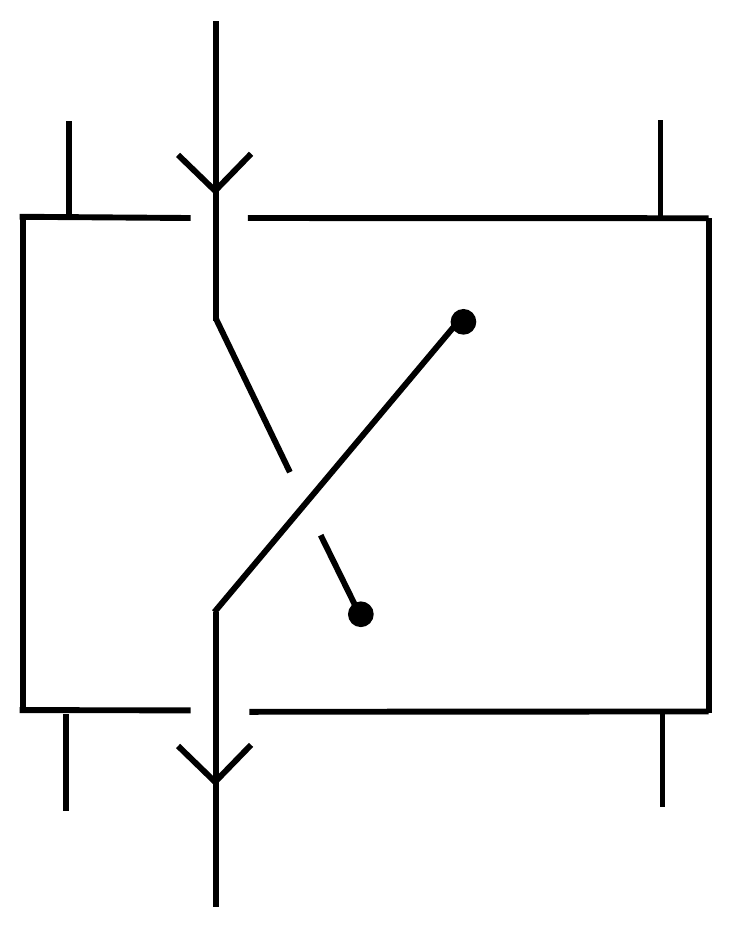}}
  \hspace{.4cm} \longleftrightarrow \hspace{.4cm}
  \raisebox{-35pt}{\includegraphics[height=1in]{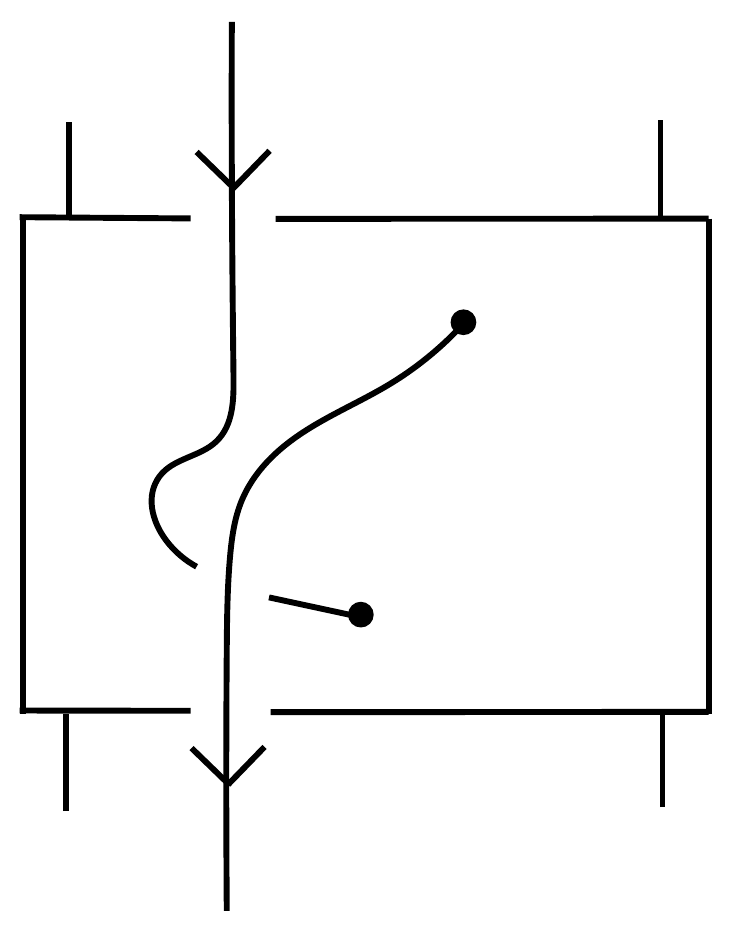}}
  \hspace{.4cm} \longleftrightarrow \hspace{.4cm}
  \raisebox{-27pt}{\includegraphics[height=.78in]{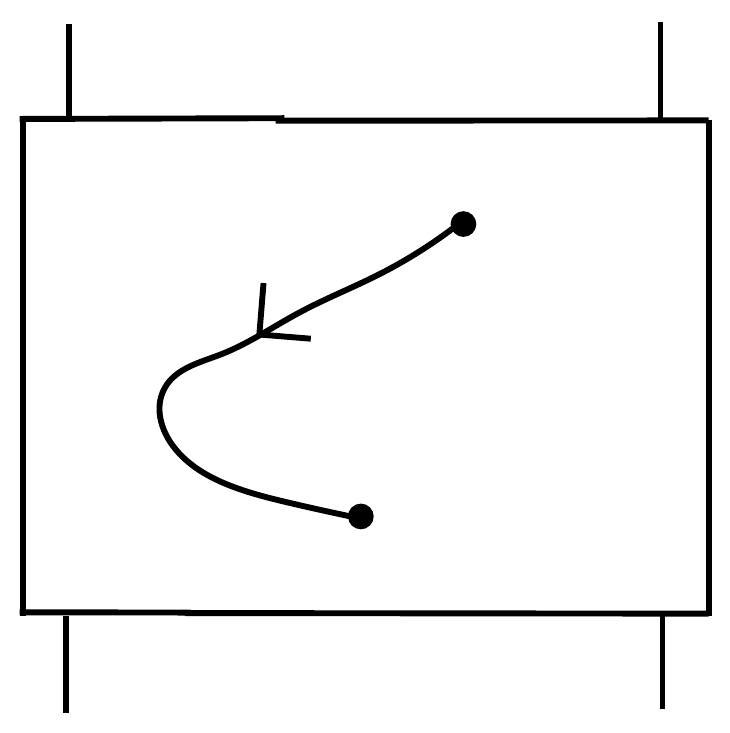}}
  \put(-243,30){\fontsize{9}{9}$o$}
  \put(-243,-30){\fontsize{9}{9}$o$}
  \put(-139,30){\fontsize{9}{9}$o$}
  \put(-139,-30){\fontsize{9}{9}$o$} \put(-190,20){\small{braid}}
  \put(-193,10){\small{isotopy}}
  \put(-95,10){\small{\fontsize{9}{9}$L_o$-move}}
  \put(-260,-45){\small{Left \fontsize{9}{9}$+L_o$-move}}
  \]
  \[
  \raisebox{-35pt}{\includegraphics[height=1in]{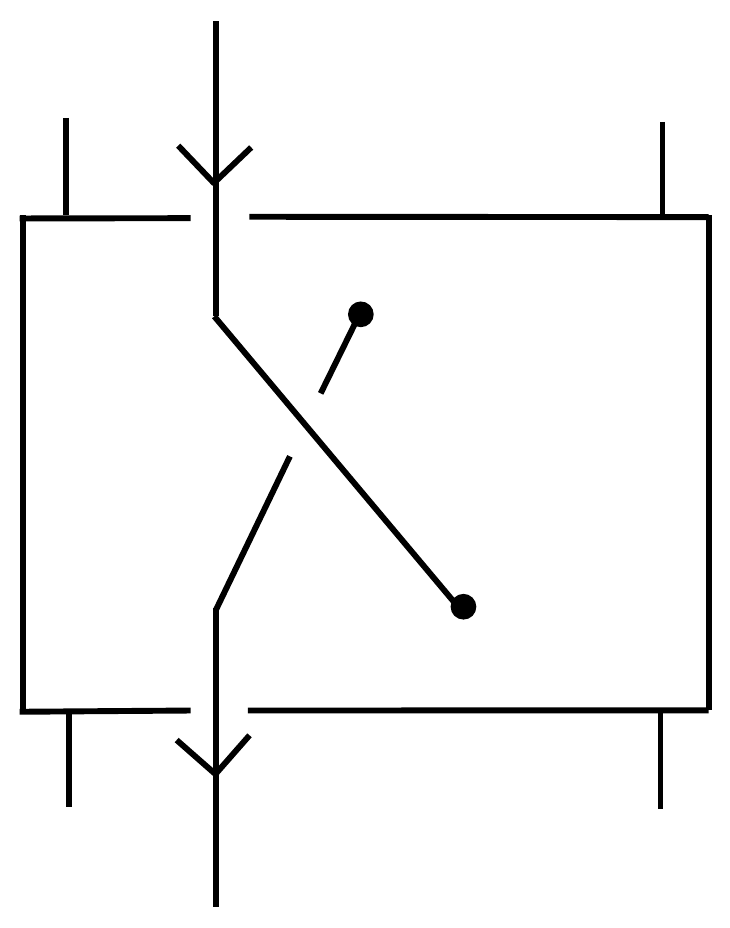}}
  \hspace{.4cm} \longleftrightarrow \hspace{.4cm}
  \raisebox{-35pt}{\includegraphics[height=1in]{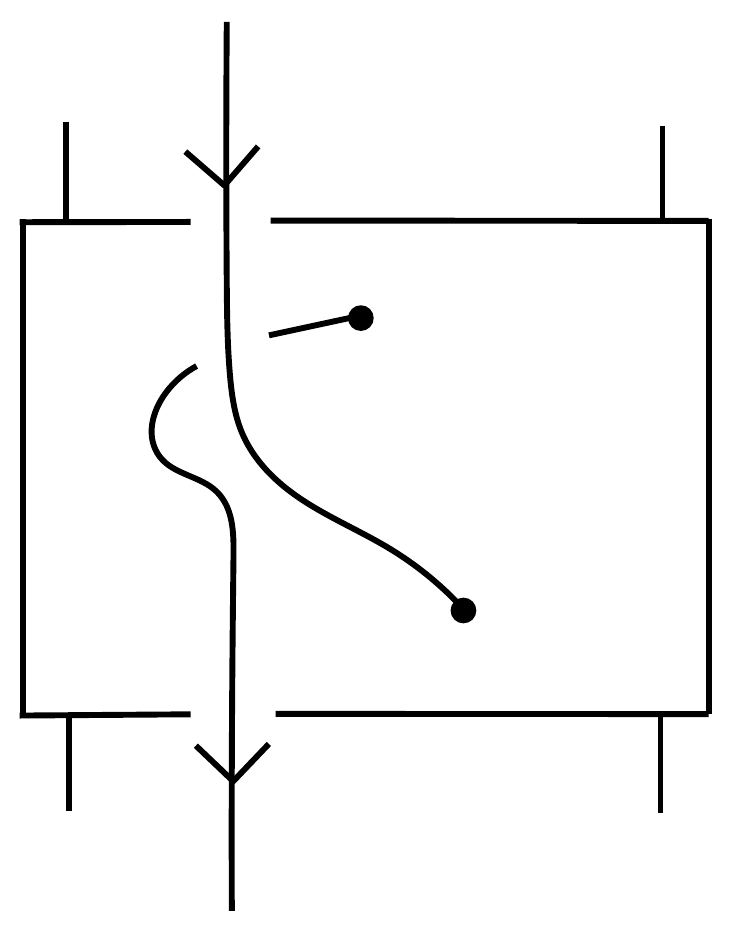}}
  \hspace{.4cm} \longleftrightarrow \hspace{.4cm}
  \raisebox{-27pt}{\includegraphics[height=.78in]{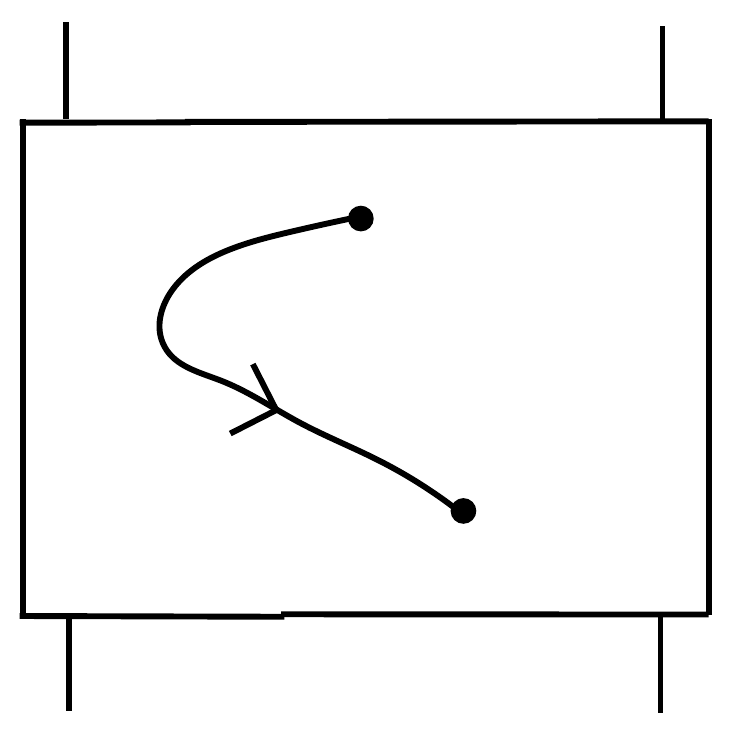}}
  \put(-243,30){\fontsize{9}{9}$o$}
  \put(-243,-30){\fontsize{9}{9}$o$}
  \put(-139,30){\fontsize{9}{9}$o$}
  \put(-139,-30){\fontsize{9}{9}$o$} \put(-190,20){\small{braid}}
  \put(-193,10){\small{isotopy}}
  \put(-95,10){\small{\fontsize{9}{9}$L_o$-move}}
  \put(-260,-45){\small{Left \fontsize{9}{9}$-L_o$-move}}
  \]
  \caption{Left $\pm L$-moves via basic
    $L$-moves} \label{prop:crosstype}
\end{figure}
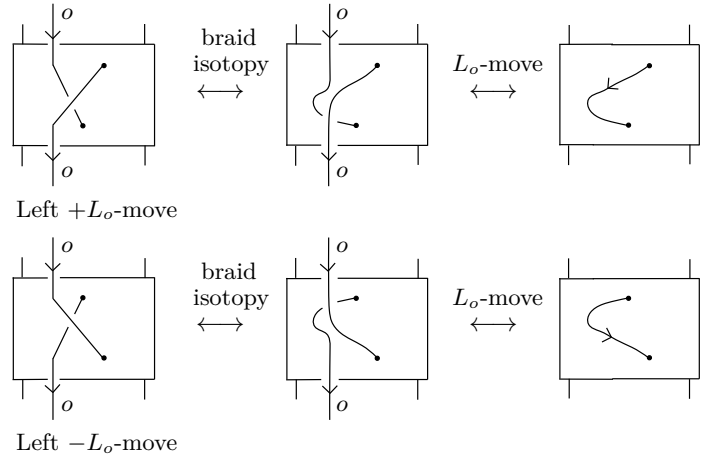

We remark that the $L$-equivalence for classical braids introduced in~\cite{L, LR} comprises classical braid isotopy and the right $L$-move for classical braids. Thus, the $TL$-equivalence, when restricted to classical braids, is the same as the $L$-equivalence. In other words, the $L$-equivalence for classical braids extends to trivalent braids (as long as the $L$-moves are applied away from trivalent vertices, as we explained in the previous discussion).

Finally, we define conjugation by elementary trivalent braids $\sigma_i$ and $\sigma_i^{-1}$ in $TB_n^n$. Given a trivalent braid $b \in TB_n^n$, we say that the braids $b \sigma_i^{\pm 1} \sim \sigma_i^{\pm 1} b$, where $1 \leq i \leq n-1$, are related by \textbf{elementary conjugation} (see Figure~\ref{fig:conj2}). Since $\sigma_i$ is invertible in $TB_n^n$ with inverse $\sigma_i^{-1}$, the elementary conjugation in $TB_n^n$ has the following equivalent form: $b \sim \sigma_i b \sigma_i^{-1}$ or $b \sim \sigma_i^{-1} b \sigma_i$.

The statement of Markov's theorem~\cite{M} for classical braids makes use of conjugation by $\sigma_i^{\pm 1}$. But when employing the $L$-moves, the elementary conjugation can be dropped from the $L$-move Markov-type theorem, since elementary conjugation for classical braids follows from $L$-equivalence (see~\cite{L, LR}). The same holds for trivalent braids, as we now prove.

\begin{lemma} \label{conj} Elementary conjugation in a trivalent braid
  can be realized by a sequence of $L$-moves together with braid
  isotopy.
\end{lemma}

\begin{proof}
  The proof is illustrated in Figure~\ref{fig:conj} (compare with~\cite{HL}). We start with an $(n,n)$ trivalent braid of the form $\sigma_i^{-1}b$, where $b \in TB_n^n$ and $1 \leq i \leq n-1$, and obtain the trivalent braid $b\sigma_i^{-1}$ through a sequence of $L$-moves and trivalent braid isotopy.
\end{proof}

\begin{figure}[ht]
  \[
  \raisebox{-35pt}{\includegraphics[height=1in]{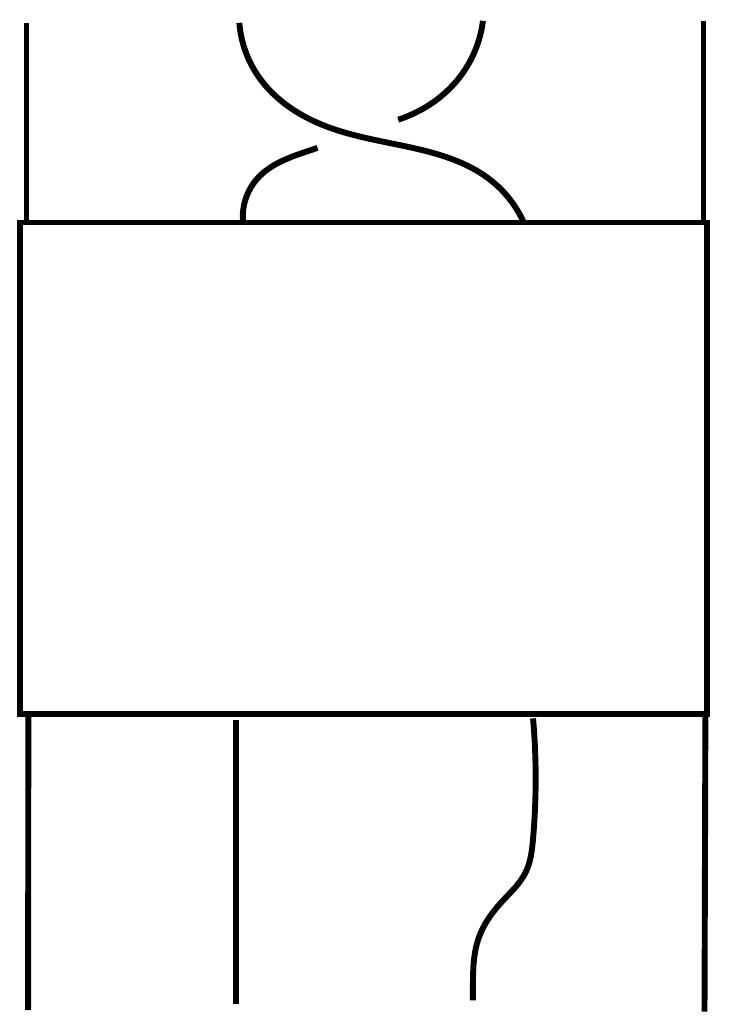}}
  \hspace{.4cm} \longleftrightarrow \hspace{.4cm}
  \raisebox{-35pt}{\includegraphics[height=1in]{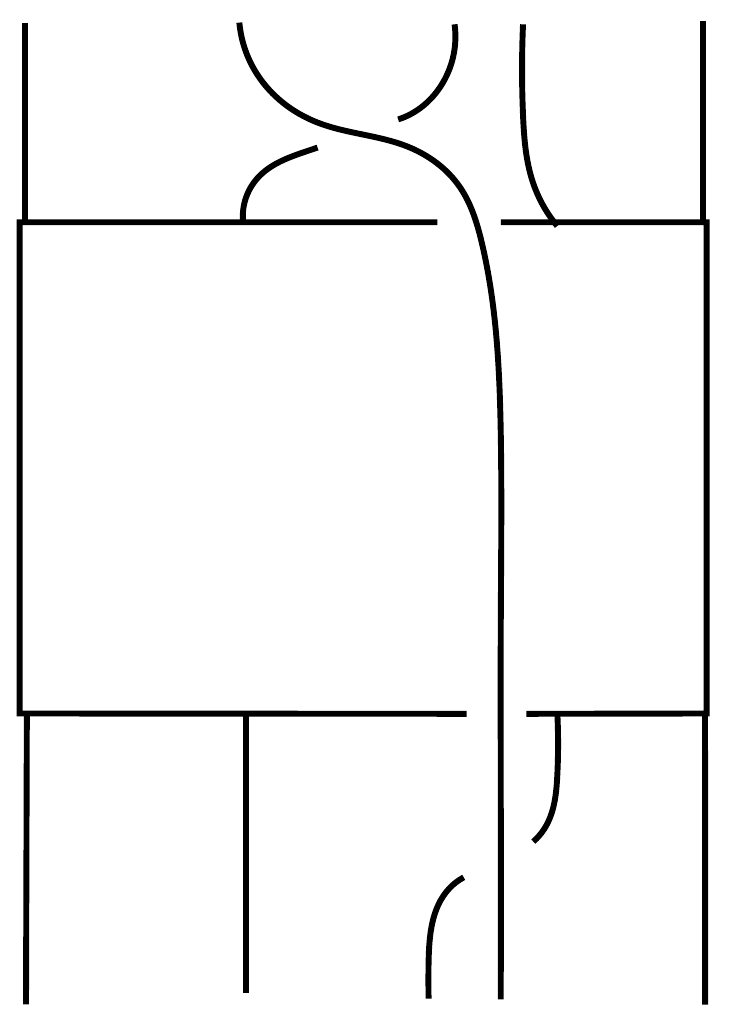}}
  \hspace{.4cm} \thicksim \hspace{.4cm}
  \raisebox{-35pt}{\includegraphics[height=1in]{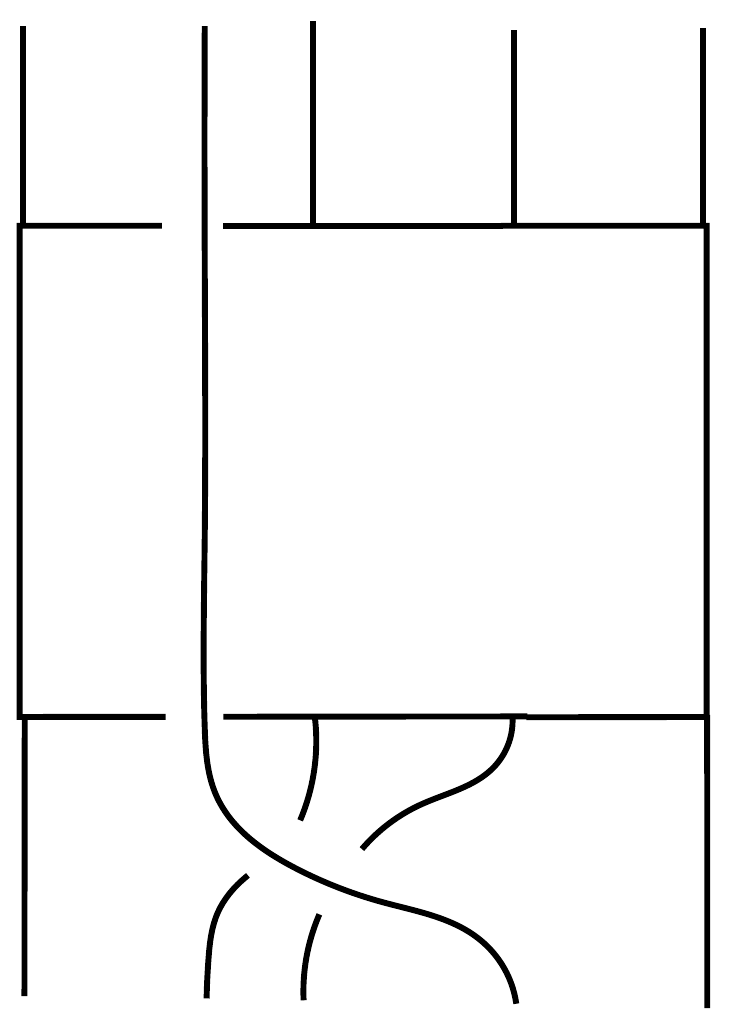}}
  \hspace{.4cm} \longleftrightarrow \hspace{.4cm}
  \raisebox{-35pt}{\includegraphics[height=1in]{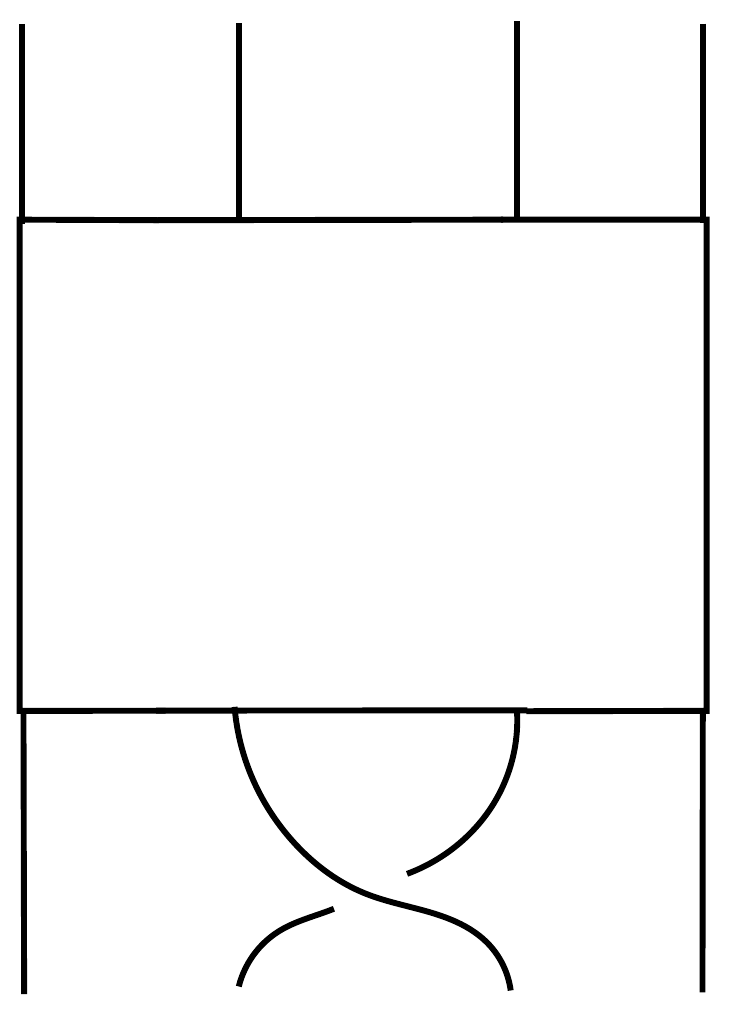}}
  \put(-210,7){\fontsize{9}{9}$o$}
  \put(-132,7){\fontsize{9}{9}$o$}
  \put(-278,10){\small{\fontsize{9}{9}$L_o$-move}}
  \put(-179,20){\small{braid}} \put(-182,10){\small{isotopy}}
  \put(-85,20){\small{Left}}
  \put(-97,10){\small{\fontsize{9}{9}$-L_o$-move}}
  \put(-331,26){\tiny{$\cdots$}}
  \put(-331,-30){\tiny{$\cdots$}}
  \put(-298,26){\tiny{$\cdots$}}
  \put(-298,-30){\tiny{$\cdots$}}
  \put(-233,26){\tiny{$\cdots$}}
  \put(-233,-30){\tiny{$\cdots$}}
  \put(-199,26){\tiny{$\cdots$}}
  \put(-199,-30){\tiny{$\cdots$}}
  \put(-147,26){\tiny{$\cdots$}}
  \put(-147,-30){\tiny{$\cdots$}}
  \put(-112,26){\tiny{$\cdots$}}
  \put(-112,-30){\tiny{$\cdots$}}
  \put(-47,26){\tiny{$\cdots$}}
  \put(-47,-30){\tiny{$\cdots$}}
  \put(-13,26){\tiny{$\cdots$}} \put(-13,-30){\tiny{$\cdots$}}
  \]
  \caption{Elementary conjugation in terms of
    $L$-moves} \label{fig:conj}
\end{figure}
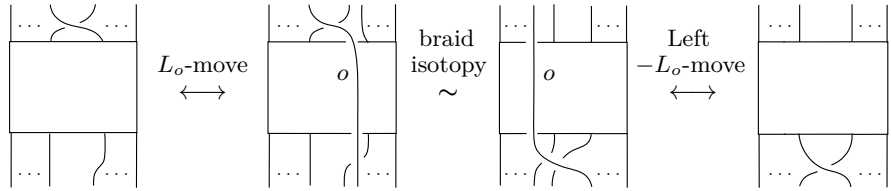

We aim to show that there is a 1-1 correspondence between the isotopy types of STG diagrams and the $TL$-equivalence classes of trivalent braids. Now, it can be easily seen that different choices when applying the braiding algorithm affect the output of the final braid. In addition, local isotopy changes in an STG diagram may induce a different final braid, upon our braiding algorithm. However, the following theorem asserts that every instance of isotopy between STG diagrams
can be translated in terms of $TL$-equivalence of trivalent braids.

\begin{theorem}[\textbf{$L$-move Markov-type theorem for STGs}] \label{LMarkov} 
Two well-oriented spatial trivalent graphs are isotopic if and only if any two of their corresponding trivalent braids are $TL$-equivalent.
\end{theorem}


\subsection{Proof of Theorem~\ref{LMarkov}} \label{sec:ProofLMarkov}
It is clear from the definition of $L$-moves and braid isotopy that $TL$-equivalent braids have isotopic closures. Therefore, we only need  to show that isotopic STG diagrams yield $TL$-equivalent braids upon our braiding algorithm. Throughout this proof, diagrams shall be assumed to be in general position. Also, isotopy moves will be considered strictly between diagrams in general position. In order to extend the proof for arbitrary STG diagrams (not necessarily in general position) it is essential to show that different choices when bringing a diagram into general position do not affect the final braid. Therefore, using Lemma~\ref{lemma:isotopyRegPosition} will complete the proof that any two isotopic STG diagrams yield $TL$-equivalent braids, upon our braiding algorithm.

The proof will be divided into two parts: The first part consists of analyzing the different choices made during the braiding process, and showing that each of these yield the same final braid, up to $TL$-equivalence; these choices amount to how the subdivision points are assigned, and the labels for free up-arcs. The second part will address isotopy between STG diagrams, and thus show it does not affect the final braid;  for this, we analyze the direction sensitive moves and the extended Reidemeister moves for STG diagram.

Our main approach for the proof is the following. For a given STG diagram we shall consider only the local region in which an isotopy move takes place. We assume that all other up-arcs outside such local region have been braided already. By the triangle condition, this choice does not affect the final braid, and, thus, we have the liberty to compare the braided portions corresponding to such local regions and conclude that the final braids are $TL$-equivalent.

For the first part of the proof, we shall assume that the diagram under consideration is equipped with a choice of subdivision points. To this end, in order to compare the effect of different choices of subdivision points for a given diagram, we need the following lemmas (see Lemmas 4.1 and 4.2 in~\cite{LR} for detailed
proofs).

\begin{lemma}\label{lemM1}
  If we add an extra subdivision point to an up-arc of an STG diagram, the corresponding braids differ by basic $L$-moves.
\end{lemma}

\begin{lemma}\label{lemM2}
  When we braid a free up-arc, which we have the choice of labeling ``u'' or ``o'', the resulting braid is independent of this choice, up to $TL$-equivalence.
\end{lemma}

\begin{remark}
  The following are consequences of the previous two lemmas.
  \begin{enumerate}
  \item If we have a chain of overlapping sliding triangles of free up-arcs so that we have a free choice of labeling for the whole chain then, up to $TL$-equivalence, this choice does not affect the final braid.
  \item If by adding a subdivision point on an up-arc we have a choice for relabeling the resulting new up-arcs so that the triangle condition is still satisfied, then the resulting braids are $TL$-equivalent.
  \end{enumerate}
\end{remark}

\begin{corollary}
  Given any two subdivisions, $S_1$ and $S_2$, of an STG diagram which will satisfy the triangle condition with appropriate labelings, the resulting braids are $TL$-equivalent.
\end{corollary}
\begin{proof}
  Recall that whenever a subdivision satisfies the triangle condition, any refinement satisfies it as well. Using Lemmas~\ref{lemM1} and~\ref{lemM2}, the statement follows by considering the subdivision $S_1 \cup S_2$.
\end{proof}

For the second part of the proof, we shall first examine the choices made when bringing an STG diagram into general position. These amount to different applications of the direction sensitive moves.

\begin{lemma}   \label{dir sen}
  Spatial trivalent graph diagrams in general position that differ by direction sensitive moves correspond to trivalent braids that differ by $L$-moves.
\end{lemma}

\begin{proof}
  When putting an STG diagram in general position, the first thing we do is shift its vertices into regular position using the conventions introduced in Figures~\ref{vgp} and \ref{vgp2}. In the next few paragraphs we check that different choices when applying these conventions do not affect, up to $TL$-equivalence, the final braid. The only cases where we have different choices are for vertices incident with at least two up-arcs.

  Consider either a $Y$- or $\lambda$-type vertex incident with only one up-arc (as shown in the first row of Figure~\ref{vgp} and Figure~\ref{vgp2}). In any case, there is only one option for shifting such vertex into regular position.

  For the case of a $Y$-type vertex incident with exactly two up-arcs, we apply an $R5$ move between the up-arcs (this is shown in the second row of Figure~\ref{vgp}).  We have a choice for the type of crossing (positive or negative) to add when the move is applied.  Note that the diagrams obtained from either choice differ by a switch move (recall Figure~\ref{twist}). Therefore, we need to verify that the switch move of this type does not affect the final braid, up to $TL$-equivalence. In Figure \ref{twist move} we consider the braided portions obtained from a switch move performed on the right hand side of a vertex. It is clear that the braids (1) and (2) in Figure~\ref{twist move} differ by planar isotopy.  The case where the two up-arcs are on the left hand side of a $Y$-type vertex is treated similarly (this can be seen by reflecting the diagrams in Figure~\ref{twist move} across a vertical axis).

  \begin{figure}[ht]
    \[
    \raisebox{-25pt}{\includegraphics[height =
      .6in]{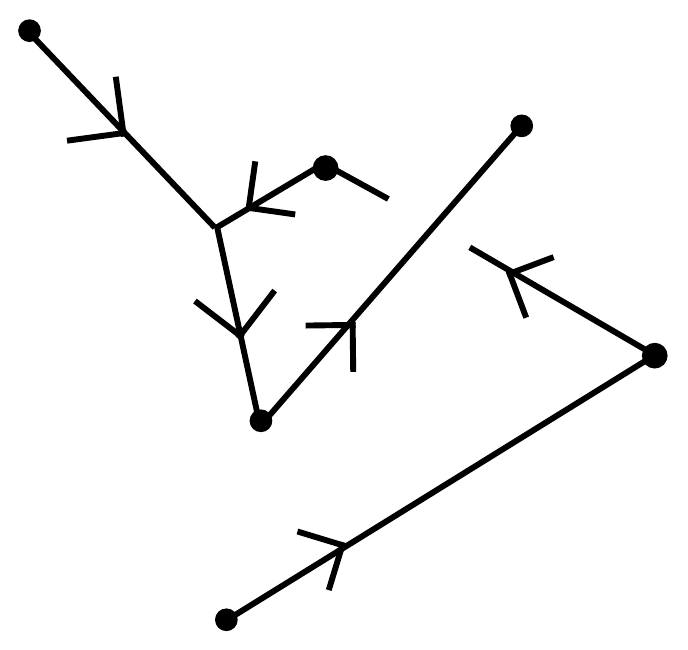}} \hspace{.2cm} \longrightarrow
    \hspace{.2cm} \raisebox{-40pt}{\includegraphics[height =
      1.2in]{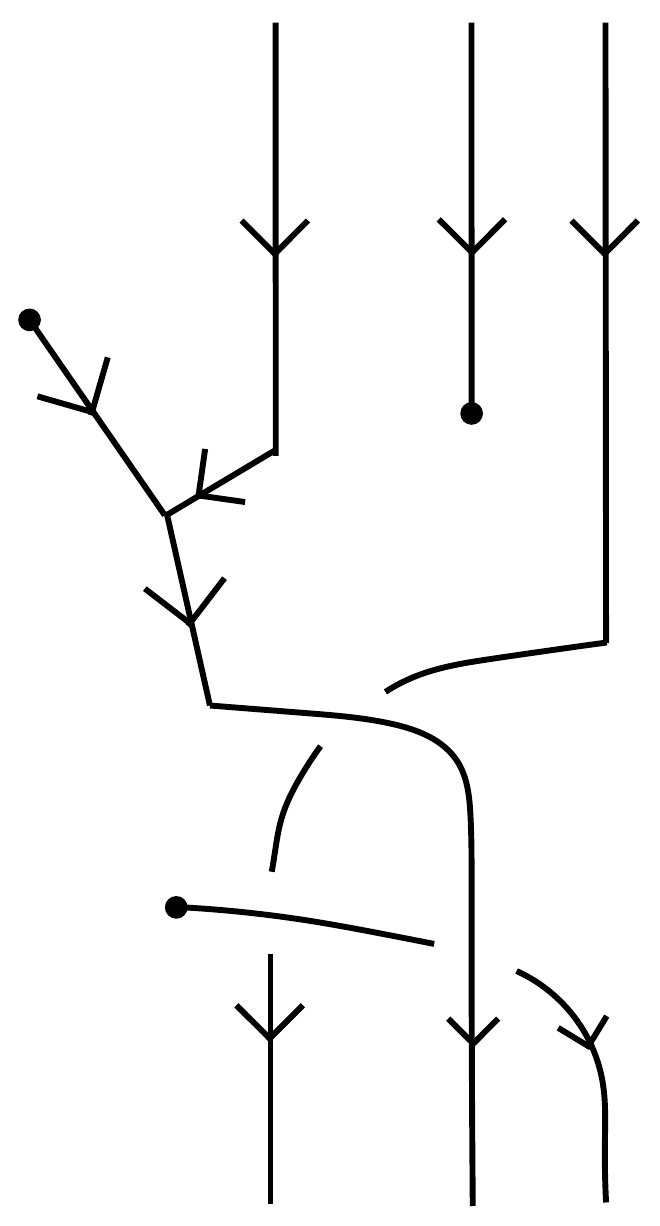}} \hspace{.2cm} \thicksim
    \hspace{.2cm} \raisebox{-40pt}{\includegraphics[height =
      1.2in]{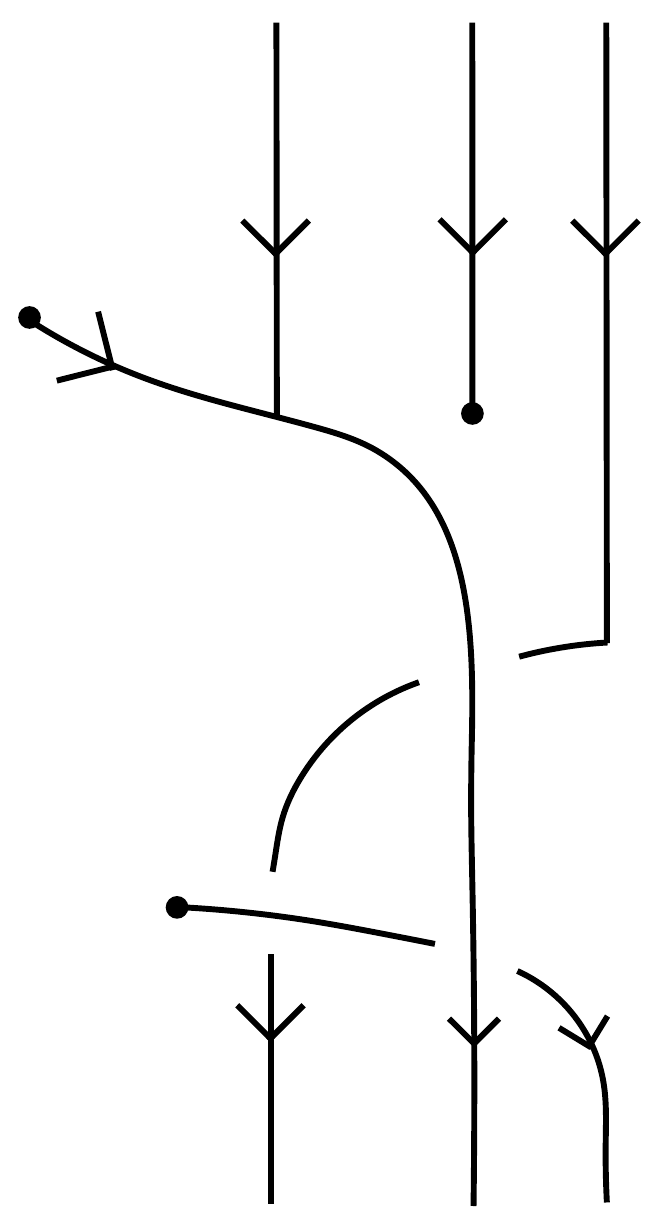}} \hspace{.2cm} \longleftrightarrow
    \hspace{.2cm} \raisebox{-40pt}{\includegraphics[height =
      1.2in]{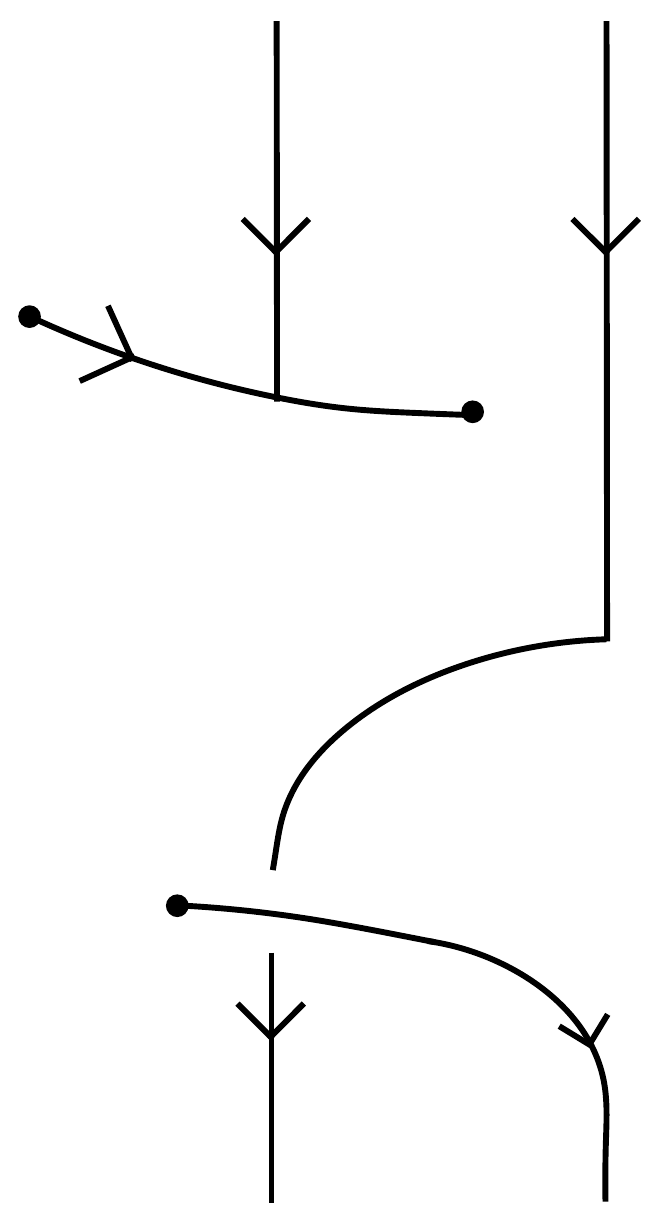}} \hspace{.2cm} \longleftrightarrow
    \hspace{.2cm} \raisebox{-40pt}{\includegraphics[height =
      1.2in]{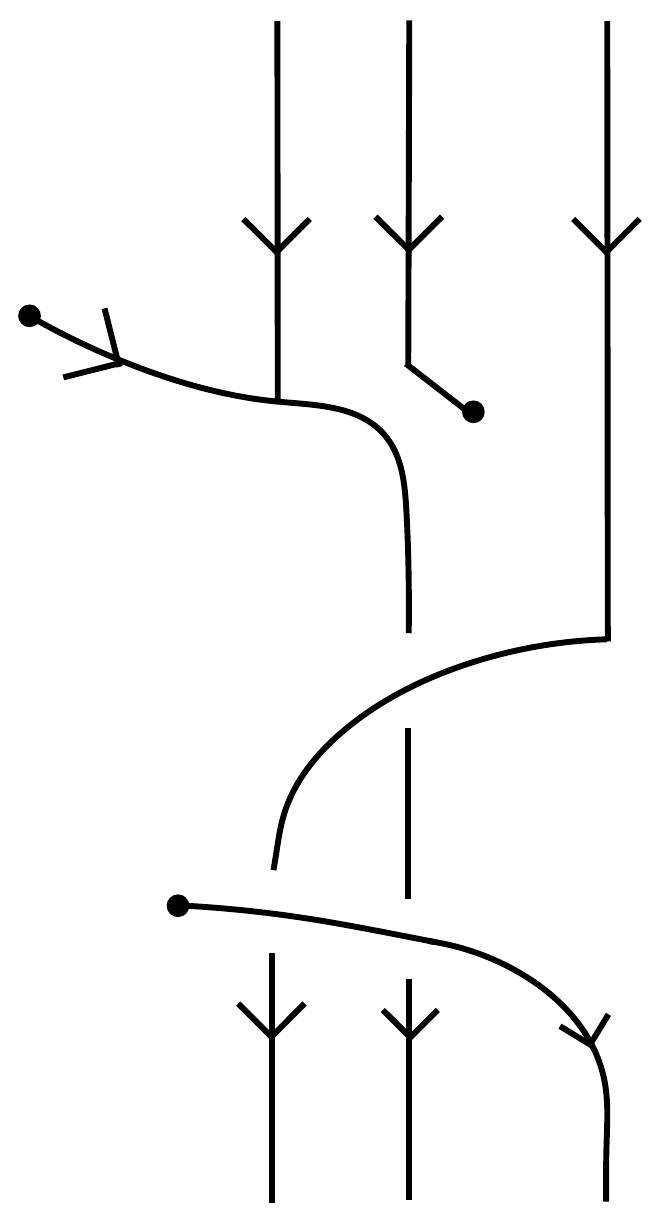}} \put(-330
    ,-25){\fontsize{9}{9}$o$} \put(-270
    ,38){\fontsize{9}{9}$u$} \put(-270
    ,-38){\fontsize{9}{9}$u$} \put(-255
    ,38){\fontsize{9}{9}$o$} \put(-255
    ,-38){\fontsize{9}{9}$o$} \put(-245
    ,38){\fontsize{9}{9}$o$} \put(-245
    ,-38){\fontsize{9}{9}$o$} \put(-199
    ,38){\fontsize{9}{9}$u$} \put(-199
    ,-38){\fontsize{9}{9}$u$} \put(-184
    ,38){\fontsize{9}{9}$o$} \put(-184
    ,-38){\fontsize{9}{9}$o$} \put(-174
    ,38){\fontsize{9}{9}$o$} \put(-174
    ,-38){\fontsize{9}{9}$o$} \put(-116
    ,38){\fontsize{9}{9}$u$} \put(-116
    ,-38){\fontsize{9}{9}$u$}
    \put(-92,38){\fontsize{9}{9}$o$} \put(-92
    ,-38){\fontsize{9}{9}$o$}
    \put(-35,38){\fontsize{9}{9}$u$} \put(-35
    ,-38){\fontsize{9}{9}$u$}
    \put(-25,38){\fontsize{9}{9}$u$} \put(-25
    ,-38){\fontsize{9}{9}$u$}
    \put(-11,38){\fontsize{9}{9}$o$} \put(-11
    ,-38){\fontsize{9}{9}$o$}
    \put(-23,-55){\fontsize{9}{9}$(1)$} \put(-318,
    10){\small{braiding}} \put(-233, 10){\small{braid}} \put(-235,
    -10){\small{isotopy}} \put(-162,
    10){\small{$L_o$-move}} \put(-80, 10){\small{$L_u$-move}}
    \]
    \begin{center}
      \rule{\textwidth}{.5pt}
    \end{center}
    \[
    \raisebox{-25pt}{\includegraphics[height =
      .7in]{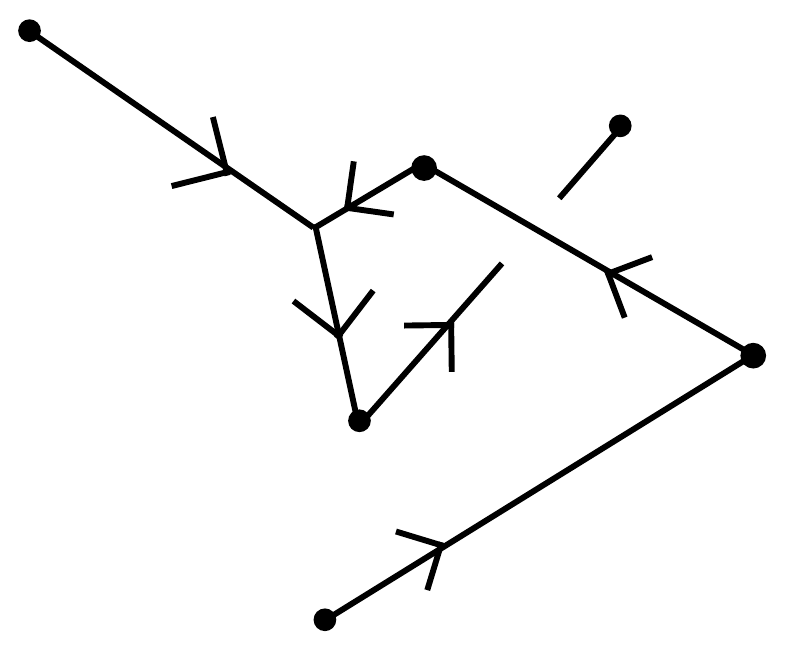}} \hspace{.2cm} \longrightarrow
    \hspace{.2cm} \raisebox{-40pt}{\includegraphics[height =
      1.2in]{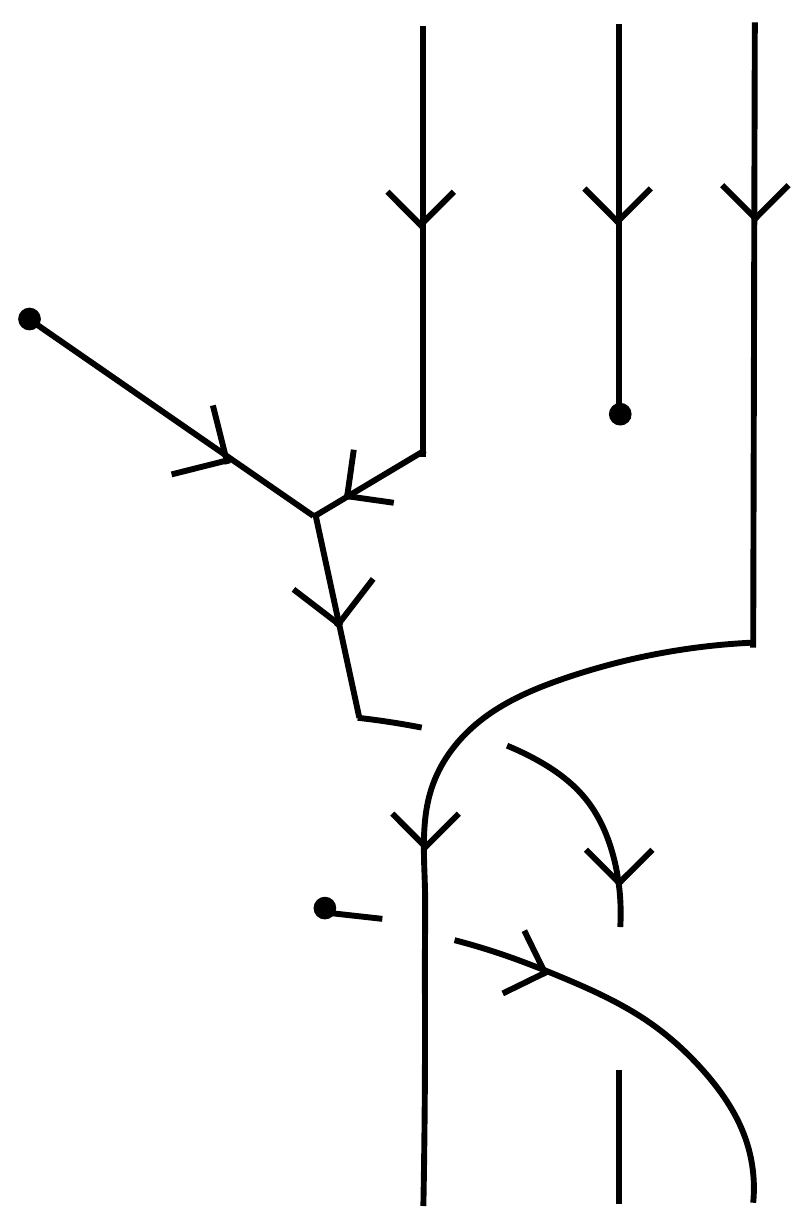}} \hspace{.2cm} \thicksim
    \hspace{.2cm} \raisebox{-40pt}{\includegraphics[height =
      1.2in]{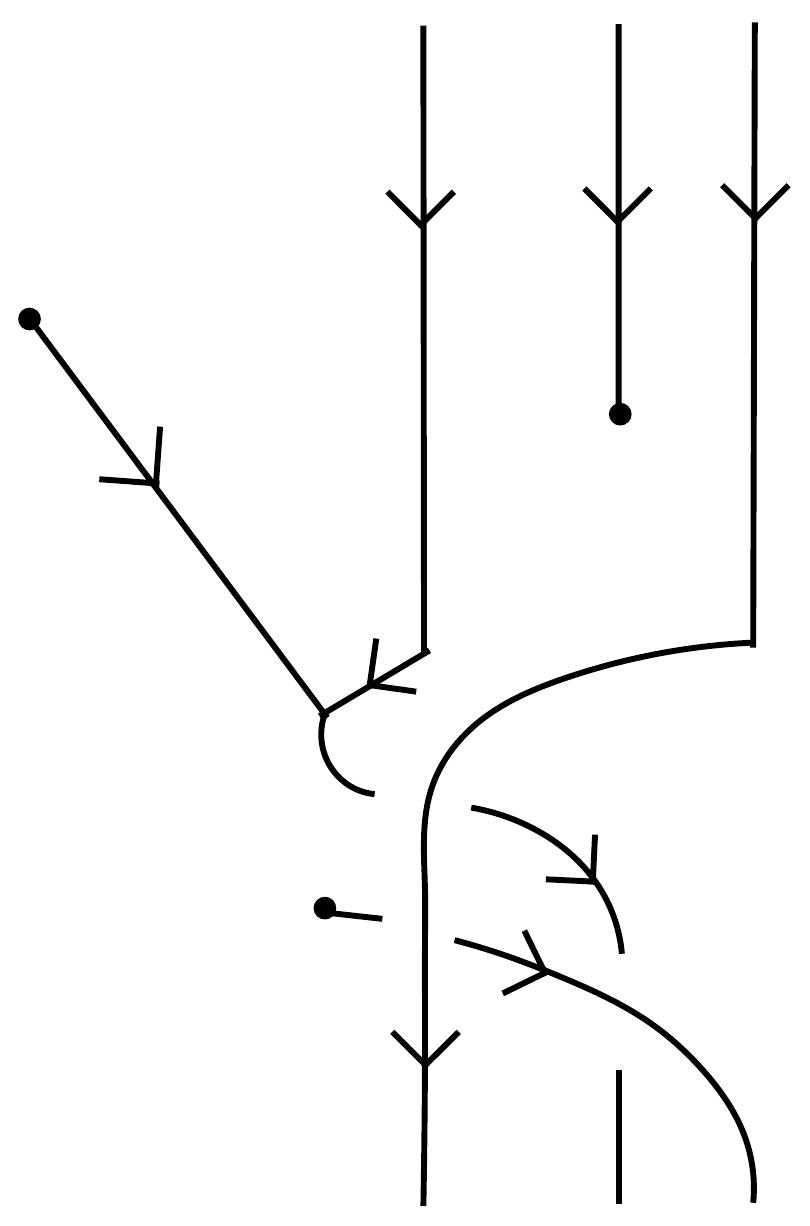}} \hspace{.2cm}
    \longleftrightarrow \hspace{.2cm} \put(-220,
    -20){\fontsize{9}{9}$o$} \put(-145,
    38){\fontsize{9}{9}$o$} \put(-145,
    -38){\fontsize{9}{9}$o$} \put(-133,
    38){\fontsize{9}{9}$u$} \put(-133,
    -38){\fontsize{9}{9}$u$} \put(-122,
    38){\fontsize{9}{9}$o$} \put(-122,
    -38){\fontsize{9}{9}$o$} \put(-64,
    38){\fontsize{9}{9}$o$} \put(-64,
    -38){\fontsize{9}{9}$o$} \put(-52,
    38){\fontsize{9}{9}$u$} \put(-52,
    -38){\fontsize{9}{9}$u$} \put(-41,
    38){\fontsize{9}{9}$o$} \put(-41,
    -38){\fontsize{9}{9}$o$} \put(-202, 10){\small{braiding}}
    \put(-109, 10){\small{braid}} \put(-111, -10){\small{isotopy}}
    \put(-28, 10){\small{$L_o$-move}}
    \]
    \[
    \raisebox{-40pt}{\includegraphics[height =
      1.2in]{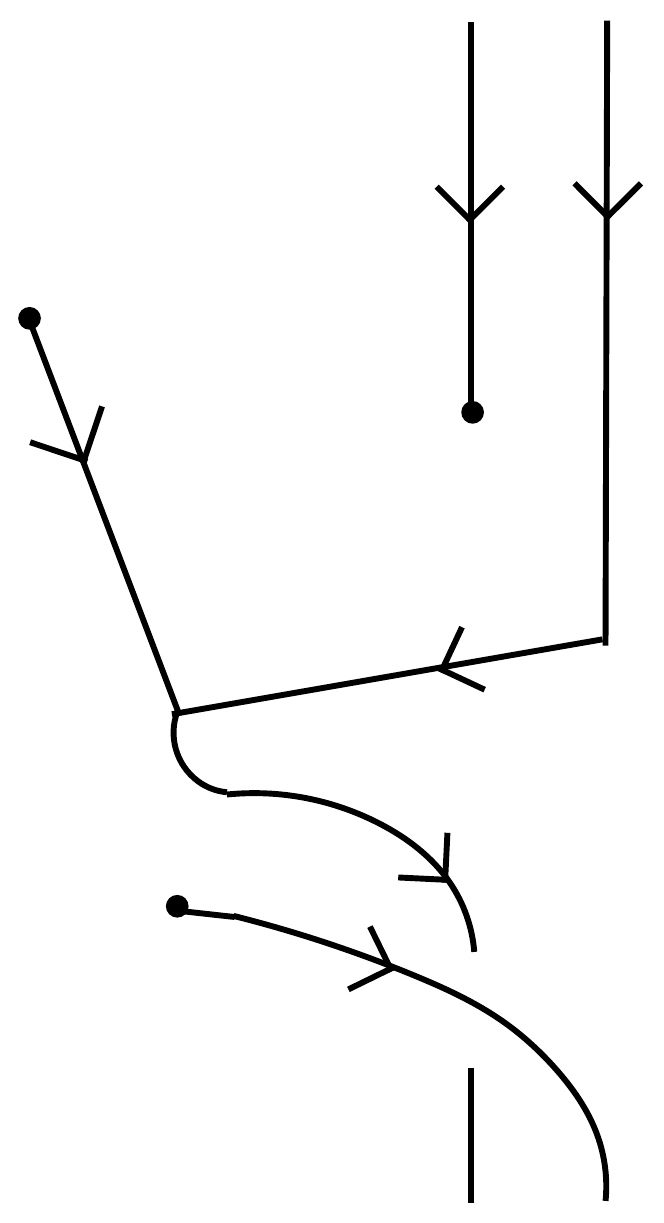}} \hspace{.2cm} \longleftrightarrow
    \hspace{.2cm} \raisebox{-40pt}{\includegraphics[height =
      1.2in]{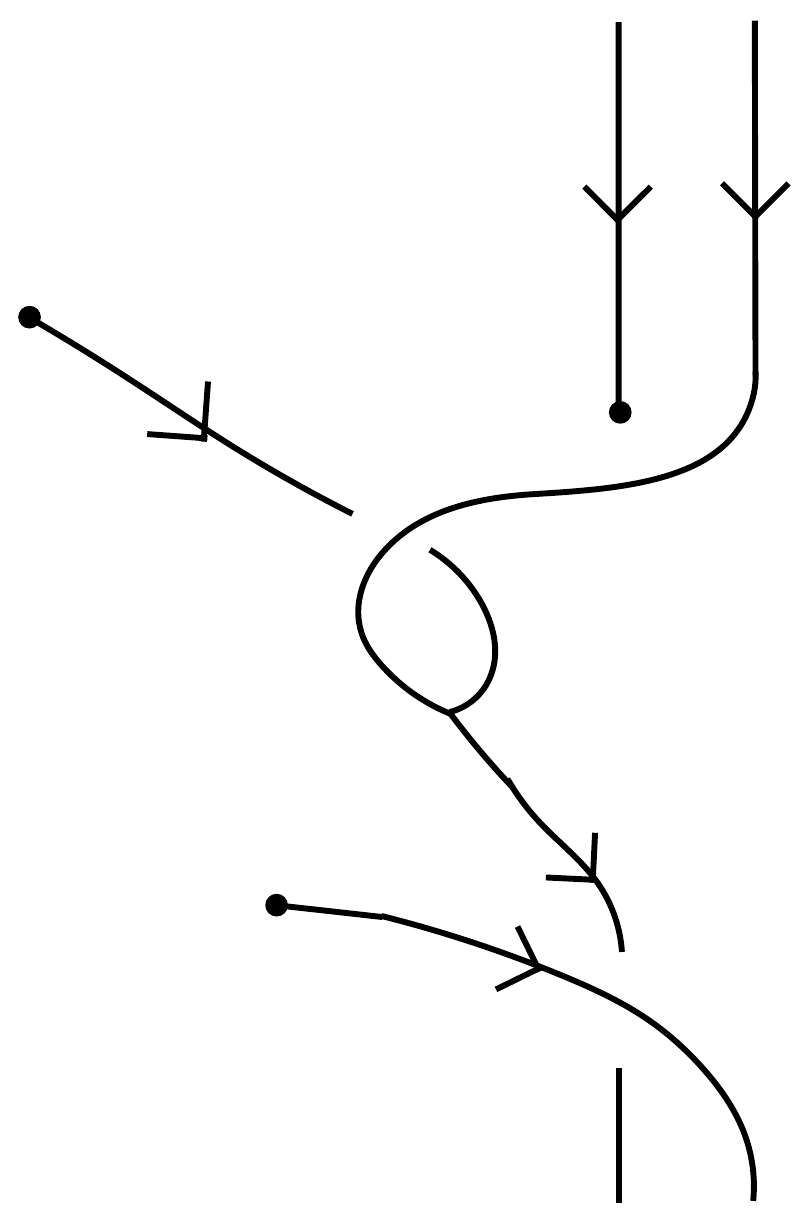}} \hspace{.2cm} \longleftrightarrow
    \hspace{.2cm} \raisebox{-40pt}{\includegraphics[height =
      1.2in]{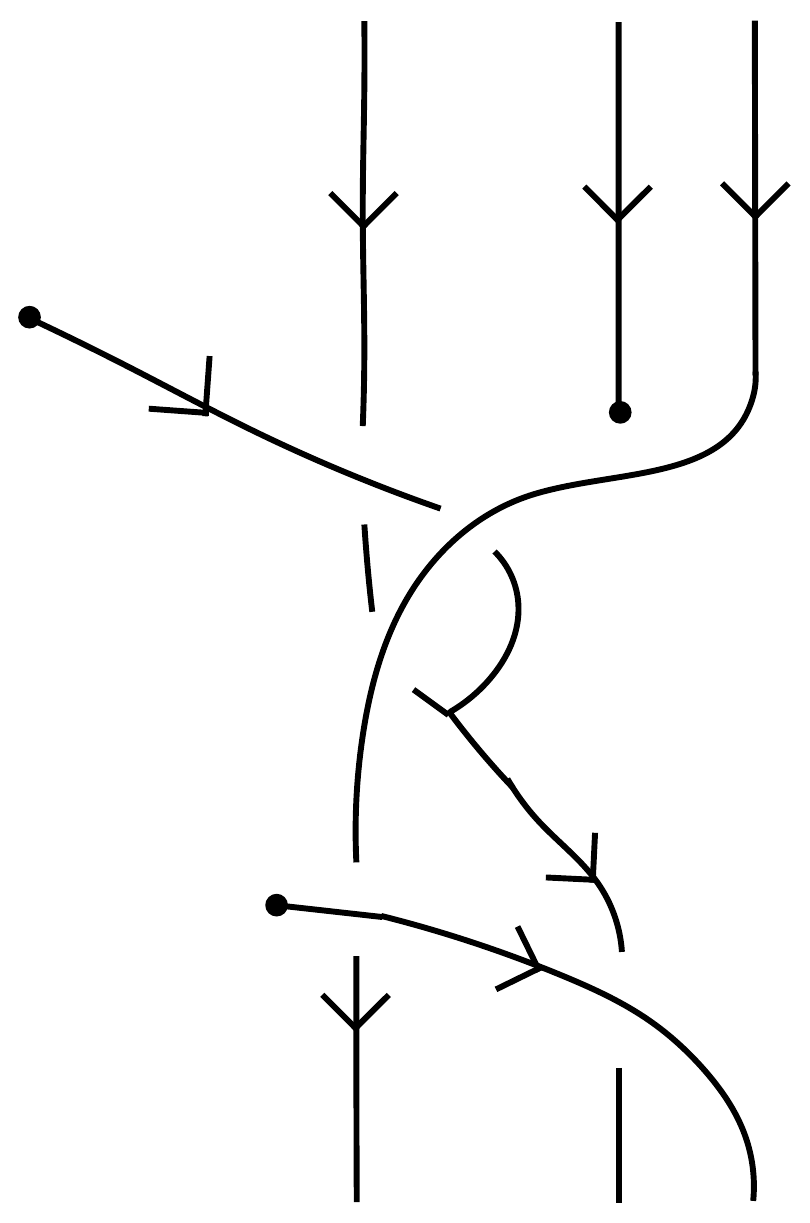}} \hspace{.2cm} \longleftrightarrow
    \hspace{.2cm} \raisebox{-40pt}{\includegraphics[height =
      1.2in]{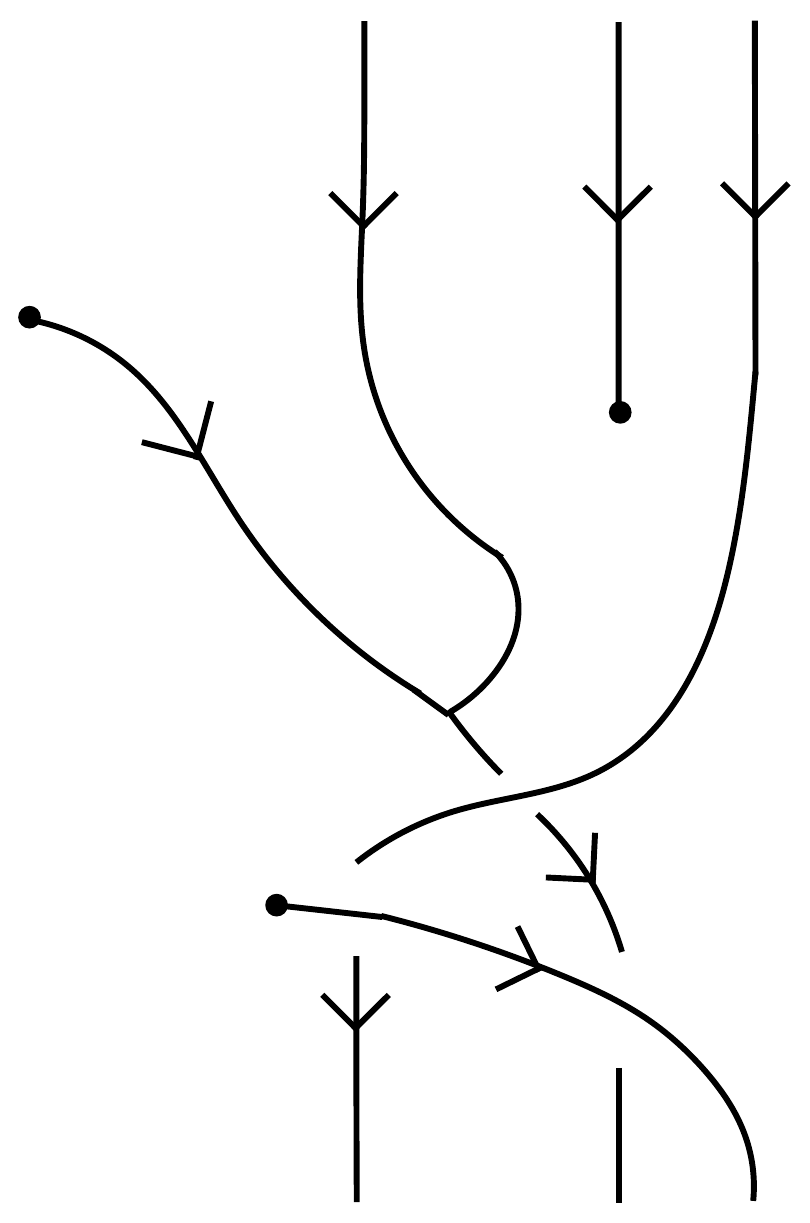}} \put(-299,
    38){\fontsize{9}{9}$u$} \put(-299,
    -38){\fontsize{9}{9}$u$} \put(-288,
    38){\fontsize{9}{9}$o$} \put(-289,
    -38){\fontsize{9}{9}$o$} \put(-206,
    38){\fontsize{9}{9}$u$} \put(-206,
    -38){\fontsize{9}{9}$u$} \put(-195,
    38){\fontsize{9}{9}$o$} \put(-196,
    -38){\fontsize{9}{9}$o$} \put(-134,
    38){\fontsize{9}{9}$u$} \put(-134,
    -38){\fontsize{9}{9}$u$} \put(-114,
    38){\fontsize{9}{9}$u$} \put(-114,
    -38){\fontsize{9}{9}$u$} \put(-103,
    38){\fontsize{9}{9}$o$} \put(-103,
    -38){\fontsize{9}{9}$o$} \put(-42,
    38){\fontsize{9}{9}$u$} \put(-42,
    -38){\fontsize{9}{9}$u$} \put(-22,
    38){\fontsize{9}{9}$u$} \put(-22,
    -38){\fontsize{9}{9}$u$} \put(-11,
    38){\fontsize{9}{9}$o$} \put(-11,
    -38){\fontsize{9}{9}$o$}
    \put(-23,-55){\fontsize{9}{9}$(2)$} \put(-274,
    10){\small{br.
        $R5$}} \put(-175, 20){\small{Left}} \put(-188,
    10){\small{$+L_u$-move}} \put(-88, 20){\small{br.
        $R4$}} \put(-88, 10){\small{br. $R5$}}
    \]
    \caption{Checking a switch move on a $Y$-type vertex with two
      up-arcs} \label{twist move}
  \end{figure}

  For the case of a $\lambda$-type vertex incident with two up-arcs we perform an $R5$ move between the up-arcs (as displayed in the second row of Figure~\ref{vgp2}). Similar to the case of a $Y$-type vertex, we have a choice for the type of crossing introduced by the $R5$ move.  Once again, the braids resulting from either choice differ by a switch move. In Figure~\ref{twist move2} we show that this version of the switch move applied on the right hand side of a
  $\lambda$-type vertex does not affect the final braid up to $TL$-equivalence; specifically, the braids (1) and (2) in Figure~\ref{twist move2} differ by planar
  isotopy.  The case where the two up-arcs lie on the left hand side of a $\lambda$-type vertex follows similarly, and thus we omit it to avoid repetition.

  \begin{figure}[ht]
    \[
    \raisebox{-25pt}{\includegraphics[height =
      .6in]{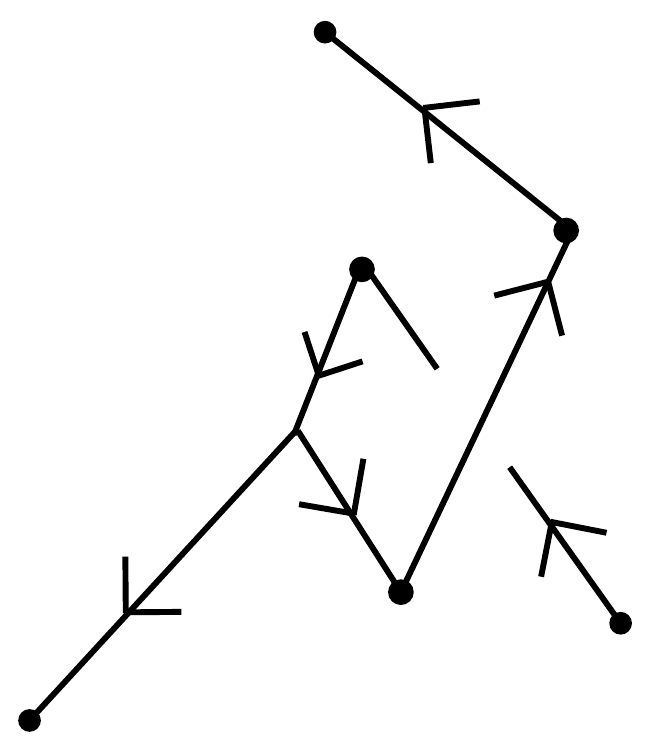}} \hspace{.2cm} \longrightarrow
    \hspace{.2cm} \raisebox{-40pt}{\includegraphics[height =
      1.2in]{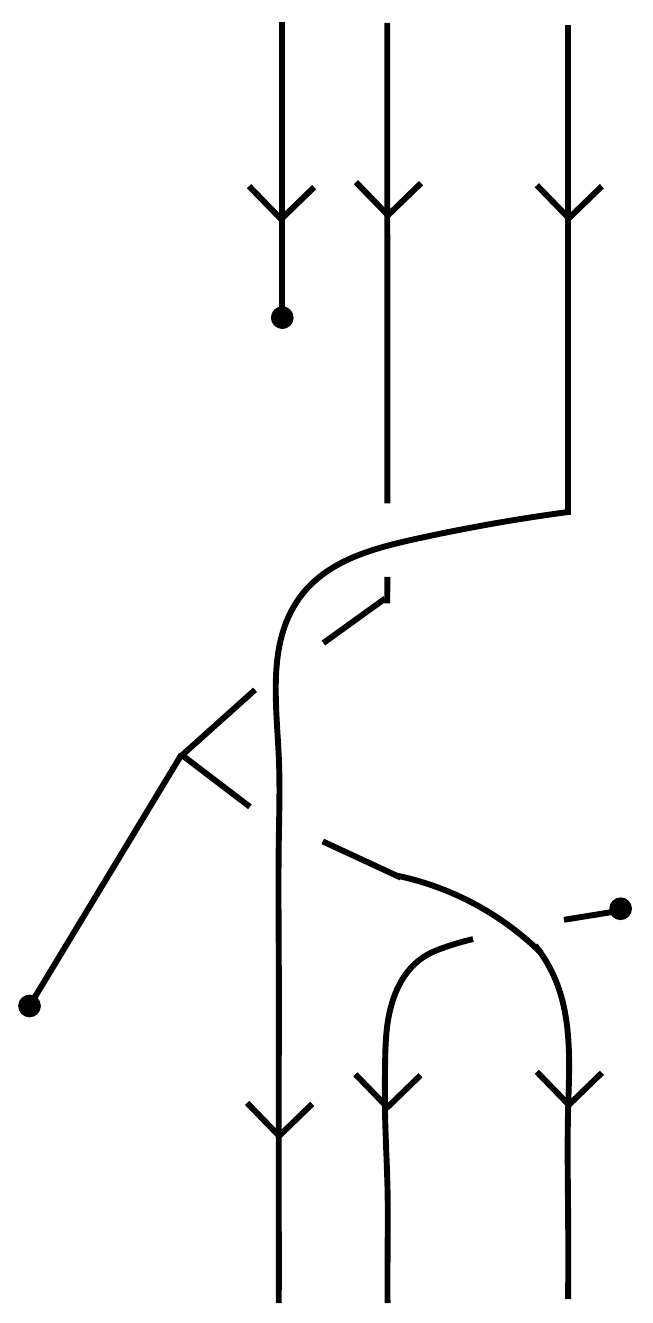}} \hspace{.2cm} \longleftrightarrow
    \hspace{.2cm} \raisebox{-40pt}{\includegraphics[height =
      1.2in]{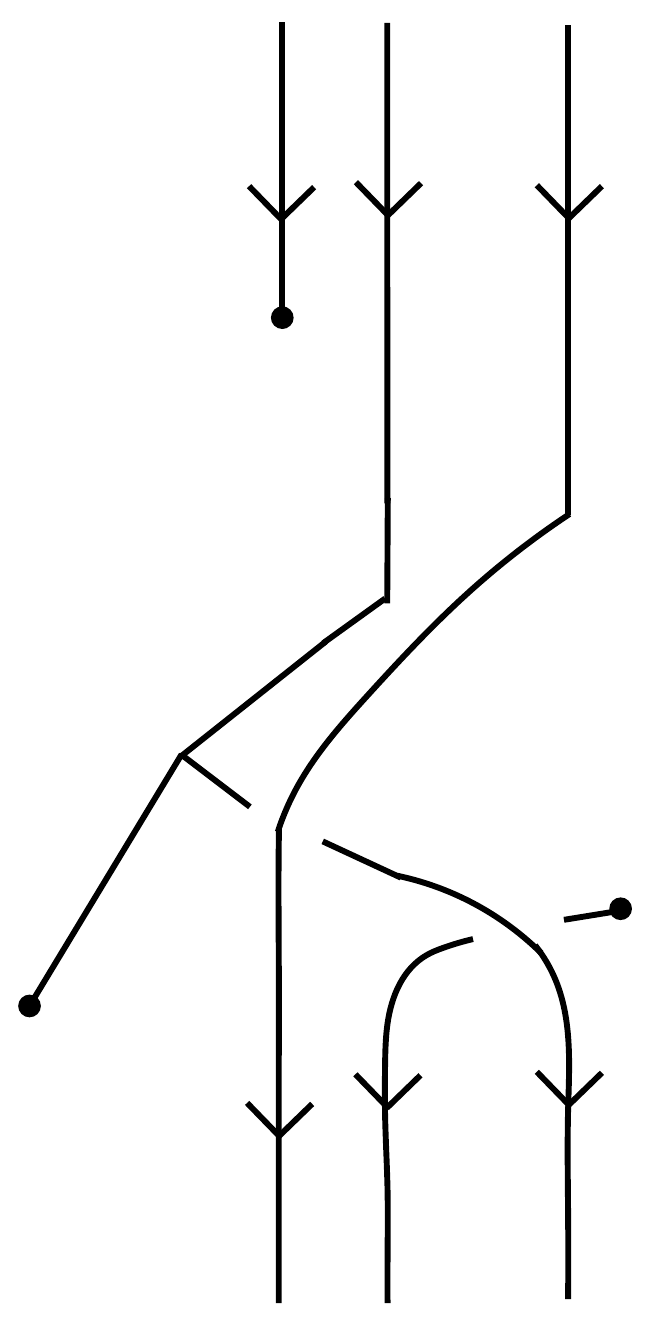}} \hspace{.2cm} \longleftrightarrow
    \hspace{.2cm} \put(-200,
    20){\fontsize{9}{9}$o$} \put(-140,
    38){\fontsize{9}{9}$o$} \put(-140,
    -38){\fontsize{9}{9}$o$} \put(-132,
    38){\fontsize{9}{9}$u$} \put(-132,
    -38){\fontsize{9}{9}$u$} \put(-120,
    38){\fontsize{9}{9}$o$} \put(-120,
    -38){\fontsize{9}{9}$o$} \put(-61,
    38){\fontsize{9}{9}$o$} \put(-61,
    -38){\fontsize{9}{9}$o$} \put(-53,
    38){\fontsize{9}{9}$u$} \put(-53,
    -38){\fontsize{9}{9}$u$} \put(-41,
    38){\fontsize{9}{9}$o$} \put(-41,
    -38){\fontsize{9}{9}$o$} \put(-185, 10){\small{braiding}}
    \put(-100, 10){\small{br.
        $R2$}} \put(-17, 20){\small{right}} \put(-28,
    10){\small{$+L_o$-move}}
    \]
    \[
    \raisebox{-40pt}{\includegraphics[height =
      1.2in]{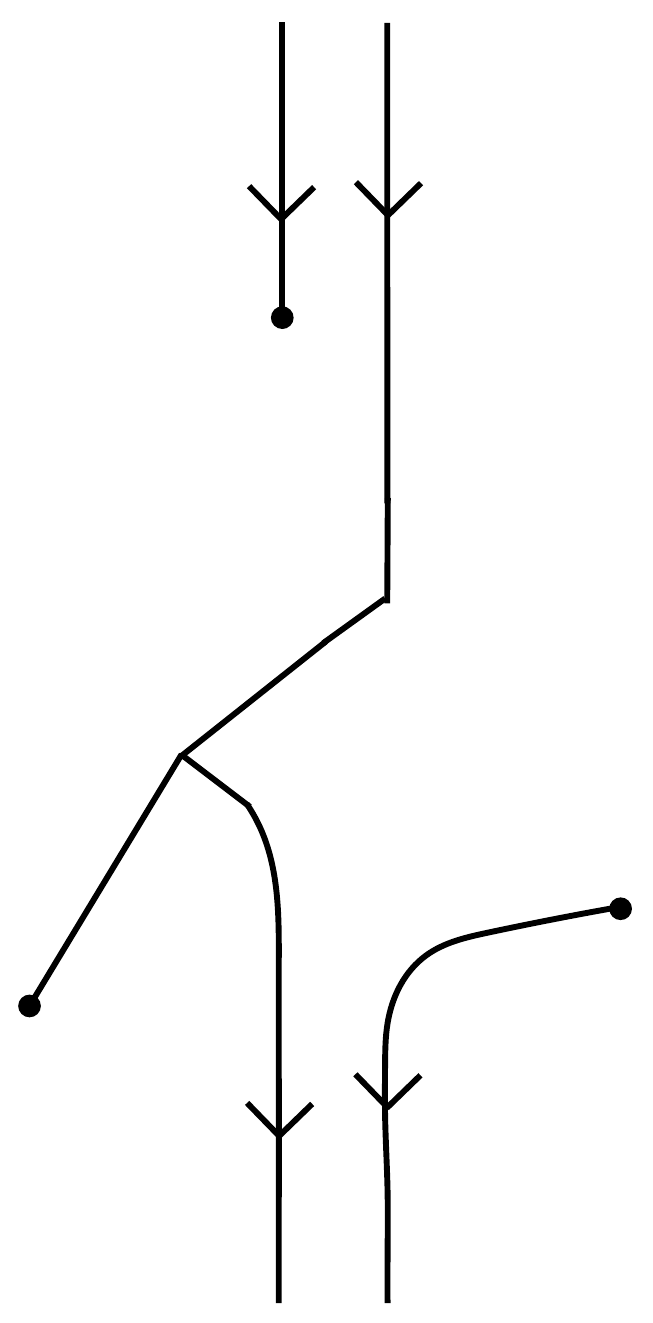}} \hspace{.2cm} \longleftrightarrow
    \hspace{.2cm} \raisebox{-40pt}{\includegraphics[height =
      1.2in]{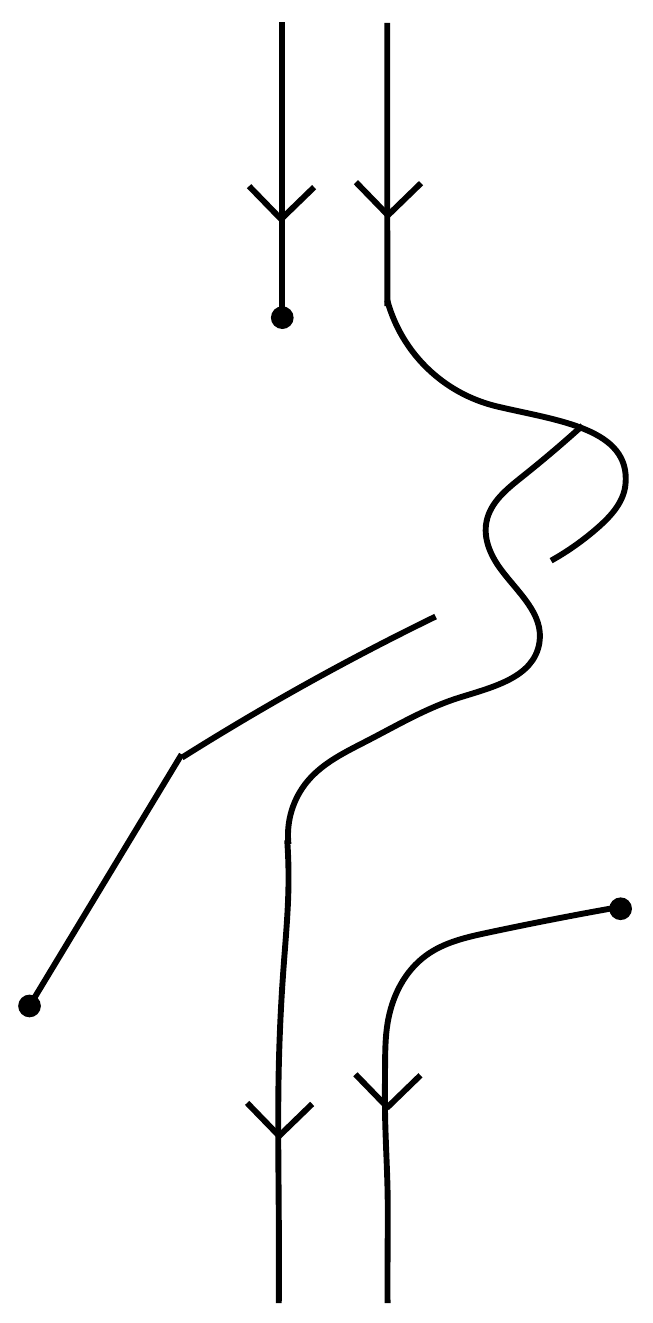}} \hspace{.2cm} \longleftrightarrow
    \hspace{.2cm} \raisebox{-40pt}{\includegraphics[height =
      1.2in]{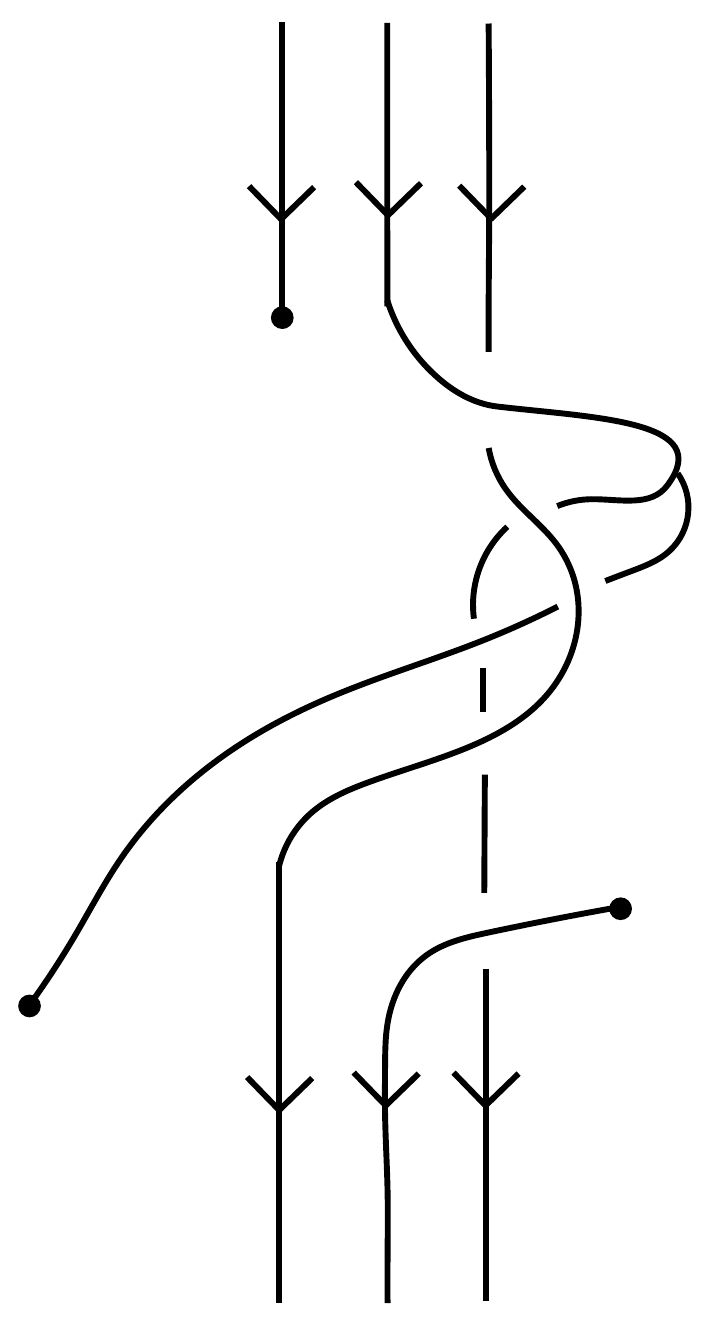}} \hspace{.2cm} \longleftrightarrow
    \hspace{.2cm} \raisebox{-40pt}{\includegraphics[height =
      1.2in]{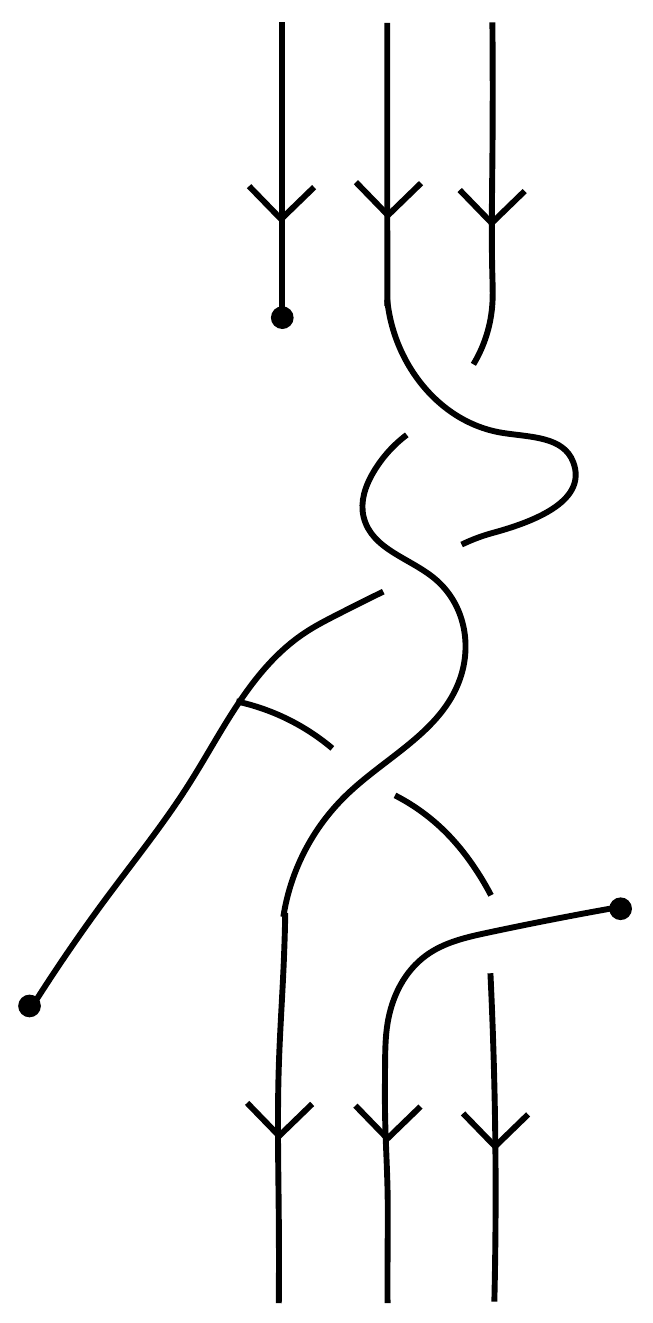}} \put(-269,
    38){\fontsize{9}{9}$o$} \put(-269,
    -38){\fontsize{9}{9}$o$} \put(-261,
    38){\fontsize{9}{9}$u$} \put(-261,
    -38){\fontsize{9}{9}$u$} \put(-190,
    38){\fontsize{9}{9}$o$} \put(-190,
    -38){\fontsize{9}{9}$o$} \put(-183,
    38){\fontsize{9}{9}$u$} \put(-183,
    -38){\fontsize{9}{9}$u$} \put(-112,
    38){\fontsize{9}{9}$o$} \put(-112,
    -38){\fontsize{9}{9}$o$} \put(-105,
    38){\fontsize{9}{9}$u$} \put(-105,
    -38){\fontsize{9}{9}$u$} \put(-98,
    38){\fontsize{9}{9}$u$} \put(-98,
    -38){\fontsize{9}{9}$u$} \put(-30,
    38){\fontsize{9}{9}$o$} \put(-30,
    -38){\fontsize{9}{9}$o$} \put(-23,
    38){\fontsize{9}{9}$u$} \put(-23,
    -38){\fontsize{9}{9}$u$} \put(-16,
    38){\fontsize{9}{9}$u$} \put(-16,
    -38){\fontsize{9}{9}$u$} \put(-232, 20){\small{planar}} \put(-232,
    10){\small{isotopy}} \put(-231, -10){\small{br.
        $R5$}} \put(-145, 20){\small{left}} \put(-158,
    10){\small{$-L_u$-move}} \put(-70, 20){\fontsize{9}{9}br.
      $R4$} \put(-70, 10){\fontsize{9}{9}br. $R5$}
    \]
    \[\hspace{.2cm} \longleftrightarrow \hspace{.2cm}
    \raisebox{-40pt}{\includegraphics[height =
      1.2in]{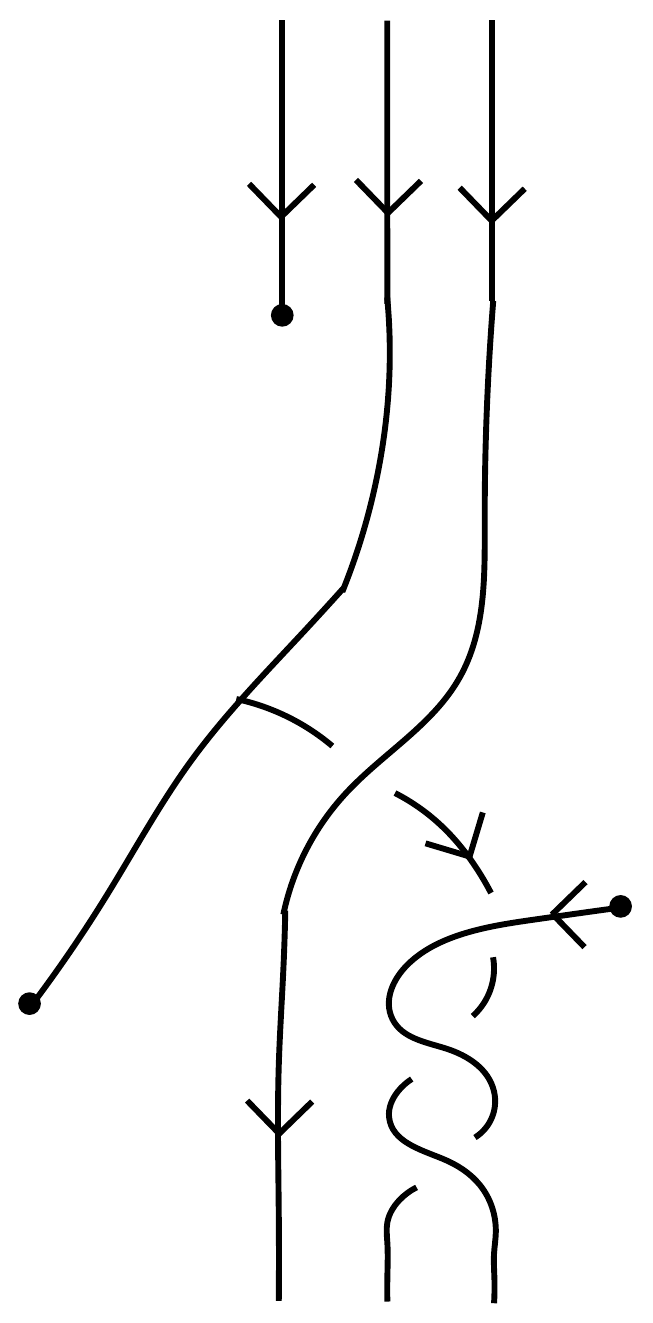}} \hspace{.2cm} \longleftrightarrow
    \hspace{.2cm} \raisebox{-40pt}{\includegraphics[height =
      1.2in]{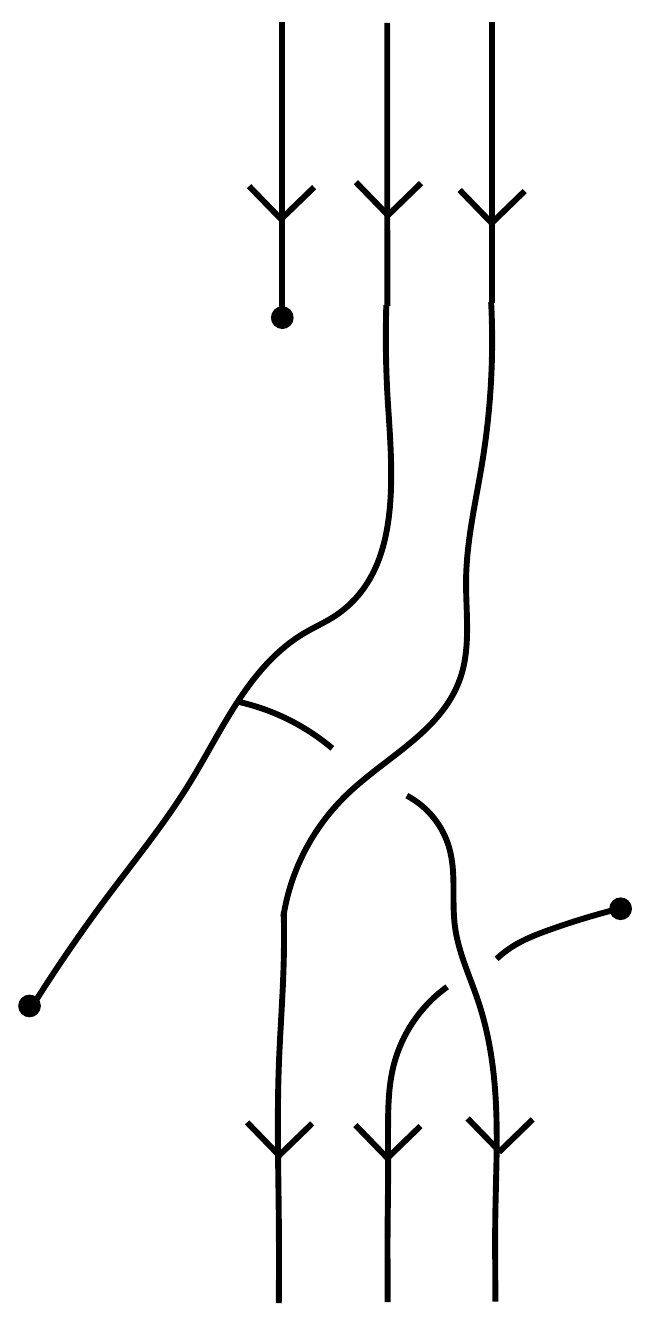}} \put(-109,
    38){\fontsize{9}{9}$o$} \put(-109,
    -38){\fontsize{9}{9}$o$} \put(-102,
    38){\fontsize{9}{9}$u$} \put(-102,
    -38){\fontsize{9}{9}$u$} \put(-95,
    38){\fontsize{9}{9}$u$} \put(-95,
    -38){\fontsize{9}{9}$u$} \put(-30,
    38){\fontsize{9}{9}$o$} \put(-30,
    -38){\fontsize{9}{9}$o$} \put(-23,
    38){\fontsize{9}{9}$u$} \put(-23,
    -38){\fontsize{9}{9}$u$} \put(-16,
    38){\fontsize{9}{9}$u$} \put(-16,
    -38){\fontsize{9}{9}$u$} \put(-160, 10){\small{conjugation}}
    \put(-67,10){\small{$R2$}} \put(-24,-55){\small{$(1)$}}
    \]
    \begin{center}
      \rule{\textwidth}{.5pt}
    \end{center}
    \[
    \raisebox{-25pt}{\includegraphics[height =
      0.6in]{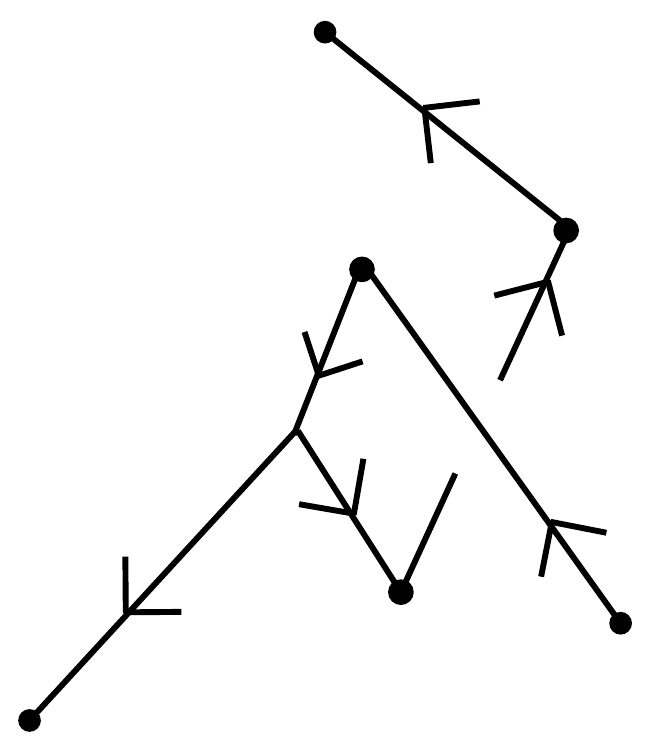}} \hspace{.2cm} \longrightarrow
    \hspace{.2cm} \raisebox{-40pt}{\includegraphics[height =
      1.2in]{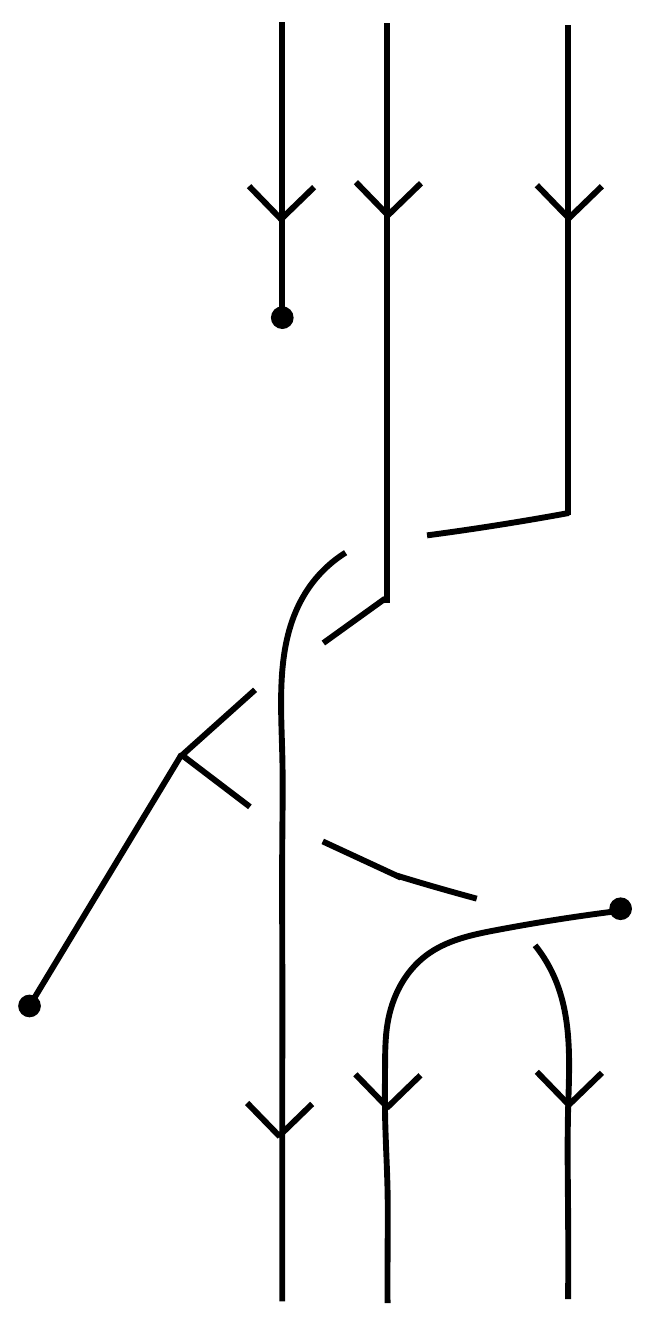}} \hspace{.2cm} \thicksim
    \hspace{.2cm} \raisebox{-40pt}{\includegraphics[height =
      1.2in]{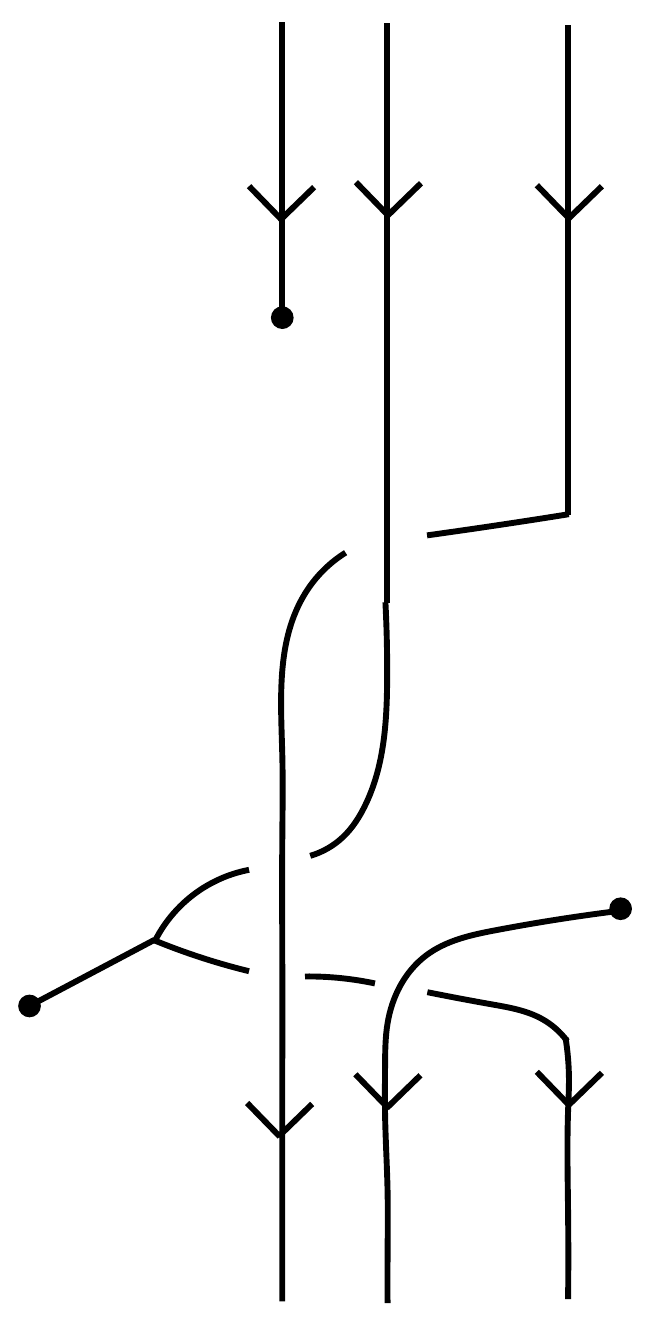}} \hspace{.2cm} \longleftrightarrow
    \hspace{.2cm} \put(-193,
    20){\fontsize{9}{9}$o$} \put(-127,
    38){\fontsize{9}{9}$o$} \put(-127,
    -38){\fontsize{9}{9}$o$} \put(-120,
    38){\fontsize{9}{9}$o$} \put(-120,
    -38){\fontsize{9}{9}$o$} \put(-109,
    38){\fontsize{9}{9}$u$} \put(-109,
    -38){\fontsize{9}{9}$u$} \put(-61,
    38){\fontsize{9}{9}$o$} \put(-61,
    -38){\fontsize{9}{9}$o$} \put(-53,
    38){\fontsize{9}{9}$o$} \put(-53,
    -38){\fontsize{9}{9}$o$} \put(-43,
    38){\fontsize{9}{9}$u$} \put(-43,
    -38){\fontsize{9}{9}$u$} \put(-174, 10){\small{braiding}}
    \put(-92, 20){\small{braid}} \put(-94, 10){\small{isotopy}}
    \put(-18, 18){\small{basic}} \put(-25, 10){\small{$L_o$-move}}
    \]
    \[
    \raisebox{-40pt}{\includegraphics[height =
      1.2in]{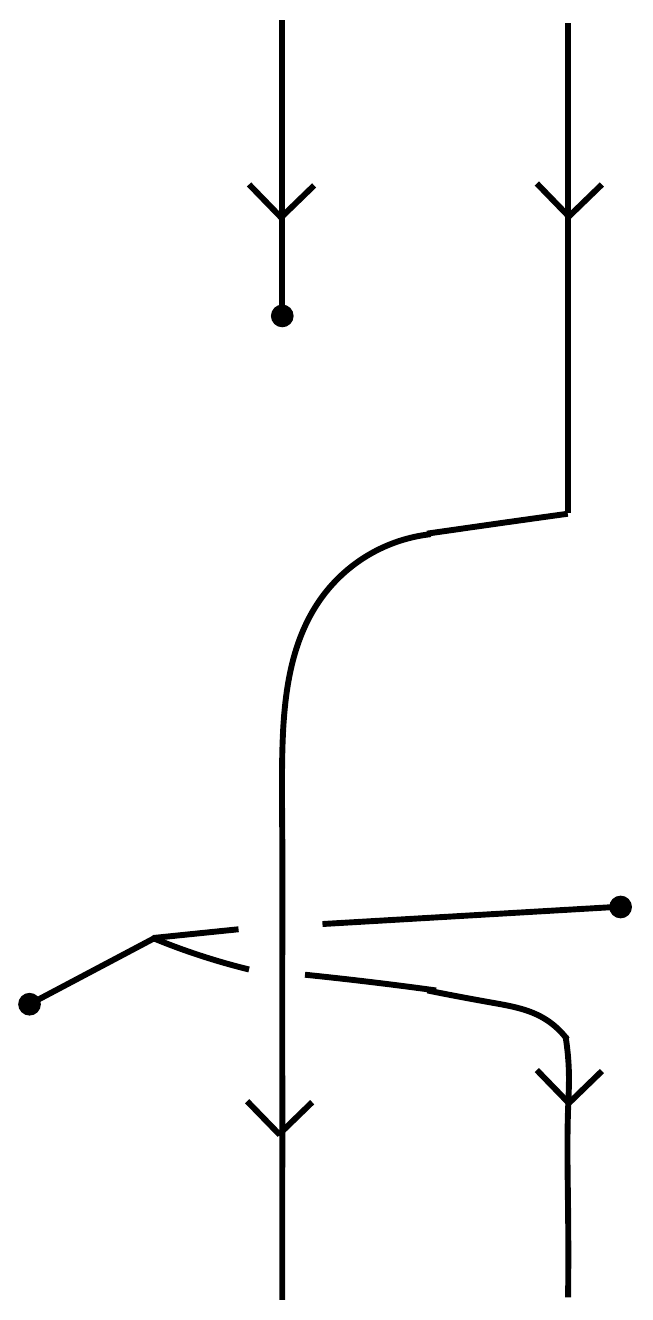}} \hspace{.2cm} \longleftrightarrow
    \hspace{.2cm} \raisebox{-40pt}{\includegraphics[height =
      1.2in]{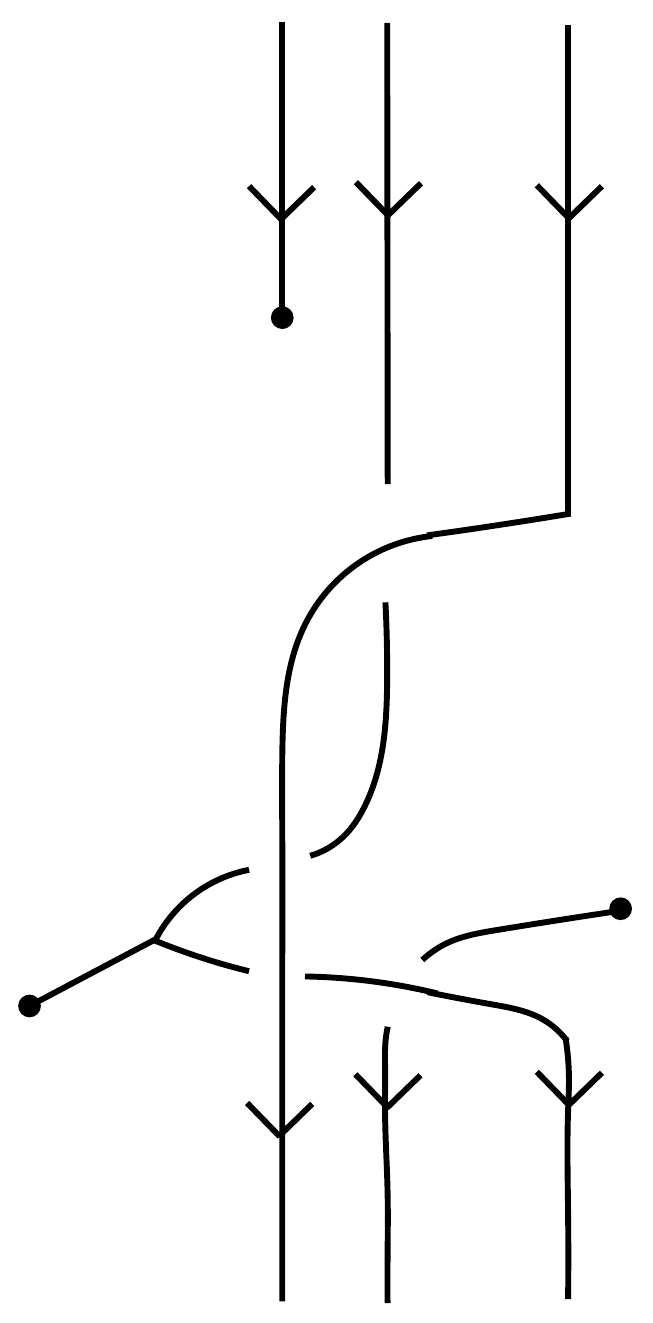}} \hspace{.2cm} \longleftrightarrow
    \hspace{.2cm} \raisebox{-40pt}{\includegraphics[height =
      1.2in]{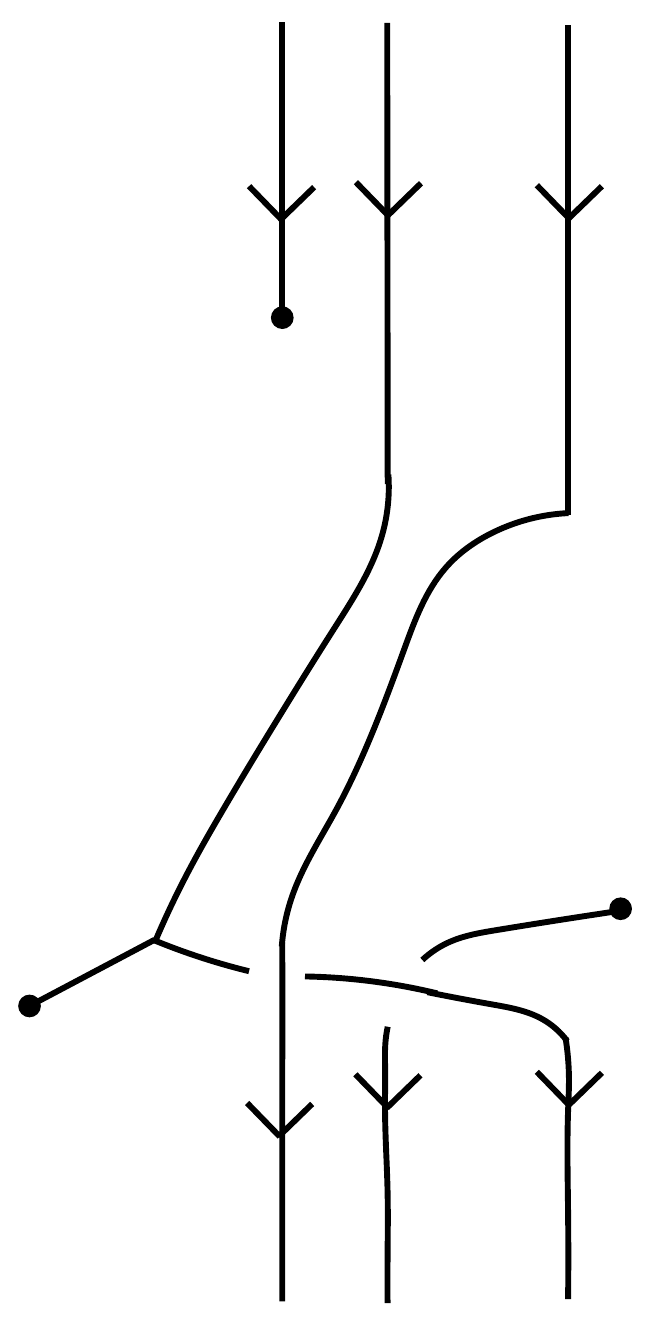}} \put(-186,
    38){\fontsize{9}{9}$o$} \put(-186,
    -38){\fontsize{9}{9}$o$} \put(-168,
    38){\fontsize{9}{9}$u$} \put(-168,
    -38){\fontsize{9}{9}$u$} \put(-108,
    38){\fontsize{9}{9}$o$} \put(-108,
    -38){\fontsize{9}{9}$o$} \put(-101,
    38){\fontsize{9}{9}$u$} \put(-101,
    -38){\fontsize{9}{9}$u$} \put(-90,
    38){\fontsize{9}{9}$u$} \put(-90,
    -38){\fontsize{9}{9}$u$} \put(-30,
    38){\fontsize{9}{9}$o$} \put(-30,
    -38){\fontsize{9}{9}$o$} \put(-23,
    38){\fontsize{9}{9}$u$} \put(-23,
    -38){\fontsize{9}{9}$u$} \put(-12,
    38){\fontsize{9}{9}$u$} \put(-12,
    -38){\fontsize{9}{9}$u$}
    \put(-23,-55){\fontsize{9}{9}$(2)$} \put(-145, 18){\small{basic}}
    \put(-152,
    10){\small{$L_u$-move}} \put(-72, 10){\fontsize{9}{9}br. $R2$}
    \]
    \caption{Checking a switch move on a $\lambda$-type vertex with two
      up-arcs} \label{twist move2}
  \end{figure}

  Now consider a $Y$-type vertex $v$ incident with three up-arcs. In order to shift $v$ into regular position we first perform planar isotopy on one arc
  and an $R5$ move between the remaining two up-arcs. This means we have four choices for shifting $v$ into regular position (see the last row of Figure~\ref{vgp}). We want to compare the braided portions obtained from each choice of shifting $v$ into regular position and show that they are $TL$-equivalent. In order to do this, we shall first compare choices that result in diagrams which differ by a switch move on the right hand side of $v$ (see
  Figure \ref{UUU Twist}). We leave it as an exercise for the reader to show the braids in Figure \ref{UUU Twist} are $TL$-equivalent. (One can use a similar approach to that used in Figure \ref{twist move2}.) The case for diagrams that differ by a switch move on the left hand side of $v$ is treated similarly. Finally, we wish to compare a choice where the crossing is on the right hand side of $v$ to one where the crossing is on the left hand side of $v$. Figure~\ref{R4R1} illustrates a way to relate such diagrams via $R1$, $R4$ and swing moves (this shall be enough once we check the swing moves and braid isotopy below).

  \begin{figure}
    \[
    \raisebox{-20pt}{\includegraphics[height=.5in]{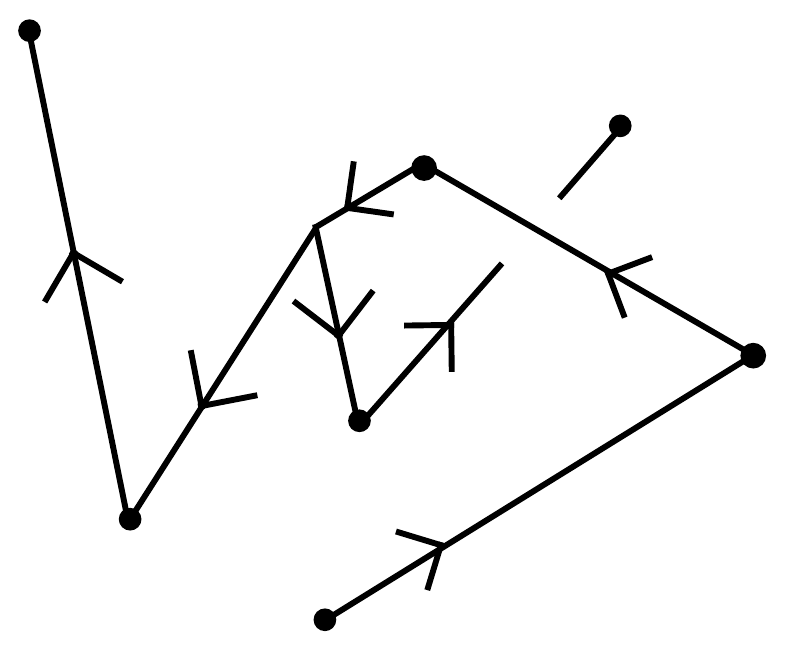}}
    \hspace{0.4cm} \longleftrightarrow \hspace{0.4cm}
    \raisebox{-20pt}{\includegraphics[height=.5in]{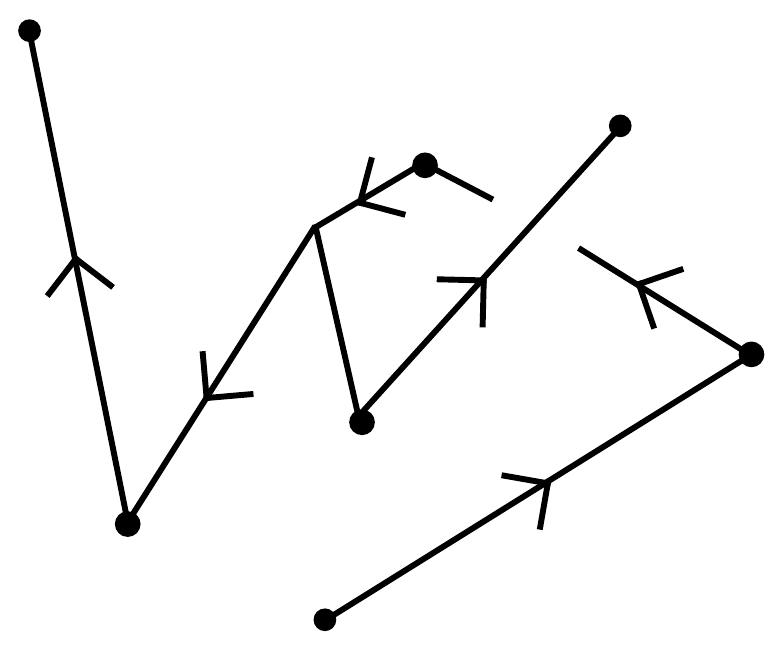}}
    \put(-138,8){\fontsize{9}{9}$o$}
    \put(-95,-13){\fontsize{9}{9}$o$}
    \put(-48,8){\fontsize{9}{9}$o$} \put(-5,-13){\fontsize{9}{9}$o$}
    \]
    \vspace{.1cm}
    \[
    \downarrow \hspace{3cm} \downarrow \put(-136,0){\small{braiding}}
    \put(0,0){\small{braiding}}
    \]
    \vspace{.1cm}
    \[
    \raisebox{-30pt}{\includegraphics[height=1in]{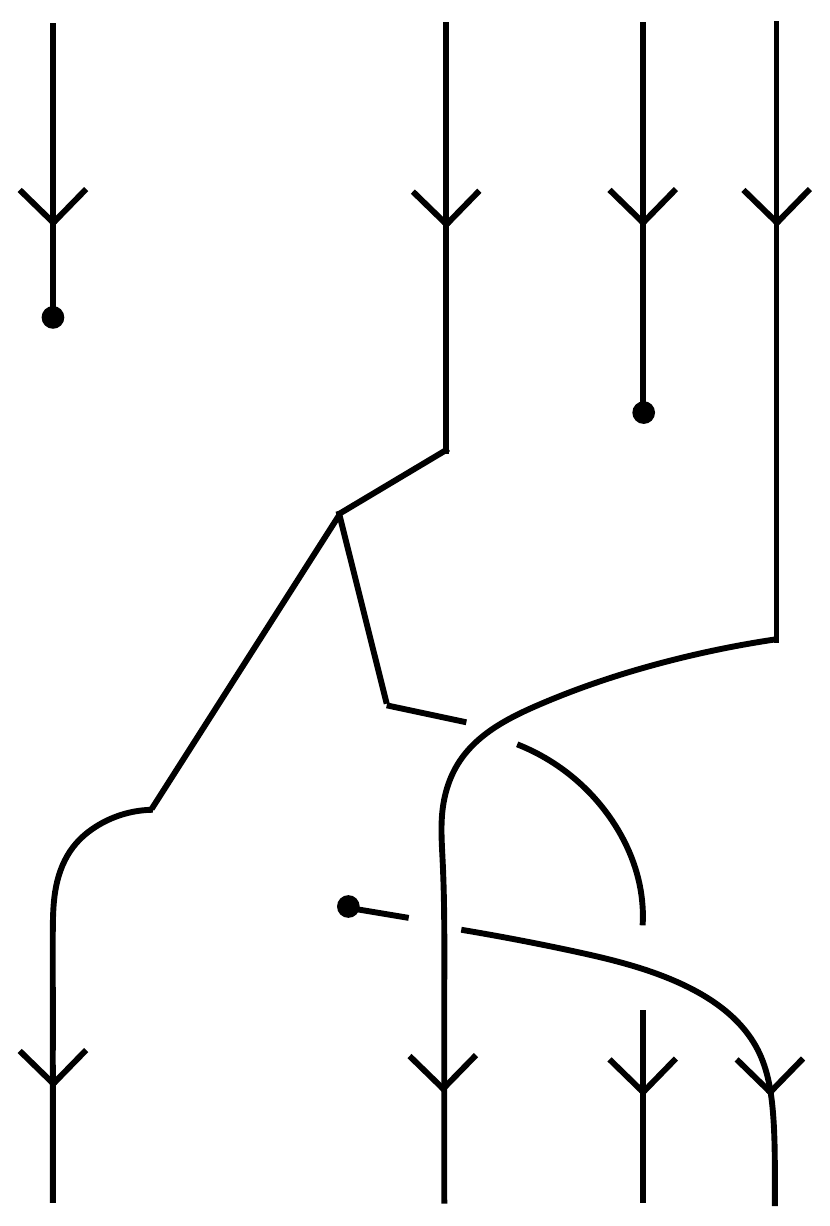}}
    \hspace{1.8cm}
    \raisebox{-30pt}{\includegraphics[height=1in]{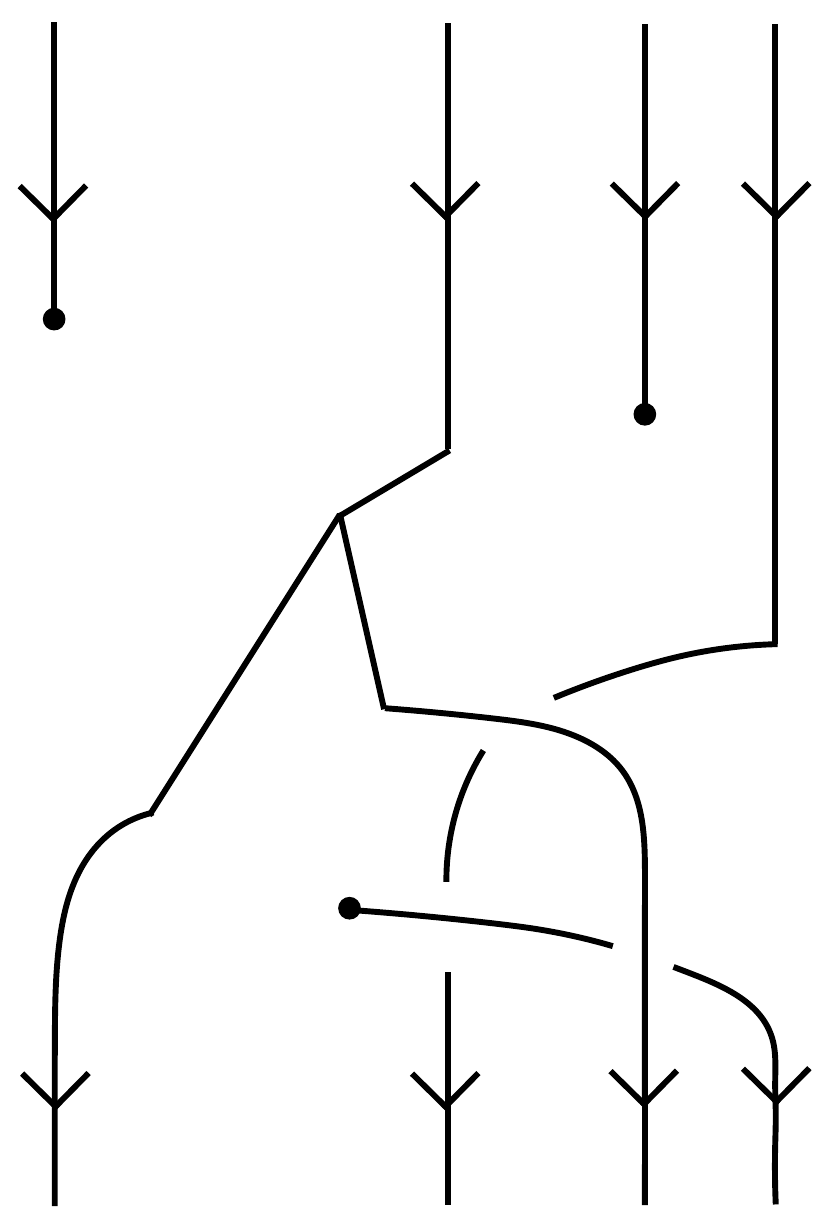}}
    \put(-152,36){\fontsize{9}{9}$o$}
    \put(-152,-28){\fontsize{9}{9}$o$}
    \put(-130,36){\fontsize{9}{9}$o$}
    \put(-130,-28){\fontsize{9}{9}$o$}
    \put(-119,36){\fontsize{9}{9}$u$}
    \put(-119,-28){\fontsize{9}{9}$u$}
    \put(-109,36){\fontsize{9}{9}$o$}
    \put(-109,-28){\fontsize{9}{9}$o$}
    \put(-52,36){\fontsize{9}{9}$o$}
    \put(-52,-28){\fontsize{9}{9}$o$}
    \put(-30,36){\fontsize{9}{9}$u$}
    \put(-30,-28){\fontsize{9}{9}$u$}
    \put(-19,36){\fontsize{9}{9}$o$}
    \put(-19,-28){\fontsize{9}{9}$o$}
    \put(-9,36){\fontsize{9}{9}$o$} \put(-9,-28){\fontsize{9}{9}$o$}
    \]
    \caption{} \label{UUU Twist}
  \end{figure}

  \begin{figure}
    \[
    \raisebox{-25pt}{\includegraphics[height = 0.75in]{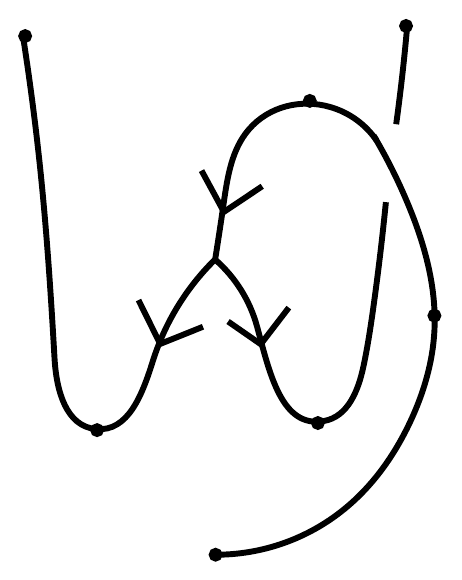}}
    \hspace{.2cm} \longleftrightarrow \hspace{.2cm}
    \raisebox{-40pt}{\includegraphics[height = 1in]{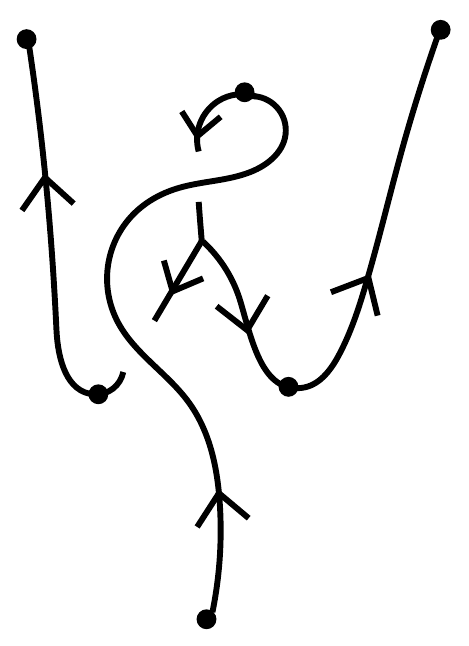}}
    \hspace{.2cm} \longleftrightarrow \hspace{.2cm}
    \raisebox{-40pt}{\includegraphics[height = 1in]{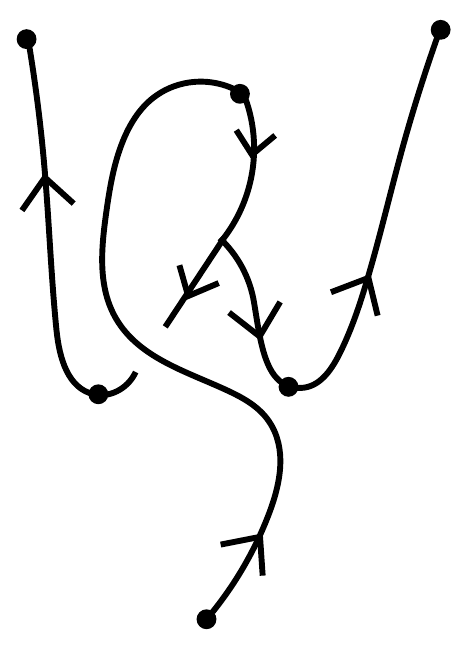}}
    \hspace{.2cm} \longleftrightarrow \hspace{.2cm}
    \raisebox{-25pt}{\includegraphics[height = 0.75in]{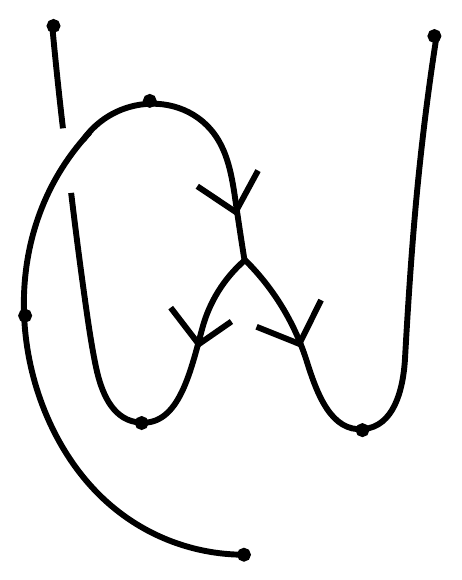}}
    \put(-244,10){\small{$R4$}}
    \put(-154,10){\small{$R1$}} \put(-72,20){\small{swing}}
    \put(-72,10){\small{move}} \put(-246,-10){\small{swing}}
    \put(-246,-20){\small{move}}
    \]
    \caption{} \label{R4R1}
  \end{figure}

  For the last instance of shifting a vertex into regular position, we need to consider a $\lambda$-type vertex incident with three up-arcs. Once again, we allow for
  four different choices to put such vertex into regular position (see the third row of Figure \ref{vgp2}). A similar argument as the one in the paragraph above shows that these choices yield $TL$-equivalent braids.

  Now we check the elimination of horizontal arcs. This amounts to planar isotopy between diagrams in general position. For the case of an up-arc, planar isotopy can be treated by subdividing an up-arc (we refer the reader to \cite{LR} for details). The most interesting case of planar isotopy of a down-arc is verified in  Figure~\ref{Darc}. The remaining cases can be derived easily from the previous one.

  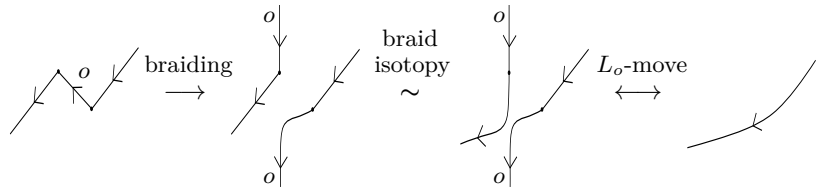
\begin{figure}
    \[
    \raisebox{-15pt}{\includegraphics[height=.5in,
      width=.7in]{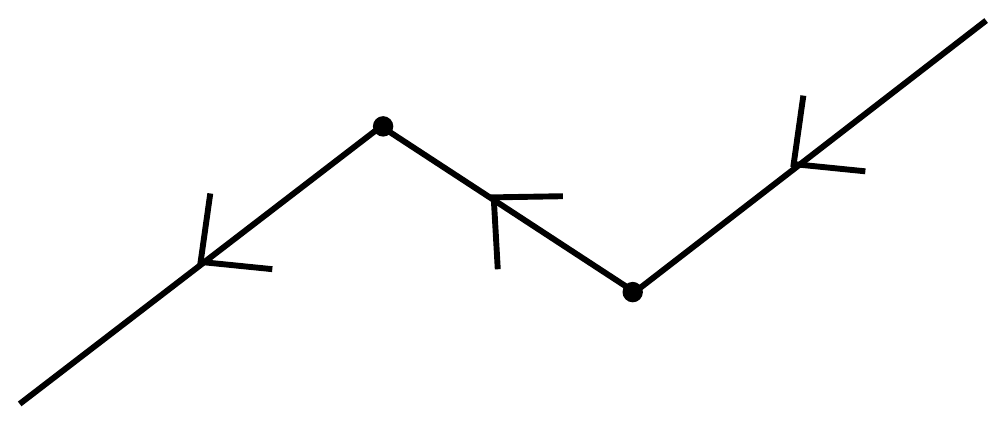}} \hspace{.2cm} \longrightarrow
    \hspace{.2cm} \raisebox{-35pt}{\includegraphics[height = 1in,
      width = .7in]{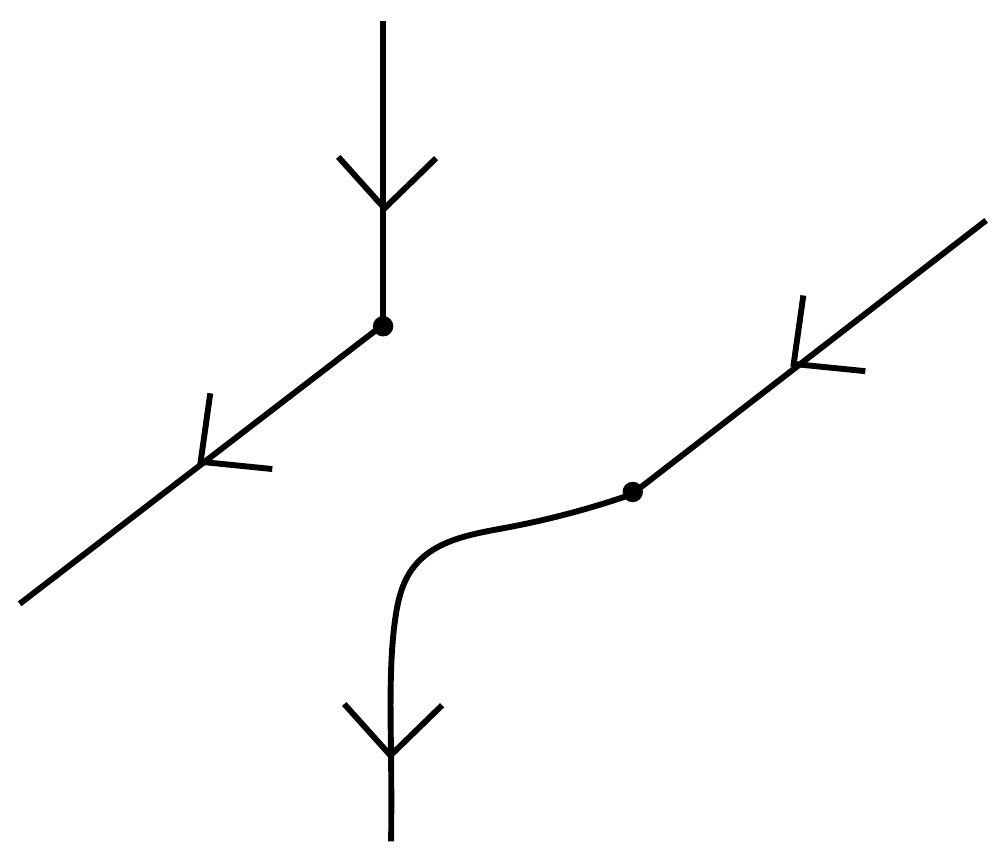}} \hspace{.4cm} \thicksim
    \hspace{.4cm} \raisebox{-35pt}{\includegraphics[height = 1in,
      width = .7in]{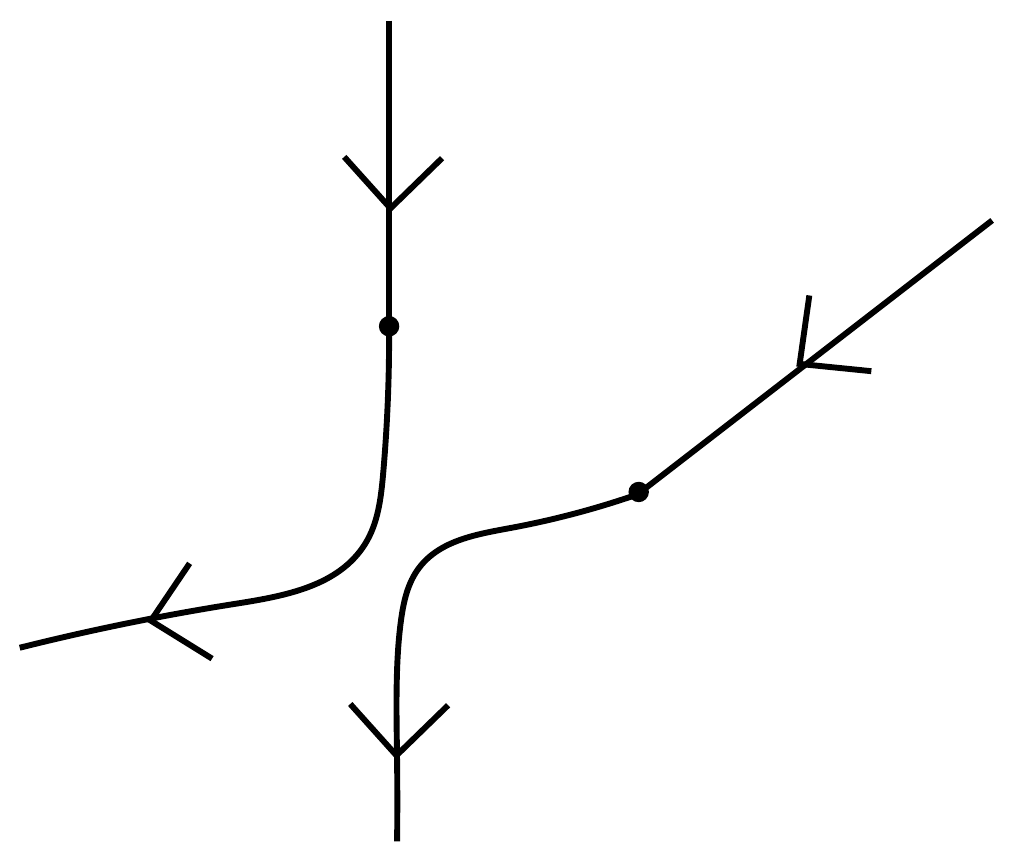}} \hspace{.2cm}
    \longleftrightarrow \hspace{.2cm}
    \raisebox{-20pt}{\includegraphics[height = .5in, width
      =.7in]{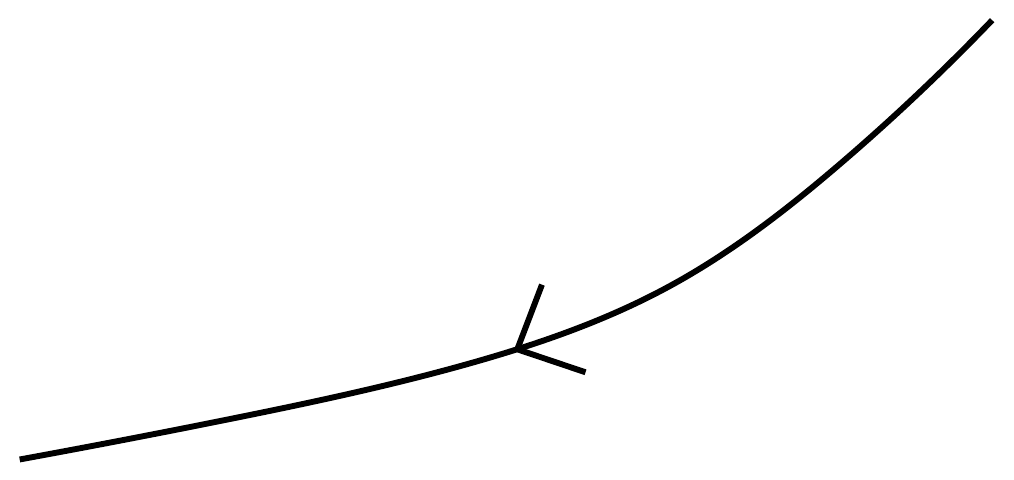}}
    \put(-280,8){\fontsize{9}{9}$o$} \put(-210,
    29){\fontsize{9}{9}$o$} \put(-210,
    -32){\fontsize{9}{9}$o$} \put(-123,
    29){\fontsize{9}{9}$o$} \put(-123,
    -32){\fontsize{9}{9}$o$} \put(-255, 10){\small{braiding}}
    \put(-165, 20){\small{braid}} \put(-168, 10){\small{isotopy}}
    \put(-84, 10){\small{$L_o$-move}}
    \]
    \caption{Planar isotopy on a down-arc} \label{Darc}
  \end{figure}

  Correcting horizontal alignment of either crossings, subdivision points or vertices amounts to small vertical shifts, which yield\---- up to braid isotopy\---- the same trivalent braid. In Figure~\ref{valign} we illustrate the correcting shifts for vertically aligned subdivision points. Note that the final braids are the same, up to planar isotopy. The remaining instances of vertical alignment can be treated similarly.

  \begin{figure}[ht]
    \[
    \raisebox{-30pt}{\includegraphics[height=1in,
      width=.6in]{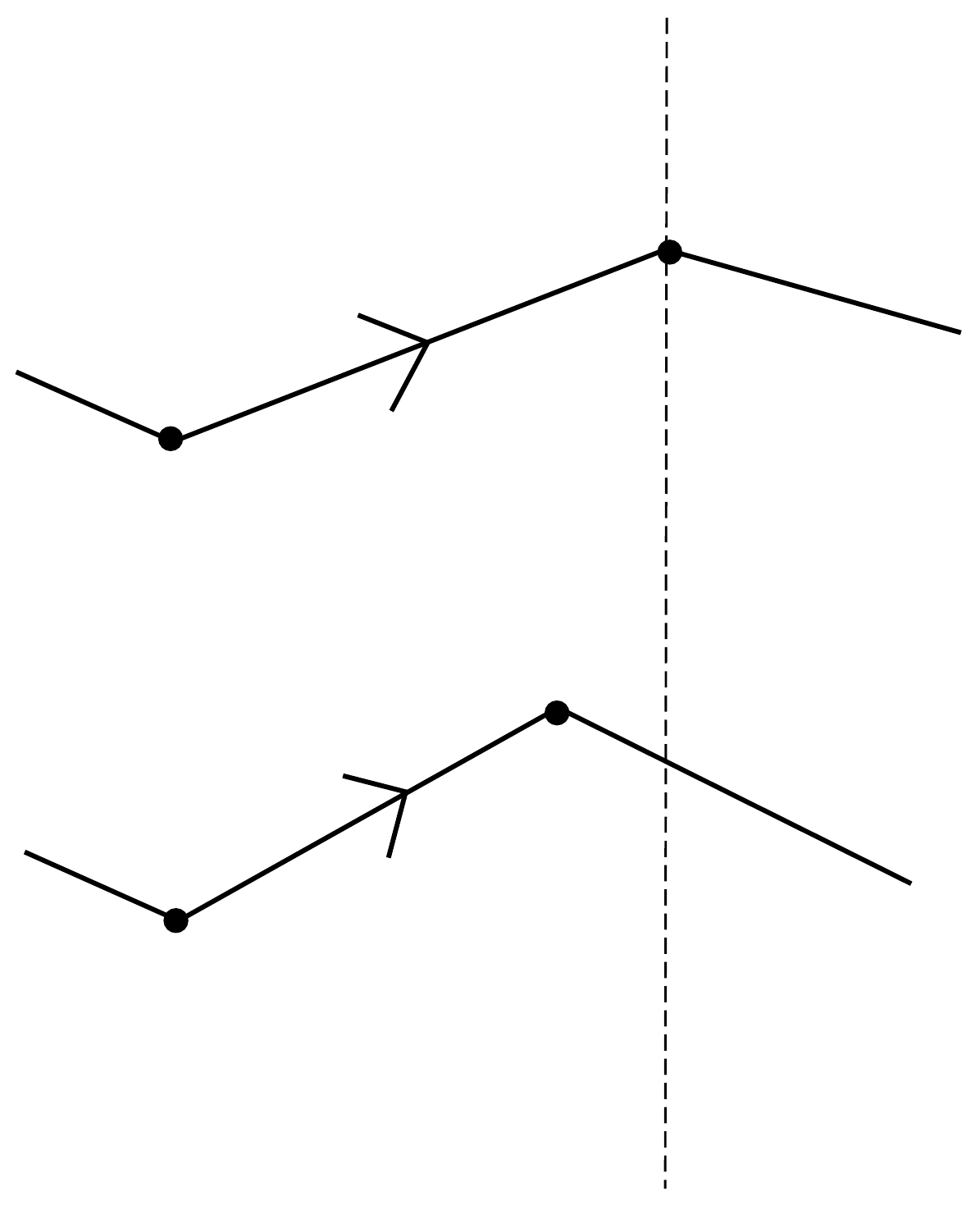}} \hspace{.4cm} \longrightarrow
    \hspace{.4cm} \raisebox{-30pt}{\includegraphics[height=1in,
      width=.6in]{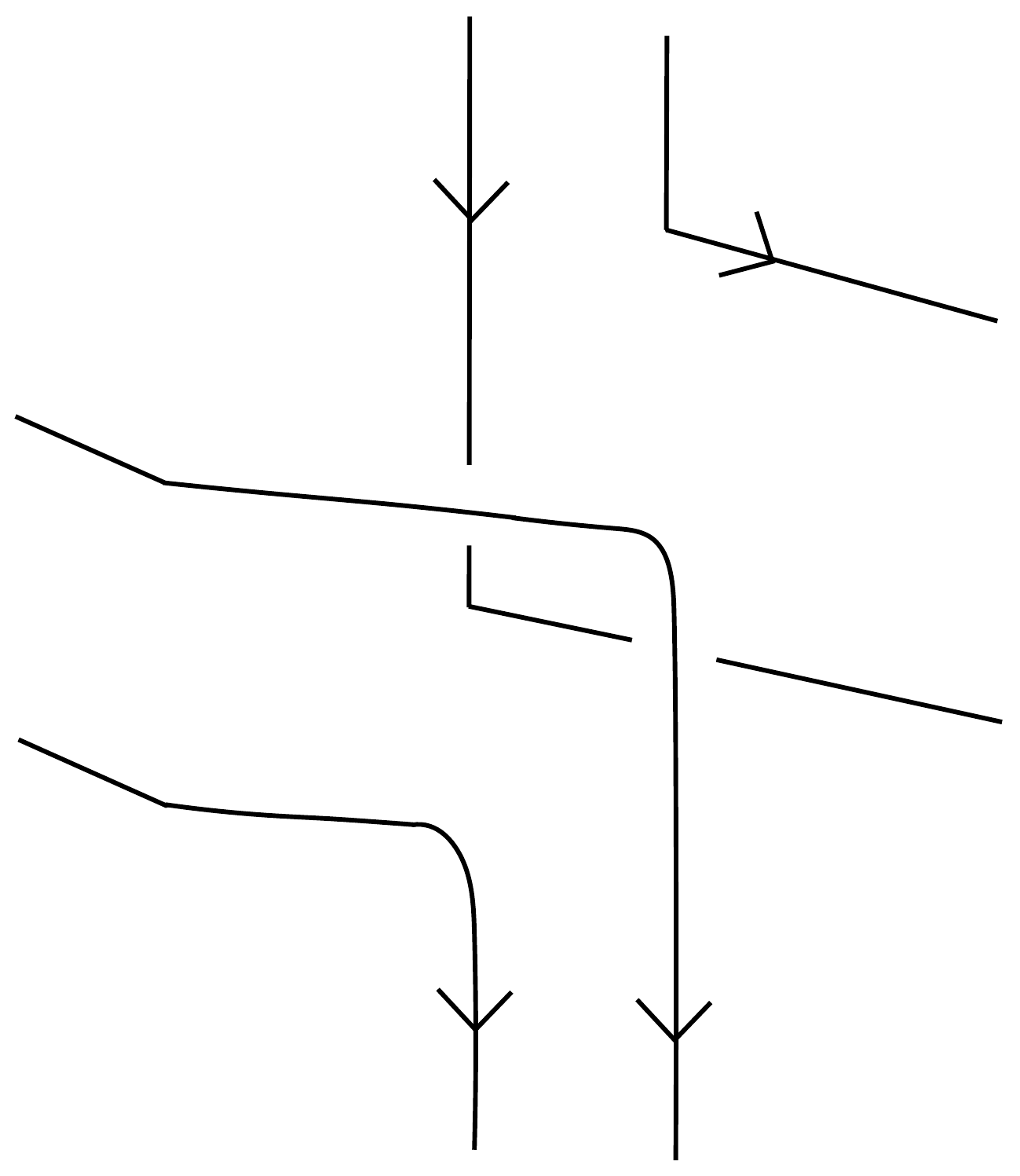}} \hspace{.4cm}
    \longleftrightarrow \hspace{.4cm}
    \raisebox{-30pt}{\includegraphics[height=1in,
      width=.6in]{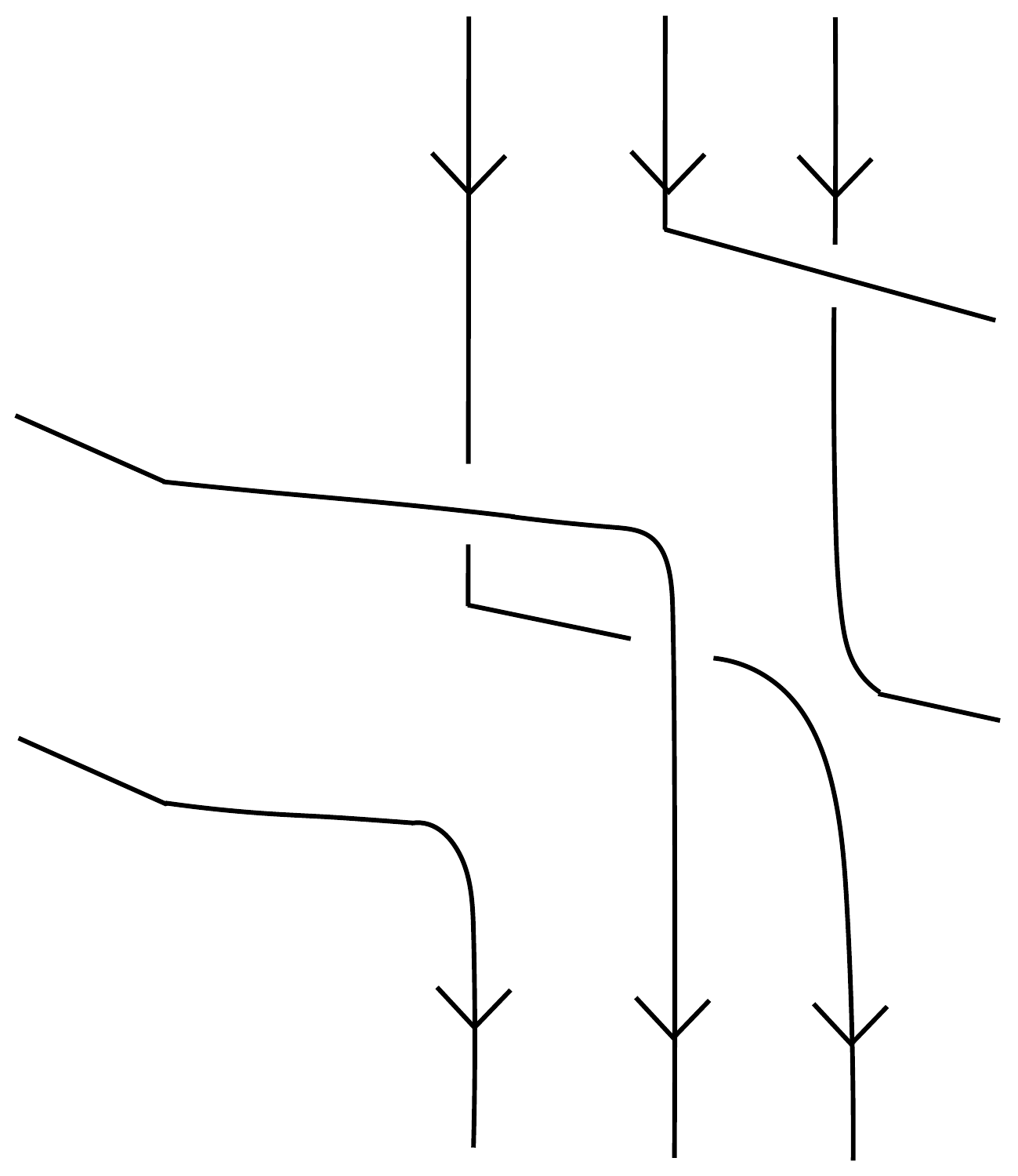}} \put(-205,
    25){\fontsize{9}{9}$o$} \put(-210,
    -2){\fontsize{9}{9}$u$} \put(-30,
    35){\fontsize{9}{9}$u$} \put(-21,
    35){\fontsize{9}{9}$o$} \put(-14,
    35){\fontsize{9}{9}$u$} \put(-30,
    -27){\fontsize{9}{9}$u$} \put(-21,
    -27){\fontsize{9}{9}$o$} \put(-13,
    -27){\fontsize{9}{9}$u$} \put(-111,
    35){\fontsize{9}{9}$o$} \put(-120,
    35){\fontsize{9}{9}$u$} \put(-120,
    -27){\fontsize{9}{9}$u$} \put(-111,
    -27){\fontsize{9}{9}$o$} \put(-84,
    10){\small{$L_u$-move}} \put(-170, 10){\small{braiding}}
    \]
    \vspace{.4cm}
    \[
    \raisebox{-30pt}{\includegraphics[height=1in,
      width=.6in]{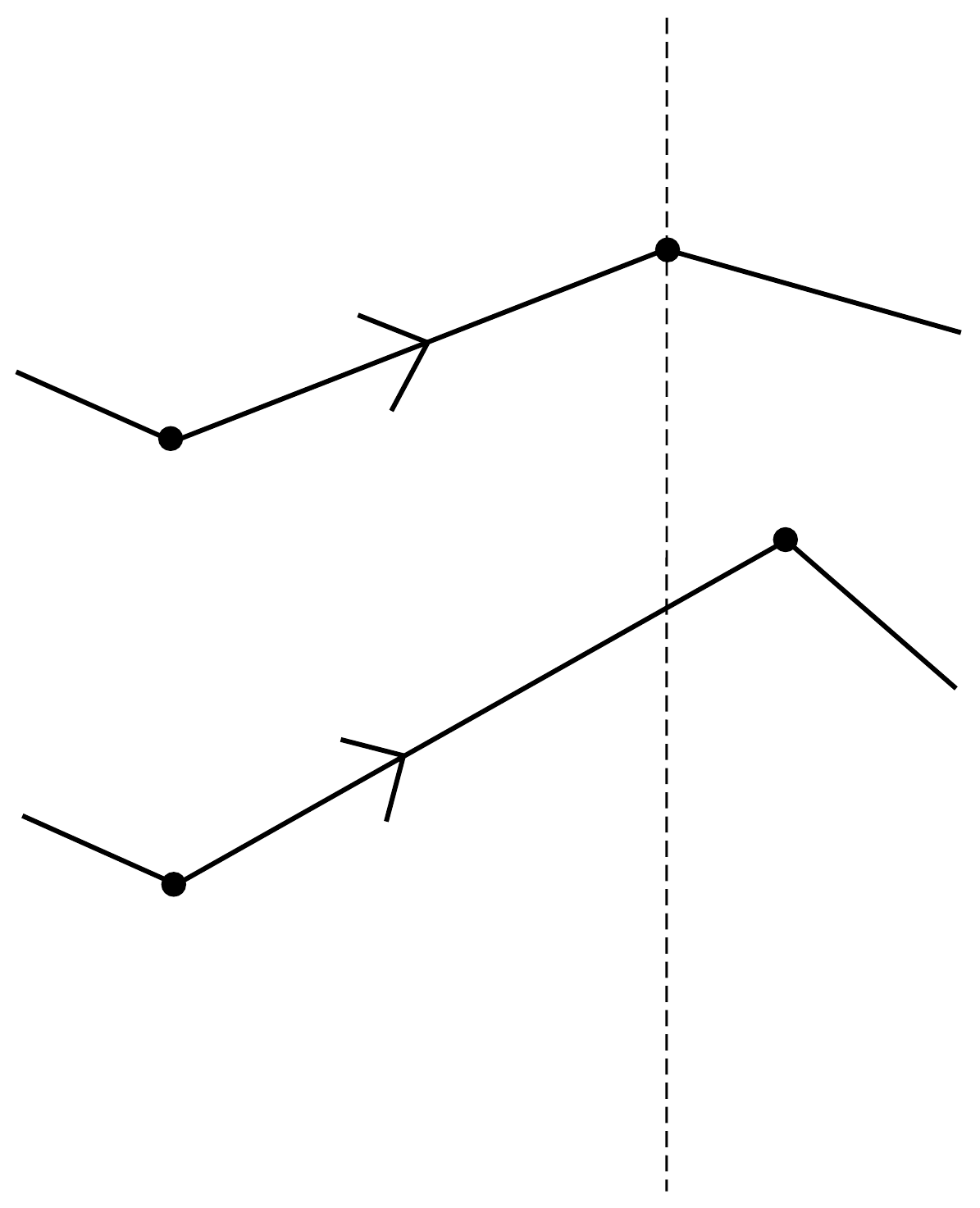}} \hspace{.4cm} \longrightarrow
    \hspace{.4cm} \raisebox{-30pt}{\includegraphics[height=1in,
      width=.6in]{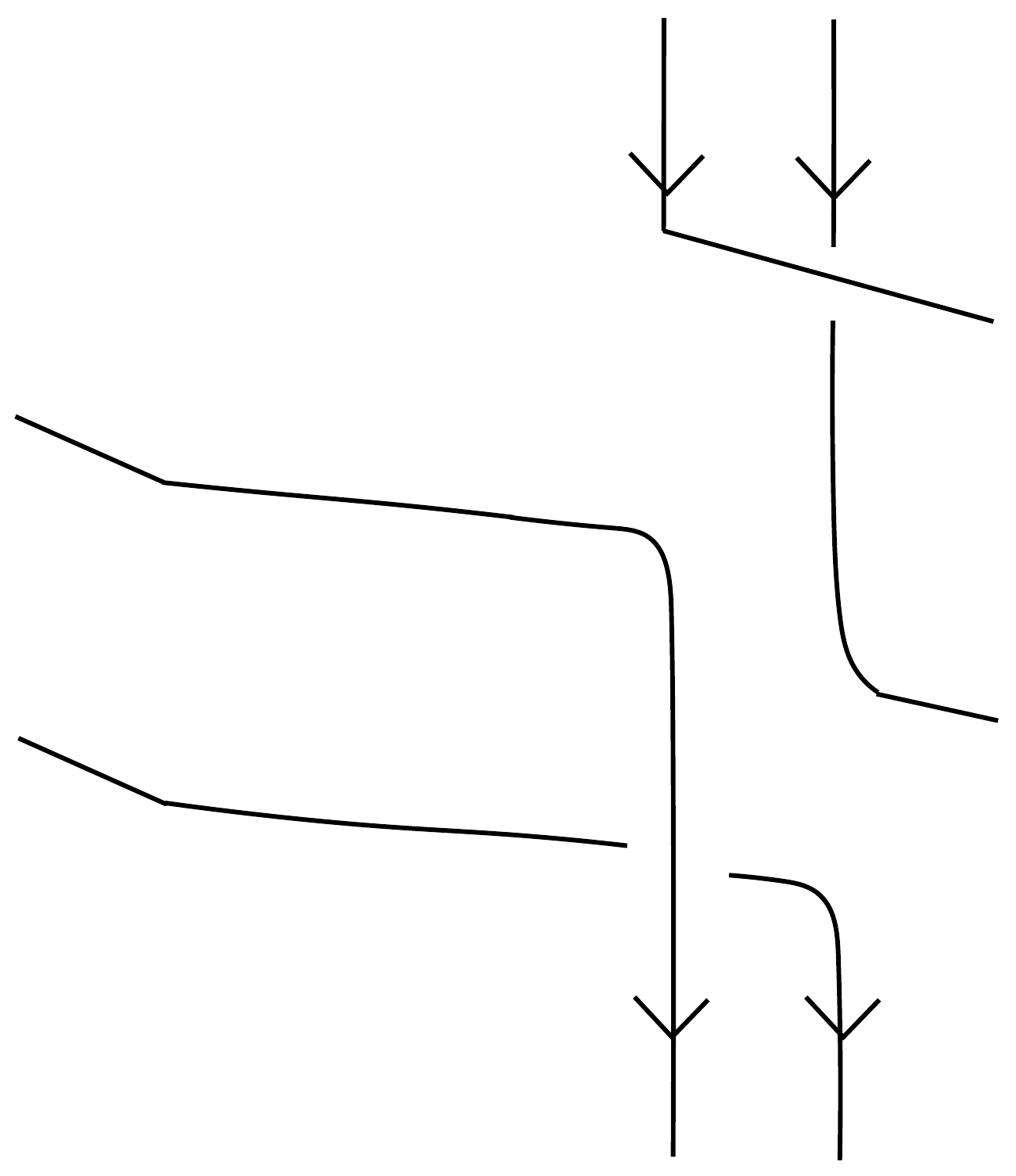}} \hspace{.3cm}
    \longleftrightarrow \hspace{.4cm}
    \raisebox{-30pt}{\includegraphics[height=1in,
      width=.6in]{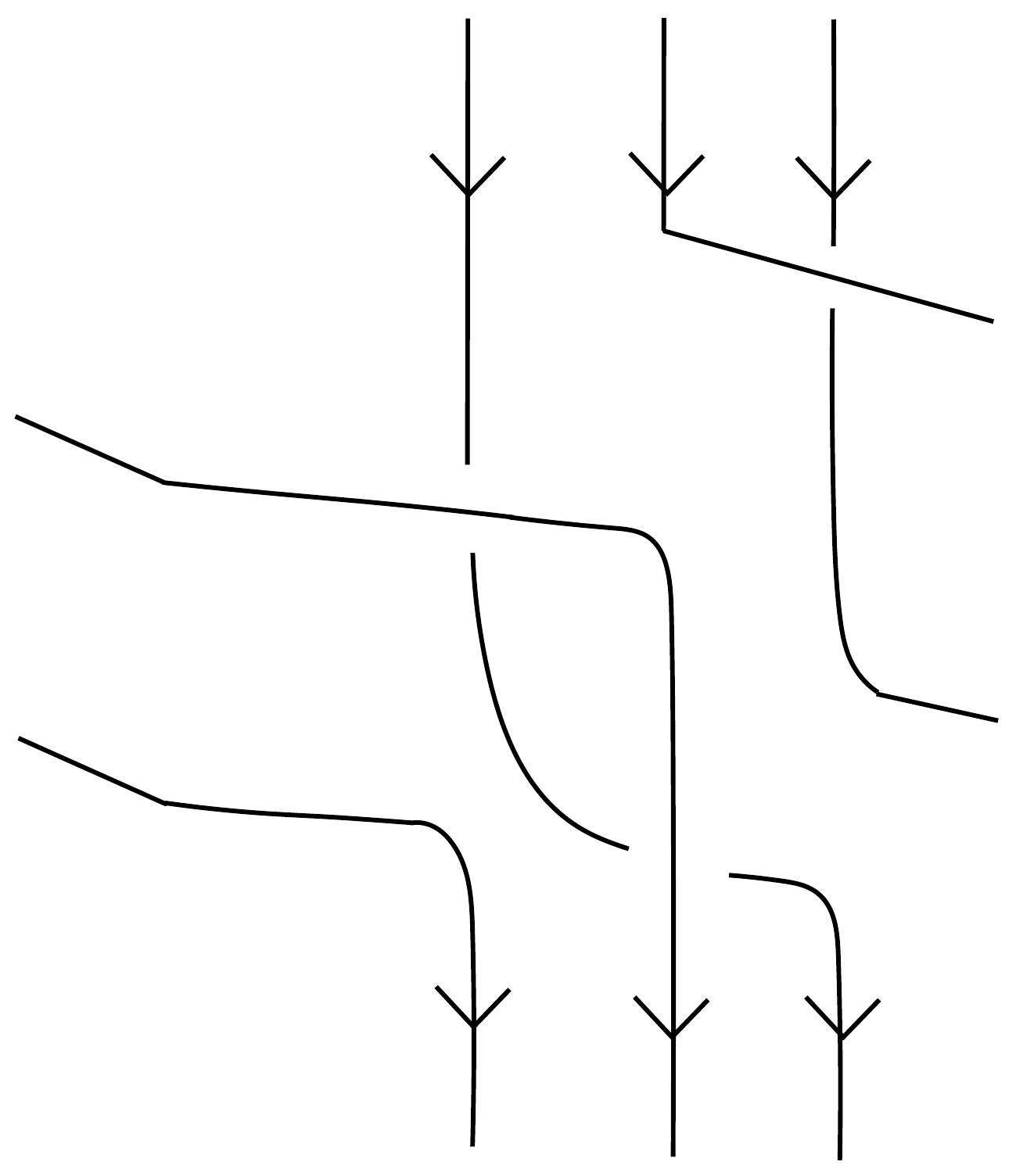}} \put(-205,
    25){\fontsize{9}{9}$o$} \put(-200,
    4){\fontsize{9}{9}$u$} \put(-30,
    35){\fontsize{9}{9}$u$} \put(-21,
    35){\fontsize{9}{9}$o$} \put(-14,
    35){\fontsize{9}{9}$u$} \put(-30,
    -27){\fontsize{9}{9}$u$} \put(-21,
    -27){\fontsize{9}{9}$o$} \put(-14,
    -27){\fontsize{9}{9}$u$} \put(-110,
    35){\fontsize{9}{9}$o$} \put(-101,
    35){\fontsize{9}{9}$u$} \put(-101,
    -27){\fontsize{9}{9}$u$} \put(-109,
    -27){\fontsize{9}{9}$o$} \put(-83,
    10){\fontsize{9}{9}{$L_u$-move}} \put(-170, 10){\small{braiding}}
    \]
    \caption{Planar shifts of vertically aligned subdivision
      points} \label{valign}
  \end{figure}
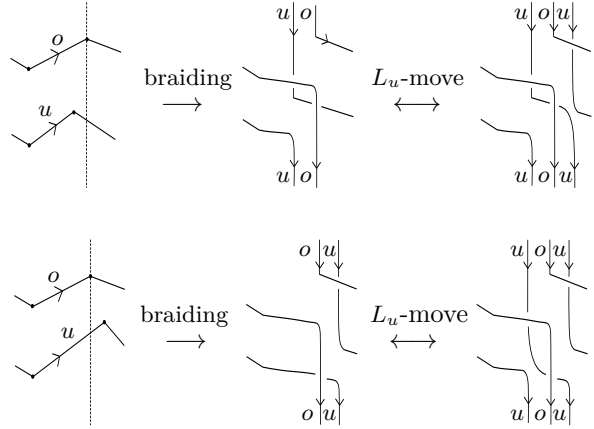

  Finally, in Figures~\ref{swing1} and~\ref{swing2} we show that the swing moves also yield $TL$-equivalent trivalent braids. This completes the proof of the lemma.
\end{proof}

\begin{figure}
  \[
  \raisebox{-15pt}{\includegraphics[height=.59in]{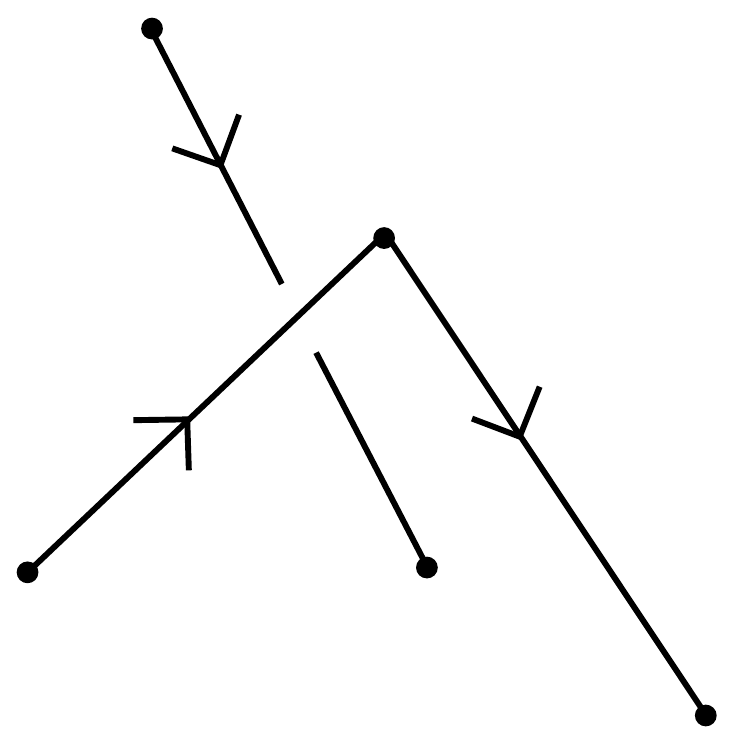}}
  \hspace{.4cm} \longrightarrow \hspace{.4cm}
  \raisebox{-30pt}{\includegraphics[height=1in]{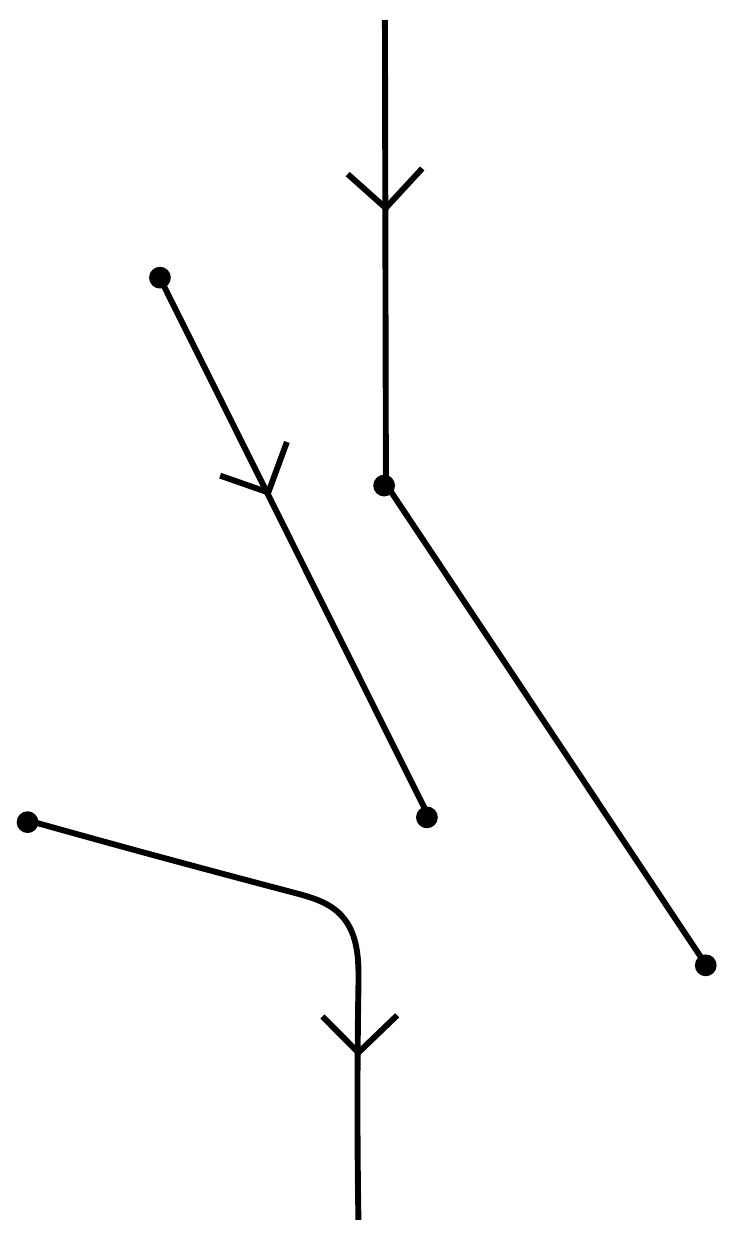}}
  \hspace{.4cm} \thicksim \hspace{.4cm}
  \raisebox{-30pt}{\includegraphics[height=1in]{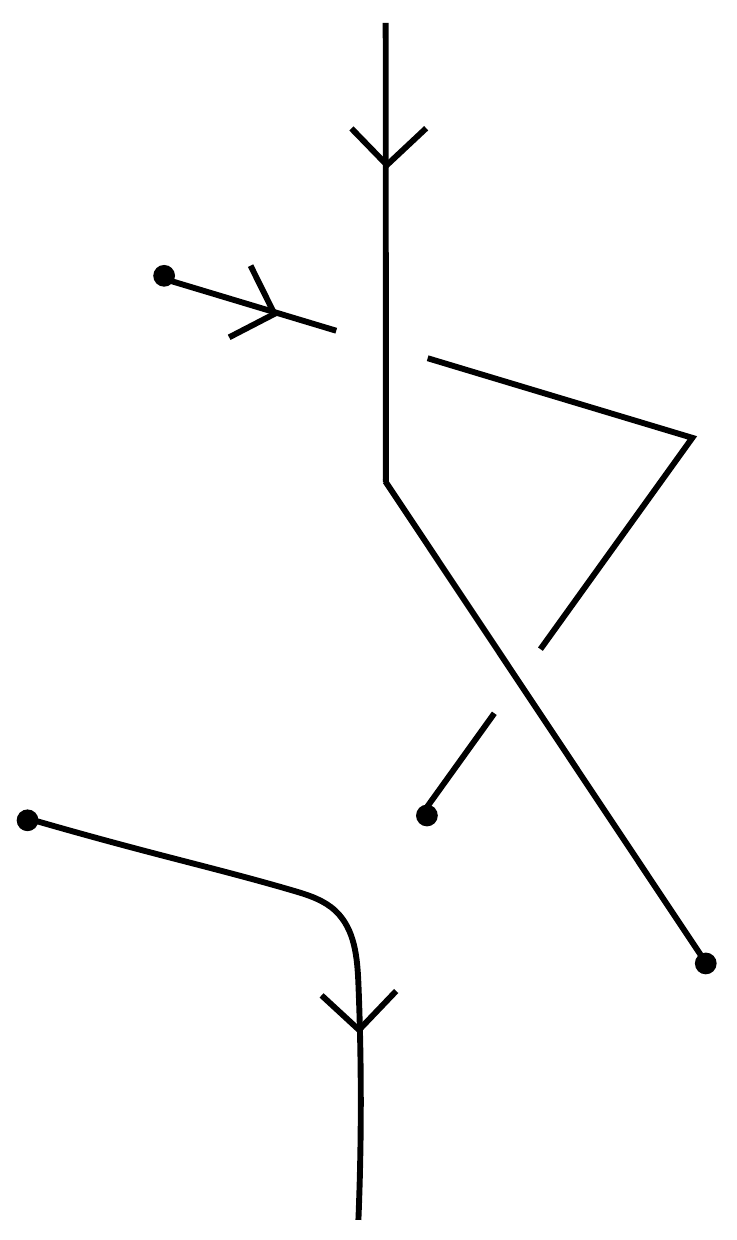}}
  \hspace{.4cm} \longleftarrow \hspace{.4cm}
  \raisebox{-15pt}{\includegraphics[height=.59in]{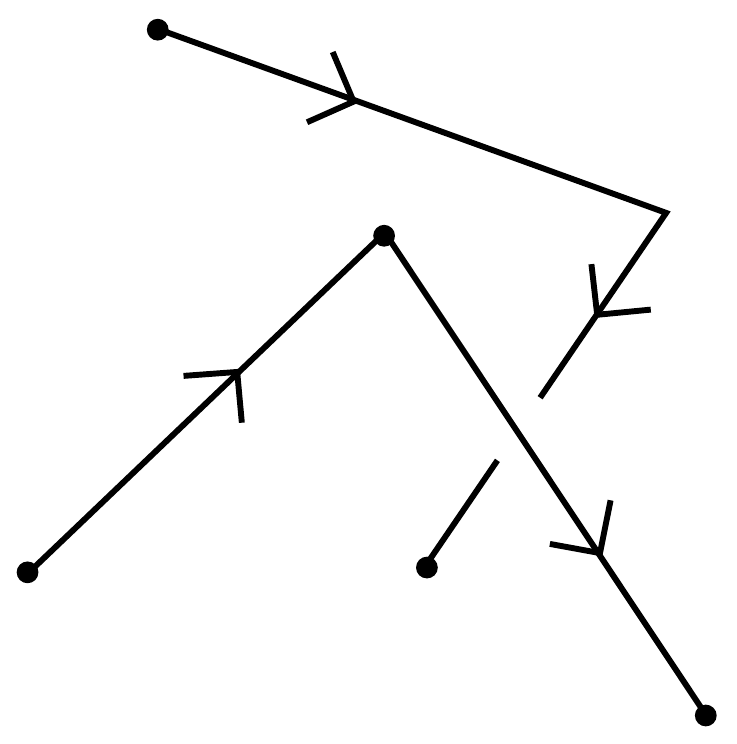}}
  \put(-193,
  35){\fontsize{9}{9}$o$} \put(-194,
  -27){\fontsize{9}{9}$o$} \put(-114,
  35){\fontsize{9}{9}$o$} \put(-115,
  -27){\fontsize{9}{9}$o$} \put(-249, 10){\small{braiding}} \put(-80,
  10){\small{braiding}} \put(-159, 20){\small{braid}} \put(-162,
  10){\small{isotopy}}
  \]
  \caption{The first case of the swing moves} \label{swing1}
\end{figure}

\begin{figure}
  \[
  \raisebox{-15pt}{\includegraphics[height=.55in]{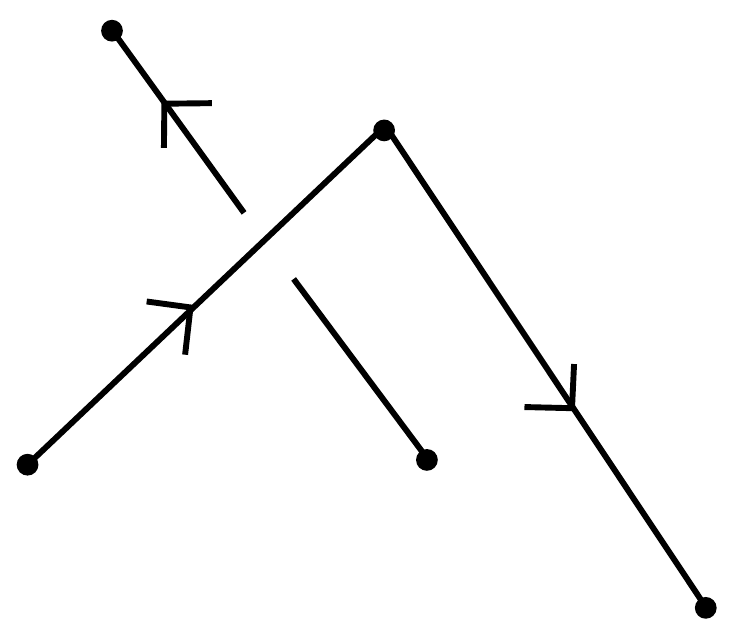}}
  \hspace{.4cm} \longrightarrow \hspace{.4cm}
  \raisebox{-30pt}{\includegraphics[height=1in]{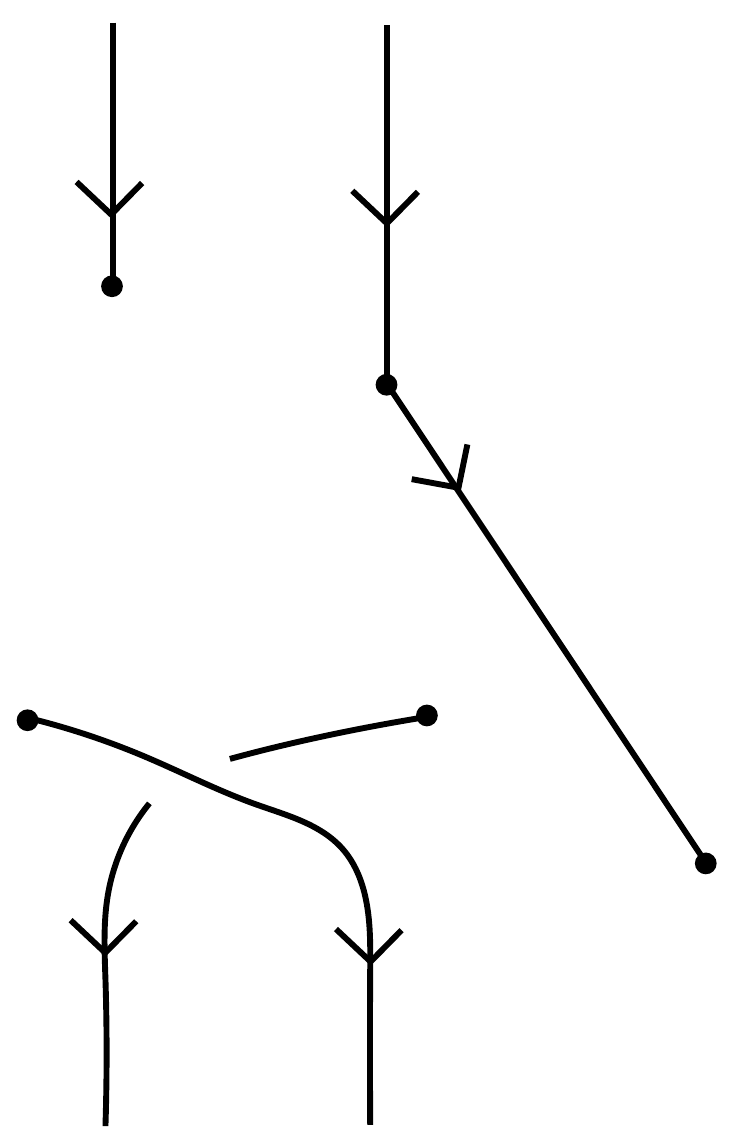}}
  \hspace{.4cm} \thicksim \hspace{.4cm}
  \raisebox{-30pt}{\includegraphics[height=1in]{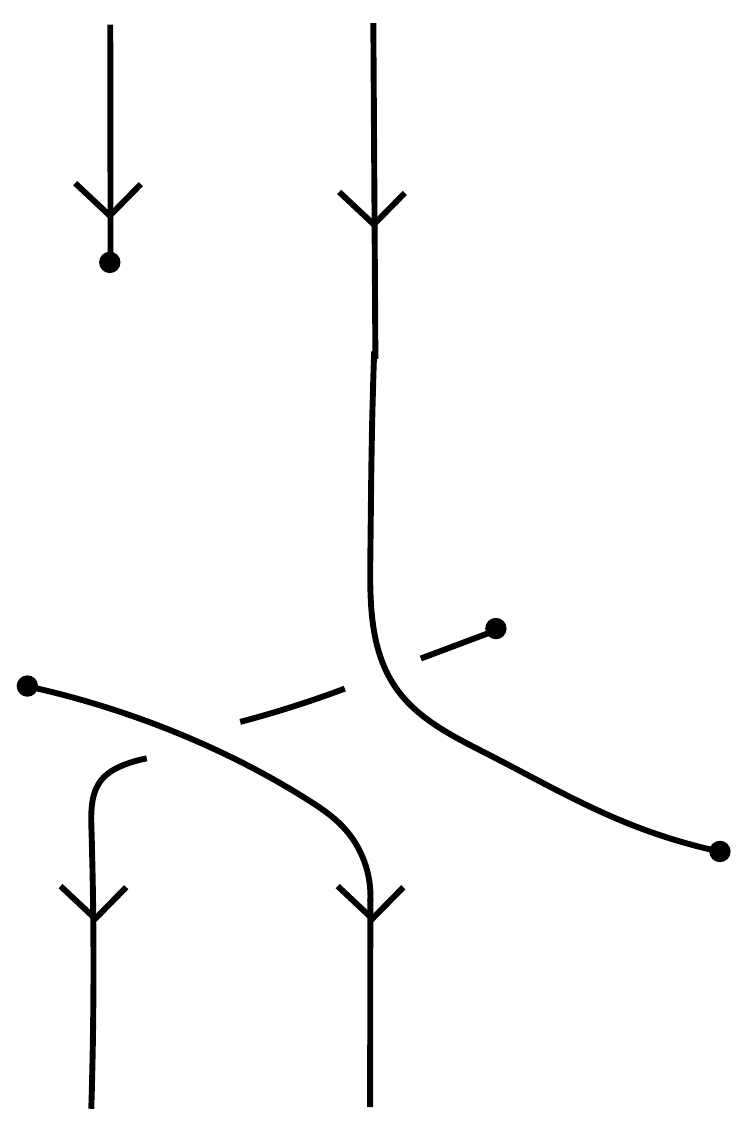}}
  \hspace{.4cm} \longleftrightarrow \hspace{.4cm}
  \raisebox{-30pt}{\includegraphics[height=1in]{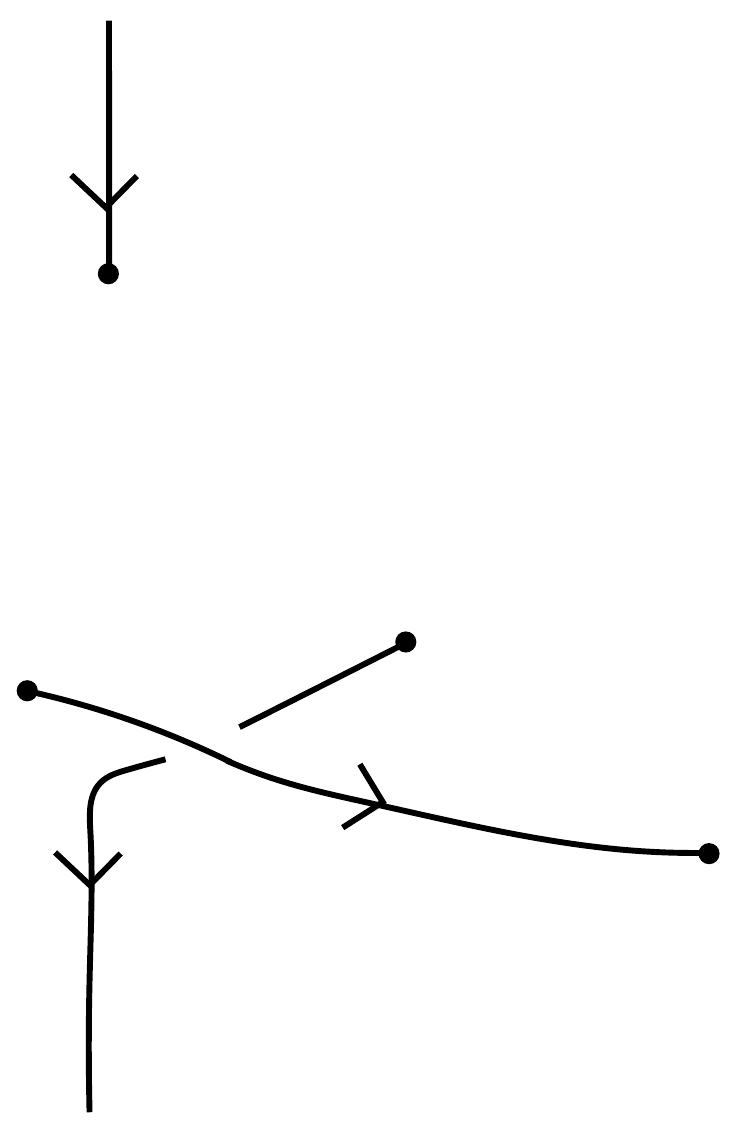}}
  \put(-224,
  35){\fontsize{9}{9}$u$} \put(-225,
  -27){\fontsize{9}{9}$u$} \put(-207,
  35){\fontsize{9}{9}$o$} \put(-207,
  -27){\fontsize{9}{9}$o$} \put(-142,
  35){\fontsize{9}{9}$u$} \put(-142,
  -27){\fontsize{9}{9}$u$} \put(-124,
  35){\fontsize{9}{9}$o$} \put(-124,
  -27){\fontsize{9}{9}$o$} \put(-47,
  35){\fontsize{9}{9}$u$} \put(-48,
  -27){\fontsize{9}{9}$u$} \put(-82, 20){\small{basic}} \put(-88,
  10){\small{$L_o$-move}} \put(-264, 10){\small{braiding}} \put(-170,
  20){\small{braid}} \put(-173, 10){\small{isotopy}}
  \]
  \vspace{0cm}
  \begin{center}
    \rule{\textwidth}{.5pt}
  \end{center}
  \vspace{0cm}
  \[
  \raisebox{-15pt}{\includegraphics[height=.59in]{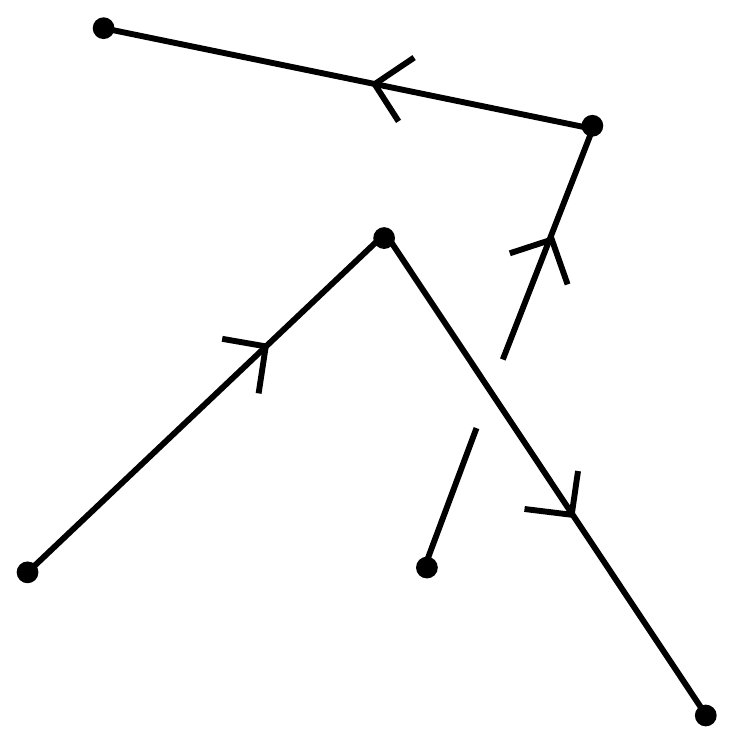}}
  \hspace{.4cm} \longrightarrow \hspace{.4cm}
  \raisebox{-30pt}{\includegraphics[height=1in]{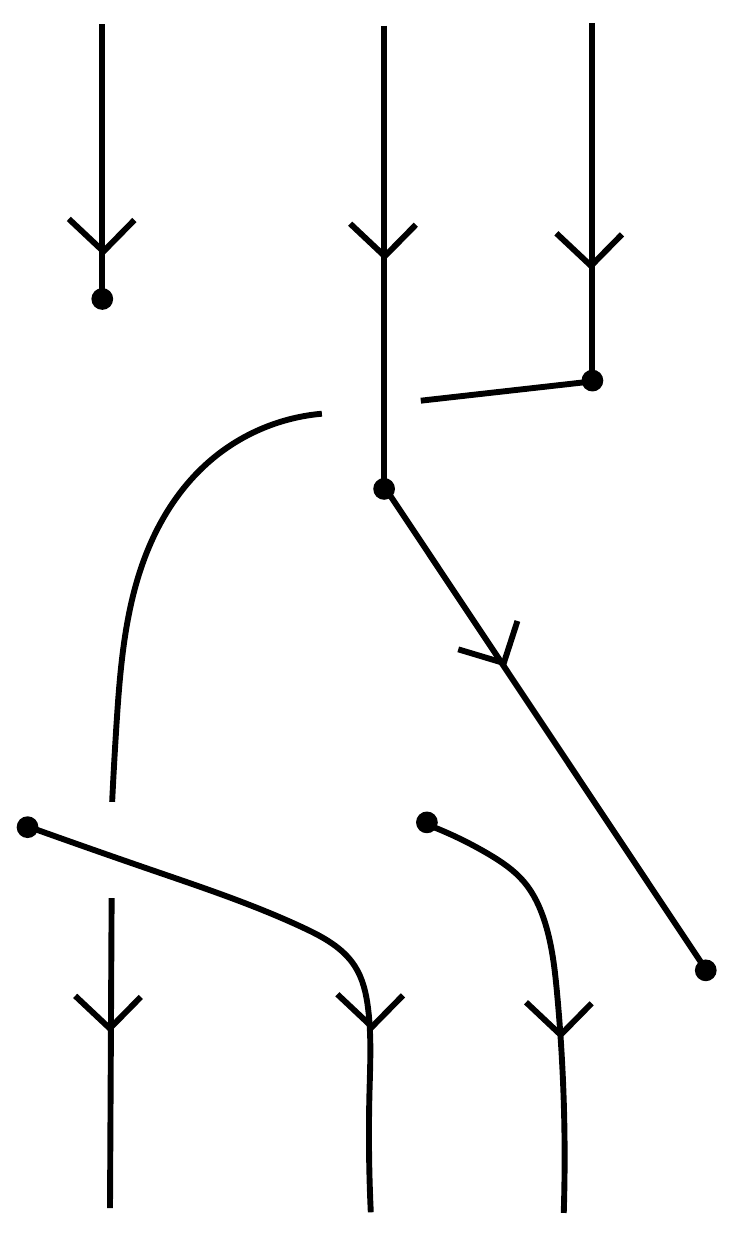}}
  \hspace{.4cm} \thicksim \hspace{.4cm}
  \raisebox{-30pt}{\includegraphics[height=1in]{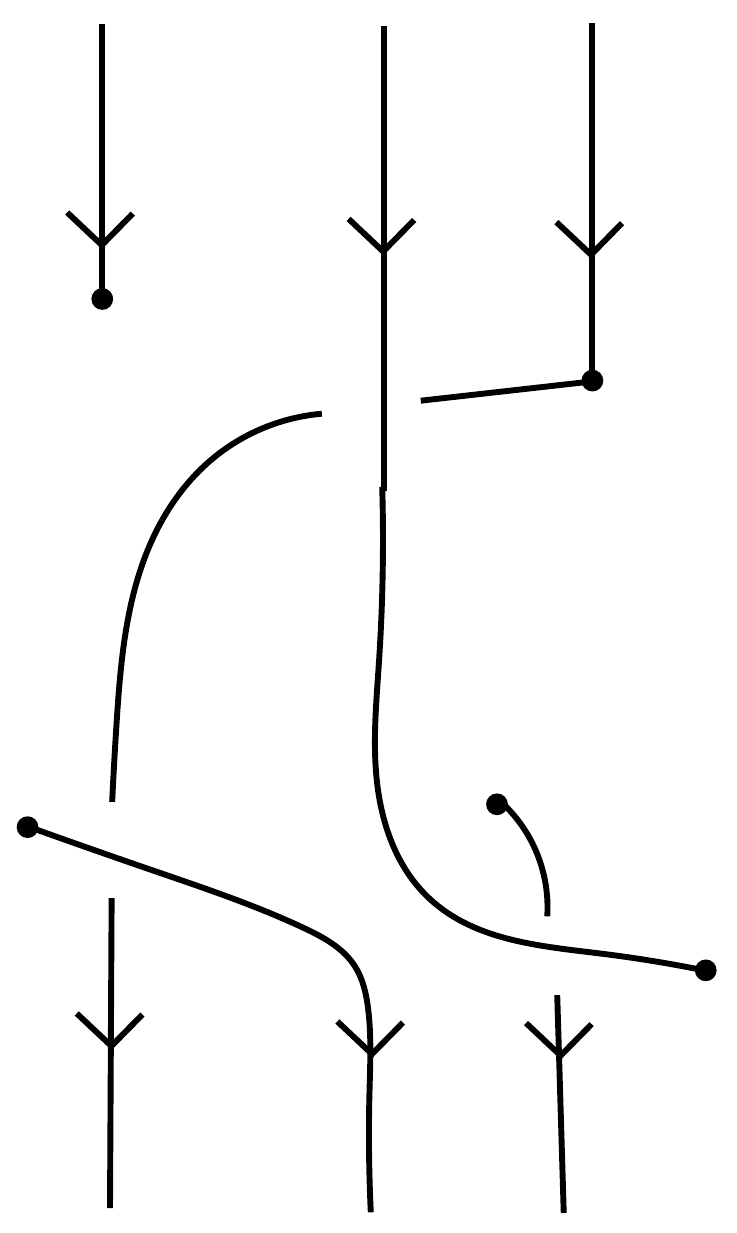}}
  \put(-192,
  28){\fontsize{9}{9}$u$} \put(-200,
  10){\fontsize{9}{9}$o$} \put(-124,
  35){\fontsize{9}{9}$u$} \put(-123,
  -27){\fontsize{9}{9}$u$} \put(-107,
  35){\fontsize{9}{9}$o$} \put(-107,
  -27){\fontsize{9}{9}$o$} \put(-95,
  35){\fontsize{9}{9}$u$} \put(-96,
  -27){\fontsize{9}{9}$u$} \put(-44,
  35){\fontsize{9}{9}$u$} \put(-44,
  -27){\fontsize{9}{9}$u$} \put(-27,
  35){\fontsize{9}{9}$o$} \put(-28,
  -27){\fontsize{9}{9}$o$} \put(-15,
  35){\fontsize{9}{9}$u$} \put(-17,
  -27){\fontsize{9}{9}$u$} \put(-161, 10){\small{braiding}} \put(-72,
  20){\small{braid}} \put(-75, 10){\small{isotopy}}
  \]
  \vspace{.4cm}
  \[
  \hspace{.4cm} \longleftrightarrow \hspace{.4cm}
  \raisebox{-30pt}{\includegraphics[height=1in]{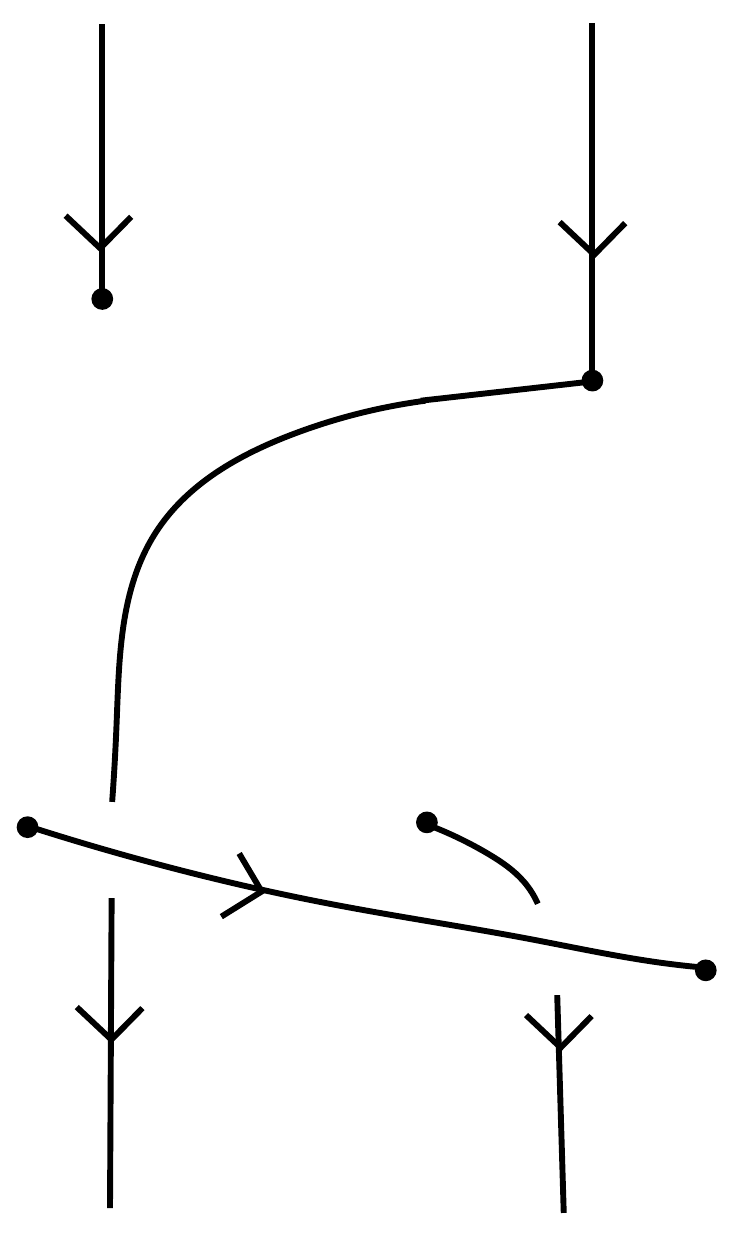}}
  \hspace{.4cm} \thicksim \hspace{.4cm}
  \raisebox{-30pt}{\includegraphics[height=1in]{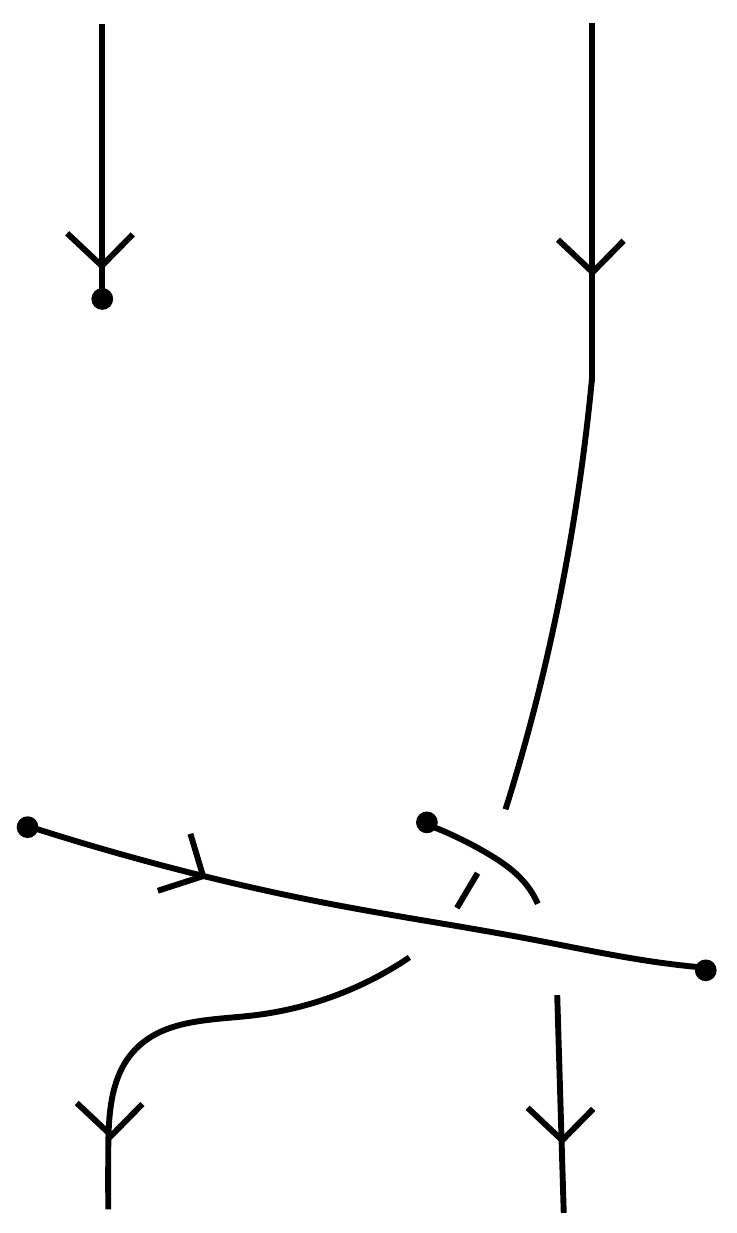}}
  \hspace{.4cm} \longleftrightarrow \hspace{.4cm}
  \raisebox{-30pt}{\includegraphics[height=1in]{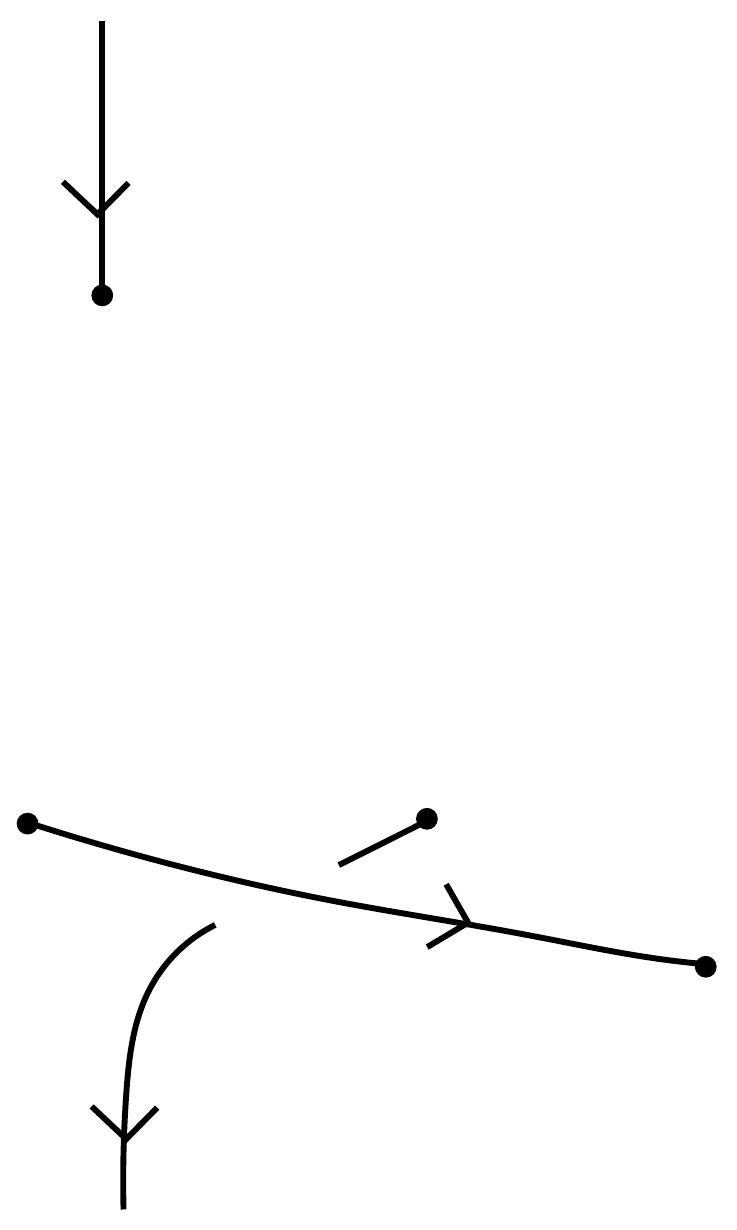}}
  \put(-213,
  35){\fontsize{9}{9}$u$} \put(-213,
  -27){\fontsize{9}{9}$u$} \put(-185,
  35){\fontsize{9}{9}$u$} \put(-186,
  -27){\fontsize{9}{9}$u$} \put(-134,
  35){\fontsize{9}{9}$u$} \put(-135,
  -27){\fontsize{9}{9}$u$} \put(-106,
  35){\fontsize{9}{9}$u$} \put(-108,
  -27){\fontsize{9}{9}$u$} \put(-44,
  35){\fontsize{9}{9}$u$} \put(-44,
  -27){\fontsize{9}{9}$u$} \put(-248, 20){\small{basic}} \put(-254,
  10){\small{$L_o$-move}} \put(-165, 20){\small{braid}} \put(-168,
  10){\small{isotopy}} \put(-78, 20){\small{right}} \put(-86,
  10){\small{$-L_u$-move}}
  \]
  \caption{The second case of the swing moves} \label{swing2}
\end{figure}
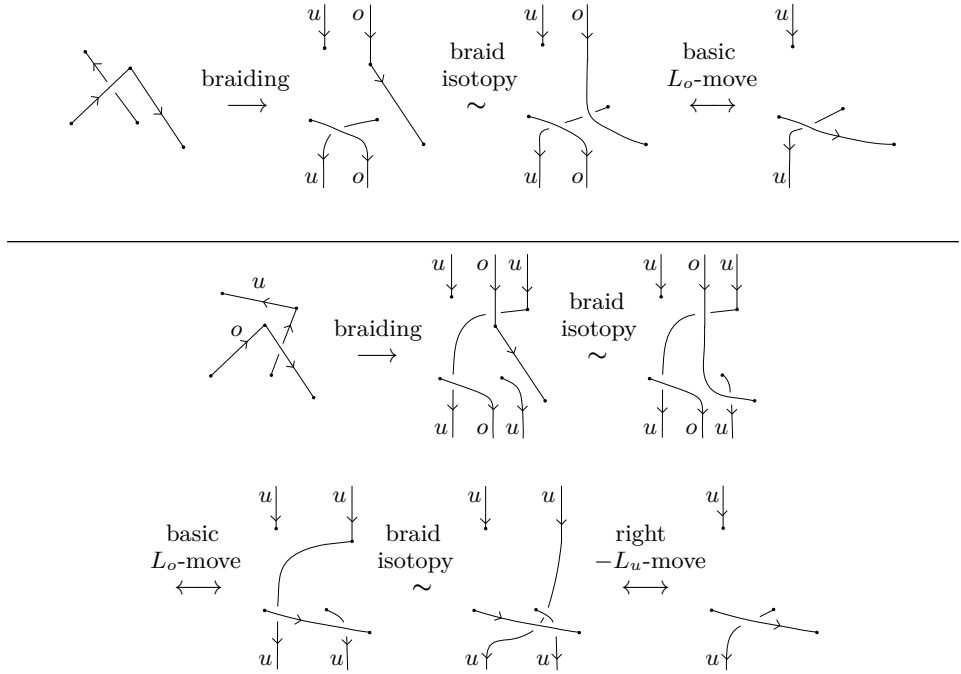


We will now show that ambient isotopy does not affect the braiding process.

\begin{lemma}
  The extended Reidemeister moves yield $TL$-equivalent braids.
  \label{iso moves}
\end{lemma}

\begin{proof}
  In Figure \ref{R1}, we verify one version of the $R1$ move; the braids corresponding to the two diagrams involved in the move are equivalent up to braid isotopy and $L$-moves.

  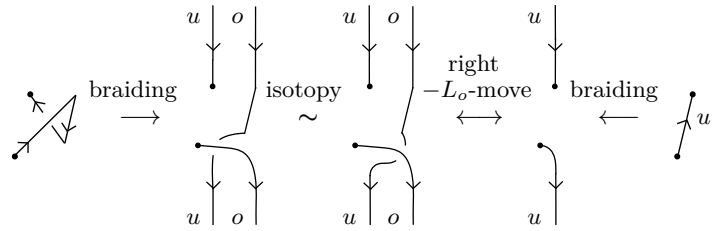
\begin{figure}
    \[\raisebox{-15pt}{\includegraphics[height=.4in]{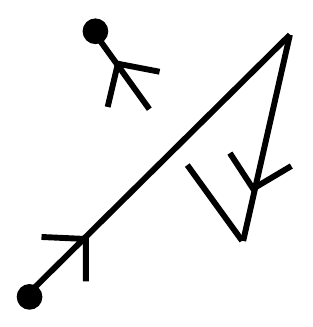}}
    \hspace{.5cm} {\longrightarrow} \hspace{.4cm}
    \raisebox{-40pt}{\includegraphics[height=1.2in]{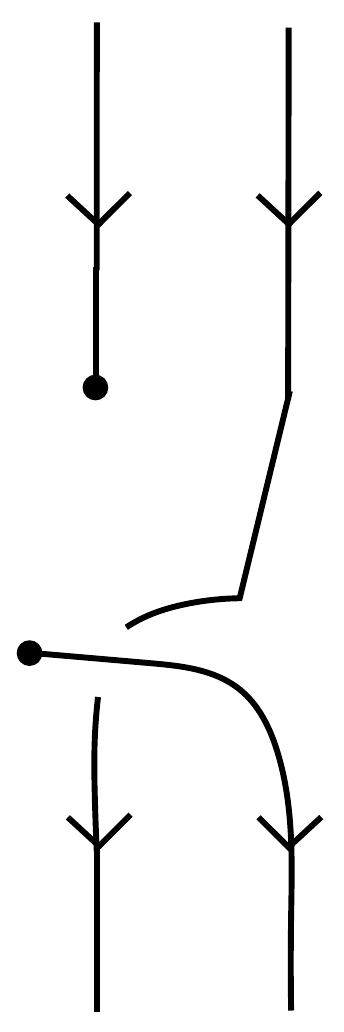}}
    \hspace{.4cm} {\thicksim} \hspace{.4cm}
    \raisebox{-40pt}{\includegraphics[height=1.2in]{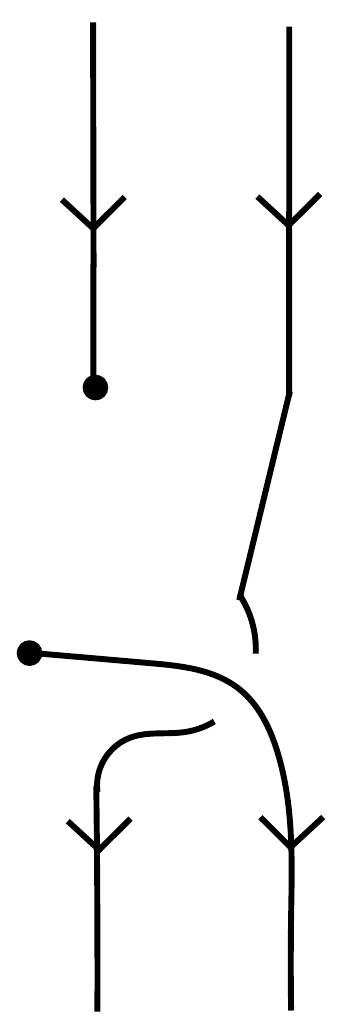}}
    \hspace{.4cm} {\longleftrightarrow} \hspace{.4cm}
    \raisebox{-40pt}{\includegraphics[height=1.2in]{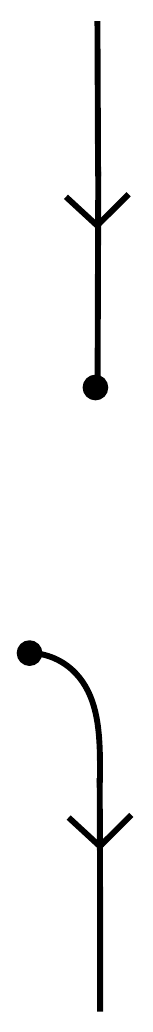}}
    \hspace{.4cm} {\longleftarrow} \hspace{.4cm}
    \raisebox{-15pt}{\includegraphics[height=.4in]{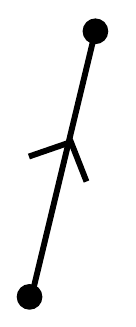}}
    \put(-195,38){\fontsize{9}{9}$u$}
    \put(-195,-38){\fontsize{9}{9}$u$}
    \put(-178,38){\fontsize{9}{9}$o$}
    \put(-178,-38){\fontsize{9}{9}$o$}
    \put(-136,38){\fontsize{9}{9}$u$}
    \put(-136,-38){\fontsize{9}{9}$u$}
    \put(-119,38){\fontsize{9}{9}$o$}
    \put(-119,-38){\fontsize{9}{9}$o$}
    \put(-66,38){\fontsize{9}{9}$u$}
    \put(-66,-38){\fontsize{9}{9}$u$}
    \put(-2,0){\fontsize{9}{9}$u$} 
    \put(-232,10){\small{braiding}}
    \put(-165,10){\small{isotopy}}
    \put(-106,10){\small{$-L_o$-move}} \put(-96,20){\small{right}}
    \put(-50,10){\small{braiding}}
    \]
    \caption{An $R1$ move} \label{R1}
  \end{figure}

  The proof of a version of the $R2$ move with one up-arc is given in Figure \ref{R2}. After braiding the up-arcs in the two diagrams involved in the move, we obtain two braids that differ by an $L_u$-move and braid isotopy. The more interesting case is an $R2$ move with two up-arcs is shown in Figure~\ref{R2D2}. Again, the trivalent braids associated with the two sides of the move are $TL$-equivalent.

  \begin{figure}
    \[\raisebox{-18pt}{\includegraphics[height=.6in]{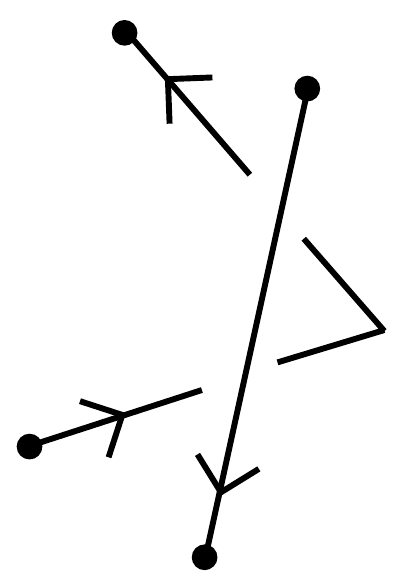}}
    \hspace{.2cm} \longrightarrow \hspace{.2cm}
    \raisebox{-40pt}{\includegraphics[height=1.2in]{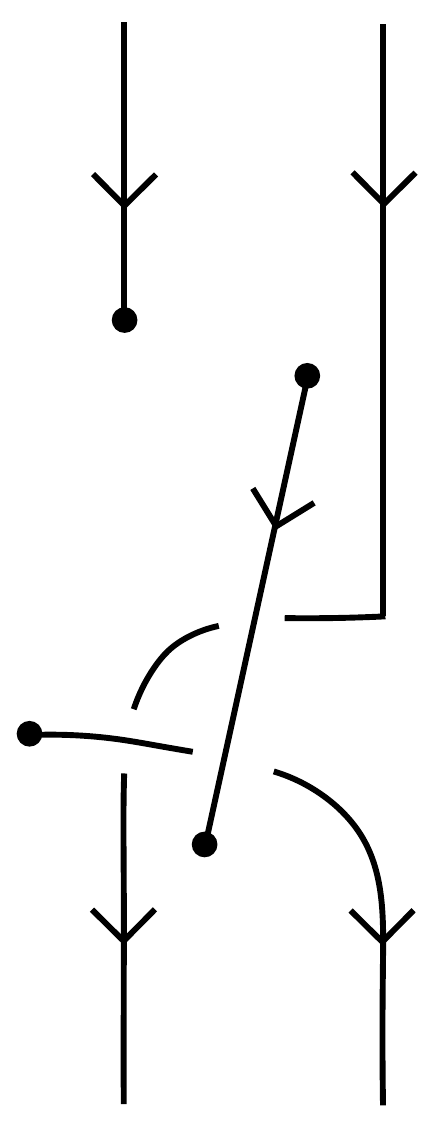}}
    \hspace{.25cm} \thicksim \hspace{.2cm}
    \raisebox{-40pt}{\includegraphics[height=1.2in]{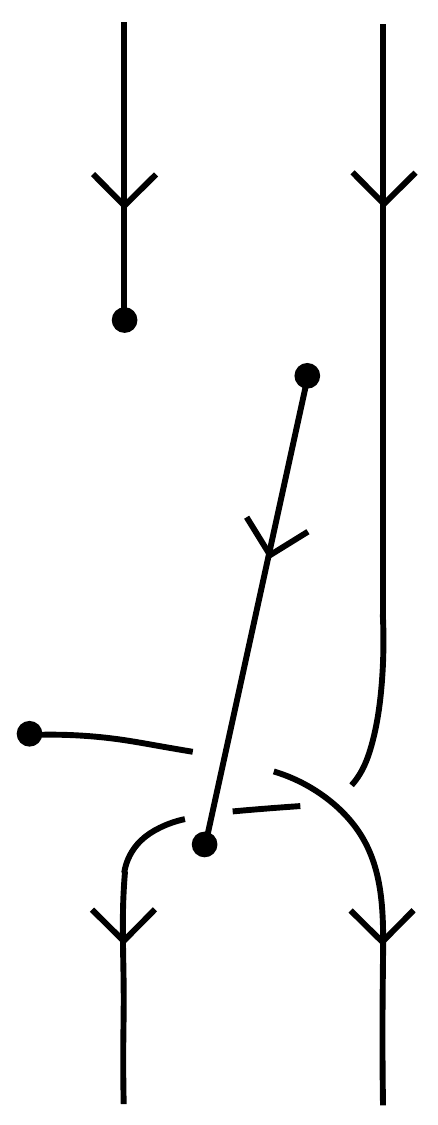}}
    \hspace{.2cm} \longleftrightarrow \hspace{.2cm}
    \raisebox{-40pt}{\includegraphics[height=1.2in]{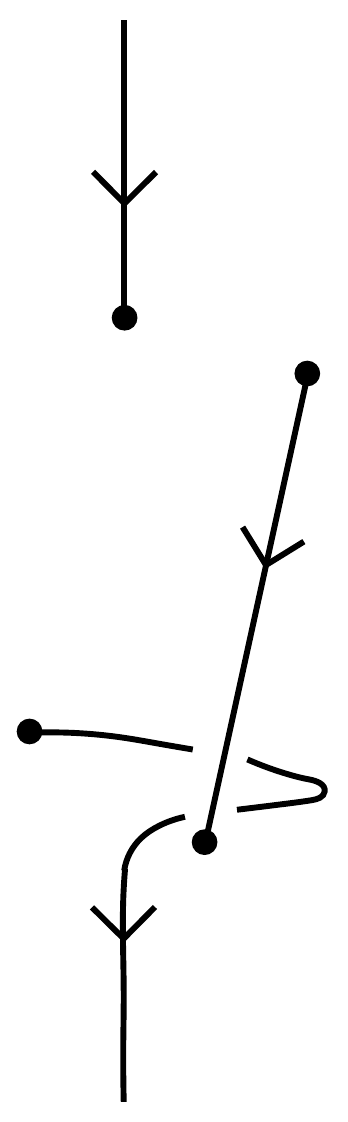}}
    \hspace{.2cm} \sim \hspace{.2cm}
    \raisebox{-40pt}{\includegraphics[height=1.2in]{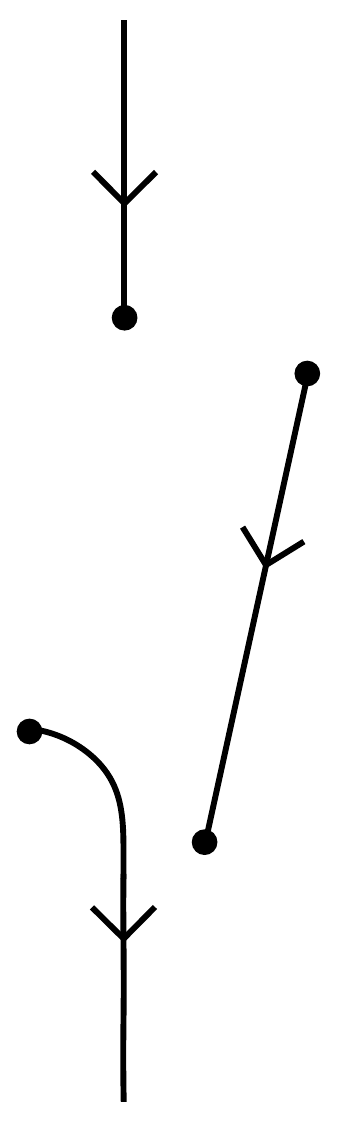}}
    \hspace{.2cm} \longleftarrow \hspace{.2cm}
    \raisebox{-20pt}{\includegraphics[height=.6in]{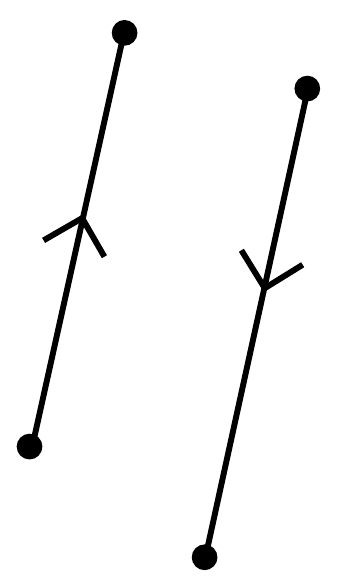}}
    \put(-264,40){\fontsize{9}{9}$u$}
    \put(-264,-38){\fontsize{9}{9}$u$}
    \put(-243,40){\fontsize{9}{9}$u$}
    \put(-243,-38){\fontsize{9}{9}$u$}
    \put(-205,40){\fontsize{9}{9}$u$}
    \put(-205,-38){\fontsize{9}{9}$u$}
    \put(-183,40){\fontsize{9}{9}$u$}
    \put(-183,-38){\fontsize{9}{9}$u$}
    \put(-135,40){\fontsize{9}{9}$u$}
    \put(-135,-38){\fontsize{9}{9}$u$}
    \put(-82,40){\fontsize{9}{9}$u$}
    \put(-82,-38){\fontsize{9}{9}$u$} \put(-300,10){\small{braiding}}
    \put(-230,10){\small{br.
        $R3$}} \put(-163,20){\small{Right}}
    \put(-173,10){\small{$-L_u$-move}} \put(-110,10){\small{br.
        $R2$}} \put(-58,10){\small{braiding}}
    \]
    \caption{An $R2$ move with one up-arc} \label{R2}
  \end{figure}

  \begin{figure} 
    \[
    \raisebox{-15pt}{\includegraphics[height=.51in]{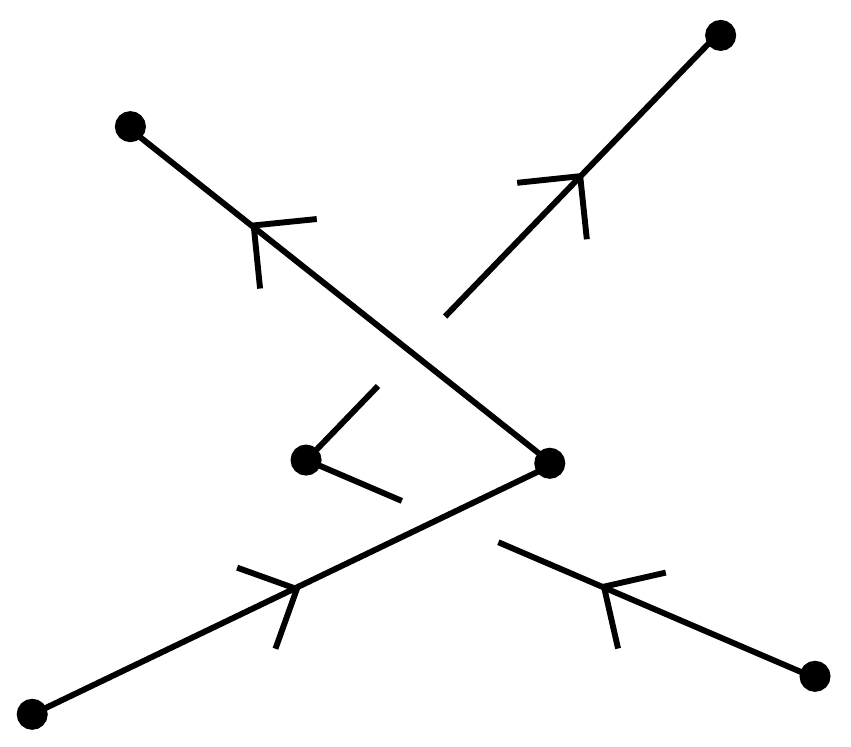}}
    \hspace{.4cm} \longrightarrow \hspace{.4cm}
    \raisebox{-55pt}{\includegraphics[height=1.3in]{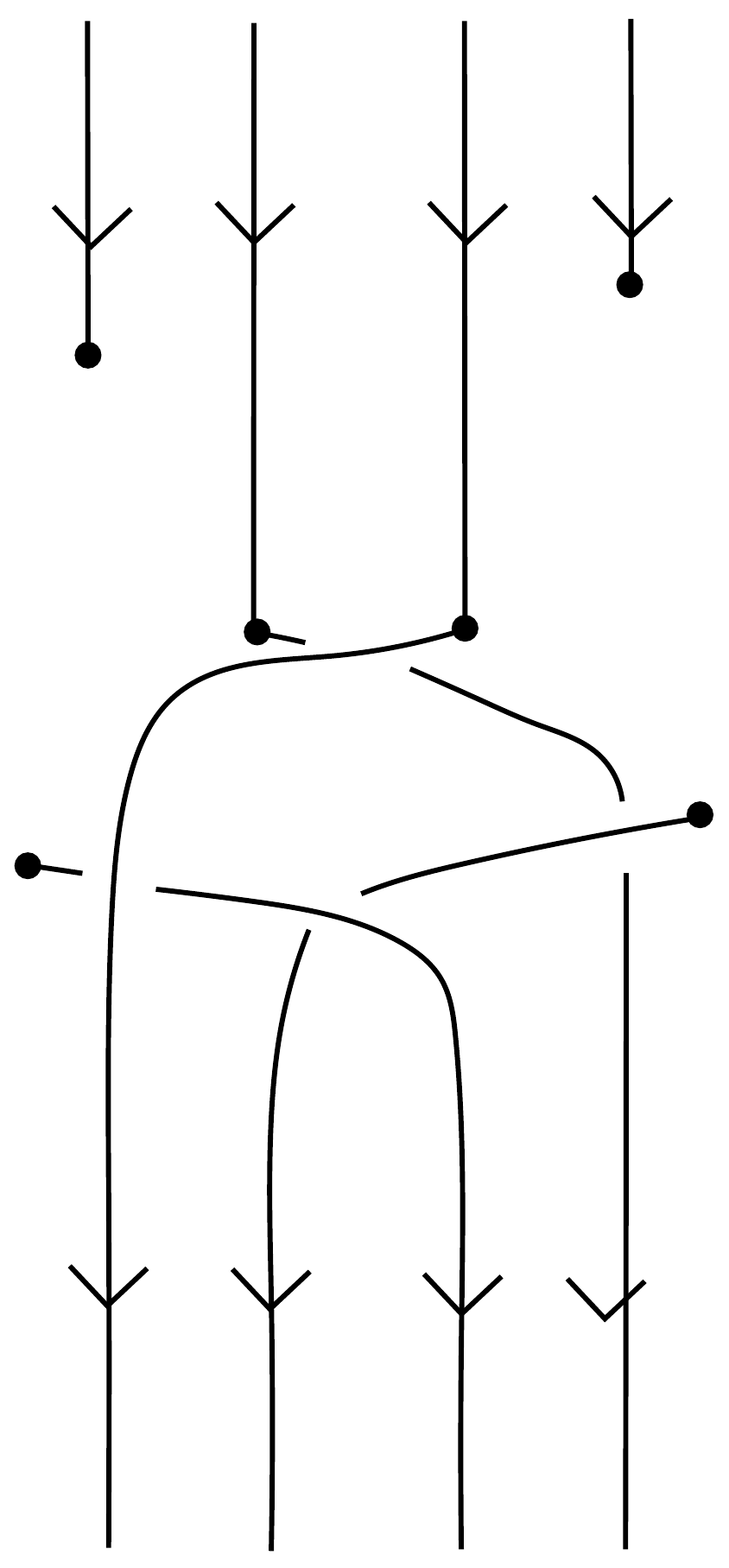}}
    \hspace{.4cm} \thicksim \hspace{.4cm}
    \raisebox{-55pt}{\includegraphics[height=1.3in]{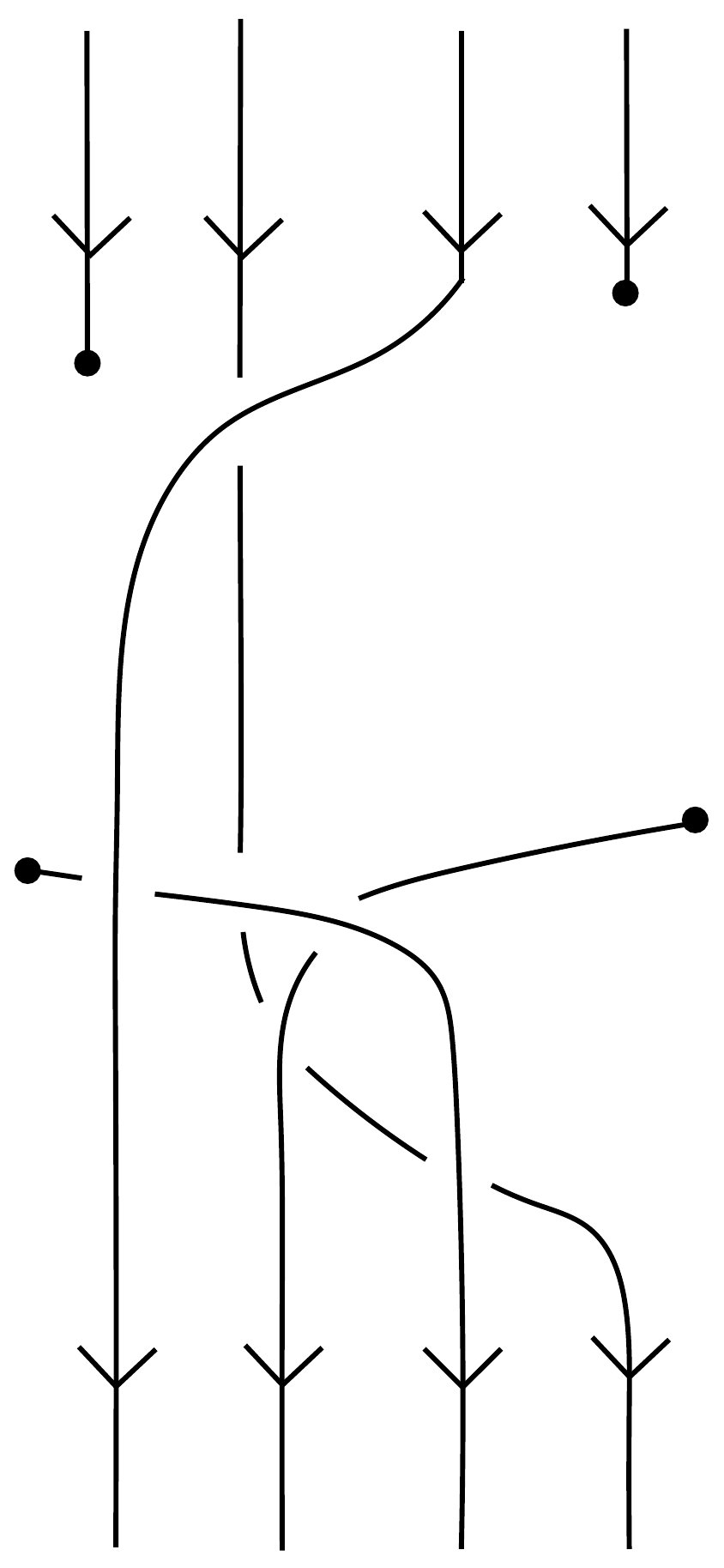}}
    \hspace{.4cm} \longleftrightarrow \hspace{.4cm}
    \raisebox{-55pt}{\includegraphics[height=1.3in]{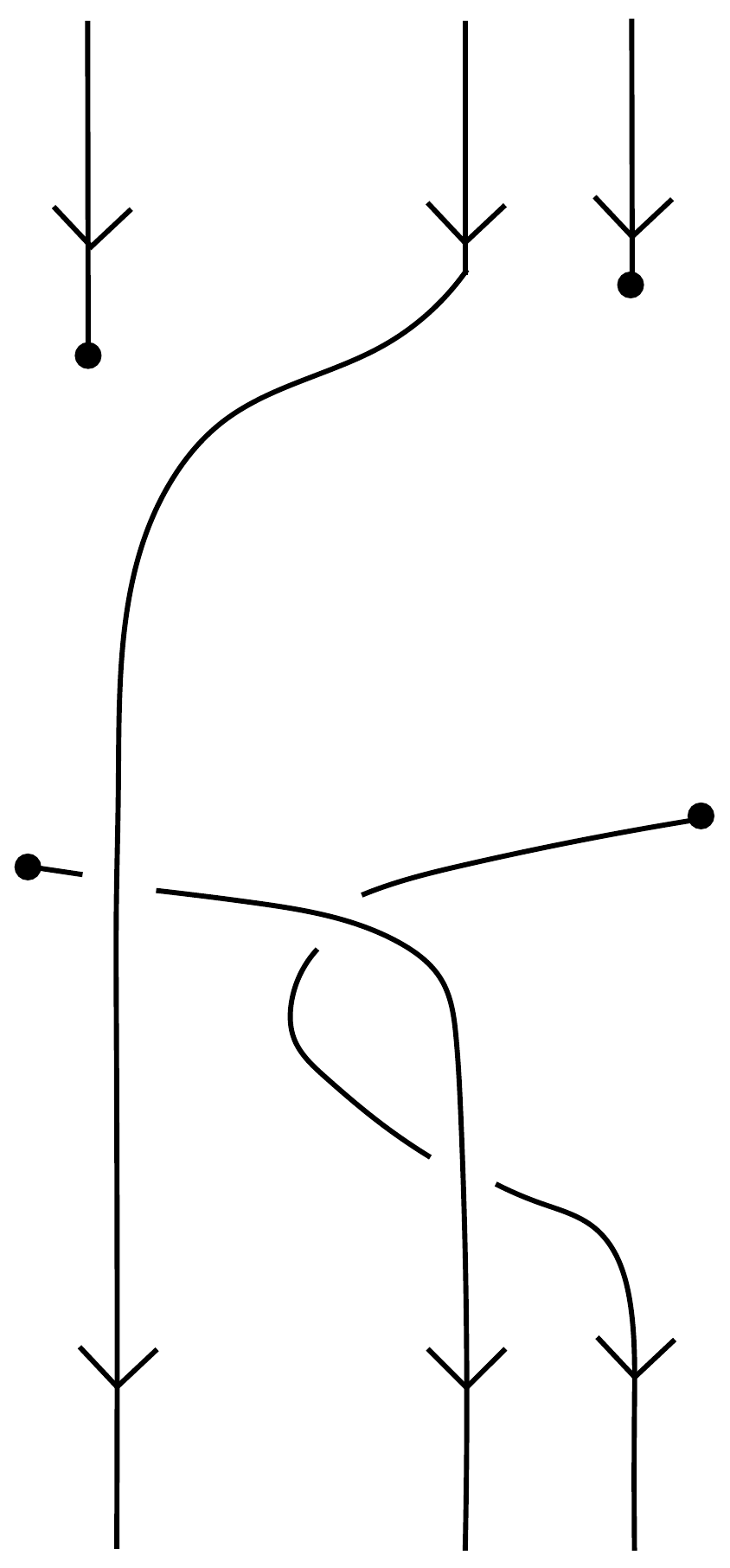}}
    \put(-216,
    31){\fontsize{9}{9}$o$} \put(-214,
    -52){\fontsize{9}{9}$o$} \put(-206,
    31){\fontsize{9}{9}$u$} \put(-204,
    -52){\fontsize{9}{9}$u$} \put(-193,
    31){\fontsize{9}{9}$o$} \put(-193,
    -52){\fontsize{9}{9}$o$} \put(-183,
    31){\fontsize{9}{9}$u$} \put(-183,
    -52){\fontsize{9}{9}$u$} \put(-136,
    31){\fontsize{9}{9}$o$} \put(-134,
    -52){\fontsize{9}{9}$o$} \put(-127,
    31){\fontsize{9}{9}$u$} \put(-124,
    -52){\fontsize{9}{9}$u$} \put(-113,
    31){\fontsize{9}{9}$o$} \put(-113,
    -52){\fontsize{9}{9}$o$} \put(-103,
    31){\fontsize{9}{9}$u$} \put(-103,
    -52){\fontsize{9}{9}$u$} \put(-46,
    31){\fontsize{9}{9}$o$} \put(-46,
    -52){\fontsize{9}{9}$o$} \put(-24,
    31){\fontsize{9}{9}$o$} \put(-24,
    -52){\fontsize{9}{9}$o$} \put(-13,
    31){\fontsize{9}{9}$u$} \put(-13,
    -52){\fontsize{9}{9}$u$} \put(-253, 10){\small{braiding}}
    \put(-168, 10){\small{br.
        $R3$}} \put(-75, 20){\small{Left}} \put(-90,
    10){\small{\fontsize{9}{9}$+L_u$-move}}
    \]
    \vspace{.7cm}
    \[
    \raisebox{-15pt}{\includegraphics[height=.51in]{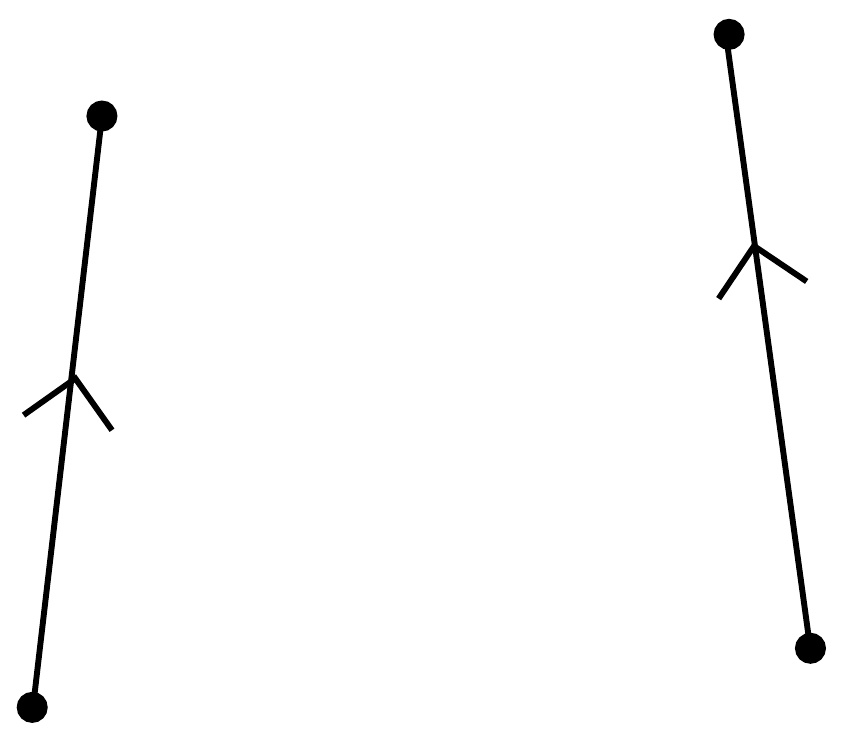}}
    \hspace{.4cm} \longrightarrow \hspace{.4cm}
    \raisebox{-55pt}{\includegraphics[height=1.3in]{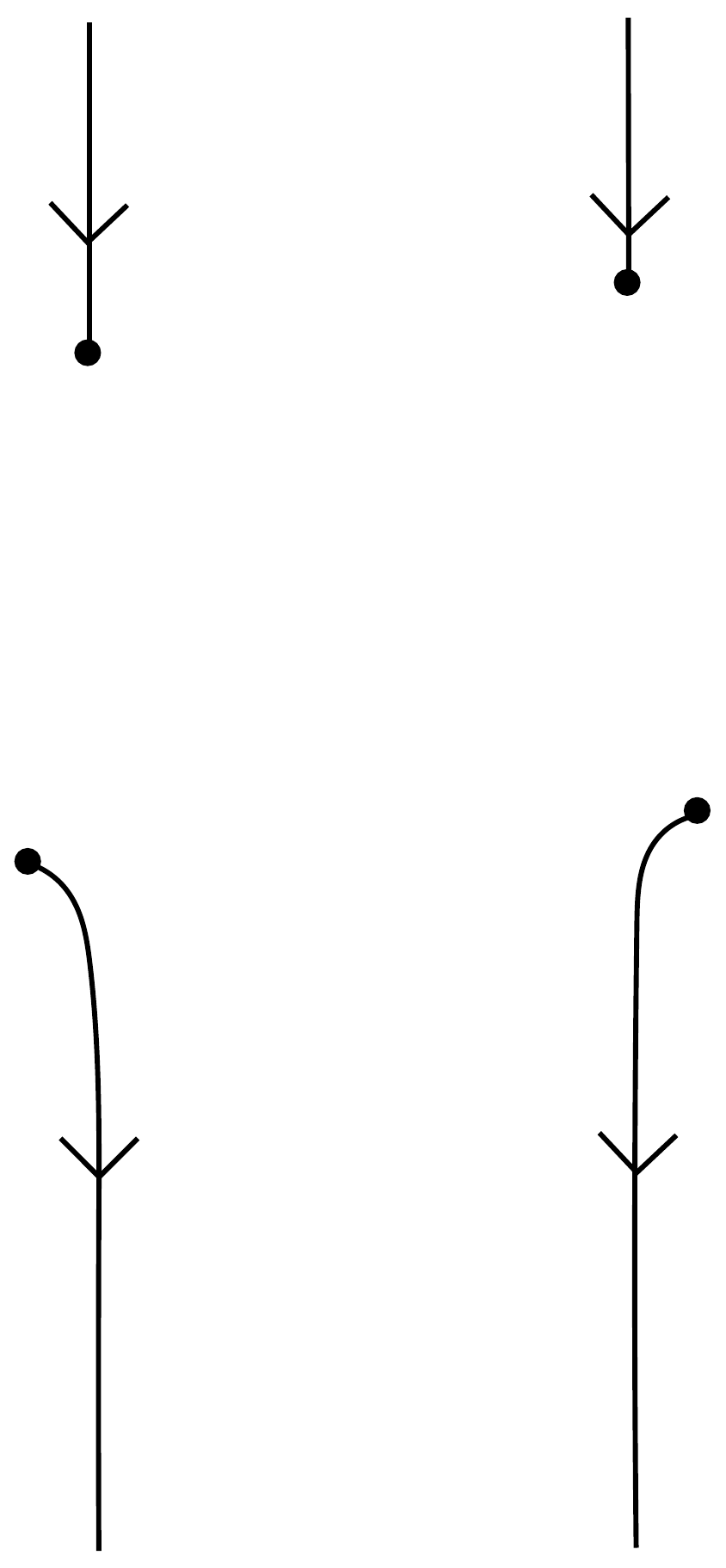}}
    \hspace{.4cm} \thicksim \hspace{.4cm}
    \raisebox{-55pt}{\includegraphics[height=1.3in]{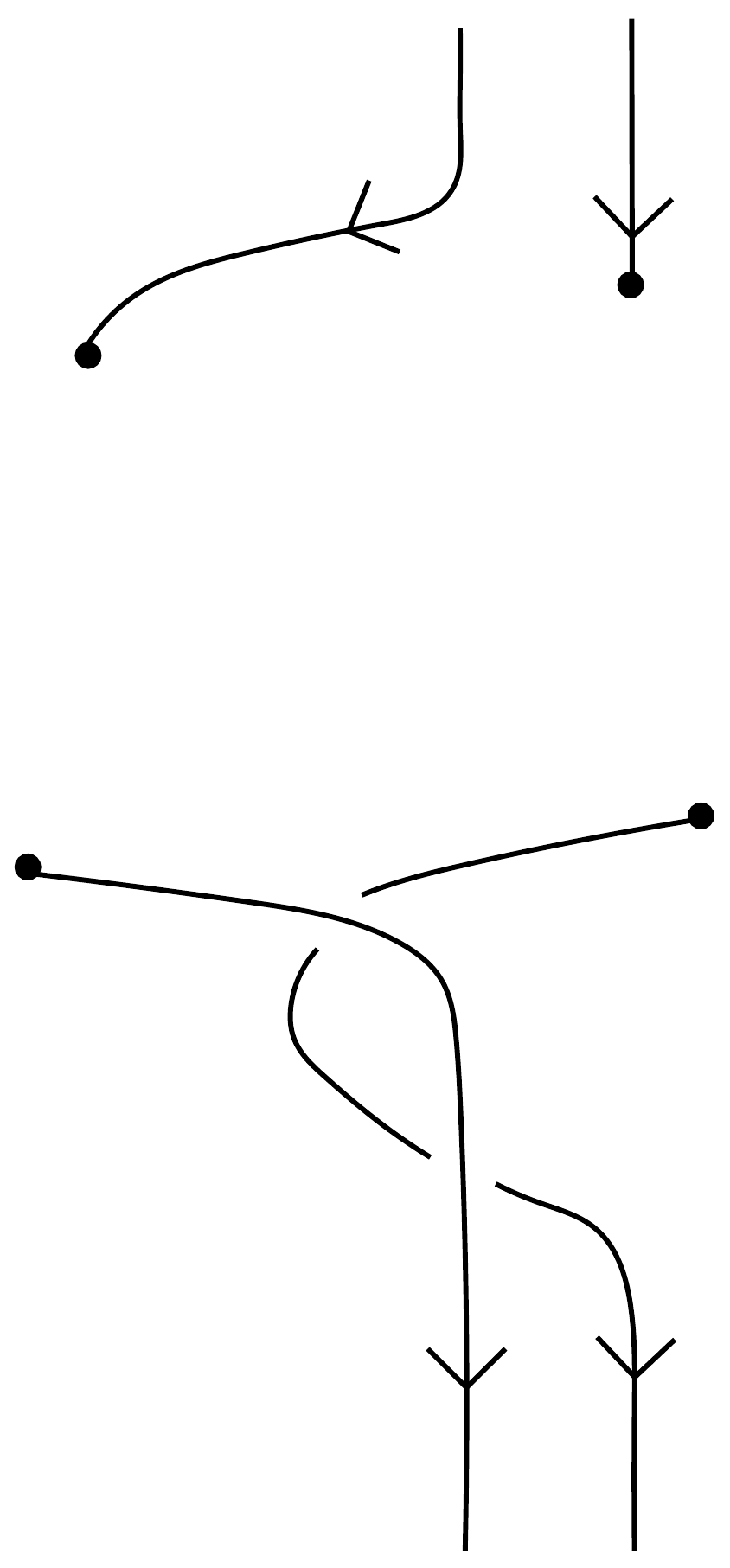}}
    \hspace{.4cm} \longleftrightarrow \hspace{.4cm}
    \raisebox{-55pt}{\includegraphics[height=1.3in]{R2/r22.pdf}}
    \put(-304,
    10){\fontsize{9}{9}$o$} \put(-272,
    13){\fontsize{9}{9}$u$} \put(-216,
    31){\fontsize{9}{9}$o$} \put(-214,
    -52){\fontsize{9}{9}$o$} \put(-183,
    31){\fontsize{9}{9}$u$} \put(-183,
    -52){\fontsize{9}{9}$u$} \put(-113,
    31){\fontsize{9}{9}$o$} \put(-113,
    -52){\fontsize{9}{9}$o$} \put(-104,
    31){\fontsize{9}{9}$u$} \put(-104,
    -52){\fontsize{9}{9}$u$} \put(-46,
    31){\fontsize{9}{9}$o$} \put(-46,
    -52){\fontsize{9}{9}$o$} \put(-24,
    31){\fontsize{9}{9}$o$} \put(-24,
    -52){\fontsize{9}{9}$o$} \put(-13,
    31){\fontsize{9}{9}$u$} \put(-13,
    -52){\fontsize{9}{9}$u$} \put(-253, 10){\small{braiding}}
    \put(-168, 10){\small{br. $R2$}} 
    \put(-79, 20){\small{basic}} 
    \put(-86,10){\small{\fontsize{9}{9}$L_o$-move}}
    \]
    \caption{An $R2$ move with two up-arcs} \label{R2D2}
  \end{figure}

  For the $R3$ move, we rely on the $R2$ moves which have been already verified. In Figure~\ref{R3} we consider a version of the move with one up-arc. The other oriented versions of the $R3$ move can be verified similarly, by applying braid isotopy (namely the $R2$ move) and then a version of the $R3$ move that was already verified.

  \begin{figure}
    \[\raisebox{-15pt}{\includegraphics[width=.75in]{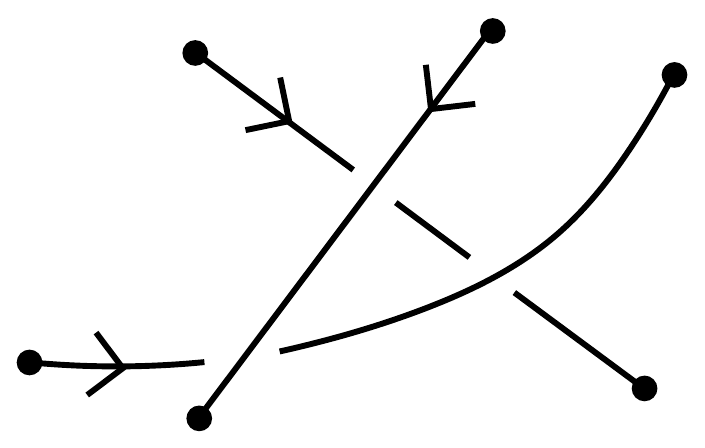}}
    \hspace{.1cm} \longleftrightarrow \hspace{.1cm}
    \raisebox{-15pt}{\includegraphics[width=.75in]{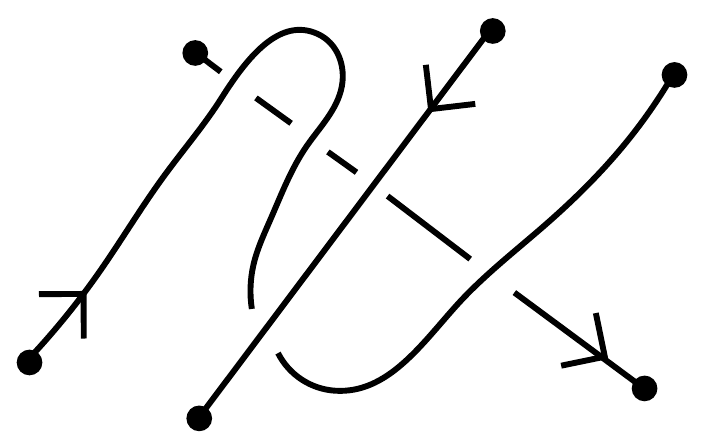}}
    \hspace{.18cm} \sim \hspace{.18cm}
    \raisebox{-15pt}{\includegraphics[width=.75in]{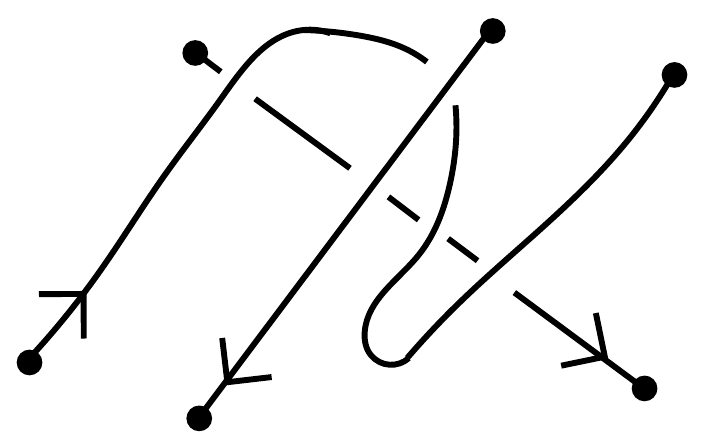}}
    \hspace{.1cm} \longleftrightarrow \hspace{.1cm}
    \raisebox{-15pt}{\includegraphics[width=.75in]{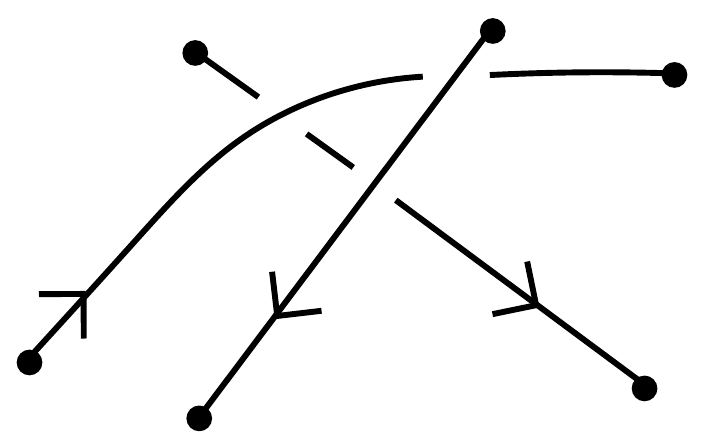}}
    \put(-237,10){\fontsize{9}{9}$R2$} \put(-158,20){\small{braid}}
    \put(-155,10){\small{R3}} \put(-73,10){\fontsize{9}{9}$R2$}
    \]
    \caption{An $R3$ move with one up-arc} \label{R3}
  \end{figure}

  In addition to the traditional Reidemeister moves, there are two extended Reidemeister moves, namely $R4$ and $R5$.

  The two basic versions of the $R4$ move on a regular $Y$-type vertex with an arc sliding under the vertex are shown in Figures~\ref{R4a} and~\ref{R4b}.  Figure~\ref{R4a} shows the basic braid isotopy case where all strands in the diagram are oriented downward. Figure~\ref{R4b} shows the move where the sliding strand is an up-arc. We reduce this move to the version of the $R4$ move in braid form, by employing first an $R2$ move (in a similar way as we did for the considered version of the $R3$ move). The move with an arc sliding over a $Y$-type vertex is treated similarly. No other orientations on $Y$-type vertices need to be considered, since we are assuming our diagram is in general position. The same method can be applied to the various versions of an $R4$ move on a $\lambda$-type vertex.
  
  \begin{figure}
    \[
    \raisebox{-30pt}{\includegraphics[height=0.8in]{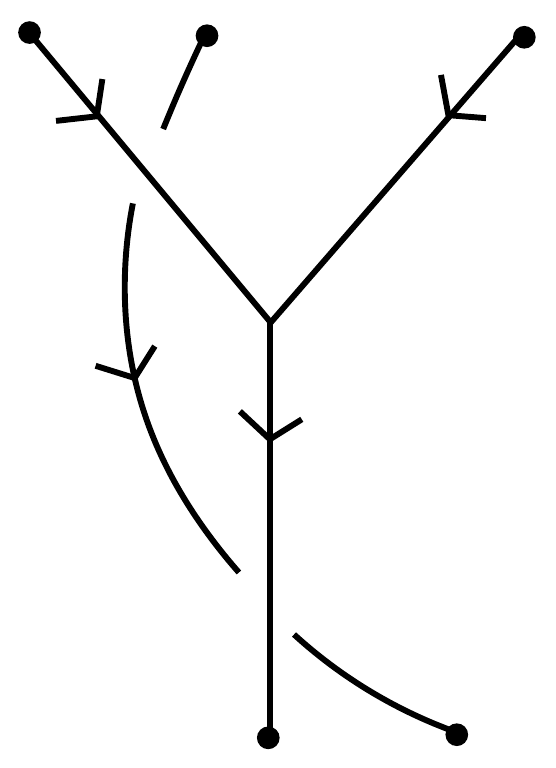}}
    \hspace{.1cm} \longleftrightarrow \hspace{.1cm}
    \raisebox{-30pt}{\includegraphics[height=0.8in]{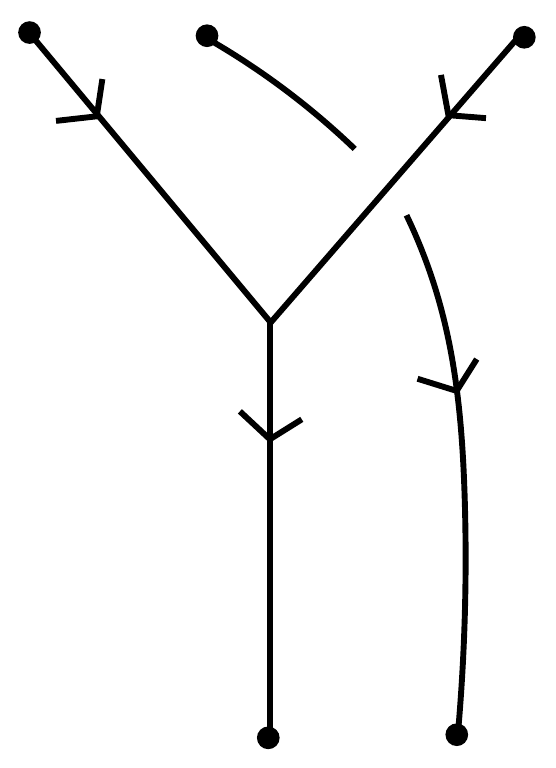}}
    \]
    \caption{An $R4$ move in braid form and on a
      $Y$-type vertex} \label{R4a}
  \end{figure}
  
  \begin{figure}
    \[
    \raisebox{-30pt}{\includegraphics[height=0.8in]{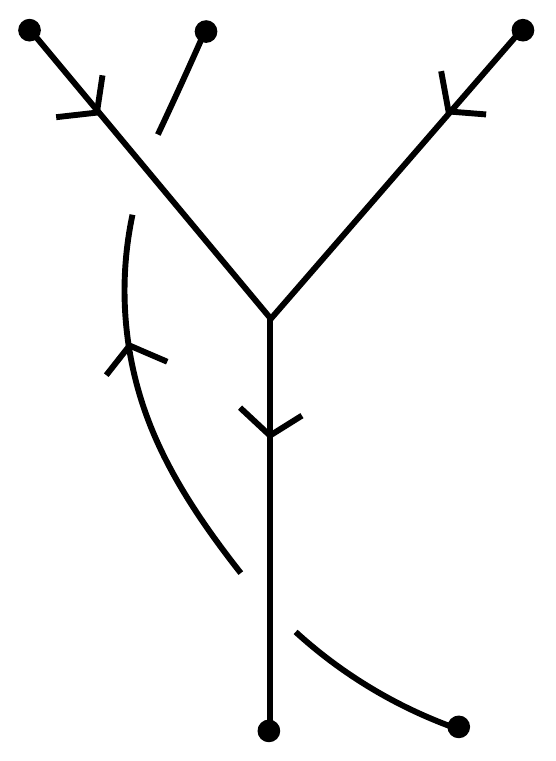}}
    \hspace{.1cm} \longleftrightarrow \hspace{.1cm}
    \raisebox{-30pt}{\includegraphics[height=0.8in]{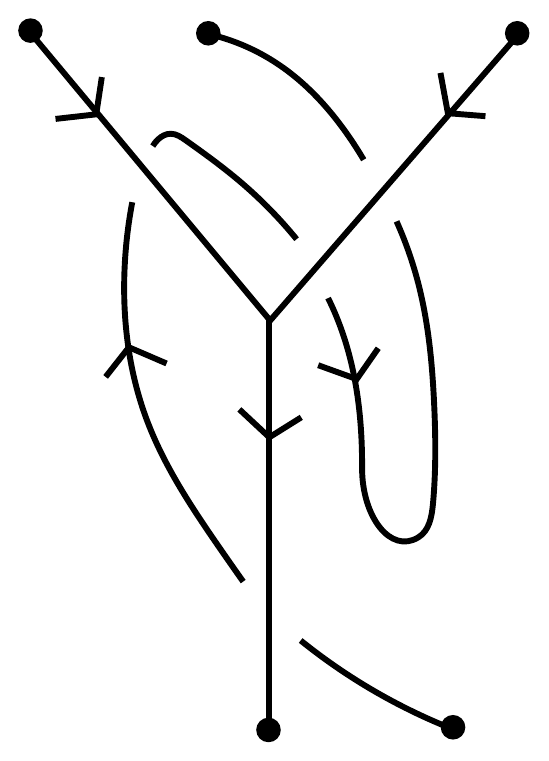}}
    \hspace{.18cm} \sim \hspace{.18cm}
    \raisebox{-30pt}{\includegraphics[height=0.8in]{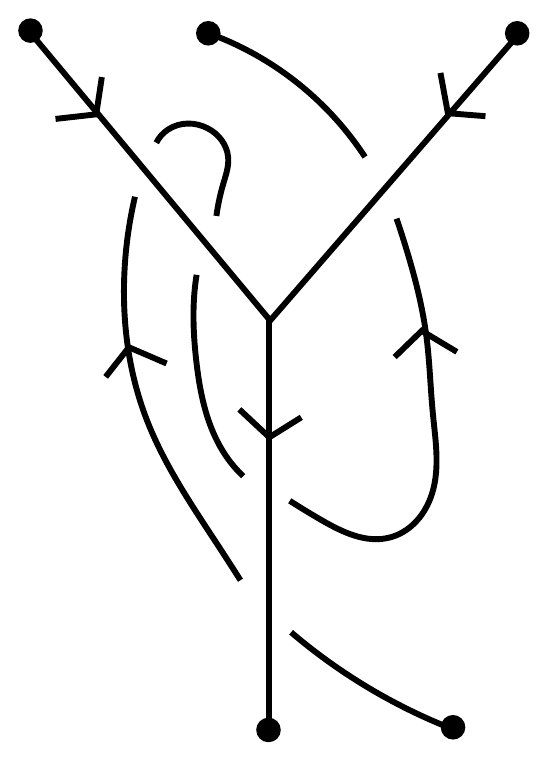}}
    \hspace{.1cm} \longleftrightarrow \hspace{.1cm}
    \raisebox{-30pt}{\includegraphics[height=0.8in]{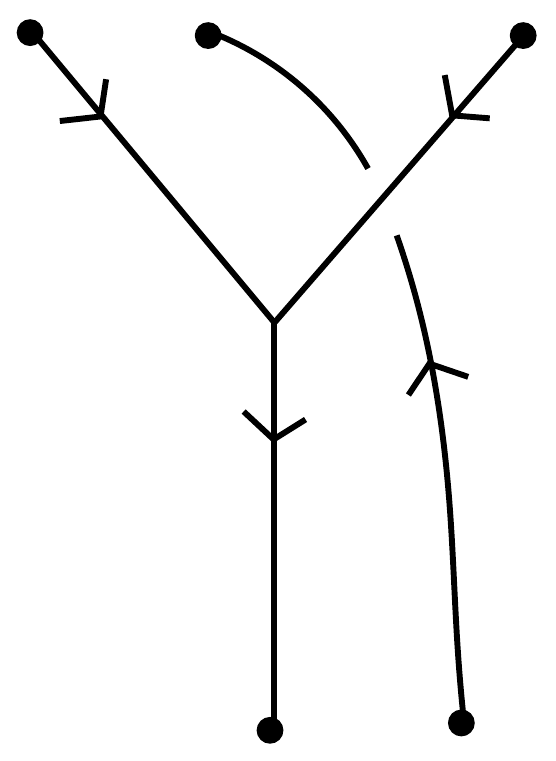}}
    \put(-198,10){\fontsize{9}{9}$R2$} \put(-136,20){\small{braid}}
    \put(-133,10){\small{R4}} \put(-63,10){\fontsize{9}{9}$R2$}
    \]
    \caption{An $R4$ move with an up-arc and on a $Y$-type vertex} \label{R4b}
  \end{figure}
 
 Since we are assuming all vertices in our diagram are in regular position, we see that the only case needed to be verified for the $R5$ move on a $Y$-type vertex is that shown in Figure~\ref{R5}. However, this move is in braid form and, therefore, it is part of trivalent braid isotopy. This argument applies to the $R5$ move on a $\lambda$-type vertex as well.
 
 \begin{figure}
   \[\includegraphics[height=0.8in]{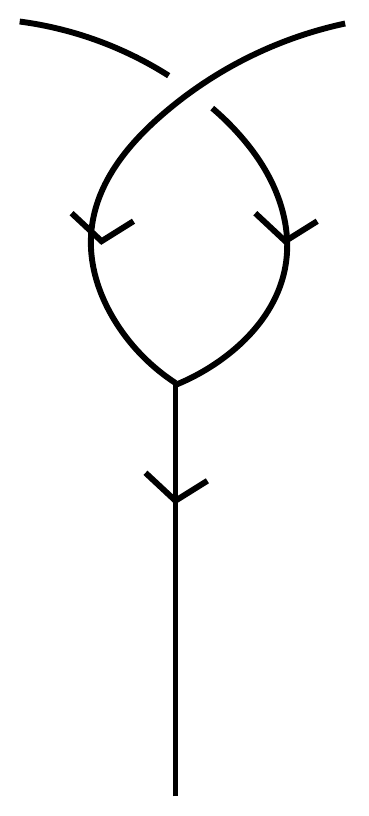} \hspace{.1cm}
   \raisebox{20pt}{$\longleftrightarrow$}
   \hspace{.1cm} \includegraphics[height=0.8in]{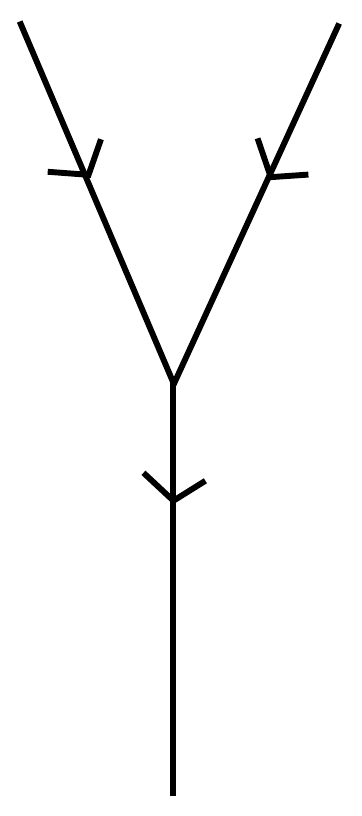}
   \hspace{.1cm}
    \raisebox{20pt}{$\longleftrightarrow$}
    \hspace{.1cm} \includegraphics[height=0.8in]{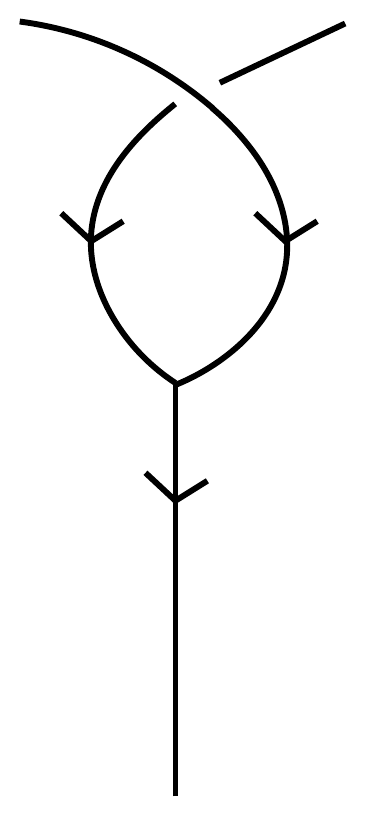}
    \]
    \caption{The $R5$ move in braid form and on a
      $Y$-type vertex} \label{R5}
  \end{figure}

We have shown that our braiding algorithm applied to the two sides of any of the extended Reidemeister moves yield trivalent braids that are $TL$-equivalent. This concludes the proof of the $L$-move Markov-type theorem for trivalent braids and spatial trivalent graphs (Theorem~\ref{LMarkov}).
\end{proof}

\subsection{Algebraic Markov Theorem} \label{sec:ProofAMarkov}

In this section, we state and prove an algebraic Markov-type theorem for trivalent braids and their closures. We use the elementary trivalent braids introduced in Section~\ref{sec:braidsintro} to define a set of algebraic moves that define an equivalence relation on trivalent braids. This algebraic equivalence relation can replace the geometric $TL$-equivalence used in Theorem~\ref{LMarkov}.

We use $b$ to represent an arbitrary trivalent braid in $TB_n^n$.  We can also embed $b$ into $TB_{n+1}^{n+1}$ by adding an extra strand to the right of $b$; we use the same symbol, $b$, to denote the resulting braid with an extra strand.  Using this operation, we can think of $TB_n^n$ being embedded into $TB_{n+1}^{n+1}$, so we define $TB:= \cup_{n=1}^\infty TB_n^n$.

\begin{theorem}[\textbf{Algebraic Markov-type theorem for STGs}] \label{AMarkov}
 Two well-oriented spatial trivalent graphs are isotopic if and only if any two corresponding trivalent braids differ by a finite sequence of braid relations in $TB$ and the following moves:
  \begin{enumerate}

  \item[(i)] Elementary conjugation (conjugation by $\sigma_i$ and $\sigma_i^{-1}$; see Figure~\ref{fig:conj2}):\\
    $\sigma_i b \sim b\sigma_i \,\,\, \text{and} \,\,\, \sigma_i^{-1} b \sim b\sigma_i^{-1}, \,\,\, \text{where}\,\, b, \sigma_i^{\pm 1} \in TB^n_n, \, \, 1 \leq i \leq n-1$\\

  \item[(ii)] Right stabilization (see Figure~\ref{fig:stab}):\\
    $ b c \sim b \sigma_n^{\pm 1} c, \, \, \text{where} \,\, b, c \in
    TB^n_n \, \text{and}\, \, b \sigma_n^{\pm 1} c \in TB_{n+1}^{n+1}$
  \end{enumerate}
\end{theorem}

\begin{figure}[ht]
  \[
  \raisebox{-35pt}{\includegraphics[height=1in]{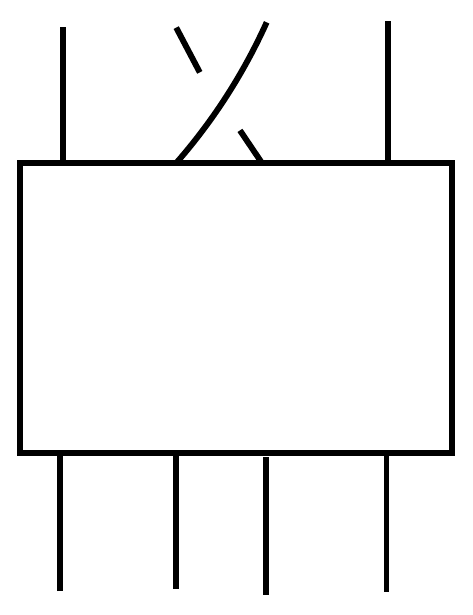}}
  \hspace{0.5cm} \sim \hspace{0.5cm}
  \raisebox{-35pt}{\includegraphics[height=1 in]{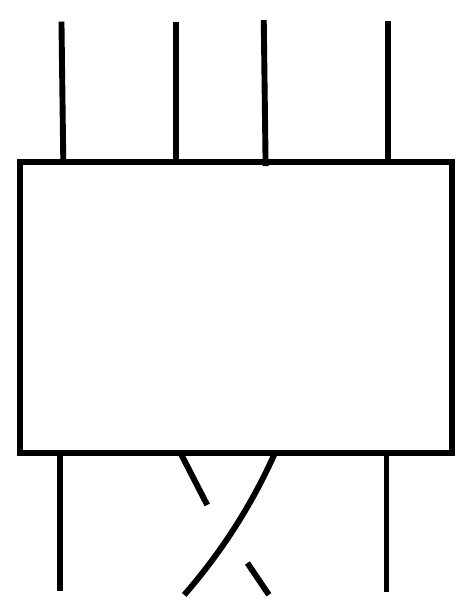}} \put(-130,
  0){\fontsize{9}{9}$b$} \put(-30,
  0){\fontsize{9}{9}$b$} \put(-135,
  38){\fontsize{7}{7}$i$} \put(-125,
  38){\fontsize{7}{7}$i+1$} \put(-40,
  38){\fontsize{7}{7}$i$} \put(-30, 38){\fontsize{7}{7}$i+1$}
  \]
  \caption{Conjugation by $\sigma_i$}\label{fig:conj2}
\end{figure}

\begin{figure}[ht]
  \[
  \raisebox{-35pt}{\includegraphics[height=1.1in]{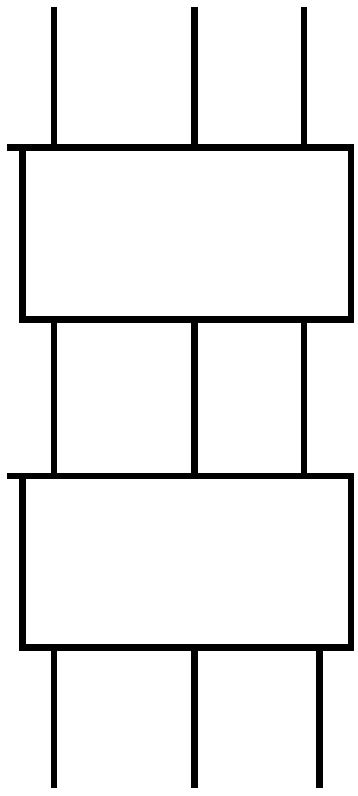}}
  \hspace{.2in} {\sim} \hspace{.2in}
  \raisebox{-35pt}{\includegraphics[height=1.1in]{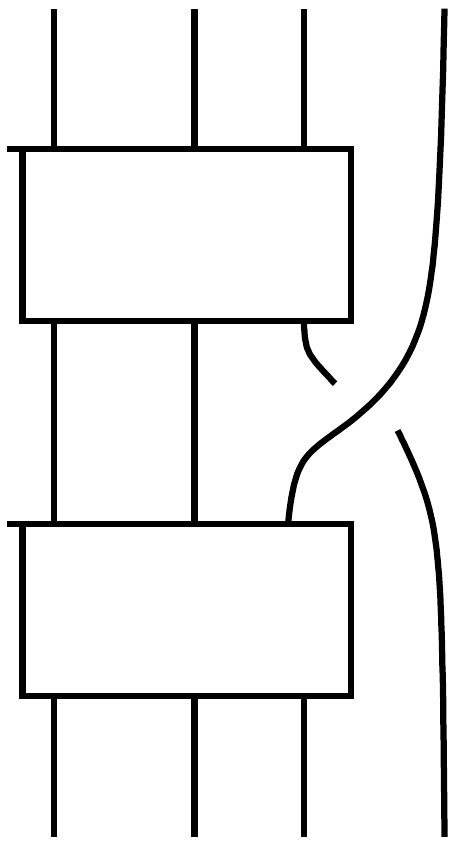}}
  \put(-97,20){\fontsize{9}{9}$b$}
  \put(-30,20){\fontsize{9}{9}$b$} \put(-97,
  -15){\fontsize{9}{9}$c$} \put(-30, -15){\fontsize{9}{9}$c$}
  \]
  \caption{Right stabilization by $\sigma_n$}\label{fig:stab}
\end{figure}

\begin{proof}
  We note first that braid isotopy is part of both $TL$-equivalence and the algebraic equivalence of Theorem~\ref{AMarkov}.

  It is easy to see that the closures of two trivalent braids that are related by trivalent braid isotopy and a finite sequence of right stabilization and elementary conjugation are isotopic STG diagrams.

  For the converse, let $b_1$ and $b_2$ be trivalent braids that yield isotopic STG diagrams upon the closure operation. By Theorem~\ref{LMarkov}, we know that $b_1$ and $b_2$ are $TL$-equivalent. Therefore, it suffices to show that the right $L$-moves for trivalent braids follow from the algebraic moves of  Theorem~\ref{AMarkov}. In Figure~\ref{stab-proofAMarkov}, we show that the right $L_u$-move can be obtained from right stabilization, elementary conjugation, and trivalent braid isotopy. 
  \begin{figure}[ht]
  \[
  \raisebox{-35pt}{\includegraphics[height = 1in]{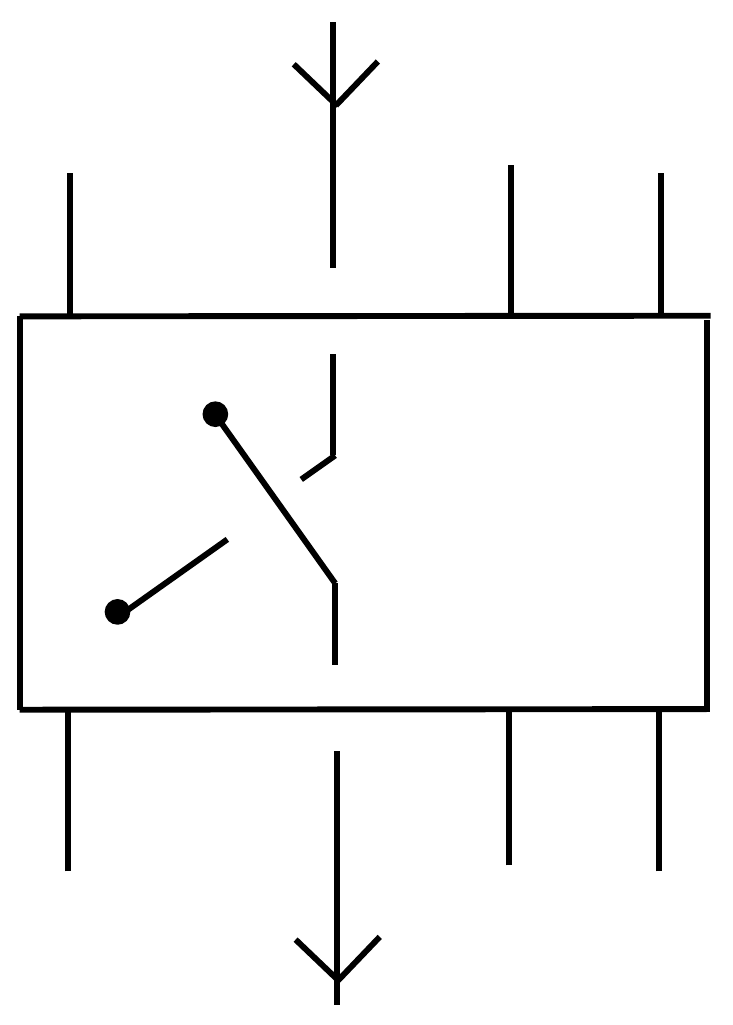}}
  \hspace{.3cm} \thicksim \hspace{.5cm}
  \raisebox{-35pt}{\includegraphics[height = 1in]{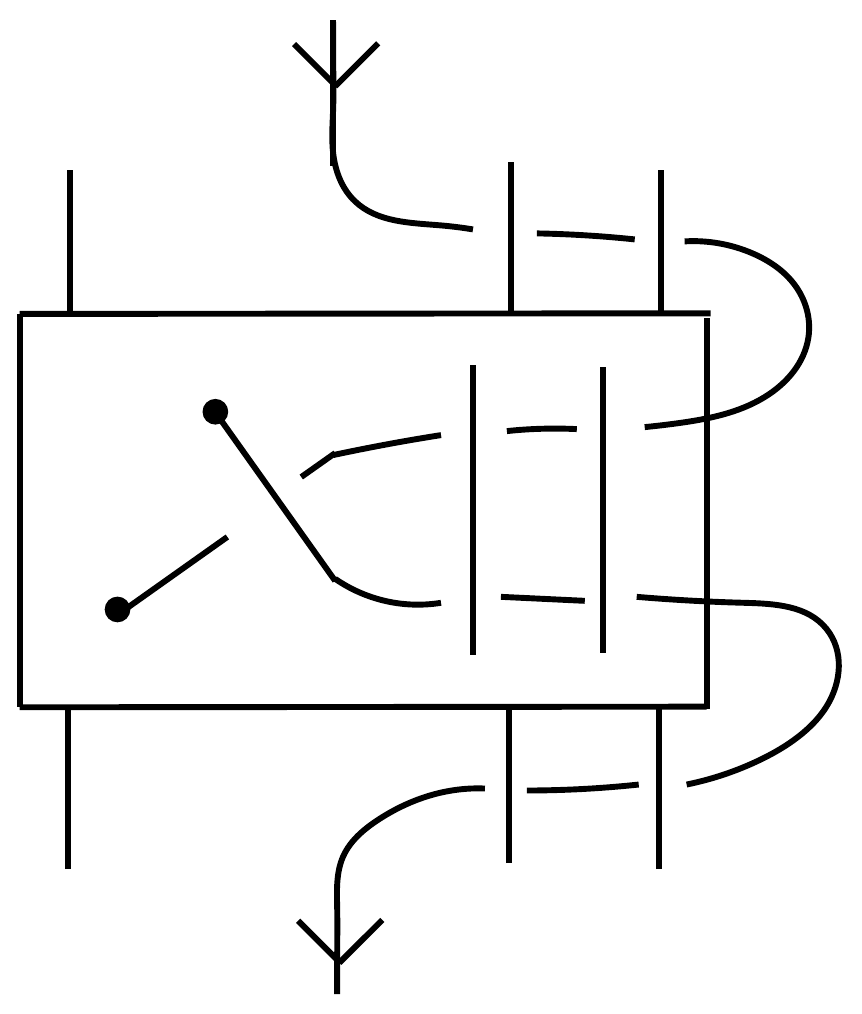}}
  \hspace{.3cm} \longleftrightarrow \hspace{.2cm}
  \raisebox{-35pt}{\includegraphics[height = 1in]{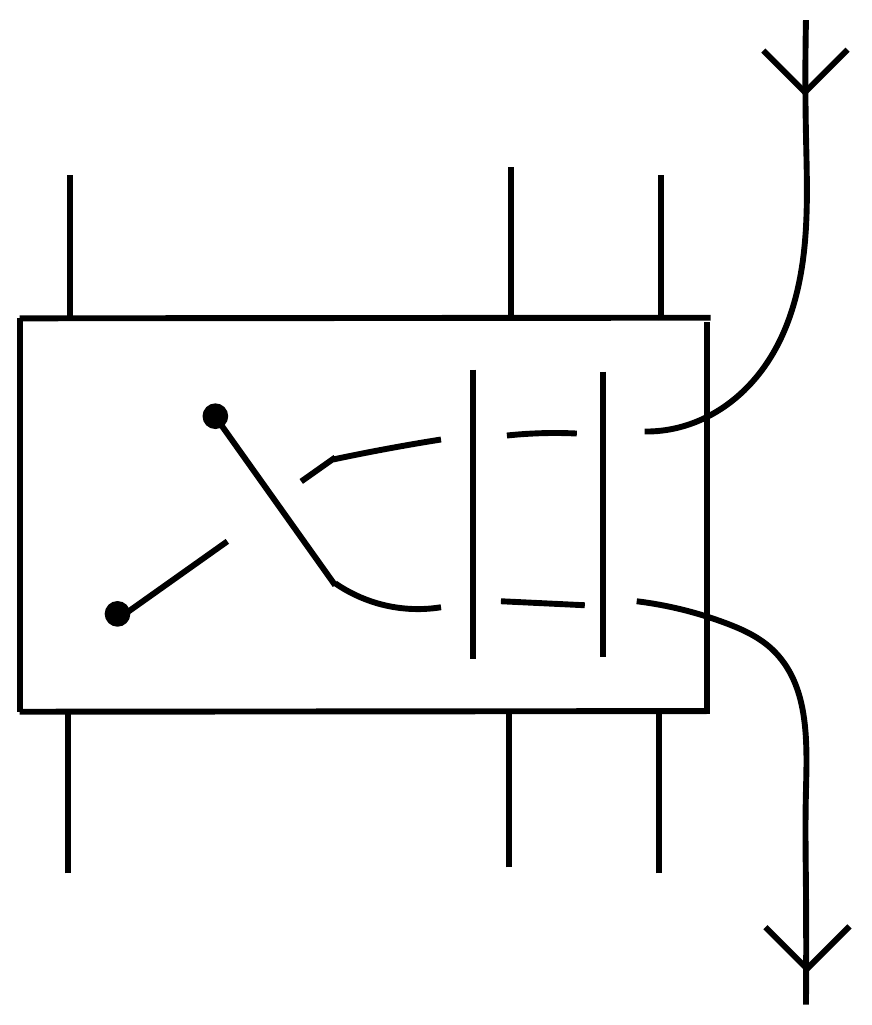}}
  \hspace{.2cm} \thicksim \hspace{.2cm} \put(-210,20){\small{braid}}
  \put(-212,10){\small{isotopy}} \put(-107,10){\small{conj.}}
  \put(-13,10){\small{R3}}
  \]
  \[
  \raisebox{-35pt}{\includegraphics[height = 1in]{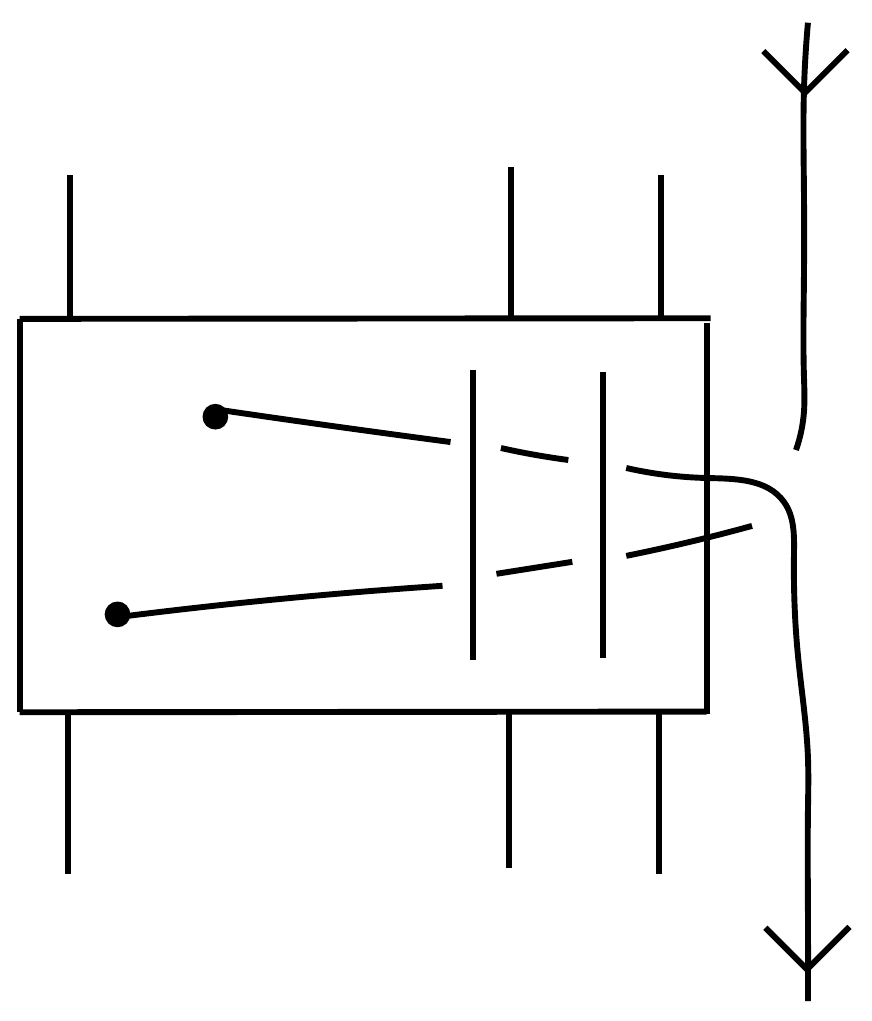}}
  \hspace{.2cm} \longleftrightarrow \hspace{.2cm}
  \raisebox{-30pt}{\includegraphics[height = 0.8in]{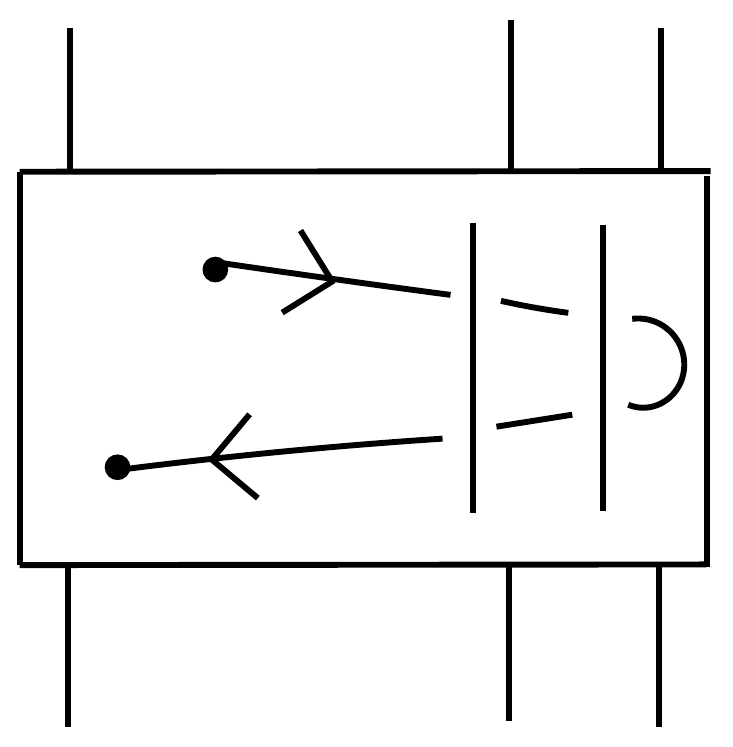}}
  \hspace{.2cm} \thicksim \hspace{.2cm}
  \raisebox{-30pt}{\includegraphics[height = 0.8in]{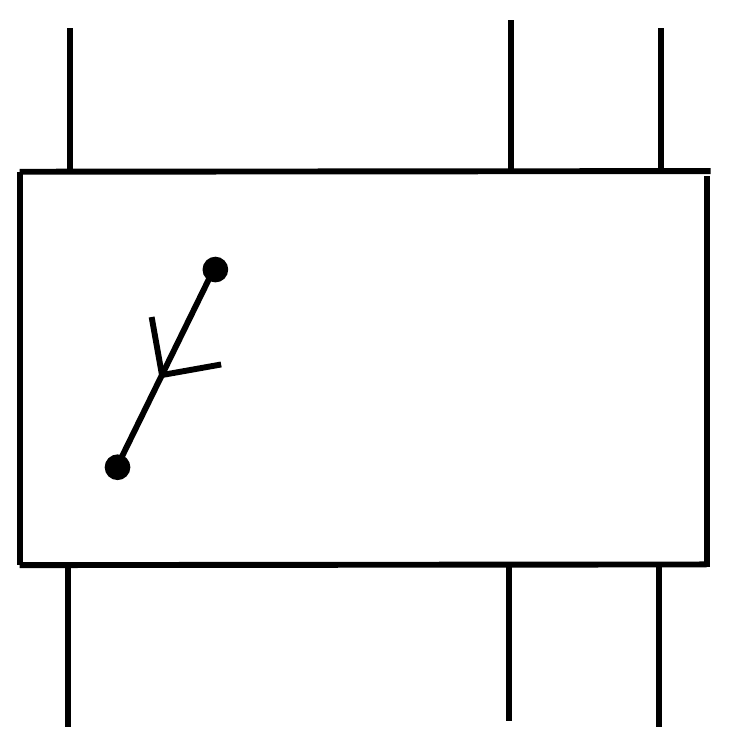}}
  \put(-165,10){\small{stab.}}  \put(-75,10){\small{R2}}
  \]
  \caption{Right $L_u$-move follows from the algebraic
    moves}\label{stab-proofAMarkov}
\end{figure}

The proof for the right $L_o$-move follows similarly. This completes the proof of Theorem~\ref{AMarkov}.
\end{proof}

\textbf{Final comments.} We provided two Markov-type theorems for trivalent braids and spatial trivalent graphs. As in the case of classical braids, our algebraic Markov-type theorem requires two moves, besides braid isotopy: conjugation by the elementary braids $\sigma_i$ or $\sigma_i^{-1}$ and right stabilization. In our case, right stabilization is not done in the bottom of a braid but between two trivalent braids. For classical braids, due to conjugation, bottom right stabilization is equivalent to right stabilization between braids. This is not the case for our approach for trivalent braids, since we do not have conjugation by the elementary trivalent braids $y_i$ and $\lambda_i$.

We defined the $TL$-equivalence for trivalent braids as an extension of the $L$-equivalence for classical braids. $TL$-equivalence encompasses trivalent braid isotopy and right $L$-moves. We used the $TL$-equivalence to prove a one-move Markov-type theorem for trivalent braids. Then, we used this $L$-move Markov-type theorem to prove a more algebraic Markov-type theorem for trivalent braids.

\vspace{0.2cm}

\textbf{Acknowledgements.} We gratefully acknowledge support from the NSF Grant DMS--1460151 through the \textit{Research Experience for Undergraduates} (REU) Program at California State University, Fresno. The first author was also partially supported by Simons Foundation collaboration grant $\#$ 355640.

\end{document}